\documentclass[12pt,a4paper,graybox,envcountchap,sectrefs]{SNmono}
\usepackage{SNmonoArxivPatch}

\usepackage{type1cm}
\usepackage{makeidx}
\usepackage{graphicx}
\usepackage{multicol}
\usepackage[bottom]{footmisc}

\usepackage{amsbsy,amsfonts,amssymb,amscd,amsthm,mathtools,bm}
\usepackage{mathrsfs,float,color}
\usepackage[mathcal]{euscript}

\usepackage{pxfonts}

\usepackage{enumitem}
\usepackage{epstopdf}
\usepackage[normalem]{ulem}
\usepackage{xcolor}

\usepackage[
  colorlinks=true,
  linkcolor=blue,
  citecolor=green,
  urlcolor=blue,
  plainpages=false,
  pdfpagelabels
]{hyperref}
\hypersetup{
  linktocpage,
  pdftitle={Mostly Contracting Random Maps},
  pdfauthor={Pablo G. Barrientos and Dominique Malicet}
}

\usepackage{tikz}
\usetikzlibrary{arrows.meta,positioning}
\usepackage{centernot}

\theoremstyle{plain}
\newtheorem{mainthm}{Theorem}
\newtheorem{mainprop}{Proposition}
\newtheorem{maincor}[mainprop]{Corollary}

\newtheorem{thm}{Theorem}[section]
\newtheorem{cor}[thm]{Corollary}
\newtheorem{lem}[thm]{Lemma}
\newtheorem{claim}[thm]{Claim}
\newtheorem{prop}[thm]{Proposition}

\newtheorem{defi}[thm]{Definition}

\theoremstyle{definition}
\newtheorem{exap}[thm]{Example}
\newtheorem{rem}[thm]{Remark}

\renewcommand{\theenumii}{\alph{enumii}}

\DeclarePairedDelimiter{\tnorm}{\vert\mkern-2mu\vert\mkern-2mu\vert}
{\vert\mkern-2mu\vert\mkern-2mu\vert}

\newcommand{\eqdef}{\stackrel{\scriptscriptstyle\rm def}{=}}

\DeclareMathOperator{\diam}{diam}

\let\D\relax
\DeclareMathOperator{\D}{{\mathrm{d}}}
\DeclareMathOperator{\Dp}{{\mathrm{d}}^+}
\DeclareMathOperator{\Dpm}{{\mathrm{d}}}
\DeclareMathOperator{\EE}{\mathrm{Lip}_{\mathbb{P}}}

\newcommand{\REM}[1]{}
\newcommand{\del}[1]{}

\newcommand{\CE}{C^1_{\smash{\mathbb{P}}}}
\newcommand{\DD}{\D_{C^{\smash[b]{1}}}}
\newcommand{\DE}[1]{\mathrm{Diff}^{#1}_{\mathbb{P}}}
\newcommand{\Dbeta}{\D^{\smash{ n,\beta}}_{\smash{\mathrm{L}}}}
\newcommand{\LP}{\mathrm{L}_{\mathbb{P}}(m)}
\newcommand{\LPC}{\mathrm{L}_{\smash{\mathbb{P}}}^{\mkern-5mu \raisebox{1.2pt}{\tiny \rm{cpt}}}(m)}

\newcommand{\R}{\mathbb{R}}

\newcommand{\N}{\mathbb{N}}

\let\E\relax
\newcommand{\E}{\mathbb{E}}

\makeatletter
\newcommand{\exterior}[1]{\mathop{\mathpalette\exterior@{#1}}}
\newcommand{\exterior@}[2]{%
  \raisebox{\depth}{%
    \fontsize{\sf@size}{0}%
    \m@th
    $\ifx#1\displaystyle\textstyle\else#1\fi\bigwedge$%
  }%
  ^{\mspace{-2mu}#2}%
  \kern-\scriptspace
}
\makeatother

\makeatletter
\AtBeginDocument{%
  \def\@biblabel#1{[#1]}%
  \def\@lbibitem[#1]#2{%
    \item[{[#1]}\hfill]%
    \if@filesw
      {\let\protect\noexpand
       \immediate\write\@auxout{\string\bibcite{#2}{#1}}}%
    \fi
    \ignorespaces
  }%
}
\makeatother

\makeindex

\begin{document}

~\vspace{-1cm}
\author{Pablo G. Barrientos \and Dominique Malicet}
\title{Mostly Contracting Random Maps}
\institute{%
Pablo G. Barrientos \at Instituto de Matem\'atica e Estat\'istica, UFF, Niter\'oi, Brazil \\
\email{pgbarrientos@id.uff.br}
\and
Dominique Malicet \at LAMA, Universit\'e Gustave Eiffel, Batiment Coppernic, 5 boulevard Descartes \\
\email{mdominique@crans.org}}

\pagenumbering{arabic}
\pagestyle{plain}

\maketitlebis 

\begin{abstractbis}
In this monograph, we study the long-term behavior of the iteration of a random map consisting of Lipschitz transformations on a compact metric space, independently and randomly selected according to a fixed probability measure.
Such a random map is said to be \emph{mostly contracting} if all Lyapunov exponents associated with stationary measures are negative. This requires introducing the notion of maximal Lyapunov exponent in this general context of Lipschitz transformations on compact metric spaces. We show that this class is open with respect to the appropriate topology and satisfies the strong law of large numbers for non-uniquely ergodic systems, the limit theorem for the law of random iterations, Palis' global conjecture, and that the associated annealed Koopman operator is quasi-compact. This implies many statistical properties such as central limit theorems, large deviations, statistical stability, and the continuity and H\"older continuity of Lyapunov exponents.

Examples from this class of random maps include random products of circle \(C^1\) diffeomorphisms, interval \(C^1\) diffeomorphisms onto their images, and \(C^1\) diffeomorphisms of a Cantor set on a line, all considered under the assumption of no common invariant measure. This class also includes projective actions of locally constant linear cocycles under the assumptions of simplicity of the top Lyapunov exponent and a suitable irreducibility condition. Finally, it encompasses the classical theory of finite-state Markov chains, viewed as random walks on a finite set.

One of the main tools to prove the above results is the generalization of \emph{Kingman's subadditive ergodic theorem} and the \emph{uniform Kingman's subadditive ergodic theorem} for general \emph{Markov operators}. Another key ingredient is the \emph{exponential local contraction theorem}, which can be viewed as a non-smooth metric analogue of the contraction mechanism behind stable manifold theory. These results are of independent interest, as they may have broad applications in other contexts.
\end{abstractbis}


\thispagestyle{empty}

\arxivtableofcontents

\vspace{0.4cm}


\chapter{Introduction}
\abstract{ {
This chapter introduces the framework of mostly contracting Lipschitz random
maps on compact metric spaces and gives an overview of the main results of the
monograph. We introduce maximal Lyapunov exponents through local Lipschitz
constants, and formulate the mostly contracting condition in terms of the
negativity of the Lyapunov exponents associated with stationary measures. We
then describe the principal consequences of this condition: exponential local
contraction, finiteness of stationary regimes, strong laws of large numbers,
Palis' conjecture on the finiteness of physical measures, quasi-compactness and
spectral gap properties of the annealed Koopman operator, central limit theorems, large
deviations, and statistical stability. The chapter also summarizes the main
classes of examples treated later, including locally constant linear cocycles
and random products of one-dimensional diffeomorphisms. In particular,  we establish the continuity, H\"older continuity, and several limit theorems for the top Lyapunov exponent of these classes of examples. } 
}

\section{{From linear cocycles to nonlinear random maps}}
This work builds upon the idea that many results for locally constant linear cocycles, driven by a Bernoulli shift, might also have nonlinear counterparts in the context of random walks of diffeomorphisms (on the circle). Furstenberg's seminal work indicated that the extreme Lyapunov exponents associated with random products of independent and identically distributed unimodular matrices are non-zero~\cite{Fur:63}. Analogous conclusions for random products of diffeomorphisms on smooth manifolds were first proved by Crauel~\cite{crauel:1990}. Subsequent advancements by Malicet~\cite{Mal:17}, Barrientos and Malicet~\cite{BM20}, and Malicet and Militon~\cite{MM23}
expanded these findings to include random products of circle homeomorphisms, conservative diffeomorphisms, and Cantor diffeomorphisms, respectively. 
In this {monograph}, we find and enhance the strong statistical properties of well-distributed orbits established by Le Page~\cite{page1982theoremes} for random products of matrices, extending these results to random walks on diffeomorphisms of the circle, interval diffeomorphisms onto their images, and diffeomorphisms of a Cantor set on a line, among other systems.
We also revisit (and improve) the results of Furstenberg and Kifer~\cite{furstenberg1983random},  Hennion~\cite{hennion1984loi} and Le Page~\cite{page1989regularite} on the continuity and H\"older continuity of Lyapunov exponents for linear cocycles, extending these results to random products of diffeomorphisms on the circle.

\section{Bernoulli random maps} \label{ss:random-map} 
Let $(X, d)$ be a compact metric space\index{metric structures!0ambient metric@\(d\), ambient metric on \(X\)} and denote the semigroup of Lipschitz transformations on $X$ by $\text{Lip}(X)$\index{space of functions!\(\operatorname{Lip}(X)\), Lipschitz transformations}. We focus on random walks on $\text{Lip}(X)$ defined by the compositions $g_n = f_{n-1} \circ \dots \circ f_0$, where each $f_k$ is independently and randomly selected according to the same probability measure on $\text{Lip}(X)$. In a slightly more abstract sense, let $(T , \mathscr{A}, p)$ be a probability space, and consider the Bernoulli product space\index{Bernoulli product space} $(\Omega , \mathscr{F}, \mathbb{P}) = (T^\mathbb{N}, \mathscr{A}^\mathbb{N}, p^\mathbb{N})$. We deal with measurable functions $f : \Omega \times X \to X$, referred to as (Bernoulli) \emph{random maps}\index{random maps!(Bernoulli) random map}, such that
$$
f_\omega \eqdef f(\omega, \cdot) = f_{\omega_0} \quad \text{for $\mathbb{P}$-a.e.~$\omega = (\omega_i)_{i \geq 0} \in \Omega$}.
$$
We introduce the following non-autonomous iterations:
$$
f^0_\omega = \mathrm{id} \quad \text{and} \quad f^n_\omega = f_{\omega_{n-1}} \circ \dots \circ f_{\omega_0} \quad \text{for} \ \ n > 0 \ \ \text{and $\mathbb{P}$-a.e.~$\omega = (\omega_i)_{i \geq 0} \in \Omega$}.
$$
To study these compositions, we introduce the following canonical skew-product\index{skew-product@canonical skew-product} 
\begin{equation*}\label{eq:skew}
F: \Omega \times X \to \Omega \times X, \quad F(\omega, x) = (\sigma \omega, f_{\omega}(x)),
\end{equation*}
where $\sigma: \Omega \to \Omega$ is the shift operator. A probability measure $\mu$ on $X$ is said to be a \emph{stationary measure} of $f$ (or $f$-stationary for short) if $\bar{\mu} = \mathbb{P} \times \mu$ is a  $F$-invariant measure. 
Similarly, we say that $\mu$ is  \emph{ergodic}\index{stationary measures!ergodicity} if $\bar{\mu}$ is ergodic.


We are interested in studying the $f$-stationary probability measures, which capture the typical behavior of the random walk. In the case where $X$ is a manifold, these measures naturally give rise to Lyapunov exponents,  which we will assume are negative. To accommodate the general scenario, we generalize the concept of the maximal Lyapunov exponent when $X$ is not a manifold.

\section{Maximal Lyapunov exponents for Lipschitz random maps} \label{sec:lyapunov-def}
A function {$g:X\to X$} is said to be \emph{Lipschitz}\index{Lipschitz!$g:X\to X$, Lipschitz function} if there is a constant $C\geq 0$ such that $d(g(x),g(y))\leq C d(x,y)$. The infimum of $C\geq 0$ such that the above inequality holds is called the \emph{Lipschitz constant}\index{Lipschitz!\(\operatorname{Lip}(g)\), (global) Lipschitz constant} and denoted by~$\mathrm{Lip}(g)$.
The \emph{local Lipschitz constant}\index{Lipschitz!B@\(Lg(x)\), (pointwise) local Lipschitz constant} of  $g$ at $z\in X$ is defined as
\begin{align*}
   Lg(z) \eqdef \inf_{r>0} L_rg(z) = \lim_{r\to 0^+} L_rg(z)
 \end{align*}
with
\begin{align*}
   L_rg(z)\eqdef \mathrm{Lip}(g|_{B(z,r)})=\sup\left\{   \frac{d(g(x),g(y))}{d(x,y)} \, : \ x,y\in B(z,r), \ x\not=y \right \} \index{Lipschitz!A@\(L_rg(x)\), flat local Lipschitz constant}
\end{align*}
where $B(z,r)$ denotes the open ball centered at $z$ and of radius $r$.
We have that $\mathrm{Lip}(g)$ is greater than or equal to the supremum of $Lg(z)$ over all $z\in X$. As we prove in Lemma~\ref{lem1}, the function $Lg: X \to [0,\infty)$ is upper semicontinuous and
satisfies the chain rule inequality
\[L(g \circ h)(z) \leq Lg(h(z))\cdot Lh(z).\] \index{Lipschitz!chain rule for local Lipschitz constants}
If $g$ is a diffeomorphism of a Riemannian manifold, then $Lg(z) = \|Dg(z)\|$.

Let \(f:\Omega\times X \to X\) be a random map as introduced earlier. We say that $f$ is a \emph{continuous} (resp.~\emph{$C^r$ differentiable}) \emph{random map}\index{random maps!continuous random map} if $f_\omega:X\to X$ is a continuous (resp.~$C^r$ differentiable)\index{random maps!PP@$C^r$ differentiable random map} map for $\mathbb{P}$-a.e.~$\omega\in\Omega$. Similarly, $f$ is said to be a \emph{Lipschitz random map}\index{random maps!Lipschitz random map} if $\mathrm{Lip}(f_\omega)<\infty$ for $\mathbb{P}$-a.e.~$\omega\in \Omega$. 
For each \(\omega \in \Omega\) and \(x \in X\), define \(L_n(\omega,x) = Lf_\omega^n(x)\). By the chain rule, we derive the  inequality
\begin{equation} \label{eq:subaditive1}
\log L_{n+m} \leq \log L_n\circ F^m + \log L_m
\end{equation}
where $F$ is the canonical skew-product associated with $f$. 

Let $\bar{\mu}$ be an $F$-invariant probability measure on $\Omega\times X$ with first marginal~$\mathbb{P}$.  
In view of~\eqref{eq:subaditive1}, under the standard \(\bar{\mu}\)-integrability condition of \(\log^+ L_1\), by Kingman's subadditive ergodic theorem~\cite[Thm.~I.1]{ruelle1979ergodic}, there exists an $F$-invariant measurable function $\lambda: \Omega\times X \to \mathbb{R}\cup \{-\infty\}$, called the \emph{maximal pointwise Lyapunov exponent}\index{maximal Lyapunov exponents!$\lambda(\omega,x)$, pointwise exponent} such that $\lambda^+\eqdef \max\{0,\lambda\}$ is $\bar{\mu}$-integrable and
\[
     \lambda(\omega,x)=\lim_{n\to\infty} \frac{1}{n} \log Lf^n_\omega(x) \quad \text{$\bar{\mu}$-a.e.~$(\omega,x)\in \Omega\times X$}.
\]
Moreover, the \emph{maximal Lyapunov exponent}\index{maximal Lyapunov exponents!$\lambda(\mu)$, exponent of $\bar{\mu}=\mathbb{P}\times \mu$}  of \(\mu\) is given by
\begin{align*}
   \lambda(\bar{\mu}) &\eqdef \inf_{n \ge 1} \frac{1}{n} \int \log Lf^n_\omega(x)\, d\bar{\mu} \\ &= \lim_{n\to \infty} \frac{1}{n} \int \log Lf^n_\omega(x) \, \, d\bar{\mu} = \int \lambda(\omega,x) \, d\bar{\mu} \in [-\infty,\infty).
\end{align*}
To ensure the subadditivity property for any 
probability measure with first marginal $\mathbb{P}$, we introduce the more encompassing uniform integrability condition:
\begin{equation} \label{eq:integrability}
   \int \log^+ \mathrm{Lip}(f_\omega) \, d\mathbb{P} <\infty. \index{integrability conditions!uniform log-Lipschitz integrability}
\end{equation}
If $X$ is a Riemannian manifold,
$\lambda(\bar{\mu})$ coincides with the maximal average Lyapunov exponent from Oseledets' theorem for the linear cocycle over the skew-product $F$.


%
%


\section{Mostly contracting random maps} \label{ss:mostlycontracting}

For notational simplicity, given an $f$-stationary measure $\mu$, we write $\lambda(\mu)=\lambda(\mathbb{P}\times \mu)$. 
\begin{defi} \label{def:mostly-contracting} We say that a Lipschitz random map $f:\Omega\times X \to X$ is \emph{mostly contracting}\index{random maps!mostly contracting} if the integrability condition~\eqref{eq:integrability} holds and 
\[
\lambda(f)  \eqdef \sup\left\{\lambda(\mu) : \  \mu \ \text{is $f$-stationary} \right\}<0. \index{maximal Lyapunov exponents!lambda f@\(\lambda(f)\), exponent of a random map}
\]
\end{defi}

If the space $X$ is finite (which corresponds to the case of a Markov chain with finite states), then $f$ is locally constant, so $Lf_\omega(x)=0$ and $f$ is trivially mostly contracting. Another basic family of examples is given {by contracting random Lipschitz maps, or more generally by} contraction in mean, such as $\mathbb{E}[\mathrm{Lip}(f)]<1$ or $\mathbb{E}[\log \mathrm{Lip}(f)]<0$. {These examples include the classical setting of contracting iterated function systems.} We also see that, in general, random products of matrices give rise to mostly contracting systems (see~\S\ref{ss:locally-constant}), and the same holds for random maps of one-dimensional diffeomorphisms (see~\S\ref{ss:one-dimensional}).

One of the main objectives of this  {monograph} is to develop a unified framework showing that, under mild assumptions, the Koopman operator associated with a mostly contracting system has a spectral gap (Theorem~\ref{thmA}). This spectral property allows us to deduce a broad range of statistical consequences, such as limit theorems, large deviations,  central limit laws and regularity of Lyapunov exponents (\S\ref{ss:limit laws}--\ref{ss:large-deviation},~\S\ref{ss:CLT-linear}--\ref{ss:holde-continuty-intro} and~\S\ref{ss:CLT-one-dimensional}--\ref{ss:holder-circle}). While each of these cases has been studied separately in the literature --for Markov chains with finitely many states~\cite{HL:12}, for systems with contraction in mean~\cite{peigne1993iterated,pollicott2001contraction,HH:01}, and for random matrix products~\cite{page1982theoremes,page1989regularite,BouLac:85,benoist2016random} among others-- our results provide a common perspective that both recovers and extends these classical settings. In particular, we obtain new results on random products of matrices in non-irreducible settings and on random maps of one-dimensional diffeomorphisms, which, to our knowledge, have not been treated previously. {More generally, the metric-space nature of the theory is not merely technical. 
For instance, in recent work~\cite{BMNNT}, we extend the spectral properties established for the Koopman operator to the setting of mostly contracting random maps \(g\) on a compact metric space \(X\) driven by a Markov chain on a finite state space \(T\). This is done by encoding \(g\) as a Bernoulli random map \(f\) on the enlarged phase space \(T\times X\), which allows one to apply Theorem~\ref{thmA} to \(f\).}

One of the main steps to obtain our results is to generalize the \emph{Kingman's subadditive ergodic theorem} and the \emph{uniform Kingman's subadditive ergodic theorem} for Markov operators. See Theorems~\ref{Kingmanuniform} and \ref{Kingman}. These generalizations may have an independent interest for the study of Markov processes.
From these results, we derive in Theorem~\ref{prop:equivalence} that $\lambda(f)$ coincides with the maximum of $\lambda(\mu)$ over the ergodic $f$-stationary measures~$\mu$. We also introduce the 
\emph{maximal annealed Lyapunov exponent at $x$}\index{maximal Lyapunov exponents!lambda x@\(\lambda(x)\), pointwise annealed exponent} as
$$
\lambda(x)\eqdef
\limsup_{n\to\infty} \
\frac{1}{n} \int \log Lf^n_\omega(x) \, d\mathbb{P}.
$$
and prove that,
{for every \(f\)-stationary probability measure \(\mu\), this limit exists for \(\mu\)-a.e.~\(x\in X\) and coincides almost surely with the pointwise Lyapunov exponent, that~is,
\[
\lambda(\omega,x)=\lambda(x)
\qquad\text{for }(\mathbb P\times \mu)\text{-a.e.~}(\omega,x)\in \Omega\times X.
\]
In the literature, it is standard that the pointwise Lyapunov exponent \(\lambda(\omega,x)\) does not depend on $\omega$ almost surely~\cite [Theorem~3.2]{liu1995smooth}. However, the conclusion here is stronger: we identify the pointwise Lyapunov exponent \(\lambda(\omega,x)\) with the deterministic function $\lambda(x)$ obtained from the annealed Lyapunov exponent. Moreover, 
}
$$
 \lambda(\mu)=\int \lambda(x)\, d\mu \quad \text{and}  \quad
\lambda(f)
=\sup \left\{ \lambda(x):  x \in X \right\}.
$$
We also obtain that
$$
\lambda(f) =   
\inf_{n\geq 1}
\max_{x\in X} \frac{1}{n} \int \log Lf^n_\omega(x) \, d\mathbb{P} = \lim_{n\to\infty} \max_{x\in X} \frac{1}{n} \int \log Lf^n_\omega(x) \, d\mathbb{P}.
$$
By these relations, mostly contracting is equivalent to any of the following conditions:
\begin{enumerate}[label=(\roman*), leftmargin=0.9cm]
  \item  $\lambda(\mu)<0$ for every ergodic $f$-stationary measure,   \item $\lambda(x)<0$ for all $x\in X$,
  \item there are $a>0$, $n\in \mathbb{N}$ such that $\int \log Lf^n_\omega(x) \, d\mathbb{P} < -a$ for any $x\in X$.
\end{enumerate}

The term ``mostly contracting'' was originally introduced for partially hyperbolic diffeomorphisms, essentially implying that all Lyapunov exponents along the invariant center bundle are negative, mimicking item (ii) above. This concept was first proposed by Bonatti and Viana~\cite{BV:00} as a somewhat technical condition that guarantees the existence and finiteness of physical measures. A related notion of mostly contracting following item (iii) was posed by Dolgopyat~\cite{Dol:00} to study the mixing properties of such systems and further derive several limit theorems within a similar context~\cite{Dol:13}. Later, Andersson~\cite{And:10} established that the mostly contracting property, as previously introduced in the literature, is equivalent to the property that every (ergodic) Gibbs u-state has negative central Lyapunov exponents, corresponding to item (i). Additionally, in~\cite{And:10}, it was shown that these mostly contracting partially hyperbolic diffeomorphisms form an open set in the space of $C^{1+\epsilon}$ diffeomorphisms and exhibit stability of the physical measures under small perturbations of the map in question. This result was subsequently generalized to the $C^1$ topology by Yang~\cite{Yang:21}. For a comprehensive collection of examples of mostly contracting partially hyperbolic diffeomorphisms, refer to~\cite{ViaYan:13,DolYanVia:16}. 

We show that most of the results obtained for mostly contracting partially hyperbolic diffeomorphisms have analogues in the context of Lipschitz random maps. In particular, we establish the openness of the set of mostly contracting random maps with respect to an appropriate topology in the space 
of Lipschitz random maps. This is not immediately evident from Definition~\ref{def:mostly-contracting}. However, due to the equivalent condition (iii) above, we deduce this property in Proposition~\ref{prop:open-mostly}, emphasizing that for $C^1$ differentiable random maps of a compact Riemannian manifold, the considered topology is coarser than the $C^1$-topology induced by the $C^1$-distance along the fiber maps. We also prove the finiteness of the number of physical measures (see item~(b) in Corollary~\ref{rem:measure-mean-quasicompactness}), central limit theorems~\S\ref{ss:limit laws}, large deviations~\S\ref{ss:large-deviation}, statistical stability properties~\S\ref{ss:statistical-stability}, and remarkable examples of mostly contracting random maps~\S\ref{ss:locally-constant} and~\S\ref{ss:one-dimensional} among other consequences of our main results.

\section{Local contraction} \label{ss:local}
Assuming that the maximal Lyapunov exponent is negative, one can anticipate an exponential local contraction from the Pesin theory on stable manifolds for random dynamical systems, as worked out in~\cite{kifer2012ergodic,liu1995smooth}. However, within our Lipschitz framework, the direct application of Pesin's theory on stable manifolds is not feasible. In Theorem~\ref{contraction}, we aim to establish a non-smooth Pesin's contraction result for general skew-products, asserting that the principle of exponential contraction persists, provided that the maximal Lyapunov exponent is negative.  As a consequence, in the setting of random maps, we establish the \emph{local contraction property},   that is,  for every $x\in X$, for $\mathbb{P}$-a.e.~$\omega\in \Omega$, there exists a neighborhood $B=B(\omega)$ of $x$ such that $\diam f^n_\omega(B) \to 0$ as $n\to\infty$. Here $\diam A=\sup \{d(x,y): x,y\in A\}$\index{metric structures!1diameter@\(\diam\), diameter of subsets of \(X\)} denotes the diameter of a set $A$.
In fact, we  show the following exponential contraction result:

\begin{mainthm} \label{cor:local-contraction}
Let $(X,d)$ and $(\Omega,\mathscr{F},\mathbb{P})$ be a compact metric space and a Bernoulli product probability space, respectively, and consider a mostly contracting random map $f:\Omega\times X\to X$.  Then there exists $q<1$ such that for every $x\in X$, for $\mathbb{P}$-a.e.~$\omega\in \Omega$, there exist a neighborhood $B$ of $x$ and a constant $C>0$ such that 
$$
\diam f^n_\omega(B) \leq Cq^n \quad  \text{for all  $n\geq 1$}.
$$
\index{local contraction! exponential local contraction property}
\end{mainthm}

The above theorem brings several interesting consequences from~\cite[Sec.~4]{Mal:17}, some of which we summarize {below.  For convenience, we first introduce some terminology used in the next statement.

\begin{defi}\label{def:topological-notions-random}
Let \(f:\Omega\times X\to X\) be a random map. A closed set \(Y\subset X\) is said to be
\emph{\(f^k\)-invariant}\index{invariant sets} if
$f_\omega^k(Y)\subset Y$ for \(\mathbb{P}\)-a.e.~\(\omega\in\Omega\). 
We say that $f$ is
\begin{enumerate}[label=(\roman*), leftmargin=0.8cm]
    \item \emph{minimal}\index{random maps!(minimality) minimal random map} if there is no proper closed \(f\)-invariant subset of \(X\);
    
    \item \emph{proximal}\index{random maps!(proximality) proximal random map} if for every \(x,y\in X\) there exists \(\omega\in\Omega\) such that
    \[
    \liminf_{n\to\infty} d(f_\omega^n(x),f_\omega^n(y))=0;
    \]
    
    \item \emph{mingled}\index{random maps!(mingledness) mingled random map} if for every \(k\ge 1\) there do not exist two disjoint non-empty closed
    \(f^k\)-invariant subsets of \(X\). In other words, for any $k\geq 1$ there is a unique minimal $f^k$-invariant closed set; 
    
    \item \emph{uniquely ergodic random map}\index{random maps!uniquely ergodic random map} if it admits a unique \(f\)-stationary probability measure;
    
    \item \emph{aperiodic}\index{random maps!(aperiodicity) aperiodic random map} if there do not exist \(p\ge 2\) and pairwise disjoint non-empty closed sets
    \(F_1,\dots,F_p\subset X\) such that
    \[
    f_\omega(F_i)\subset F_{i+1}
    \quad\text{for }1\le i<p
    \quad\text{and}\quad
    f_\omega(F_p)\subset F_1 \quad \text{for \(\mathbb{P}\)-a.e.\ \(\omega\in\Omega\).}
    \]
\end{enumerate}
\end{defi}}
\begin{maincor}
\label{rem:measure-mean-quasicompactness}
{Let $(X,d)$ and $(\Omega,\mathscr{F},\mathbb{P})$ be a compact metric space and a Bernoulli product probability space, and consider a continuous random map $f:\Omega\times X \to X$ with the local contraction property.} Then, there are only finitely many ergodic $f$-stationary probability measures $\mu_1, \dots, \mu_r$. 
 Their topological supports are 
{pairwise disjoint. These supports are precisely the minimal \(f\)-invariant closed subsets of \(X\).} \\[-0.2cm]

Moreover, the following statements hold:

\begin{enumerate}[itemsep=0.5cm, leftmargin=0.7cm, label=(\alph*)] \index{strong law of large numbers}
\item \textbf{Strong law of large numbers:} For any probability measure $\nu$ on $X$,  
\begin{equation}
\label{eq:union_vacia}
\nu(B_\omega(\mu_1)\cup \dots \cup B_\omega(\mu_r))=1 \quad \text{for $\mathbb{P}$-a.e.~$\omega \in \Omega$}
\end{equation}
where
$$B_\omega(\mu_i)=\left\{x\in X:  \lim_{n\to\infty}\frac{1}{n}\sum_{k=0}^{n-1}\delta_{f_\omega^k(x)}=\mu_i \ \text{in the weak$^*$ topology }  \right\}.$$
Furthermore, for any $x \in X$ and for $\mathbb{P}$-a.e.~$\omega\in\Omega$,  the sequence of points $\{f^n_\omega(x)\}_{n\geq 0}$ accumulates to one of the topological supports of these measures. \\[-0.2cm]
 
If  $f$ is either minimal,
proximal
or mingled,
then  $f$ is uniquely ergodic.
\index{Palis' global conjecture}
\item \textbf{Palis' global conjecture:} Assuming that  $\Omega$ is a Polish space, for any  probability measure $\nu$ on $X$, 
there are finitely many ergodic $F$-invariant probability measures $\bar{\mu}_1,\dots,\bar{\mu}_s$ on $\Omega\times X$ such that
\begin{enumerate}[leftmargin=1.5cm,label=(\roman*)]
  \item $\bar{\nu}(B(\bar{\mu}_i))>0$, i.e., $\bar{\mu}_i$ is a physical measure\index{stationary measures!physical measure} with respect to $\bar{\nu}=\mathbb{P}\times \nu$,  
  \item $\bar{\nu}\left(B(\bar{\mu}_1)\cup \dots \cup B(\bar{\mu}_s)\right)=1$,
\end{enumerate} 
where
$$\quad \ \ \ B(\bar{\mu}_i)=\left\{(\omega,x)\in \Omega\times X:  \lim_{n\to\infty}\frac{1}{n}\sum_{k=0}^{n-1}\delta_{F^k(\omega,x)}=\bar{\mu}_i \ \text{in the weak$^*$ topology }  \right\}.$$
Moreover, $\bar{\mu}_i=\mathbb{P}\times \mu_j$ for some $j\in \{1,\dots,r\}$. Hence $s\leq r$.

\item \textbf{Uniform pointwise convergence:}  
{Let \(P:C(X)\to C(X)\) be the annealed Koopman operator\index{operator!annealed Koopman operator}
defined  by
\[
P\phi(x)\eqdef \int \phi(f_\omega(x))\,d\mathbb{P},
\qquad \phi\in C(X),
\]}
where $C(X)$\index{space of functions!$C@$C(X)$, real-valued continuous functions} is the space of real-valued continuous functions   on $X$ with the uniform norm $\|\cdot \|_\infty$\index{metric structures!distance uniform norm@\(\Vert\cdot\Vert_\infty\), uniform norm on \(C(X)\)} and consider the sequence of operators 
$$A_n\eqdef \frac{1}{n}\sum_{i=0}^{n-1} P^i.$$ 
Then, there exists a  finite-rank projection $\Pi$ such that
$$
\lim_{n\to\infty} \|A_n\phi-\Pi\phi\|_\infty = 0, \quad \text{for all $\phi \in C(X)$}.$$
\noindent Moreover, if $f$ is 
\begin{enumerate}[label=$\bullet$,]
\item uniquely ergodic,  
then $\Pi$ is a projection on the space of constant functions; 
\item aperiodic,
then
$$
\lim_{n\to \infty}\|P^n\phi-\Pi\phi\|_\infty =0, \quad \text{for all $\phi \in C(X)$.}
$$
\end{enumerate}

\index{law of random iterations}
\item \textbf{Limit theorem for the law of the random iterations:} For every $x\in X$, denote by $$\mu_{n}^{x}\eqdef (X^x_n)_*\mathbb{P}$$ the law of the random variable $X_{n}^{x}(\omega) = f^{n}_{\omega}(x)$. Assuming that $f$ is aperiodic, then  $\mu_{n}^{x}$ converges in the weak$^*$ topology to an $f$-stationary probability measure $\mu^{x}\in \{\mu_1,\dots,\mu_r\}$. Moreover, the convergence is uniform in $x$ in the sense that 
\[
\sup_{x \in X} \left| \int \phi \, d\mu_{n}^{x} - \int\phi \, d\mu^{x} \right| \xrightarrow[n \to \infty]{} 0 \quad \text{for all $\phi\in C(X)$.} \\[0.2cm]
\]
\end{enumerate}
\end{maincor}

When $f$ is uniquely ergodic, item (a) becomes the so-called  \emph{strong law of large numbers}: for every $x\in X$, taking $\nu=\delta_x$ in~\eqref{eq:union_vacia},
$$
   \lim_{n\to\infty} \frac{1}{n}\sum_{k=0}^{n-1} \phi(f^k_\omega(x))=\int \phi \, d\mu_1  \quad \text{for $\mathbb{P}$-a.e.~$\omega\in \Omega$  and every $\phi\in C(X)$}
$$
 where $\mu_1$ is the unique $f$-stationary measure ($r=1$).  Breiman established the strong law of large numbers in~\cite{Bre:60} for any uniquely ergodic continuous random map.
 Item~(a) extends Breiman's theorem to not necessarily uniquely ergodic continuous random maps under the assumption of local contraction, in particular, to mostly contracting random maps.
 Item~(b) corresponds to Palis' global conjecture on the finiteness of attractors~\cite{palis2000global} (see~\S\ref{ss:palis-conjeture} for more details). Although this item was not formally proven in~\cite{Mal:17}, it can be easily deduced, as demonstrated in Theorem~\ref{thm:palis-conjecture}.
 %
%
Furthermore, item~(d) was stated in~\cite[Thm.~C]{Mal:17} for random walks of homeomorphisms of the circle. However, its proof only uses~\cite[Prop.~4.14]{Mal:17}, which merely requires continuous aperiodic random maps with the local contraction property. In fact, item~(d) is an extension of~\cite[Thm.~1]{kaijser1978limit}. Indeed,~\cite[Condition~C]{kaijser1978limit} implies that $f$ is mingled and 
 \cite[Condition~G]{kaijser1978limit} is an equivalent definition of mostly contracting (see item~(iii) in~\S\ref{ss:mostlycontracting}) but with a stronger integrability requirement.  \\

Let us describe some relations between the various topological notions introduced:

\begin{rem}\label{rem:topological-notions}
Let \(f:\Omega\times X\to X\) be a continuous random map. Then 
\begin{enumerate}[leftmargin=0.55cm,itemsep=0.2cm]
    \item \emph{\(f\) is mingled if and only if \(f\) is aperiodic and has a unique minimal closed \(f\)-invariant subset.}
    Indeed, if \(f\) has two distinct minimal closed invariant subsets, then they are disjoint, so \(f\) is not mingled.
    Likewise, if there exist pairwise disjoint non-empty closed sets \(F_1,\dots,F_p\) such that
    \(f_\omega(F_i)\subset F_{i+1}\) for \(1\le i<p\) and \(f_\omega(F_p)\subset F_1\) for
    \(\mathbb{P}\)-a.e.\ \(\omega\in\Omega\), then each \(F_i\) is \(f^p\)-invariant, and again \(f\) is not mingled.
    Conversely, assume that \(f\) is aperiodic and has a unique minimal closed \(f\)-invariant subset \(F\).
    If \(F'\) is a minimal \(f^p\)-invariant closed subset for some \(p\ge 1\), then \(F'\subset F\).
    By~\cite[Lemma~4.15]{Mal:17}, the restriction of \(f^p\) to \(F\) is minimal, hence \(F'=F\).
    Therefore \(f\) is mingled.
    
\item \emph{If \(f\) is minimal and \(X\) is connected, then \(f\) is mingled.}
Minimality implies that \(X\) itself is the unique minimal closed invariant subset.
If \(f\) were not aperiodic, then the sets appearing in the definition of aperiodicity
would give a decomposition of \(X\) into pairwise disjoint non-empty closed subsets,
contradicting connectedness. Hence, item~(i) applies. {A deterministic irrational rotation
of \(\mathbb S^1\) is a simple example.}

    \item \emph{If \(f\) is proximal, then \(f\) is mingled.}
    Indeed, every iterate \(f^k\) is still proximal.
    In particular, for every \(k\ge 1\), there is a unique minimal closed \(f^k\)-invariant subset.
    Thus \(f\) is mingled.
   { A basic example is a strict contraction of a compact interval.}

    \item \emph{Under the assumption of local contraction, \(f\) is uniquely ergodic if and only if there is a unique minimal closed \(f\)-invariant subset.}
    This is precisely~\cite[Prop.~4.8]{Mal:17}.
    {In particular, once local contraction is available, topological uniqueness of the minimal set already forces
    uniqueness of the stationary measure.
    }
\end{enumerate}
\end{rem}

{
\subsection{Examples on finite-phase space}
Assume that \(X=\{1,\dots,m\}\) is finite and is endowed with the discrete metric. In this case, every random map
\(f:\Omega\times X\to X\) is locally constant, hence \(f\) is mostly contracting and by Theorem~\ref{cor:local-contraction} it satisfies the local contraction property.
In particular, by Remark~\ref{rem:topological-notions}(iv), \(f\) is uniquely ergodic if and only if it has a unique minimal
\(f\)-invariant subset; and by Remark~\ref{rem:topological-notions}(i), \(f\) is mingled if and only if it is aperiodic and uniquely ergodic.
Thus, in the finite-state setting, the topological notions introduced above are related by
\[
\text{proximal}
\Longrightarrow
\text{mingled}
\Longrightarrow
\text{unique minimal set}
\Longleftrightarrow
\text{uniquely ergodic}.
\]
Moreover, minimality also implies unique ergodicity, as does proximality, while in general
\(\text{minimal}\centernot\Rightarrow \text{proximal}\) and \(\text{proximal}\centernot\Rightarrow \text{minimal}\).
The examples below show that none of the converses holds in general.

To relate this to the classical presentation of a finite-state Markov chain, recall that
\(\Omega=T^{\mathbb N}\) and \(f_\omega=f_{\omega_0}\) for \(\mathbb P\)-a.e.\ \(\omega=(\omega_n)_{n\ge0}\).
The annealed Koopman operator from item~\emph{(c)} of Corollary~\ref{rem:measure-mean-quasicompactness} is the transition operator
of the Markov chain induced by \(f\). More precisely, set
\[
m_{xy}\eqdef p(\{t\in T: f_t(x)=y\}), \qquad x,y\in X,
\]
so that \(M=(m_{xy})_{x,y\in X}\) is a stochastic matrix, and
\[
P\phi(x)=\sum_{y\in X} m_{xy}\phi(y), \qquad \phi\in C(X)\simeq \mathbb R^m.
\]
Equivalently, one may consider the directed graph with an arrow \(x\to y\) whenever \(m_{xy}>0\), that is,
whenever \(f_t(x)=y\) for some \(t\in \mathrm{supp}\,p\). A subset \(Y\subset X\) is \(f\)-invariant if and only if no arrow leaves \(Y\),
and the minimal \(f\)-invariant subsets are precisely the closed irreducible classes of this graph. \\[-0.25cm]

The following examples illustrate the relations between the notions introduced above.

\begin{enumerate}[label=(\alph*), leftmargin=0.7cm, itemsep=0.35cm]

\item \emph{Minimal does not imply mingled; equivalently, on a finite set, uniquely ergodic does not imply mingled.}
Let \(T=\{a\}\), \(p(a)=1\), \(X=\{1,2,3\}\), and
\[
f_a(1)=2,\qquad f_a(2)=3,\qquad f_a(3)=1.
\]
Then the stochastic matrix and directed graph are
\[
M=
\begin{pmatrix}
0&1&0\\
0&0&1\\
1&0&0
\end{pmatrix}
\qquad
\begin{tikzpicture}[baseline={(current bounding box.center)},>=Latex,
  every node/.style={circle,draw,inner sep=1pt}]
\node (1) at (90:0.8) {\scriptsize \(1\)};
\node (2) at (210:0.7) {\scriptsize \(2\)};
\node (3) at (330:0.7) {\scriptsize \(3\)};
\draw[->] (1) to[bend left=11] (2);
\draw[->] (2) to[bend left=11] (3);
\draw[->] (3) to[bend left=11] (1);
\end{tikzpicture}
\]
Hence \(X\) is the unique minimal \(f\)-invariant subset, so \(f\) is minimal and uniquely ergodic.
However, \(f\) is not aperiodic, since
\[
F_1=\{1\},\qquad F_2=\{2\},\qquad F_3=\{3\}
\]
form a periodic decomposition of period \(3\). Therefore \(f\) is not mingled (and in particular not proximal).
This is the periodic case in item~\emph{(c)} of Corollary~\ref{rem:measure-mean-quasicompactness}:
the powers \(P^n\) do not converge, whereas the Ces\`aro averages \(n^{-1}\sum_{j=0}^{n-1}P^j\) do.

\item \emph{Mingled does not imply proximal.}
Let \(T=\{a,b\}\) with \(p(a),p(b)>0\) and $p(a)+p(b)=1$, let \(X=\{1,2,3\}\), and define
\[
\hspace{.6cm} f_a(1)=2,\quad f_a(2)=2,\quad f_a(3)=3,
\quad
f_b(1)=3,\quad f_b(2)=3,\quad f_b(3)=2.
\]
Then the associated stochastic matrix and directed graph are
\[
M=
\begin{pmatrix}
0&p(a)&p(b)\\
0&p(a)&p(b)\\
0&p(b)&p(a)
\end{pmatrix}
\qquad
\begin{tikzpicture}[baseline={(current bounding box.center)},>=Latex,
  every node/.style={circle,draw,inner sep=1.2pt}]
\node (1) at (0,1.05) {\scriptsize \(1\)};
\node (2) at (-1.0,0) {\scriptsize \(2\)};
\node (3) at (1.0,0) {\scriptsize \(3\)};
\draw[->] (1) to[bend left=10] (2);
\draw[->] (1) to[bend right=10] (3);
\draw[->] (2) to[loop left] (2);
\draw[->] (3) to[loop right] (3);
\draw[->] (2) to[bend left=14] (3);
\draw[->] (3) to[bend left=14] (2);
\end{tikzpicture}
\]
Thus \(\{2,3\}\) is the unique minimal \(f\)-invariant subset, so \(f\) is uniquely ergodic.
Moreover, the minimal class \(\{2,3\}\) is aperiodic because it contains loops; hence \(f\) is mingled.
On the other hand, \(f\) is not proximal: starting from \(2\) and \(3\), every composition of \(f_a\) and \(f_b\)
keeps the pair \(\{2,3\}\) invariant, so these points never~meet.

\item \emph{Proximal does not imply minimal.}
Let \(T=\{a,b\}\) with \(p(a),p(b)>0\) and $p(a)+p(b)=1$, let \(X=\{1,2,3\}\), and define
\[
f_a(x)=1 \quad \text{for all }x\in X,
\qquad
f_b(x)=x \quad \text{for all }x\in X.
\]
Then the associated stochastic matrix and directed graph are
\[
M=
\begin{pmatrix}
1&0&0\\
p(a)&p(b)&0\\
p(a)&0&p(b)
\end{pmatrix}
\qquad
\begin{tikzpicture}[baseline={(current bounding box.center)},
>=Latex,
  every node/.style={circle,draw,inner sep=1.2pt}]
\node (1) at (0,1.4) {\scriptsize \(1\)};
\node (2) at (-1.0,1.8) {\scriptsize \(2\)};
\node (3) at (1.0,1.8) {\scriptsize \(3\)};
\draw[->] (1) to[loop above] (1);
\draw[->] (2) to[bend left=10] (1);
\draw[->] (2) to[loop left] (2);
\draw[->] (3) to[bend right=10] (1);
\draw[->] (3) to[loop right] (3);
\end{tikzpicture}
\]
The random map \(f\) is proximal (and hence mingled), because one occurrence of the symbol \(a\) sends every point to \(1\).
Moreover, since \(\{1\}\subsetneq X\) is an \(f\)-invariant subset, \(f\) is not minimal.
\end{enumerate}
}


\section{Quasi-compactness of annealed Koopman operator} \label{sec:koopman} 
We consider the so-called \emph{annealed Koopman operator} associated with the random map $f$ defined by
 \begin{equation}\label{eq:koopman0}
 P\varphi(x)\eqdef 
  \int \varphi \circ f_\omega(x) \, d\mathbb{P} \quad \text{for } \ \varphi\in C(X).
\end{equation}
Note that $P$ is a \emph{Markov operator}, meaning that $P\varphi \geq 0$ for all $\varphi \geq 0$ and $P1_X=1_X$, where $1_A$ denotes the characteristic function of a set $A$ in $X$. 

{Quasi-compactness is a classical spectral mechanism in the study of Markov operators. Roughly speaking, it means that, after separating a finite-dimensional part of the spectrum that governs the main asymptotic behavior, the remaining iterates decay exponentially fast. As a consequence, quasi-compactness is one of the standard routes to exponential convergence to equilibrium, limit theorems, large deviations, and stability properties. This operator-theoretic point of view goes back to the study of Markov chains and has also become fundamental in smooth dynamics, for instance, in the analysis of expanding and hyperbolic systems; see, for example,~\cite{HH:01,Baladi00}.

In our setting, quasi-compactness provides the link between the geometric contraction properties of the random dynamics and the statistical properties of the annealed evolution of observables. To state our next main result, we first recall some definitions.}


Let $(\mathcal{E},\|\cdot\|)$ be a Banach space  contained in $C(X)$ such that $1_X \in \mathcal{E}$. A Markov operator  $Q:C(X) \to C(X)$ is called \emph{quasi-compact}\index{spectral properties!quasi-compact operator} on $\mathcal{E}$ if $Q: \mathcal{E} \to \mathcal{E}$ is bounded (with respect to the norm $\|\cdot\|$) and
there exists a compact linear operator $R:\mathcal{E}\to \mathcal{E}$ such that $\|Q^n - R\|_{\text{op}} < 1$ for some $n \in \mathbb{N}$.
Furthermore, $Q$ is said to have a \emph{spectral gap}\index{spectral properties!spectral gap} on $\mathcal{E}$ if it is quasi-compact and if $Q\varphi = \lambda \varphi$ with $|\lambda|=1$, then $\lambda$ must equal $1$, and $\varphi$ must be a constant function.

%

For $0<\alpha\leq 1$, we denote by $C^{\alpha}(X)$ the space\index{space of functions!\(C^\alpha(X)\), H\"older continuous functions} of H\"older maps $\varphi : X \rightarrow \mathbb{R}$ with the associated norm $\|\varphi\|_{\alpha} = \|\varphi\|_{\infty} + |\varphi|_\alpha$\index{metric structures!distance Holder norm@\(\Vert\cdot\Vert_\alpha\), H\"older norm on \(C^\alpha(X)\)}, where
$$
   |\varphi|_\alpha \eqdef \sup \left\{\frac{|\varphi(x)-\varphi(y)|}{d(x,y)^\alpha} : x, y \in X, \ x \ne y\right\} < \infty.\index{metric structures!distance Holder aseminorm@\(\vert\cdot\vert_\alpha\), H\"older seminorm on \(C^\alpha(X)\)}
$$
Now we state the following main result. 

\begin{mainthm} \label{thmA}
Let $(X , d)$ and $(\Omega, \mathscr{F}, \mathbb{P})$ be a compact metric space and a Bernoulli product probability space respectively. Consider a mostly contracting random map $f: \Omega \times X \to X$ such that
\begin{equation}\label{eq:integral_condition}
  \int \mathrm{Lip}(f_\omega)^\beta \, d\mathbb{P} < \infty \quad \text{for some $\beta > 0$.} \index{integrability conditions!exponential Lipschitz moment}
\end{equation}
Let $P: C(X) \to C(X)$ be the annealed Koopman operator associated with $f$. Then there exists $\alpha_0 > 0$ such that $P$ is a quasi-compact operator on $C^{\alpha}(X)$ for any $0 < \alpha \leq \alpha_0$. Moreover, 
\begin{enumerate}[label=\textup{\alph*)},leftmargin=1cm]
    \item $1$ is a simple eigenvalue of $P$ if and only if $f$ is uniquely ergodic;
    \item $1$ is the unique eigenvalue of $P$ with modulus $1$ if and only if $f$ is aperiodic;
    \item $P$ has a spectral gap if and only if $f$ is mingled.
\end{enumerate}
\end{mainthm}

\begin{rem} \label{rem:thmA-finite-moment}
The integrability condition~\eqref{eq:integral_condition} is stronger than the condition~\eqref{eq:integrability} required to define the notion of mostly contracting. This condition is usually called the \emph{exponential moment condition} when considering random products of matrices (see~\eqref{eq:exponential-moment-condition}). If $\Omega = T^{\mathbb{N}}$ is a compact set and $f_\omega(x)$ depends continuously on $\omega$, this condition is immediately fulfilled. In particular, it holds when $T$ is finite. 
\end{rem}

Quasi-compactness and the spectral gap are strong properties with many consequences. Under the integrability condition~\eqref{eq:integral_condition}, these properties enable us to strengthen the conclusions of items~(c) and (d) in Corollary~\ref{rem:measure-mean-quasicompactness}. First, the convergence of $A_n = \frac{1}{n} \sum_{k=0}^{n-1} P^k$ to a finite-rank projection $\Pi$ is no longer pointwise but in operator norm in the space $C^\alpha(X)$. Similarly, $P^n$ converges in operator norm to $\Pi$ if $f$ is aperiodic. Second, it is standard that the convergence is always exponentially fast, or equivalently, the law $\mu_n^x$ of $X_n^x$ converges exponentially fast to a stationary measure for the dual metric of $C^\alpha(X)$, in particular in the Wasserstein distance. Among other consequences, one can deduce that, for any $\phi \in C^\alpha(X)$,  the projection $\Pi\phi$ is also in $C^\alpha(X)$, and that there exists $\psi \in C^\alpha(X)$ solving the Poisson equation: $\phi - \Pi\phi = \psi - P\psi$. More involved consequences will be discussed in the following sections.

\begin{question}
    Let $f$ be an aperiodic continuous random map. Is the exponentially fast local contraction property established in Theorem~\ref{cor:local-contraction} sufficient to ensure that $\|P^n \phi - \Pi \phi\|_\infty \to 0$ exponentially fast for any $\phi \in C^\alpha(X)$?
\end{question}

{The main difficulty in our setting is to prove that the annealed Koopman operator is quasi-compact for the broad class of mostly contracting random maps. Indeed, the mostly contracting condition is formulated in terms of the negativity of Lyapunov exponents over all stationary measures and does not, a priori, yield a direct contraction estimate for the operator itself. A main step in the proof of Theorem~\ref{thmA} is to bridge this gap by showing that every mostly contracting random map \(f\) satisfying the exponential moment condition~\eqref{eq:integral_condition} is locally contracting on average, in the following sense.}

\begin{defi} \label{def:local-contracting-average} A random map $f:\Omega \times X \to X$ is said to be \emph{locally contracting on average}\index{contraction on average!local contracting on average} if there exist $r > 0$, $0<\alpha\leq1$, $q < 1$, and $n \geq 1$ such that 
$$
\int d(f^n_\omega(x), f^n_\omega(y))^\alpha \, d\mathbb{P} \leq q \, d(x, y)^\alpha \quad \text{for any $x, y \in X$ with $d(x, y) < r$.}
$$
\end{defi}
Introducing the new equivalent distance $d_\alpha$ on $(X, d)$ defined by $d_\alpha(x, y) = d(x, y)^\alpha$, one can see that Definition~\ref{def:local-contracting-average} coincides with a similar notion in~\cite[Def.~1]{stenflo2012survey}. Moreover, Stenflo shows in~\cite[Lemma~1]{stenflo2012survey} that if $(X, d_\alpha)$ is a geodesic space, then locally contracting on average implies (global) \emph{contracting on average}\index{contraction on average!global contracting on average}, that is,   
$$
 \int d_\alpha(f^n_\omega(x), f^n_\omega(y)) \, d\mathbb{P} \leq q \,  d_\alpha(x, y) \quad \text{for all $x, y \in X$.}
$$
Notice that, in general, although $(X, d)$ was geodesic, the metric space $(X, d_\alpha)$ is not necessarily geodesic. Global contracting on average Lipschitz random maps are proximal and uniquely ergodic (cf.~\cite[Thm.~8]{barnsley2008v} and~\S\ref{ss:global_contraction}). Moreover, according to~\cite{peigne1993iterated, HH:01, walkden2013transfer}, the annealed Koopman operator has a spectral gap on $C^\alpha(X)$. We revisit these results in Proposition~\ref{prop:quasi-compact} by proving that locally contracting on average and the exponential moment condition~\eqref{eq:integral_condition} imply the so-called Lasota-Yorke type inequality, from which we deduce that the annealed Koopman operator is quasi-compact on $C^\alpha(X)$.

In the sequel, we show several results that follow as a consequence of Theorem~\ref{thmA}. 

\section{Central limit laws} \label{ss:limit laws}
Let $\mu_1,\dots,\mu_r$ be probability measures given in Corollary~\ref{rem:measure-mean-quasicompactness}. Note that the existence of such measures is also implied by the quasi-compactness of the annealed Koopman operator $P$ on $C^\alpha(X)$ by~\cite[Thm.~2.4]{Her:94}.  We also have from~Theorem~\ref{thm:Herver-improved} (cf.~\cite[Thm.~2.3]{Her:94}), non-negative functions $g_1,\dots,g_r$ in $C^\alpha(X)$ such that
$$
  g_i=1  \ \ \text{on} \  \mathrm{supp} \mu_i, \quad \int g_i \, d\mu_j =0 \ \ \text{if $i\not = j$} \quad \text{and} \quad \sum_{i=1}^r g_i=1.
$$
Taking into account that $C^\alpha(X)$ is identified with the space of Lipschitz functions with the equivalent metric $d(\cdot,\cdot)^\alpha$ on $X$, according to~\cite[Prop.~3.1 and Prop.~2.2]{HH:01}, we immediately get as a consequence of Theorem~\ref{thmA} and \cite[Thm.~A' and B']{HH:01} the following results: 

\index{central limit theorem}
\begin{maincor}[Central limit theorem of mixed type] \label{maincor2}
  Under the assumption of Theorem~\ref{thmA}, if $\phi \in C^\alpha(X)$ satisfies $\int \phi\, d\mu_i=0$ for all $i=1,\dots,r$, then,  for any probability measure $\nu$ on $X$,
  $$
    \frac{1}{\sqrt{n}} \,S_n\phi(\omega,x) \eqdef  \frac{1}{\sqrt{n}} \sum_{i=1}^{n} \phi(f^i_\omega(x)) \longrightarrow  \sum_{i=1}^r \nu_i \, \mathcal{N}(0,\sigma_i^2) 
  $$
  in distribution with respect to $\bar{\nu}=\mathbb{P}\times\nu$, where $\nu_1+\dots+\nu_r=1$ with $\nu_i=\int g_i \, d\nu \geq 0$ and  
  {\(\sum_{i=1}^r \nu_i\, \mathcal N(0,\sigma_i^2)\) denotes the mixture (convex combination) of centered normal distributions with standard deviation}   
$$
     \sigma^2_i \eqdef \lim_{n\to\infty}   \int  (S_n\phi)^2  \, d\bar{\mu}_i \geq 0  \quad \text{and} \quad \bar{\mu}_i\eqdef \mathbb{P}\times \mu_i \qquad i=1,\dots,r.
$$
Moreover, $\sigma_i=0$ if and only if there exists $\psi \in C^\alpha(X)$ such that
$$\phi(f_\omega(x))=\psi(x)-\psi(f_\omega(x)) \quad  \text{for \  $\bar{\mu}_i$-a.e.~$(\omega,x)\in \Omega\times X$}.$$
\end{maincor}

Given $z\in [-\infty,\infty]$, we denote the cumulative distribution function  of the standard normal distribution $\mathcal{N}(0,1)$ by
$$
  \Phi(z)\eqdef \frac{1}{\sqrt{2\pi}} \int_{-\infty}^z e^{-\frac{s^2}{2}}\, ds.
$$

\index{Berry-Esseen theorem}
\begin{maincor}[Central limit theorem with rate of convergence] \label{maincor3}
  Under the assumption of Theorem~\ref{thmA}, if $\phi \in C^\alpha(X)$ satisfies $\int \phi\, d\mu_i=0$ and $\sigma_i=\sigma_i(\phi)>0$ for all $i=1,\dots,r$,  then for any  probability measure $\nu$ on $X$ satisfying $\int \phi \, d\nu>0$,  
  $$
    \sup_{u\in \mathbb{R}} \left|\bar{\nu}\left(\frac{1}{\sqrt{n}} \, S_n\phi  \leq u  \right)- \sum_{i=1}^r \nu_i \,\Phi\left(\frac{u}{\sigma_i}\right)\right| =O\left(n^{-1/2}\right) \quad \text{where} \ \ \bar{\nu}=\mathbb{P}\times \nu.
  $$
  \end{maincor}

\begin{rem}
   Corollaries~\ref{maincor2} and~\ref{maincor3} were stated for $\phi\in C^\alpha(X)$ with $\int \phi\, d\mu_i=0$ for all $i=1,\dots,r$. If this condition does not hold, then these theorems will apply to the function $\varphi=\phi-\sum_{i=1}^r \int \phi \,d\mu_i \cdot g_i$.
\end{rem}

\begin{rem}
    When $r=1$, there is a unique stationary measure $\mu$ and we obtain the classical CLT: for $\phi\in C^\alpha(X)$, $\frac{1}{\sqrt{n}}(S_n\phi-\int \phi d\mu_1)\to \mathcal{N}(0,\sigma^2)$ in distribution, with $O(n^{-1/2})$ speed if $\sigma>0$.
\end{rem}

\begin{rem}
The \emph{local central limit theorem}
 and \emph{renewal theorem} can also be derived as consequences of Theorem~\ref{thmA}, following the results of \cite[Thm.~C' and D']{HH:01}, under additional arithmetic assumptions on~$\phi$.
\end{rem}

\section{Large deviations} \label{ss:large-deviation} 

In what follows, we assume that $f$ is uniquely ergodic, i.e., $r=1$ in Corollary~\ref{rem:measure-mean-quasicompactness}. 
We denote by $\mu$ the unique $f$-stationary measure and $\sigma^2=\lim_{n\to\infty} \int (S_n\phi)^2\, d\bar{\mu}$ where $\bar{\mu}=\mathbb{P}\times \mu$. 
Hennion and Herv\'e~\cite[Thm.~E]{HH:01} 
obtained a sharp upper bound on the rate function for the large deviations
of the process $S_n\phi$ assuming that $\sigma>0$. If one does not care about sharp estimates, one can follow the approach described in~\cite[Rem.~5.2]{duarte2016lyapunov} to get just an upper bound on the decay of the deviation without assuming positive variance. Thus, as a consequence of 
Theorem~\ref{thmA}, we obtain the following: \index{large deviations}
\begin{maincor}[Large deviations] 
\label{maincor4}
Let $f$ be a uniquely ergodic mostly contracting random map under the assumption~\eqref{eq:integral_condition} and consider
$\phi \in C^\alpha(X)$ satisfying  $\int \phi \, d\mu=0$. Then, there exist constants $h>0$ and $C>0$ such that 
for any $\epsilon>0$ small,  {probability measure $\nu$ on $X$} and $n$ large enough, 
$$
\bar{\nu}(S_n\phi>\epsilon n) \leq  C e^{-n h\epsilon^2},  \qquad \bar\nu=\mathbb{P}\times \nu.
$$
If $\sigma>0$, there is a non-negative strictly convex function $c(\epsilon)$ vanishing only at $\epsilon=0$ such that 
$$
\lim_{n\to\infty} \frac{1}{n}\log \bar{\nu}(S_n\phi>\epsilon n)=-c(\epsilon). 
$$
Furthermore,  if either $f$ is mingled or $\nu=\mu$, for any   $\psi\geq 0$ in $C^\alpha(X)$ such that $\int \psi \,d\nu >0$,   
$$
\lim_{n\to\infty} \frac{1}{n} \log \int_{\{S_n\phi>\epsilon n\}  } \psi(f^n_\omega(x)) \,d\bar{\nu} =-c(\epsilon). 
$$ 
\end{maincor}



\begin{rem} Corollaries~\ref{maincor2}, \ref{maincor3} and~\ref{maincor4}  can be also  applied to more general observables~$\phi$. See~\S\ref{ss:limit-theorems-HH} for more details. 
\end{rem}

\section{Statistical stability} \label{ss:statistical-stability}
Let $X$ be a compact Riemannian manifold and consider the space\index{space of random maps!$\CE(X)$, $C^1$ random maps} $\CE(X)$ of $C^1$ differentiable random maps $f:\Omega \times X \to X$, up to $\mathbb{P}$-almost identification of fiber maps. That is, $f=g$ if and only if $f_\omega=g_\omega$ for $\mathbb{P}$-a.e.~$\omega\in \Omega$. For $r\in \{0,1\}$, we consider the metric 
 \begin{equation} \label{eq:C1distacia-random-maps}
 \D_{C^r}(f,g)=\int d_{C^r}(f_\omega,g_\omega) \, d\mathbb{P}, \qquad \text{for $f, g \in \CE(X)$}
 \index{metric structures!Cr distance@\(\D_{C^r}\), averaged \(C^r\)-distance on \(C^r_{\smash{\mathbb{P}}}(X)\)}
 \end{equation}
 where 
 $d_{C^r}$ denotes the $C^r$-distance\index{metric structures!Cr@\(d_{C^r}\), \(C^r\)-distance on $C^r$ transformations} 
 on the space of $C^1$ differentiable transformations of $X$. Notice that clearly $\D_{C^0}\leq \D_{C^1}$. Also observe that the notion of mostly contracting can be extended to the considered equivalent classes of random maps, since Lyapunov exponents do not change by $\mathbb{P}$-almost identification of fiber maps.  

\index{stationary measures!statistical stability}
\begin{mainprop} \label{prop:statistical-stability} Let $X$ be a compact Riemannian manifold and consider a mostly contracting random map $f\in \CE(X)$  satisfying~\eqref{eq:integral_condition}. Then, there exists a neighborhood $\mathcal{B}$ of $f$ in $(\CE(X),\D_{C^1})$ such that
the number of ergodic $f$-stationary probability measures is upper semicontinuous~in~$\mathcal{B}$. 
In addition,  
 if $f$ is uniquely ergodic, then
 \begin{enumerate}[leftmargin=1cm]
     \item every $g\in \mathcal{B}$ has a unique stationary measure $\mu_g$,
     \item $\mu_g$ varies continuously in the weak$^*$ topology with $g\in \mathcal{B}$,
     \item  for any  H\"older continuous function $\phi:X\to \mathbb{R}$, the map $
g \in \mathcal{B} \mapsto \int \phi \, d\mu_g$  is H\"older continuous. Moreover, there exist $C> 0$ and $\gamma\in(0,1]$ such that 
$$
   \left|\int \phi \, d\mu _g - \int \phi \, d\mu_h\right| \leq C  \D_{C^0}(g,h)^\gamma  \quad \text{for every $g,h\in \mathcal{B}$.}
$$
 \end{enumerate}
\end{mainprop}

We prove in Theorem~\ref{prop:statistical-stability-length} a slightly more general result for {not necessarily uniquely ergodic mostly contracting} Lipschitz random maps on metric spaces.  In our context, statistical stability typically refers to a scenario in which all ergodic stationary measures persist and vary continuously with small perturbations in the dynamics. Hence, under the assumption of the previous proposition, uniquely ergodic random maps exhibit statistical stability.  {Moreover, since the number of ergodic stationary measures is upper semicontinuous and Theorem~\ref{prop:statistical-stability-length} yields the statistical stability in the stratum where the number of stationary measures remains constant,  the persistence property holds on an open and dense subset of
the mostly contracting random maps. See
Theorem~\ref{prop:generic-statistical-stability}.}

Following~\cite{andersson2020statistical}, one can relax the concept of statistical stability, stating that a random map $f$ is \emph{weakly statistically stable} if, given any neighborhood $\mathbb{U}$ in the weak$^*$ topology of the stationary measures of $f$, there exists a neighborhood $\mathcal{U}$ of $f$ such that, for any $g\in \mathcal{U}$, every stationary measure of $g$ belongs to $\mathbb{U}$. We prove in Theorem~\ref{thm:Weak statistically stability} that every continuous random map is weakly statistically stable.

\section{Locally constant linear cocycles}  \label{ss:locally-constant}
{A further fundamental class} of mostly contracting random maps arises from the classical study of products of independent identically distributed random variables in the group $\mathrm{GL}(m)$\index{metric spaces!W@\( \operatorname{GL}(m)\), general linear group} of real invertible $m \times m$ matrices, with $m \geq 2$. 

Let $A:\Omega \to \mathrm{GL}(m)$ be a measurable map such that 
$$A(\omega) = A(\omega_0) \quad \text{for $\mathbb{P}$-a.e.~$\omega = (\omega_i)_{i \ge 0} \in \Omega$.} 
$$
This map is usually called a \emph{locally constant linear cocycle}\index{linear cocycles!locally constant linear cocycle} (over the shift $\sigma:\Omega \to \Omega$). For $\mathbb{P}$-a.e.~$\omega = (\omega_i)_{i \geq 0} \in \Omega$, define 
$$
A^0(\omega) = \mathrm{Id} \quad \text{and} \quad A^n(\omega) = A({\omega_{n-1}}) \cdot \ldots \cdot A({\omega_0}) \quad \text{for $n \geq 1$}.
$$
Following~\cite[Def.~I.2.1 and Def.~III.5.5]{BouLac:85}, by the standard $\mathbb{P}$-integrability assumption of $\log^{+} \|A\|$, we can define the \emph{Lyapunov exponents}\index{linear cocycles!Lyapunov exponents} $\lambda_1(A) \geq \ldots \geq \lambda_m(A)$ associated with $A$ inductively for $1 \leq p \leq m$ by 
$$
\sum_{i=1}^p  \lambda_i(A) = 
\lim_{n \to \infty} \frac{1}{n} \int \log \|\exterior{p} A^n(\omega)\| \, d\mathbb{P}.
$$
Here $\|\exterior{p} M\|$ denotes the supremum of $\|Mu_1 \wedge \dots \wedge Mu_p\|$ with $u_1 \wedge \dots \wedge u_p$ ranges over the unit vectors in the exterior power $\exterior{p} \mathbb{R}^m$ for a matrix $M \in \mathrm{GL}(m)$. 

We say that $A$ is \emph{quasi-irreducible}\index{linear cocycles!quasi-irreducible cocycle} if there is no  proper subspace $L$ of $\mathbb{R}^m$ that is invariant under $A(\omega)$ for $\mathbb{P}$-a.e.~$\omega \in \Omega$ and $\lambda_1(A|_L) < \lambda_1(A)$ where
\begin{equation} \label{eq:reducible-Lyapunov-exponent}
\lambda_1(A|_L) \eqdef \lim_{n \to \infty} \frac{1}{n} \int \log \big\|A^n(\omega) \big\arrowvert_{L}\, \big\| \, d\mathbb{P}.
\end{equation} 
Notice that this condition is weaker than the classical condition of irreducibility. Namely, $A$ is said to be \emph{irreducible}\index{linear cocycles!irreducible cocycle} if there is no proper subspace $L$ of $\mathbb{R}^m$ such that $A(\omega)L = L$ for $\mathbb{P}$-a.e.~$\omega \in \Omega$. Similarly, $A$ is said to be \emph{strongly irreducible}\index{linear cocycles!strongly irreducible cocycle} if there is no finite family of proper subspaces invariant under $A(\omega)$ for $\mathbb{P}$-a.e.~$\omega \in \Omega$. Hence, strong irreducibility implies irreducibility, which clearly implies quasi-irreducibility.

\begin{defi} \label{def:equator}
The \emph{equator}\index{linear cocycles!equator subspace} subspace $E$ is the sum of all $A$-invariant subspaces $W \subset \mathbb{R}^m$ such that $\lambda_1(A|_W) < \lambda_1(A)$. In other words, $E$ is the largest $A$-invariant subspace of $\mathbb{R}^m$ where the first Lyapunov exponent satisfies $\lambda_1(A|_E) < \lambda_1(A)$. Conversely, the \emph{pinnacle}\index{linear cocycles!pinnacle subspace} subspace $F$ is defined as the intersection of all $A$-invariant subspaces $W \subset \mathbb{R}^m$ for which $\lambda_1(A|_W) = \lambda_1(A)$.
\end{defi}

From this definition, it follows that both the equator and the pinnacle are $A$-invariant subspaces of $\mathbb{R}^m$. These terms reflect the geometric nature of these $A$-invariant subspaces: the equator defines a boundary of slower expansion, while the pinnacle captures the maximal expansion rate. 
Notably, $A$ is quasi-irreducible if and only if its equator $E$ is the trivial \mbox{subspace of~$\mathbb{R}^m$.}

Finally, consider the projective action of $\mathrm{GL}(m)$ on the space\index{metric spaces!projective space@\(P(\mathbb R^m)\), projective space} $P(\mathbb{R}^m)$ of lines $\hat{x} = \{\alpha x : \alpha \in \mathbb{R}\}$, $x \in \mathbb{R}^m \setminus \{0\}$, endowed with the natural angular distance. Associated with $A$, we consider the \emph{projective random map}\index{random maps!\(f_A\), projective random map} 
$$f_A: \Omega \times P(\mathbb{R}^m) \to P(\mathbb{R}^m), \quad    f_A(\omega, \hat{x}) = A(\omega) \hat{x}.$$

The following proposition  {characterize}  the main property in~\cite{page1982theoremes, page1989regularite} from which Le Page derived many limit theorems for random products of matrices.

\begin{mainprop} \label{prop:projective}
{Let $A:\Omega \to \mathrm{GL}(m)$ be a locally constant linear cocycle and assume that  $\log^+ \|A^{\pm1}\|$ is $\mathbb{P}$-integrable and $\lambda_1(A) > \lambda_2(A)$. Then, the following conditions are equivalent: 
\begin{enumerate}[leftmargin=1cm, label=(\arabic*)]
    \item $A$ is quasi-irreducible,
    \item $f_A$ is uniquely ergodic,
    \item $f_A$ is mostly contracting. 
\end{enumerate} 
Moreover, under the exponential moment 
condition, 
\begin{equation}
    \label{eq:exponential-moment-condition}
 \int \left(\max\{\|A(\omega)\|, \, \|A(\omega)^{-1}\|\}\right)^\beta \, d\mathbb{P}<\infty \quad \text{for some $\beta>0$},\index{integrability conditions!exponential moment for linear cocycle}
\end{equation}
the following properties are also equivalent to the previous ones:
\begin{enumerate}[leftmargin=1cm,resume,label=(\arabic*)]
    \item $f_A$ is (globally) contracting on average,
    \item the Koopman operator associated with $f_A$ has spectral gap \\ on $C^\alpha(P(\mathbb{R}^m))$ for any $\alpha>0$ small enough. 
\end{enumerate}
Furthermore, the exponential moment and (5) also imply that $f_A$ is proximal.}
\end{mainprop}

{We show that the simplicity of the first Lyapunov exponent is a requirement to be mostly contracting, cf.~Lemma~\ref{prop:mostly-simplicity}.  Hence, as a consequence of Proposition~\ref{prop:projective},  $f_A$ is mostly contracting if and only if $A$ is quasi-irreducible and $\lambda_1(A)>\lambda_2(A)$.}    



{In addition to the exponential moment condition~\eqref{eq:exponential-moment-condition}, the classical assumptions used to obtain the spectral gap and the limit laws for the projective action are strong irreducibility and proximality. In view of the results of Guivarc'h and Raugi~\cite{GR84,GR85} (see also Gol'dsheid and Margulis~\cite{GM89}), under strong irreducibility, the proximality of \(f_A\) is equivalent to the simplicity of the maximal Lyapunov exponent, that is, to \(\lambda_1(A)>\lambda_2(A)\). Since strong irreducibility implies quasi-irreducibility, Proposition~\ref{prop:projective} shows that one can replace these classical irreducibility assumptions by the weaker hypothesis of quasi-irreducibility and simplicity of the first Lyapunov exponent. 

}

The matrix examples also admit natural multidimensional extensions to Grassmannians and flag manifolds. Fix \(1\le k\le m-1\). The cocycle \(A:\Omega\to \mathrm{GL}(m)\) induces a random map on the Grassmannian \(\mathrm{Gr}_k(\mathbb{R}^m)\) by
$(\omega,E)\mapsto A(\omega)E$.
Via the Pl\"ucker embedding
$\mathrm{Gr}_k(\mathbb{R}^m)\hookrightarrow \mathrm{P}(\exterior{k}\mathbb{R}^m)$,
this action is the restriction of the projective random map associated with the cocycle \(\exterior{k} A\). If \(\exterior{k} A\) is quasi-irreducible and $\lambda_k(A)>\lambda_{k+1}(A)$,
then
$
\lambda_1(\exterior{k} A)-\lambda_2(\exterior{k} A)=\lambda_k(A)-\lambda_{k+1}(A)>0$.
Thus, Proposition~\ref{prop:projective} applied to \(\exterior{k} A\) shows that the induced random map on \(\mathrm{Gr}_k(\mathbb{R}^m)\) is mostly contracting.  In particular, if \(A\) has simple Lyapunov spectrum and \(\exterior{k} A\) is quasi-irreducible for every \(k=1,\dots,m-1\), then each Grassmannian action is mostly contracting. Consequently, the diagonal action on
\(
\prod_{k=1}^{m-1}\mathrm{Gr}_k(\mathbb{R}^m)
\)
is mostly contracting, and therefore so is its restriction to the full flag manifold, that is, to the invariant closed subset
\[
\mathcal{F}(\mathbb{R}^m)=\{(E_1,\dots,E_{m-1})\in \prod_{k=1}^{m-1}\mathrm{Gr}_k(\mathbb{R}^m): E_1\subset \cdots \subset E_{m-1}\}.
\]

\subsection{Central limit theorem and large deviations}  \label{ss:CLT-linear}
Bellman~\cite{bellman1954limit} initiated the general theory on the limiting behavior of systems affected by noncommutative effects through the study of random products of matrices. Building on Bellman's work, Furstenberg and Kesten~\cite{FK:60} reached similar results on the central limit theorem for the entries of certain random products of positive matrices by analyzing the first Lyapunov exponent. Subsequently, Le Page~\cite{page1982theoremes} advanced the theory further by establishing the central limit theorem and large deviations for locally constant linear cocycles $A:\Omega \to \mathrm{GL}(m) $ under the assumptions that $A$ satisfies a finite exponential moment condition, is strongly irreducible, and  $f_A$ is proximal.  As recalled above, later, the works of Guivarc'h and Raugi~\cite{GR84,GR85}, as well as Gol'dsheid and Margulis~\cite{GM89},  clarified these assumptions. Specifically, they showed that under the assumption of strong irreducibility, $f_A$ is proximal if and only if $\lambda_1(A) > \lambda_2(A)$. Moreover, this is equivalent to the \emph{contracting property}\index{linear cocycles!contracting property} of $A$, meaning that there exists a sequence $g_n$ in the semigroup $\Gamma_A$ generated by the support of the associated distribution $\nu_A=A_* \mathbb{P}$ on $\mathrm{GL}(m)$ such that $\|g_n\|^{-1} g_n$ converges to a rank-one endomorphism. These conditions also imply that $\Gamma_A$ has an unbounded image in the projective general linear group~$\mathrm{PGL}(m)$.   
See~\cite[Prop.~4.4, Thm~6.1 of Chap.~III and Lem.~5.3, Cor.~8.6 in Chap.~V]{BouLac:85}.

In light of all of these relations, the following result extends the classical limit laws to the broader quasi-irreducible setting. Its proof relies on Proposition~\ref{prop:projective} together with the approach of Hennion and Herv\'e~\cite[{\S}X.4]{HH:01}.



\index{central limit theorem}
\index{large deviations}

\begin{mainprop} \label{mainprop:CLT-LD} Let $A:\Omega \to \mathrm{GL}(m)$ be a locally constant linear cocycle such that 
\begin{enumerate}[leftmargin=1cm]
    \item  $A$ satisfies the exponential moment condition~\eqref{eq:exponential-moment-condition},
    \item $\lambda_1(A)>\lambda_2(A)$, and
     \item $A$ is quasi-irreducible.
\end{enumerate}  Then there exists $\sigma \geq  0$ such that  for any unit vector $x\in \mathbb{R}^m$, 
\begin{equation} \label{eq:CLT-point-intro}
\frac{\log \|A^n(\omega)x\| - n\lambda_1(A)}{ \sqrt{n}}  
\xrightarrow[n\to +\infty]{\rm{law}}
  \mathcal{N}(0,\sigma^2)    
\end{equation} 
and there are  positive numbers $C$ and $h$, which do not depend on $x$,  such that for any $\epsilon>0$ small enough, one can find $n_0=n_0(\epsilon)\in \mathbb{N}$ for which
\begin{equation} \label{eq:LD-point-intro}
  \mathbb{P}\left( \big|\frac{1}{n}\log \|A^n(\omega)x\|-\lambda_1(A) \big| >\epsilon \right) \leq  C e^{-n h\epsilon^2}  \quad \text{for all $n\geq n_0$}.
\end{equation} 
Moreover,~\eqref{eq:CLT-point-intro} and~\eqref{eq:LD-point-intro} are also satisfied for $\log \|A^n(\omega)\|$  instead of $\log \|A^n(\omega)x\|$. 
\end{mainprop}


At first glance, the previous literature contains examples that suggest the impossibility of extending the central limit theorem to irreducible cocycles {that are not strongly irreducible}. See examples of Furstenberg and Kesten~\cite[Example~2]{FK:60}, and Benoist and Quint~\cite[Example~4.15]{BQ16}. These examples involve cocycles in $\mathrm{SL}(2)$ with Lyapunov exponents $\lambda_1(A)=\lambda_2(A)=0$. However, the previous result demonstrates that an extension beyond the classical condition of strong irreducibility is indeed possible when $\lambda_1(A)>\lambda_2(A)$. In fact, the proposition below shows that contrary to what these examples might suggest, in dimension $m=2$, we can obtain a result without the assumption of (quasi)irreducibility. 

\begin{mainprop} \label{mainprop:CLT-LD2} Let $A:\Omega \to \mathrm{GL}(2)$ be a locally constant linear cocycle such that 
\begin{enumerate}[leftmargin=1cm]
    \item  $A$ satisfies the exponential moment condition~\eqref{eq:exponential-moment-condition},
     \item $\lambda_1(A)>\lambda_2(A)$.
\end{enumerate} 
Then  there exists $\sigma \geq  0$ such that
\[ 
\frac{\log \|A^n(\omega)\| - n\lambda_1(A)}{ \sqrt{n}}  
\xrightarrow[n\to +\infty]{\rm{law}}
  \mathcal{N}(0,\sigma^2)
 \]
and there are  positive numbers $C$ and $h$  such that for any $\epsilon>0$ small enough, one can find $n_0=n_0(\epsilon)\in \mathbb{N}$ for which
\begin{equation*} 
  \mathbb{P}\left( \big|\frac{1}{n}\log \|A^n(\omega)\|-\lambda_1(A) \big| >\epsilon \right) \leq  C e^{-n h\epsilon^2}  \quad \text{for all $n\geq n_0$}.
\end{equation*} 
\end{mainprop}

A more uniform large deviation estimate was previously obtained by Duarte and Klein~\cite{duarte2020large} for finite-valued  locally constant $\mathrm{GL}(2)$-cocycles. Proposition~\ref{mainprop:CLT-LD2} holds for not necessarily compactly supported cocycles under the finite exponential moment condition.

The previous results establish a central limit theorem but do not guarantee that the limiting normal distribution is non-degenerate. 
Additional conditions are necessary to ensure that the variance is strictly positive. A positive variance not only reinforces the central limit theorem but also enables sharper results on the rate of convergence, such as the Berry-Esseen theorem and optimal estimates for large deviations. To state this result, given a cocycle $A:\Omega \to \mathrm{GL}(m)$  and an $A$-invariant subspace $F$ of $\mathbb{R}^m$, we denote by $A|_F$ the restriction of $A$ to $F$ which can be thought of as a $\mathrm{GL}(F)$-valued cocycle. 

\index{Berry-Esseen theorem}

\begin{mainprop} \label{mainprop:BE} Let $A:\Omega \to \mathrm{GL}(m)$ be a locally constant linear cocycle and denote by $F$ its pinnacle subspace on $\mathbb{R}^m$. 
{Under the assumptions of Proposition~\ref{mainprop:CLT-LD},  
$$
 \sigma=0 \iff  \dim F=1 \ \ \text{and} \ \  \|A(\omega)|_F\| \ \ \text{is constant \(\mathbb{P}\)-a.s.}
$$
In particular, if $\sigma>0$, 
}
then for any unit vector $x\in \mathbb{R}^m$, 
 \begin{equation}\label{eq:CLT12}
        \sup_{u \in \mathbb{R}} \left|\mathbb{P}\left(\frac{\log \|A^n(\omega)x\| - n\lambda_1(A)}{\sigma \sqrt{n}} \leq u  \right)- \frac{1}{\sqrt{2\pi}} \int_{-\infty}^u e^{-\frac{s^2}{2}}\, ds\right| = O\left(n^{-1/2}\right)
  \end{equation}
and
\begin{equation} \label{eq:LD12}
    \lim_{n \to \infty} \frac{1}{n} \log \mathbb{P}\left( \frac{1}{n} \log \|A^n(\omega)x\|-\lambda_1(A) > \epsilon\right) = -c(\epsilon)
\end{equation}
where $c(\cdot)$ is a non-negative, strictly convex function vanishing only at $\epsilon = 0$. 
Moreover,~\eqref{eq:CLT12} and~\eqref{eq:LD12} are also satisfied for $\log \|A^n(\omega)\|$  instead of~$\log \|A^n(\omega)x\|$. 
\end{mainprop}

It is evident that if $A$ is irreducible, then the pinnacle subspace $F$ of $A$ is $\mathbb{R}^m$. Thus, condition (iv) in the previous proposition immediately holds under the classical condition of strong irreducibility. Consequently, this proposition extends the classical results established by Le Page~\cite{page1982theoremes} to the broad class of quasi-irreducible cocycles. The following result shows the case of cocycles in $\mathrm{GL}(2)$, where we can drop the quasi-irreducibility assumption again.

\begin{mainprop} \label{mainprop:BE2} Let $A:\Omega \to \mathrm{GL}(2)$ be a locally constant linear cocycle
and denote by $E$ and $F$ its (possibly trivial) equator and pinnacle subspaces respectively. {Under the assumptions of Proposition~\ref{mainprop:CLT-LD2},
\[
\sigma=0
\iff
\left\{
\begin{aligned}
&\text{$\dim F=1$ and $\|A(\omega)|_F\|$ is $\mathbb P$-a.s.~constant or}  \\
&\text{$\dim E=1$ and $\|A(\omega)|_{\mathbb{R}^m/ E}\|$  is $\mathbb P$-a.s.~constant. }
\end{aligned}
\right.
\]
}
In particular, if $\sigma>0$, then
 \begin{equation*}
        \sup_{u \in \mathbb{R}} \left|\mathbb{P}\left(\frac{\log \|A^n(\omega)\| - n\lambda_1(A)}{\sigma \sqrt{n}} \leq u  \right)- \frac{1}{\sqrt{2\pi}} \int_{-\infty}^u e^{-\frac{s^2}{2}}\, ds\right| = O\left(n^{-1/2}\right).
  \end{equation*}
 \end{mainprop}

{An optimal large deviation principle for \(\log\|A^n(\omega)\|\) with a strictly convex rate function is not expected in the generality of Proposition~\ref{mainprop:BE2}; see Example~\ref{exap:LD}.} Assuming strong irreducibility and that semigroup generated by the support of $\nu_A=A_*\mathbb{P}$ has an unbounded image in $\mathrm{PGL}(m)$, Benoist and Quint~\cite{BQ16} established the central limit theorem with an optimal finite second-order moment condition. They also removed assumption (ii) regarding the gap between the first and second Lyapunov exponents. This follows from an algebraic lemma in~\cite[Lemma~4.3]{BQ16}, which allows one to reduce to a strongly irreducible action of $A$ in a quotient space with a proximal projective action (i.e., satisfying condition (ii)),
where we can compare the norms.

We also highlight the Berry-Esseen type bounds for $\log \|A^n(\omega)\|$ shown in Proposition~\ref{mainprop:BE}. Under the classical assumptions of strong irreducibility and proximality, Cuny, Dedecker, and Jan~\cite{CDJ17} first observed a Berry-Esseen bound of order $n^{-1/4} \sqrt{\log n}$ under a finite polynomial moment of order 3. A bound of order $n^{-1/2} \log n$ was recently established in~\cite{XGL21} when $A$ satisfies a finite exponential moment. Finally, an optimal rate of order $n^{-1/2}$ was achieved in~\cite{CDMP23} under a polynomial moment of order 4 in this setting of classical assumptions. 

In light of these observations, the following question naturally arises:
\begin{question}
    Is it possible to replace the finite exponential moment assumption in the previous results with a finite polynomial moment assumption?
\end{question}

Finally, under assumptions (i) and (ii) of the previous results, Proposition~\ref{prop:projective} gives that $A$ is quasi-irreducible if and only if $f_A$ is uniquely ergodic. Thus, since we are not assuming quasi-irreducibility in
Propositions~\ref{mainprop:CLT-LD2} and~\ref{mainprop:BE2}, these results go beyond the uniquely ergodic case (and spectral gap). In fact, in Theorem~\ref{thm:CLT-linear-general} and Corollary~\ref{cor:CLT-geral}, we obtain more general results from which all the previous propositions follow.

\subsection{Continuity of Lyapunov exponents} \label{ss:continuity-intro}
\index{Lyapunov exponents!continuity}
Let $\mathrm{L}_\mathbb{P}(m)$\index{space of linear cocycles!locally constant linear@\(\mathrm{L}_\mathbb{P}(m)\), locally constant} be the set of locally constant linear measurable cocycles $A:\Omega\to \mathrm{GL}(m)$ up to $\mathbb{P}$-almost identification and satisfying that $\log^+ \|A^{\pm 1}\|$ is $\mathbb{P}$-integrable. 
Consider the metric
$$
  \Dp(A,B)=\int \|A(\omega)-B(\omega)\| \, d\mathbb{P} \quad \text{for $A,B\in \mathrm{L}_\mathbb{P}(m)$}.
  \index{metric structures!2average distance@\(\Dp\), averaged one-sided distance on \(\LP(m)\)}
$$
We also denote by $A^T$ the transposed cocycle\index{linear cocycles!\(\lambada\)transposed@\(A^T\), transposed cocycle} given by $A^T(\omega)=A(\omega)^T$. 
As a consequence of the uniform Kingman subadditive ergodic theorem for Markov operators (Theorem~\ref{Kingmanuniform}), we get that 
the first Lyapunov exponent $\lambda_1(A)$ can be expressed as the so-called Furstenberg integral formula:\index{Lyapunov exponents!Furstenberg integral formula}
$$
\lambda_1(A)=\int \log \frac{\|A(\omega)x\|}{\|x\|} \, d\mathbb{P}d\mu
$$
for some ergodic stationary measure $\mu$ of the projective random map $f_A$. See Theorem~\ref{cor:Furstenberg}. This representation allows us to prove the following result as a consequence of the weak statistical stability results of $f_A$. 

\begin{mainprop} \label{mainthm:continuity}
Let $A\in \LP$ be a locally constant linear cocycle satisfying the exponential moment condition~\eqref{eq:exponential-moment-condition}. If either $A$ or $A^T$ is quasi-irreducible, then  
$\lambda_1(\cdot)$ 
is continuous at $A$ in $(\mathrm{L}_\mathbb{P}(m),\Dp)$.  
\end{mainprop}

If $A$ has at most one non-trivial invariant subspace, then, by Corollary~\ref{cor:atmostone}, $A$ or $A^T$ is quasi-irreducible. Thus, as a consequence of the above proposition, $A$ is a continuity point of~$\lambda(\cdot)$.  This consequence, as well as a similar version of the above proposition, was also obtained in~\cite{furstenberg1983random, hennion1984loi} under a uniform integrability assumption more restrictive than our condition~\eqref{eq:exponential-moment-condition}.  

Moreover, Viana and Bocker~\cite{bocker2017continuity} in dimension $m=2$  and more recently Avila, Eskin and Viana~\cite{AEV23} in any dimension showed that $\lambda_1(\cdot)$ is always continuous in the space of compactly supported distributions in $\mathrm{GL}(m)$. That is, in the space $\mathcal{P}_{ \hspace{-0.7mm}\scriptscriptstyle\rm cpt}(m)$\index{space of probability measures on \(\operatorname{GL}(m)\)!\(\mathcal{P}_{ \hspace{-0.7mm}\scriptscriptstyle\rm cpt}(m)\),  compactly supported} of probability measures $\mu$ on $\mathrm{GL}(m)$ whose support $\mathrm{supp} \, \mu$   is compact,  equipped with the smallest topology $\mathcal{T}$ that contains both the weak$^*$ topology and the pull-back of the Hausdorff topology under the support map $\mu \mapsto \mathrm{supp}\,\mu$. The topology $\mathcal{T}$ in $\mathcal{P}_{\hspace{-0.7mm}\scriptscriptstyle\rm cpt}(m)$ is metrizable
by considering
$$
\delta_{\mathcal{T}}(\nu_1,\nu_2)=\delta_W(\nu_1,\nu_2)+\delta_H(\mathrm{supp}\,\nu_1, \mathrm{supp}\, \nu_2), 
\quad \text{for $\nu_1,\nu_2\in \mathcal{P}_{ \hspace{-0.7mm}\scriptscriptstyle\rm cpt}(m)$}
\index{metric structures!3matrix distance distribution topology distance@\(\delta_{\mathcal T}\), Wasserstein-Hausdorff distance}
$$
where 
$\delta_W$ and $\delta_H$ are, respectively, the Wasserstein distance\footnote{
{$\delta_W(\nu_1,\nu_2)=\sup \left\{ 
\big|\int \varphi \, d\nu_1- \int \varphi \, d\nu_2 \big|\colon \text{$\varphi$ is 1-Lipschitz on $(\mathrm{GL}(m),\delta)$} 
\right\}, \ \nu_1,\nu_2\in \mathcal{P}_{ \hspace{-0.7mm}\scriptscriptstyle\rm cpt}
(m).
$} \index{metric structures!3matrix distance distribution topology distance W@\(\delta_W\), Wasserstein distance induced by \(\delta\)}
}
and the Hausdorff distance\footnote{$\delta_H(K_1,K_2)=\inf\left\{r>0: K_1 \subset B_r(K_2) \ \text{and} \ K_2 \subset B_r(K_1) \right\}$ where  $K_1$, $K_2$ are non-empty compact sets in $\mathrm{GL}(m)$ and 
the neighborhoods $B_r(K_1)$ and $B_r(K_2)$ are with respect to the distance~$\delta$.\index{metric structures!3matrix distance distribution topology distance Hausdorff distance@\(\delta_H\), Hausdorff distance induced by \(\delta\)}} with respect to the distance
$$\delta(g_1,g_2)=\|g_1-g_2\|+\|g_1^{-1}-g_2^{-1}\|, \quad  g_1,g_2\in \mathrm{GL}(m).\index{metric structures!3matrix distance@\(\delta\), matrix distance on \(\operatorname{GL}(m)\)}$$
In our context,  one can associate with  $A\in \mathrm{L}_\mathbb{P}(m)$ the  distribution \index{linear cocycles!associated distribution@ \(\nu_A=A_*\mathbb P\), associated distribution}$\nu_A=A_*\mathbb{P}$ in $\mathrm{GL}(m)$.  
Denote by $\LPC$\index{space of linear cocycles!Z@\(\LPC\), compactly supported} the subset of cocycles $A$ in $\LP$ whose distribution $\nu_A$ is compactly supported.  The restriction of the metric $\delta_{\mathcal{T}}$ to $\LPC$ 
is compatible with the distance 
$$
\Delta(A,B)=\Dpm(A,B)+ \delta_H(\mathrm{supp}\,\nu_A, \mathrm{supp}\, \nu_B) \quad \text{for $A,B\in \LPC$}\index{metric structures!compact-support distance@\(\Delta\), average--Hausdorff distance on \(\LPC\)}
$$
 where 
 $$
  \Dpm(A,B)=\int \delta(A(\omega),B(\omega)) 
  \, d\mathbb{P}  \quad \text{for $A,B\in \mathrm{L}_\mathbb{P}(m)$}. \index{metric structures!2averaged two-sided distance@\(\Dpm\), averaged distance on \(\LP\)}
$$
The results from~\cite{bocker2017continuity,AEV23} conclude the continuity of  $\lambda_1:(\LPC,\Delta) \to \mathbb{R}$.
Note that the exponential moment condition~\eqref{eq:exponential-moment-condition} is immediately satisfied for any cocycle $A\in \LPC$. Thus,   
Proposition~\ref{mainthm:continuity} improves these previous results in the quasi-irreducible case by considering the distance $\Dp$ instead of $\Delta$.  At the same time, one cannot extend the
continuity statement unconditionally to the whole space \((\LP,\D)\) under the
moment condition~\eqref{eq:exponential-moment-condition}. The discontinuity
example of S\'anchez and Viana~\cite{SV:20} can be realized by locally constant cocycles
 \(A_n\to A\) in \((\LP,\D)\) while
\(\lambda_1(A_n)\not\to\lambda_1(A)\).


\subsection{H\"older continuity of Lyapunov exponents} \label{ss:holde-continuty-intro} \index{Lyapunov exponents!H\"older continuity}
The following result establishes the H\"older continuity of $\lambda_1(\cdot)$, assuming simplicity of the first Lyapunov exponent. 

\begin{mainprop} \label{mainprop:linear-exponent-quasi}
    Let $A\in \LP$ be a  locally constant linear cocycle satisfying the exponential moment condition~\eqref{eq:exponential-moment-condition} and $\lambda_1(A)> \lambda_2(A)$. If either $A$ or $A^T$ is quasi-irreducible, then there is a neighborhood $\mathcal{B}$ of $A$ in $(\mathrm{L}_{\mathbb{P}}(m),\Dpm)$ such that the map $$B\in \mathcal{B} \mapsto \lambda_1(B)$$ is H\"older continuous. More specifically, there exist $C>0$ and $\gamma \in (0,1]$ such that 
    \begin{equation}
        \label{eq:holder}
    |\lambda_1(B)-\lambda_1(B')|\leq C \cdot \Dp(B,B')^\gamma \quad \text{for every $B,B' \in \mathcal{B}$.}
    \end{equation}
\end{mainprop}

As before, if \(A\) has at most one invariant subspace, then \(\lambda_1(\cdot)\) is H\"older continuous, provided that the first Lyapunov exponent is simple and the exponential moment condition holds. Similar results were first obtained by Le Page~\cite{page1989regularite} for strongly irreducible cocycles, and by Duarte and Baraviera~\cite{baraviera2019approximating} for continuous quasi-irreducible cocycles over a full shift on a compact metric space endowed with a uniform metric.\footnote{\ \ \(\mathrm{D}^+(A,B)=\|A-B\|_\infty\) for \(A,B\in \LPC\). Note that \(\Dp(A,B) \leq \mathrm{D}^+(A,B)\).} Moreover, Tall and Viana~\cite{tall2020moduli} proved the H\"older continuity of \(\lambda_1(\cdot)\) for locally constant linear cocycles in dimension \(m=2\) with simple spectrum, in the space of compactly supported probability distributions on \(\mathrm{GL}(2)\), with respect to the distance \(\delta_\mathcal{T}\) introduced above. Proposition~\ref{mainprop:linear-exponent-quasi} improves these results by considering average distances instead of uniform metrics and by allowing measurable cocycles that are not compactly supported. In the compactly supported case, we can further extend Proposition~\ref{mainprop:linear-exponent-quasi} to a class of cocycles that are not necessarily quasi-irreducible. To do so, we introduce the following definition. 

By abuse of notation, we continue to write \(E\) for the projective space \(P(E)\).

\begin{defi} \label{def:almost} 
A locally constant linear cocycle \(A \in \LP\) is called \emph{almost-irreducible}\index{linear cocycles!c@almost-irreducible cocycle} if its equator \(E\) is a \emph{strict repeller}; that is, there exist a compact set \(X \subset P(\mathbb{R}^m) \setminus E\) and an open set \(V \subset X\) such that \(A(\omega)X \subset V\) for \(\mathbb{P}\)-a.e.~\(\omega \in \Omega\).
\end{defi}

As mentioned, $A$ is quasi-irreducible if and only if the equator $E$ is trivial. In such a case, we take $X = V = P(\mathbb{R}^m)$ in the above definition, seeing quasi-irreducible cocycles as a particular case of almost-irreducible cocycles.  Nonetheless, the class of almost-irreducible cocycles also includes other cocycles that are not quasi-irreducible, as we discuss below. For these cocycles, $E$ is non-trivial and $X$ must be a compact subset of the projective space that strictly attracts the dynamics away from the equator. Under the assumption of simplicity of the first Lyapunov exponent, the existence of a compact invariant set outside the equator is characterised in Proposition~\ref{claim:F-invariant}. However, the definition of almost-irreducible cocycle imposes an even stronger condition: the dynamics is attracted into an open set $V \subset X$.  \\

\begin{exap} Take $\Omega=\{0,1\}^\mathbb{N}$, $\mathbb{P}=\frac{1}{2}(\delta_{0}+ \delta_{1})$ and $A\in \mathrm{L}_{\mathbb{P}}(2)$ given by
$$
A(\omega)=\begin{pmatrix} a & 0 \\ 0 & b \end{pmatrix}
\quad \text{with $|a|<|b|$ \ \ for every $\omega\in\Omega$. }
$$
The cocycle $A$ has the invariant  subspaces $E$ and $F$ generated by $(1,0)$ and $(0,1)$ respectively. Since $$\lambda(A|_E)=\log |a| < \log |b|=\lambda(A|_F)=\lambda_1(A),$$ the cocycle $A$ is not quasi-irreducible.  However, $A$ is almost-irreducible. To see this, let $X$ be the projective compact subset corresponding to the cone $\{(u,v)\in E \oplus F:  \|u\|\leq \frac{1}{2} \|v\| \}$ which, by abuse of notation, we also denote by~$X$.  It is not difficult to see that $A(\omega)X$ is strictly contained in $X$ since $|a|/|b| <1$. Thus,  
Definition~\ref{def:almost} holds. 
\end{exap}

\begin{mainprop} \label{mainprop:linear-exponent-almost}
   Let $A\in \LPC$ be a  locally constant linear cocycle such that $\lambda_1(A)> \lambda_2(A)$.  
   If either $A$ or $A^T$ is almost-irreducible, then there exists a neighborhood $\mathcal{B}$ of $A$ in $(\LPC,\Delta)$ such that the map $$B\in \mathcal{B} \mapsto \lambda_1(B)$$ is H\"older continuous. Moreover,~\eqref{eq:holder} holds. 
\end{mainprop}

The results  in~\cite{tall2020moduli} imply that the function $\lambda_1:(\mathrm{L}_{\smash{\mathbb{P}}}^{\hspace{-0.7mm}\raisebox{1.2pt}{\tiny \rm{cpt}}}(2),\Delta)\to \mathbb{R}$ is locally H\"older continuous on the set of cocycles with simple spectrum. Thus, the main contributions of the previous proposition are the improved estimate~\eqref{eq:holder} and the partial extension of the result to $m > 2$.

\begin{rem} Random matrix products are classically described by a distribution $\nu$ on $\mathrm{GL}(m)$. Here, we write $\nu=A_*\mathbb{P}$ where $A$ is a locally constant linear cocycle and $\mathbb{P}$ is a Bernoulli product probability measure with initial law $p$ on a measure space $(T,\mathscr{A})$. Thus, we have that $$\lambda_1(\nu)=\lambda_1(A,p).$$ 
In the finite-alphabet case $\nu=\sum_{i=1}^N p_i\delta_{A_i}$,
our estimates control the displacement of the matrices \(A_i\), with the
weights \(p_i\) fixed. Results such as those of~\cite{peres2006analytic,ADM26}  control
instead the dependence on the probability weights, in total variation, when
the support is fixed, proving an analytic dependence for any compactly supported cocycle. Hence, the two viewpoints are complementary: for nearby
pairs \((A,p)\) and \((B,q)\), 
\[
|\lambda_1(A,p)-\lambda_1(B,q)|
\leq
|\lambda_1(A,p)-\lambda_1(B,p)|+
|\lambda_1(B,p)-\lambda_1(B,q)|.
\]
The first term is controlled by the fixed-base estimates proved here, while the second belongs to the regularity theory in the probability variable.
\end{rem}

\subsection{Regularity of exponents for random matrix products}
\label{ss:wasserstein-distributions-intro}
\index{Lyapunov exponents!continuity}
\index{Lyapunov exponents!H\"older continuity}

As a consequence of the previous continuity propositions in averaged distances,
we also obtain distributional versions of these results in Wasserstein distances.
In what follows, we consider the distances
\[
   \delta^+(g,h)\eqdef \|g-h\|,
   \qquad
   \delta(g,h)
   \eqdef
   \|g-h\|+\|g^{-1}-h^{-1}\|,
   \qquad g,h\in \mathrm{GL}(m).\index{metric structures!3matrix distance@\(\delta^+\), one-sided matrix distance on \(\operatorname{GL}(m)\)}
\]
Let
$\mathcal P_1^+(m)\eqdef \mathcal P_1(\mathrm{GL}(m),\delta^+)$\index{space of probability measures on \(\operatorname{GL}(m)\)!\(\mathcal P_1^+(m)\), finite first moment for \(\delta^+\)}
 and $\mathcal P_1(m)\eqdef \mathcal P_1(\mathrm{GL}(m),\delta)$\index{space of probability measures on \(\operatorname{GL}(m)\)!\(\mathcal P_1(m)\), finite first moment for \(\delta\)} 
be the corresponding spaces of Borel probability measures $\nu$ with finite first
moment, i.e., denoting the corresponding metric on $\mathrm{GL}(m)$ by  $c(\cdot,\cdot)$, 
$$
  \int c(g,g_0) \, d\nu(g)<\infty\index{integrability conditions!finite first moment}
$$
for some, hence every, base poinp $g_0$ in $\mathrm{GL}(m)$. 
We denote by \(\delta_{\smash{W}}^+\) and \(\delta_W\) the corresponding Wasserstein distances on $\mathcal{P}^+_1(m)$ and $\mathcal{P}_1(m)$ respectively. Namely, for $c\in\{\delta^+,\delta\}$,  
$$
 c_W(\nu_1,\nu_2)\eqdef \sup \left\{ 
\big|\int \varphi \, d\nu_1- \int \varphi \, d\nu_2 \big|\colon 
\text{$\varphi$ is 1-Lipschitz on $(\mathrm{GL}(m),c)$} \right\}. \index{metric structures!3matrix distance distribution topology distance Wp@\(\delta_W^+\), Wasserstein distance induced by \(\delta^+\)}
$$
A distribution \(\nu=A_*\mathbb P\) is said to be quasi-irreducible
(resp. almost-irreducible) if the locally constant cocycle \(A\) is
quasi-irreducible (resp. almost-irreducible). Similarly, \(\nu^T\) denotes the
distribution induced by the transposed cocycle. Using the gluing lemma from
optimal transport and
Propositions~\ref{mainthm:continuity},~\ref{mainprop:linear-exponent-quasi}
and~\ref{mainprop:linear-exponent-almost}, we obtain the following
distributional form of the regularity results. First, recall that a function $g:(X,d) \to\mathbb{R}$ is said to be \emph{locally H\"older continuous} at a point $x_0$ if 
there exist a
\(d\)-neighborhood \( U \subset X \) of \(x_0\) and constants \(C>0\), 
\(\gamma\in(0,1]\) such that $|g(x)-g(x')|\leq C d(x,x')^{\gamma}$.

\begin{mainprop} \label{mainprop:regularity-distribution}
Let \(\nu_0\) be a probability measure on \(\mathrm{GL}(m)\) satisfying
\[
   \int
   \bigl(\max\{\|g\|,\|g^{-1}\|\}\bigr)^\beta \, d\nu_0(g)<\infty
   \qquad \text{for some } \beta>0 .
\]
Then the following statements hold:
\begin{enumerate}[label=\textup{(\arabic*)}, leftmargin=0.75cm]
\item
If \(\nu_0\in\mathcal P_1^+(m)\) and either \(\nu_0\) or
\(\nu_{\smash{0}}^T\) is quasi-irreducible, then the map 
\[ \lambda_1:(\mathcal P_1^+(m),\delta_W^+) \to \mathbb{R} \quad \text{is continuous at
\(\nu_0\).}\] 
\item If \(\nu_0\in\mathcal P_1(m)\) and \(\lambda_1(\nu_0)>\lambda_2(\nu_0)\), and either \(\nu_0\) or
\(\nu_0^T\) is quasi-irreducible,  then 
\[ \lambda_1:(\mathcal P_1(m),\delta_W) \to \mathbb{R} \quad \text{is locally H\"older continuous at \(\nu_0\).}\] 
\item
If \(\nu_0\in\mathcal P_{\hspace{-0.7mm}\scriptscriptstyle\rm cpt}(m)\),
\(\lambda_1(\nu_0)>\lambda_2(\nu_0)\), and either \(\nu_0\) or
\(\nu_0^T\) is almost-irreducible, then
\[
   \lambda_1:
   (\mathcal P_{\hspace{-0.7mm}\scriptscriptstyle\rm cpt}(m),
   \delta_{\mathcal T})
   \to \mathbb R \quad \text{is locally H\"older continuous at \(\nu_0\).}
\]
\end{enumerate}
\end{mainprop}

\section{Random maps of one-dimensional diffeomorphisms} \label{ss:one-dimensional}

As an application of the {{invariance principle}} developed in~\cite{Mal:17}, we can show that non-degenerate random walks on the semigroups of $C^1$ diffeomorphisms of the circle are also mostly contracting. {First, we say that a random map $f:\Omega\times X\to X$ has a \emph{common invariant probability measure}\index{common invariant probability measure} if there exists a probability measure $\mu$ on $X$ such that $(f_\omega)_*\mu=\mu$ for $\mathbb{P}$-a.e.~$\omega\in\Omega$.}. We also say that $f$ is \emph{quasi-symmetric}\index{random maps!quasi-symmetric random map} if the smallest closed semigroup of $(\mathrm{Homeo}(X),\circ)$ containing $\mathbb P$-a.e.~$f_\omega$ is actually a group. In particular, it holds if $f$ is a \emph{symmetric random map}\index{random maps!symmetric random map}, i.e., $f^{-1}$ has the same distribution as $f$.

\begin{mainprop} \label{cor:mostly-contraction-circle-cantor} \index{random maps!circle diffeomorphisms}
Let $f: \Omega \times \mathbb{S}^1 \to \mathbb{S}^1$ be a random map of $C^1$ diffeomorphisms of the circle {with no common invariant probability measure} such that {$\|\log |f'_\omega|\,\|_\infty$} is $\mathbb{P}$-integrable.   Then $f$ is mostly contracting.

As a consequence of Theorem~\ref{thmA}, 
if
\begin{equation} \label{eq:exponential-moment-circulo0}
     \int (\|f'_\omega\|_\infty)^\beta \, d\mathbb{P} <\infty \quad \text{for some $\beta>0$},
\end{equation}
the Koopman operator $P$ associated with $f$ is quasi-compact on $C^\alpha(\mathbb{S}^1)$ for every $\alpha>0$ small enough. Moreover, if $f$ is either proximal, minimal or quasi-symmetric, then $P$ has a spectral gap.   
\end{mainprop}

As a consequence of the previous proposition, we can also obtain the following result for random maps of a compact interval $\mathbb{I}$ of the real line:

\begin{mainprop} \label{cor:interval}\index{random maps!interval diffeomorphisms}
Let $f: \Omega \times \mathbb{I} \to \mathbb{I}$ be a random map of \(C^1\) diffeomorphisms of \(I\) onto their images. Assume that {$\|\log |f'_\omega|\,\|_\infty$} is $\mathbb{P}$-integrable and there are no $p, q \in \mathbb{I}$ such that $\{p, q\}$ is $f_\omega$-invariant for $\mathbb{P}$-a.e.~$\omega \in \Omega$.  
Then $f$ is mostly contracting and satisfies similar conclusions as in Proposition~\ref{cor:mostly-contraction-circle-cantor}.
\end{mainprop}

{Some results in~\cite{Mal:17} were extended in~\cite{MM23} to \(C^1\) diffeomorphisms of a Cantor set \(\mathbb K\subset \mathbb R\). Recall that, according to~\cite{MM23}, a \(C^1\) diffeomorphism of \(\mathbb K\) means a homeomorphism \(g:\mathbb K\to\mathbb K\) such that for every \(x\in \mathbb K\) there exist an open interval \(I_x\subset \mathbb R\) containing \(x\) and a \(C^1\) diffeomorphism
$ \tilde g_x:I_x\to \tilde g_x(I_x)$ 
satisfying
$$\tilde g_x|_{I_x\cap \mathbb K}=g|_{I_x\cap \mathbb K}.$$ Equivalently, by~\cite[Proposition~2.1]{malicet2023groups}, such maps are precisely the restrictions to \(\mathbb K\) of ambient \(C^1\) diffeomorphisms preserving \(\mathbb K\). In particular, the definition does not depend on the choice of a common ambient embedding for the fiber maps. With this convention, by a random map of \(C^1\) diffeomorphisms\index{random maps!Cantor diffeomorphisms} of \(\mathbb K\) we mean a random map \(f:\Omega\times \mathbb K\to \mathbb K\) such that \(f_\omega\) is a \(C^1\) diffeomorphism of \(\mathbb K\) for \(\mathbb P\)-a.e.~\(\omega\in\Omega\). As before, we obtain the following.}

\begin{mainprop} \label{prop:mostly-cantor}
Let $f:\Omega \times \mathbb{K} \to \mathbb{K}$ be a random map of $C^1$ diffeomorphisms of a Cantor set $\mathbb{K}$ where $\Omega = T^\mathbb{N}$ with $T$ finite.  {If $f$ has no common invariant probability measure},
then $f$ is mostly contracting. Moreover, the conclusions in Theorem~\ref{thmA} hold.
\end{mainprop}

In what follows, we present some applications of our results to the case of circle diffeomorphisms, but similar results also hold, with minor modifications, for interval diffeomorphisms onto their images.

\subsection{Central limit theorem and large deviations}
\label{ss:CLT-one-dimensional}
For a random map $f$ consisting of finitely many orientation-preserving circle homeomorphisms, Szarek and Zdunik ~\cite{SZ17} demonstrated the central limit theorem under the assumptions of minimality and the absence of a common invariant probability measure. According to~\cite{Mal:17}, if there is no probability measure invariant under every fiber map in the system, then the random map exhibits the local contraction property. Building on this, Gelfert and Salcedo~\cite{gelfert2024synchronization} extended the central limit theorem to random maps of finitely many circle homeomorphisms, assuming that $f$ is proximal and has the local contraction property. 
Given that both minimality and proximality imply the existence of a unique stationary measure, we extend these results to uniquely ergodic random maps of circle diffeomorphisms in the following theorem. Furthermore, the additional assumption of regularity allows us to obtain a stronger result concerning the rate of convergence in the central limit theorem. Similarly, for random maps not necessarily consisting of finitely many fiber maps but with a finite exponential moment, we derive large deviations for the finite-time Lyapunov exponent from the maximal Lyapunov exponent, thereby extending previous results in~\cite{gelfert2024synchronization} for iterated function systems (i.e., for random maps driven by a finitely supported probability measure).

\index{central limit theorem}
\index{Berry-Esseen theorem}

\begin{mainprop} \label{mainprop:CLT-circle}
Let $f: \Omega \times \mathbb{S}^1 \to \mathbb{S}^1$ be a random map of  $C^{1+\epsilon}$ diffeomorphisms of the circle with  $0<\epsilon\leq 1$	such that 
	\begin{itemize}
		\item  {$f$ has no common invariant probability measure},
		\item $f$ satisfies the exponential moment condition 
         \begin{equation} \label{eq:exponential-moment-circulo}
     \int \max\big\{\|f'_\omega\|_\infty, 
     \|(f^{-1}_\omega)'\|_\infty, |f'_\omega|_\epsilon \big\}^\beta \, d\mathbb{P}<\infty \quad \text{for some $\beta>0$};
    \end{equation}
  \item there is a unique $f$-stationary measure $\mu$. 
	\end{itemize}
	 Then, there exists $\sigma\ge 0$ such that for any $x\in \mathbb{S}^1$,
	\begin{enumerate}[itemsep=0.3cm]
		\item \textbf{Central limit theorem:} 
		$$\frac{\log |(f_\omega^n)'(x)|-n\lambda(\mu)}{\sqrt{n}}
  \xrightarrow[n\to +\infty]{\rm{law}}
  \mathcal{N}(0,\sigma^2).$$
  Moreover, if $f$ is minimal, then $\sigma>0$.
 		\item \textbf{Central limit theorem with rate of convergence:} if $\sigma>0$, then  
$$\qquad \quad \sup_{u\in\mathbb{R}}\left|\mathbb{P}\left( \frac{\log |(f_\omega^n)'(x)|-n\lambda(\mu)}{\sigma \sqrt{n}}\leq u\right) -\frac{1}{\sqrt{2\pi}} \int_{-\infty}^u e^{-\frac{s^2}{2}}\, ds\right| =O(n^{-1/2}).$$
		
\item \textbf{Large deviation:\index{large deviations}
 }There are  positive numbers $C$ and $h$, which do not depend on $x$, such that for any $\epsilon>0$ small enough, one can find $n_0=n_0(\epsilon)\in \mathbb{N}$ for which
		$$\mathbb{P}\left( \left|\frac{1}
  {n} \log |(f_\omega^n)'(x)|-\lambda(\mu) \right| >\epsilon \right) \leq  C e^{-n h\epsilon^2}.$$
Moreover, if $\sigma>0$, then 
\begin{equation*} 
    \lim_{n\to\infty}\frac{1}{n}\log \mathbb{P}\left( \left|\frac{1}{n}\log |(f_\omega^n)'(x)|-\lambda(\mu) \right| >\epsilon \right)=-c(\epsilon)
\end{equation*}
where $c(\cdot)$ is a non-negative strictly convex function vanishing only at $\epsilon=0$.		
	\end{enumerate}
  
\end{mainprop}

In dimension one, as well as for linear cocycles (see Theorem~\ref{cor:Furstenberg}), one can easily obtain the strong law of large numbers for Lyapunov exponents as a consequence of an extension due to Furstenberg~\cite[Lemma~7.3]{Fur:63} of Breiman's results (see Corollary~\ref{thm:Bierman-Furstenberg}). Note that this result does not require any contraction property.  

\index{strong law of large numbers}
\begin{mainprop} \label{mainprop:SLLN}  Let $f: \Omega \times \mathbb{S}^1 \to \mathbb{S}^1$ be a  uniquely ergodic random map  of  $C^{1}$ diffeomorphisms of the circle satisfying that $\|\log |f'_\omega|\,\|_\infty$ is $\mathbb{P}$-integrable. 
Then, for any $x\in \mathbb{S}^1$,  
		$$\lim_{n\to \infty} \frac{1}{n}
  \log |(f_\omega^n)'(x)|=\lambda(\mu)  \quad \text{for $\mathbb{P}$-a.e.~$\omega\in \Omega$}.
$$  
\end{mainprop}

\subsection{H\"older continuity of Lyapunov exponents}  \label{ss:holder-circle} \index{Lyapunov exponents!H\"older continuity}

For $r>0$, denote by $\DE{\,r}(\mathbb{S}^1)$ the set\index{space of random maps!\(\DE{\,r}(\mathbb{S}^1)\), circle $C^r$ diffeomorphisms} of random maps $f:\Omega\times \mathbb{S}^1 \to \mathbb{S}^1$ of $C^r$ diffeomorphisms of $\mathbb{S}^1$, up to $\mathbb{P}$-almost identification of fiber maps. We consider the  metric 
$$
  \D_{\smash{C^r}}^\pm(f,g)=\int d_{C^{\lfloor{r}\rfloor}}(f_\omega,g_\omega) + d_{C^{\lfloor{r}\rfloor}}(f_\omega^{-1},g_\omega^{-1}) + |f'_\omega-g'_\omega|_{r-\lfloor{r}\rfloor} \, d\mathbb{P}
  \index{metric structures!Cr distancecircle diffeomorphism distance@\(\D_{C^r}^{\pm}\), averaged distance on \(\DE{r}(\mathbb S^1)\)}
$$
for $f,g\in \DE{\,r}(\mathbb{S}^1)$. 
Note that $\D_{C^1} \leq \D_{\smash{C^1}}^\pm$  where $\D_{C^1}$ was introduced in~\eqref{eq:C1distacia-random-maps}.


\begin{mainprop} \label{mainprop:nolinear-exponent}
    Let $f\in \DE{r}(\mathbb{S}^1)$, $r>1$, be a uniquely ergodic random map with      
      {no common invariant probability measure}  and satisfying the exponential moment condition~\eqref{eq:exponential-moment-circulo}. Then, there exists a neighborhood $\mathcal{B}$ of $f$ in $(\DE{r}(\mathbb{S}^1),\D_{\smash{C^{r}}}^\pm)$ such that 
    every $g\in \mathcal{B}$ has a unique stationary measure $\mu_g$ which varies continuously in the weak$^*$ topology and
    the map $$g\in\mathcal{B} \mapsto \lambda(\mu_g)\in (-\infty,0)$$ is H\"older continuous.  
    Moreover, there exist $C>0$ and $\gamma \in (0,1]$ such that
$$
   |\lambda(\mu_g)-\lambda(\mu_h)|\leq C \cdot \D_{C^1}(g,h)^\gamma_{} \quad \text{for all $g,h\in \mathcal{B}$.}
$$ 
\end{mainprop}

{
The uniqueness assumption in Proposition~\ref{mainprop:nolinear-exponent}
serves only to present the statement in its simplest form. Later, in
Theorem~\ref{mainprop:nolinear-exponent-improved}, we prove a non-uniquely
ergodic version on strata where the number of ergodic stationary measures is
constant. As a consequence, for \(r>1\),
Corollary~\ref{cor:generic-lyapunov-circle} gives an open and dense set
of random circle diffeomorphisms in $(\DE{r}(\mathbb{S}^1),\D_{\smash{C^{r}}}^\pm)$ for which all stationary Lyapunov exponents
persist and vary H\"older continuously.
}

\section*{Organization of the  {monograph}}

The structure of this   {monograph}  is as follows.   {Chapter}~\S\ref{s:Kingman} {is devoted to proving} the~\emph{uniform Kingman's subadditive ergodic theorem} (Theorem~\ref{Kingmanuniform}) and \emph{Kingman's subadditive ergodic theorem} (Theorem~\ref{Kingman})  for Markov operators.  In \S\ref{s:mostly-contracting}, we discuss mostly contracting random maps, focusing on local contraction on average and quasi-compactness, proving Theorem~\ref{thmA}.
We also analyze the implications of global contraction, including proximality and unique ergodicity. \S\ref{s:statistical-stability}~addresses the robustness and statistical stability of mostly contracting random maps. Here, we prove Proposition~\ref{prop:statistical-stability} and show the weak statistical stability property for continuous random maps. In~\S\ref{s:linear}, we shift our attention to locally constant linear cocycles. We revisit the classical theory and prove Proposition~\ref{prop:projective}.
{In chapter~\S\ref{chap:CLT-linear} we study the limit theorems for the first Lyapunov exponent of linear cocycles, proving Proposition~\ref{mainprop:CLT-LD},~\ref{mainprop:CLT-LD2},~\ref{mainprop:BE} and~\ref{mainprop:BE2}. The continuity results for this exponent are developed in  Chapter~\S\ref{chap:continuity-linear} where we show Proposition~\ref{mainthm:continuity},~\ref{mainprop:linear-exponent-quasi}, \ref{mainprop:linear-exponent-almost} and~\ref{mainprop:regularity-distribution}.} In~\S\ref{s:examples}, we give the proof of Proposition~\ref{cor:mostly-contraction-circle-cantor}, \ref{cor:interval} and~\ref{prop:mostly-cantor}. 
\S\ref{s:circle} introduces the study of random maps of circle diffeomorphisms. We prove the H\"older continuity of the Lyapunov exponent (Proposition~\ref{mainprop:nolinear-exponent}) and present limit theorems for this class of maps (Proposition~\ref{mainprop:CLT-circle} and~\ref{mainprop:SLLN}). Finally, in~\S\ref{s:local-contraction}, we analyze the exponential local contraction property for general skew-products (Theorem~\ref{contraction}) and later for random dynamical systems, proving Theorem~\ref{cor:local-contraction}. We conclude by discussing Palis' global conjecture within the context of mostly contracting random maps. In Appendix~\ref{s:preliminar}, we state some classical preliminary lemmas on integration that will be used throughout the {monograph}. {Appendix~\ref{s:appendix-martingale-zero-one} contains the statements of some well-known martingale laws of random variables. Finally, in Appendix~\ref{s:appendix-feller-bernoulli} we compile from the literature and prove slightly more general versions of several fundamental results for random maps, which we used as a tool to prove some of our main theorems. }

\chapter{Kingman's subadditive ergodic theorems for Markov operators}
\label{s:Kingman}

\abstract{{
This chapter establishes Kingman's subadditive ergodic theorem for Markov
operators. After introducing \(P\)-subadditive sequences, we prove a uniform
theorem for Markov operators on $C(X)$ where $X$ is a compact metric space
in the
setting of upper semicontinuous extended real-valued observables. The resulting
formula identifies the asymptotic maximal growth with a variational supremum
over invariant probability measures, attained on ergodic measures. We then prove
a pointwise \(L^1\)-version, which gives almost sure convergence of normalized
\(P\)-subadditive sequences. These theorems provide the main operator-theoretic tool for the Lyapunov exponent formulas developed later in the monograph.}
}

\section{Subadditive processes}
\label{sec:P-subadditive-processes}

Let $P$ be an operator acting on a sequence $(\phi_n)_{n\geq 1}$ of extended real-valued measurable functions (or equivalence classes thereof) on~$X$.

\begin{defi} \label{def:Padditive} \index{subadditive processes!\(P\)-subadditive sequence}
\index{subadditive processes!\(P\)-additive sequence}
The sequence $(\phi_n)_{n\geq 1}$ is said to be $P$-subadditive if for every $m,n\geq 1$ we have
\begin{equation}\label{Qsub}
\phi_{n+m}\leq \phi_n+P^n\phi_m.
\end{equation}
Moreover, the sequence $(\phi_n)_{n\geq 1}$ is $P$-additive if equality holds in~\eqref{Qsub}.
\end{defi}


To illustrate, let \( f:X \to X \) be a measurable transformation which preserves a probability measure \( \mu \). Given the Koopman operator associated with \( f \), defined as \( P\phi = \phi \circ f \), the \( P \)-subadditive sequences correspond to the subadditive stochastic processes in the classical setting of Kingman's subadditive ergodic theorem.


\section{Uniform Kingman subadditive ergodic theorem} \label{sec:kigmanuniform}

The following result generalizes the uniform Kingman subadditive ergodic theorem to a broader class of Markov operators. In the context of subadditive processes associated with a measure-preserving transformation, this theorem was proved by S{\l}omczy\'nski~\cite{slomczynski1995continuous} and later independently rediscovered by Schreiber~\cite{schreiber1998growth}, Furman~\cite[Thm.~1]{Fur:97}, and Sturman and Stark~\cite[Thm.~1.5]{sturman2000semi}. Cao further extended it to random settings~\cite{cao2006growth}.
In the case of a measure-preserving transformation, the following result recovers the version of Morris~\cite{morris2013mather} by incorporating upper semicontinuous, extended real-valued functions. We now introduce the necessary notation to state the result.

Let $X$ be a compact metric space. A linear operator $P:C(X)\to C(X)$ is called a \emph{Markov operator} \index{operator!Markov operator on $C(X)$} if
\begin{enumerate}[label=(\roman*), leftmargin=1cm]
  \item $P$ is positive\index{operator!positive operator}, i.e., $P\varphi \geq 0$ for every $\varphi \geq 0$, and
  \item $P1_X=1_X$.
\end{enumerate}
From these conditions, it directly follows that \( \|P\|_{\mathrm{op}} = 1 \). Consequently, \( P \) is a contraction and a bounded operator.
Note that \( P \) can be naturally extended beyond \( C(X) \). For example, if $f$ is an upper semicontinuous extended real-valued function  on \( X \), then, since \( f = \inf\{g \in C(X): f\leq g\} \) by~Lemma~\ref{lem:Baire-teorema}, we may define \( P \) at \( f \) by \( Pf = \inf \{Pg: g \in C(X), f\leq g\} \). By duality, the Markov operator $P$ induces a linear operator $P^*$\index{operator!0@\(P^*\), dual operator} on the space of probability measures, satisfying the duality relation
$$
    \int P\phi \, d\mu = \int \phi \, dP^*\mu.
$$
A probability measure $\mu$ is $P^*$-invariant if $P^*\mu=\mu$. The set $\mathcal{I}$\index{space of measures!\(\mathcal I\), \(P^*\)-invariant probability measures} of \( P^* \)-invariant probability measures \index{operator!1@$P^*$-invariant probability measure} is convex, with its extreme points being the so-called \emph{ergodic} measures. {We denote by $\mathcal I_{\rm erg}$\index{space of measures!\(\mathcal I_{\rm erg}\), ergodic \(P^*\)-invariant measures} the set of ergodic $P^*$-invariant measures.}

{
Let $(\phi_n)_{n\geq 1}$ be a $P$-subadditive sequence. For any $\mu\in \mathcal{I}$, the $P^*$-invariance of $\mu$ implies
\[
\int \phi_{n+m} \, d\mu \leq  \int \phi_n \,d\mu + \int P^n\phi_m \, d\mu = \int \phi_n \,d\mu + \int \phi_m \, d\mu.
\]
Thus, the sequence $(\int \phi_n\, d\mu)_{n\geq 1}$ is subadditive. Moreover, if $\phi_1^+ \eqdef \max\{0,\phi_1\}$ is $\mu$-integrable, as is automatically the case, for instance, when $\phi_1$ is upper semicontinuous on a compact space, subadditivity yields $\int \phi_n \,d\mu \leq n \int \phi_1 \, d\mu <\infty$. Consequently, by Fekete's lemma (Lemma~\ref{lem:Fekete}), the following limit exists:
\begin{equation} \label{eq:Kingman2}
     \Lambda(\mu) \eqdef \lim_{n\to \infty} \frac{1}{n} \int \phi_n \, d\mu =\inf_{n\geq 1} \frac{1}{n} \int \phi_n \, d\mu \in [-\infty,\infty).
\end{equation}\index{growth rates!\(\Lambda(\mu)\), asymptotic  rate}
We then define the global supremum
\begin{equation*} 
\Lambda \eqdef \sup_{\mu \in \mathcal{I}} \Lambda(\mu).
\end{equation*}\index{growth rates!\(\Lambda\), maximal  rate}
}

\begin{mainthm}\label{Kingmanuniform} \index{Kingman's subadditive ergodic theorems! uniform Kingman for Markov operators}
Consider a compact metric space $X$ and let $P:C(X)\to C(X)$ be a Markov operator. Let $(\phi_n)_{n\geq 1}$ be a $P$-subadditive sequence of upper semicontinuous extended real-valued functions. Then,
\begin{equation} \label{eq:Kingman1}
  \Lambda  = \lim_{n\to \infty} \max_{x\in X} \frac{1}{n}\phi_n(x). 
\end{equation}
Moreover,
\begin{enumerate}[leftmargin=0.9cm, itemsep=0.25cm,label=(\roman*)]
\item $\displaystyle\Lambda  =  \max \big\{ \Lambda(\mu): \mu\in \mathcal{I}_{\rm erg} \big\}$
\item $\displaystyle\Lambda= \sup_{x\in X} \limsup_{n\to \infty} \frac{1}{n} \phi_n(x)=\sup_{\mu\in \mathcal{I}} \int \limsup_{n\to\infty} \frac{1}{n} \phi_n(x) \, d\mu$ 
\item $\displaystyle\Lambda=\lim_{n\to \infty} \sup_{\mu \in \mathcal{I}} \frac{1}{n} \int \phi_n \, d\mu \leq  \max_{\mu \in \mathcal{I}_{\rm erg}}  \frac{1}{m} \int \phi_m\, d\mu \ \ \text{for all $m\geq 1$}$.
\end{enumerate}
In equations~\eqref{eq:Kingman1} and (iii),  the limit may be replaced by the infimum.
\end{mainthm}

For clarity, we divide the proof of the above theorem into three subsections: \S\ref{ss:limits}, \S\ref{ss:inequality} and \S\ref{ss:proof}. Before proving the theorem, we give the following consequence:

\begin{cor} \label{cor:Kingman-uniform} If $(\phi_n)_{n\geq 1}$ is an $P$-additive sequence with $\int \phi_1^+ \, d\mu <\infty$, then  
$$
\Lambda(\mu)=\int \phi_1 \, d\mu <\infty \quad \text{and} \quad  \phi_n=\sum_{i=0}^{n-1} P^i\phi_1.
$$ 
Consequently, under the assumption of Theorem~\ref{Kingmanuniform},  
\begin{equation} \label{eq:formula-Furstenberg} 
    \Lambda = \max_{\mu \in \mathcal{I}_{\rm erg}} \int \phi_1\, d\mu =\lim_{n\to\infty}
    \max_{x\in X} \frac{1}{n} \phi_n
\end{equation}  
and 
\begin{equation} \label{eq:uniform-Furstenberg-Kifer}
    \inf_{\mu\in \mathcal I } \int \phi_1\, d\mu =\min_{\mu \in \mathcal{I}_{\rm erg}}  \int \phi_1\, d\mu 
    = \lim_{n\to\infty} \frac{1}{n} \min_{x\in X}   \phi_n(x). 
\end{equation}
\end{cor}

\begin{proof} If $(\phi_n)_{n\geq 1}$ is $P$-additive and $\mu\in \mathcal{I}$, then 
$$\int \phi_{m+1}\, d\mu=\int \phi_1\, d\mu + \int P\phi_m \,d\mu = \int \phi_1 \,d\mu + \int \phi_m \, d\mu.$$
Recursively, we get $\int \phi_{m+1} \, d\mu = (m+1)\int \phi_1 \, d\mu$. Then, dividing by $m+1$ and taking the limit, we have $\Lambda(\mu)=\int \phi_1\, d\mu <\infty$ under the assumption that $\phi_1^+=\max\{0,\phi_1\}$ is $\mu$-integrable.  

Now assume the hypotheses of Theorem~\ref{Kingmanuniform}, i.e.,  $\phi_n$ is additionally upper semicontinuous, $P:C(X)\to C(X)$ and $X$ is a compact metric space. Hence,  we obtain~\eqref{eq:formula-Furstenberg} from the facts that $\Lambda(\mu)=\int \phi_1 \,d\mu<\infty$,~\eqref{eq:Kingman1}  and the representation of $\Lambda$ in~(i) of Theorem~\ref{Kingmanuniform} as the maximum of $\Lambda(\mu)$ over the ergodic measures $\mu\in \mathcal{I}$.  

Finally, since $(\phi_n)_{n\geq 1}$ is $P$-additive, then also $(-\phi_n)_{n\geq 1}$ is $P$-additive. Thus, we can again use the representations of $\Lambda$ in~\eqref{eq:formula-Furstenberg} for both sequences of functions. Applying these representations to  $(-\phi_n)_{n\geq 1}$, we obtain that for every $x\in X$,
$$
\sup_{\mu \in \mathcal{I}} -\Lambda(\mu)=\max_{{\mu \in \mathcal{I}_{\rm erg}}} \int -\phi_1\, d\mu  = \lim_{n\to \infty}  \max_{x\in X} \frac{-1}{n}\phi_n(x). 
$$
This yields the chain of inequalities in~\eqref{eq:uniform-Furstenberg-Kifer} and concludes the proof.  
\end{proof}

\subsection{Existence of limits} \label{ss:limits}  Let us begin by establishing the existence of the limits in equations~\eqref{eq:Kingman1} and~(iii) in Theorem~\ref{Kingmanuniform}. We also show that these limits
can be replaced by an infimum.

\begin{lem} \label{lem:limites} It holds that
\begin{equation} \label{eq:Lambdamax}
\begin{aligned} \index{growth rates!\(\Lambda_{\max}\), global maximal rate}
\Lambda_{\max} &\eqdef \lim_{n\to  \infty} \max_{x\in X} \frac{1}{n} \phi_n(x) = \inf_{n\geq 1} \frac{1}{n} \max_{x\in X}\phi_n(x) \\ &\geq   \inf_{n\geq 1} \sup_{\mu\in \mathcal{I}} \frac{1}{n}\int \phi_n \, d\mu=\lim_{n\to \infty} \sup_{\mu\in \mathcal{I}} \frac{1}{n}\int \phi_n \, d\mu.
\end{aligned}
\end{equation}
\end{lem}
\begin{proof}
%
Observe that  \(  (\sup_{\mu \in \mathcal{I}} \int \phi_n \, d\mu)_{n\geq 1}\)  is a subadditive sequence, and by Fekete's lemma, we get the last equality in~\eqref{eq:Lambdamax}.
To prove the other inequality in~\eqref{eq:Lambdamax}, notice that since $\phi_n$ is upper semicontinuous and $X$ compact, it achieves its maximum on $X$, denoted by $M_n$. As constant functions are continuous maps, and since $P$ is Markov, \(P^mM_n = M_nP^m1_X = M_n\). Hence, again inequality~\eqref{Qsub} implies that the sequence \((M_n)_{n\ge 1}\) is subadditive. By invoking Fekete's lemma once more, we deduce the second equality~\eqref{eq:Lambdamax}. The inequality in~\eqref{eq:Lambdamax} follows trivially since $M_n \geq \int \phi_n \, d\mu$.
\end{proof}

\subsection{Main inequality} \label{ss:inequality}

We  turn to the proof of the following proposition:
\begin{prop} \label{Kingmanuniform2}
There exists a measure \( \nu \in \mathcal{I} \) such that \( \Lambda_{\max} \leq \Lambda(\nu) \).
\end{prop}

To prove  Proposition \ref{Kingmanuniform2}, we proceed in several steps. We first establish the following key inequality. Here, \( \lfloor \cdot \rfloor \) denotes the floor function.

\begin{lem}\label{Kingmanineq}
Let $m$ be a positive integer. There exists a constant $C\in\mathbb{R}$ such that for every $n\ge 2m$,
$$
\phi_n\leq \frac{1}{m}\sum_{k=0}^{k_n-1} P^k{\phi_m}+ C \qquad \text{where \ \ $k_n=\left( \left\lfloor \frac{n}{m}\right\rfloor-1\right)m$}.
$$
\end{lem}
\begin{proof}
Using~\eqref{Qsub} and induction, we can readily observe that for any $p\geq 1$,
\begin{equation}\label{eq:p}
\phi_{pm}\leq \sum_{j=0}^{p-1} P^{jm}\phi_m.
\end{equation}
For an integer $n \geq 2m$, let $p=\left\lfloor n/m \right\rfloor-1$. We can express $n$ as $n=(p+1)m+R$, where $R\in \{0,\dots,m-1\}$.
Now, for every $r \in \{0,\ldots,m-1\}$, there exists some $s=s(r) \in \{1,\dots,2m-1\}$ such that $m+R=r+s$. Consequently, $n=pm+r+s$. Defining $\phi_0=0$ and using~ \eqref{Qsub} and the above inequality, we have
\begin{align*}
\phi_n &\leq \phi_{pm+r}+P^{pm+r}\phi_{s} \leq \phi_r+P^r\phi_{pm}+P^{pm+r}\phi_{s}\\
&\leq \phi_r+P^r\left(\sum_{j=0}^{p-1} P^{jm}\phi_m\right)+P^{pm+r}\phi_{s} \leq C+\sum_{j=0}^{p-1} P^{jm+r}\phi_m
\end{align*}
where $C=2 \max\{M_1,\ldots,M_{2m-1}\}$ and, as before, $M_i=\max\{\phi_i(x): x\in X\}$. Averaging over $r$ in $\{0,\ldots,m-1\}$, we find
\[
\phi_n\leq C+\frac{1}{m}\sum_{r=0}^{m-1}\sum_{j=0}^{p-1} P^{jm+r}\phi_m = C+\frac{1}{m}\sum_{k=0}^{pm-1}P^k\phi_m
\]
completing the proof.
\end{proof}

Now we can obtain the following intermediary step:

\begin{lem} \label{lem:Lambdanu}
For every integer \( m \), there exists \( \nu_m \in \mathcal{I} \) such that
\( \Lambda_{\max} \leq \frac{1}{m} \int \phi_m \, d\nu_m\).
\end{lem}

\begin{proof}
For each \( n \geq 2m \), let \( x_n \in X \) be a point where \( \phi_n \) attains its maximum $M_n$. Thus, we have $M_n = \phi_n(x_n)$.
From Lemma~\ref{Kingmanineq}, we have
\[ M_n \leq \frac{1}{m} \sum_{k=0}^{k_n-1} P^k\phi_m(x_n) + C \]
where \( k_n = \left( \left\lfloor n/m \right\rfloor - 1 \right) m \).
Introducing the probability measure
\[ \mu_n = \frac{1}{k_n} \sum_{k=0}^{k_n-1} (P^k)^* \delta_{x_n} \]
we can recast the aforementioned inequality, upon division by \( n \), as
\begin{equation} \label{eq:Mn}
\frac{M_n}{n} \leq \frac{k_n}{n} \int \frac{\phi_m}{m} \, d\mu_n + \frac{C}{n}.
\end{equation}

Consider a subsequence \( (\mu_{n_i})_{i \geq 1} \) of \( (\mu_n)_{n \geq 2m} \) weakly* converging to a probability measure~\( \nu_m \) which depends on $m$ due to the dependence of $k_n$ on $m$. Since every probability measure has total variation norm equal to 
1, we have, 
\[ \left\| P^* \mu_n - \mu_n \right\|_{\mathrm{TV}} = \frac{1}{k_n} \left\| (P^{k_n})^* \delta_{x_n} - \delta_{x_n} \right\|_{\mathrm{TV}} \leq \frac{2}{k_n}. \]
Since \( k_n \to \infty \) as \( n \to \infty \), we deduce that $\| P^* \mu_n - \mu_n \|_{\mathrm{TV}}\to 0$.
Now, take an arbitrary continuous function  $\varphi$ 
which we may assume without loss of generality that satisfies $\|\varphi\|_{\infty}\leq 1$. Then,
\begin{align*}
\bigg|\int \varphi \, dP^*\nu_m &- \int \varphi\, d\nu_m \bigg| = \left|\int (P\varphi -\varphi) \, d\nu_m \right|
= \lim_{i \to \infty}  \left|\int (P\varphi  -  \varphi ) \, d\mu_{n_i} \right| \\
&=  \lim_{i \to \infty}  \left|\int \varphi \, dP^*\mu_{n_i} - \int \varphi\, d\mu_{n_i} \right| \leq \lim_{i\to\infty}\| P^* \mu_{n_i} - \mu_{n_i} \|_{\mathrm{TV}}= 0.
\end{align*}
Here the second equality relies on the fact that $P\varphi$ is a continuous function and $\mu_{n_i} \to \nu_m$ in the weak* topology. Since $\varphi$ is an arbitrary continuous function, this implies that  \( \nu_m \) is \( P^* \)-invariant,  thus \( \nu_m \in  \mathcal{I} \).

 By letting \( n \to \infty \) in~\eqref{eq:Mn} over the subsequence \( (n_i)_{i \geq 1} \) and invoking Lemmas~\ref{weakineq} and~\ref{lem:limites} along with the fact that \( k_n/n \to 1 \), we establish
\[ \Lambda_{\max} = \lim_{n \to \infty} \frac{M_n}{n} \leq \int \frac{\phi_m}{m} \, d\nu_m. \]
This concludes the proof.
\end{proof}

We can now complete the proof of Proposition \ref{Kingmanuniform2}.

\begin{proof}[Proof of Proposition~\ref{Kingmanuniform2}]
By Lemma~\ref{lem:Lambdanu}, for each $m \in \mathbb{N}$, there exists an invariant probability measure $\nu_m\in \mathcal{I}$ such that $\Lambda_{\max} \leq \frac{1}{m} \int \phi_m \, d\nu_m$. Let $\nu$ be the weak* limit of a subsequence of $(\nu_m)_{m \geq 1}$. Clearly, $\nu$ belongs to $\mathcal{I}$.

Fix a positive integer $n$, for sufficiently large $m$, we can express $m$ as $m = pn + r$, where $p = \lfloor m/n \rfloor$ and $r \in \{0, \ldots, n-1\}$. Using the $P$-subadditivity of $(\phi_k)_{k \geq 1}$, as in~\eqref{eq:p}, we~get
$\phi_{pn} \leq \phi_n+ P^n\phi_n+ \dots +P^{jn} \phi_n$.
Therefore, by the $P$-subadditivity of $(\phi_k)_{k \geq 1}$ and given that $\nu_m$ is a $P^*$-invariant measure, we find
\begin{align*}
\int \phi_m \, d\nu_m &\leq  
\sum_{j=0}^{p-1} \int P^{jn}\phi_n \, d\nu_m + \int P^{pn}\phi_r \, d\nu_{m} = p \int \phi_n \, d\nu_m + \int \phi_r \, d\nu_m.
\end{align*}
From this, we deduce
$$\Lambda_{\max} \leq \frac{1}{m} \int \phi_m \, d\nu_m \leq \frac{p}{m} \int \phi_n \, d\nu_m + \frac{C}{m}$$
where $C = \max\{M_1, \ldots, M_{n-1}\}$ and $M_i = \max \{\phi_i(x) : x \in X\}$. Taking $m \to \infty$ along the chosen subsequence that converges to $\nu$, and by applying Lemma~\ref{weakineq} while observing that $p/m \to 1/n$, we deduce
$$\Lambda_{\max} \leq \frac{1}{n} \int \phi_n \, d\nu.$$
Letting $n \to \infty$, we obtain the desired result.
\end{proof}

\subsection{Proof of Theorem~\ref{Kingmanuniform}} 
\label{ss:proof}

Before proving Theorem~\ref{Kingmanuniform}, we derive the following corollary from the previous results:

\begin{cor} \label{cor:Kingman} Let $\nu\in \mathcal{I}$ given in Proposition~\ref{Kingmanuniform2}.
For any \(\mu \in \mathcal{I}\), we have
\begin{equation} \label{eq:Lambdamudesigualdad}
 \Lambda(\mu) \leq \int \limsup_{n \to \infty} \frac{1}{n}\phi_n \, d\mu  \leq \sup_{x\in X} \limsup_{n \to \infty} \frac{1}{n}\phi_n(x) \leq  \Lambda_{\max} \leq   \Lambda(\nu)
\end{equation}
and
\begin{equation} \label{eq:Lambdamudesigualdad2}
 \Lambda(\nu) \leq \lim_{n\to\infty} \sup_{\mu\in \mathcal{I}}\frac{1}{n} \int \phi_n \, d\mu \leq  \Lambda.
\end{equation}
Consequently,
\begin{equation} \label{eq:Lambdaigualdad}
\begin{aligned}
  \Lambda =\sup_{\mu \in \mathcal{I}} \int \limsup_{n \to \infty} \frac{1}{n}\phi_n \, d\mu &= \sup_{x\in X}\limsup_{n\to\infty} \frac{1}{n}\phi_n(x) \\ &=\lim_{n\to\infty} \sup_{\mu\in \mathcal{I}}\frac{1}{n} \int \phi_n \, d\mu=\Lambda_{\max}=\Lambda(\nu).
\end{aligned}
\end{equation}
\end{cor}

\begin{proof} The last inequality of~\eqref{eq:Lambdamudesigualdad} follows from~Proposition~\ref{Kingmanuniform2}. Thus to establish~\eqref{eq:Lambdamudesigualdad}, we only need to show the first inequality since the remaining inequalities are immediate. To do this, we invoke Lemma~\ref{Fatou} with \(f_n = \frac{1}{n}\phi_n\) and \(g = 1 + \Lambda_{\max}\). The function \(g\) is \(\mu\)-integrable, and furthermore, \(f_n \leq g\) holds for all sufficiently large \(n\). Hence, by the reverse Fatou lemma, we immediately deduce the first inequality in~\eqref{eq:Lambdamudesigualdad} as desired. The first inequality in~\eqref{eq:Lambdamudesigualdad2} trivially holds and the second inequality follows by  Lemma~\ref{lem:limites} and $\Lambda_{\max} \leq \Lambda(\nu)\leq \Lambda$.
Taking the supremum over \(\mathcal{I}\) in~\eqref{eq:Lambdamudesigualdad} and~\eqref{eq:Lambdamudesigualdad2},  we get~\eqref{eq:Lambdaigualdad}
and complete the proof.
\end{proof}

\begin{proof}[Proof of Theorem~\ref{Kingmanuniform}]   
Lemma~\ref{lem:limites} 
and Corollary~\ref{cor:Kingman} imply all the assertions in Theorem~\ref{Kingmanuniform} except for the identity $\Lambda=\max\{\Lambda(\mu): \mu \in \mathcal{I}_{\rm erg} \}$ and statement~(iii).  
To prove the remaining equality, we use the Bauer maximum principle~\cite{bauer58}.
Recall that an upper semicontinuous extended real-valued function $f$ defined on a nonempty compact set $K$ in a Hausdorff locally convex topological vector
space is said to be \emph{convex} if
\[ f(\lambda x + (1 - \lambda) y) \leq \lambda f(x) + (1 - \lambda) f(y) \]
whenever $0 \leq \lambda \leq 1$, $x, y \in {K}$, and $\lambda x + (1 - \lambda) y \in {K}$. Bauer's maximum principle states that any upper semicontinuous convex extended real-valued function defined on a compact convex set attains its maximum at some extreme point of that set,~cf.~\cite[Thm.~2]{kunfu78}.

\begin{claim} \label{claim:uppersemicontinuous} The function $\mu\in \mathcal{I}\mapsto\Lambda(\mu)\in [-\infty,\infty)$ is  upper semicontinuous and convex where $\mathcal{I}$ is endowed with the weak* topology.
\end{claim}
\begin{proof} Clearly
$\Lambda(\mu)=\lim_{n\to\infty} \frac{1}{n}\int \phi_n \, d\mu$  is convex. The upper semicontinuity follows from Lemma~\ref{weakineq}. Indeed,
let $(\mu_k)_{k\geq 1}$ be a sequence of probability measures in $\mathcal{I}$ weak* converging to $\mu\in \mathcal{I}$. Then,
$$
 \limsup_{k\to\infty} \Lambda(\mu_k) = \limsup_{k\to \infty} \inf_{n\geq 1} \frac{1}{n}\int \phi_n \, d\mu_k \leq
 \limsup_{k\to \infty}  \frac{1}{n}\int \phi_n \, d\mu_k \leq \frac{1}{n} \int \phi_n \, d\mu
$$
for all $n\geq 1$. In particular, we have that
$$
\limsup_{k\to\infty} \Lambda(\mu_k) \leq \inf_{n\geq 1} \frac{1}{n} \int \phi_n \, d\mu = \Lambda(\mu)
$$
concluding the proof.
\end{proof}

 Since $\mathcal{I}$ is a compact convex set and its extreme points are ergodic measures, we conclude from the Bauer maximum principle and the above claim that the maximum of $\Lambda(\mu)$  for $\mu\in \mathcal{I}$ is achieved on an ergodic measure as desired. 
 
 A similar argument works to prove~(iii). Note first that, by Lemma~\ref{lem:Lambdanu}, and since $\Lambda=\Lambda_{\max}$, we have  that for each $m\geq 1$,
$$
\Lambda \leq \sup_{\mu\in\mathcal{I}} \frac{1}{m}\int \phi_m \, d\mu = \sup_{\mu\in\mathcal{I}} \frac{1}{m}\int \phi_m \, d\mu.
$$
Similar to Claim~\ref{claim:uppersemicontinuous} (see also~Lemma~\ref{weakineq}), $\mu\in\mathcal{I} \mapsto \int \phi_m \, d\mu\in [-\infty,\infty)$ is upper semicontinuous and convex. Then, the Bauer maximum principle also yields~(iii) and concludes the proof of the theorem.  
\end{proof}

\section{Kingman subadditive ergodic theorem: proof of Theorem~\ref{Kingman} } \label{sec:kingman}
The following result extends Kingman's subadditive ergodic theorem for stochastic processes involving a measure-preserving transformation to the case of Markov operators on \(L^1(\mu)\). Here $L^1(\mu)=L^1(X,\mathscr{B},\mu)$\index{space of functions!\(L^1(\mu)\), integrable functions} denotes the Banach space of real-valued $\mathscr{B}$-measurable functions on a probability space $(X,\mathscr{B},\mu)$ for which the absolute value is $\mu$-integrable, identifying functions that agree $\mu$-almost everywhere. This theorem was essentially obtained by Akcoglu and Sucheston~\cite{AS:78}, though some non-trivial arguments are required to derive the statement below. We first recall some definitions.

{The notion of a Markov operator depends on the Banach space on which the operator acts, but the standard definitions are compatible by duality. In the present \(L^1(\mu)\)-setting, an operator \(P:L^1(\mu)\to L^1(\mu)\) is called \emph{Markov}\index{operator!Markov operator on \(L^1(\mu)\)} if it is linear, positive, and preserves the integral, that is,
\[
\int P\varphi \, d\mu = \int \varphi \, d\mu
\qquad \text{for all }\varphi \in L^1(\mu).
\]
These conditions imply that \(\|P\|_{\mathrm{op}}=1\). Thus, \(P\) is a contraction, and in particular, bounded. Moreover, by duality, the integral-preserving condition is equivalent to 
$$P^*1_X=1_X,$$ where \(P^*\) denotes the dual operator of \(P\) acting on \(L^1(\mu)^*\equiv L^\infty(\mu)\). As usual, \(L^\infty(\mu)\) \index{space of functions!\(L^\infty(\mu)\), essentially bounded functions} denotes the quotient space of the bounded measurable functions $B(X)$ on \(X\)\index{space of functions!$B@\(B(X)\), bounded measurable functions}, where two such functions are identified if they are equal \(\mu\)-almost everywhere. Thus, on \(L^\infty(\mu)\), and hence, at the level of representatives in \(B(X)\), the corresponding notion of a Markov operator is that of a positive linear operator preserving the order unit \(1_X\). In this sense, the above definition is the \(L^1(\mu)\)-counterpart of the notion of Markov operator introduced earlier on \(C(X)\).}

Following~\cite[Sec.~11.3.1]{HL:12}, let $\mathrm{ba}(\mu)$ be the Banach space of bounded finitely-additive set functions $\lambda$ on $X$ such that $\lambda(A) = 0$ if $\mu(A) = 0$, with the total variation norm. Then $\mathrm{ba}(\mu)$ is isometrically isomorphic to the dual of $L^1(\mu)^*\equiv L^\infty(\mu)$ and hence to the second dual of $L^1(\mu)$. Hence, the second adjoint $P^{**}$ is an extension of $P$ to $\mathrm{ba}(\mu)$. Moreover, since for every $A\in \mathscr{B}$
$$
P^{**}\mu(A)=\int 1_A \, dP^{**}\mu = \int P^*1_A \, d\mu = \int_A P1_X \, d\mu \quad \text{and} \quad
\mu(A)=\int_A 1_X \, d\mu,
$$
we get that $P1_X=1_X$ if and only if $P^{**}\mu=\mu$. 

In the deterministic case, when $P$ is the Perron-Frobenius operator of a non-singular transformation $f$, the adjoint $P^*$ is
the Koopman operator given by $P^*\varphi = \varphi \circ f$, and $P1_X=1_X$ if and only if $f$ preserves $\mu$. 

\begin{mainthm}\label{Kingman} \index{Kingman's subadditive ergodic theorems!Kingman for Markov $L^1$-operators}
Consider a Markov operator $P:L^1(\mu)\to L^1(\mu)$ with $P1_X=1_X$ where $(X,\mathscr{B},\mu)$ is a probability space.
Let $(\phi_n)_{n\geq 1}$ be a $P$-subadditive sequence of {extended real-valued} functions such that $\phi^+_1\in L^1(\mu)$.
Then, there exists a measurable $P$-invariant function $g$ such that $g^+\in L^1(\mu)$,
\begin{equation*}
    \lim_{n\to \infty} \frac{1}{n} \phi_n(x) = g(x)  \ \ \, \text{for $\mu$-a.e.~$x\in X$   \ \ and }  \int g \, d\mu = \lim_{n\to \infty} \frac{1}{n}\int \phi_n \, d\mu \eqdef \Lambda(\mu).
\end{equation*}
Moreover, if  $\mu$ is ergodic, then  $g=\Lambda(\mu)$  \ $\mu$-almost everywhere.
\end{mainthm}

\begin{proof}
We first consider the case where $\mu$ satisfies $\Lambda(\mu)>-\infty$ and $\phi_1\in L^1(\mu)$. Define
\begin{equation} \label{eq:supperaditive}
F_n\eqdef -\phi_n+\sum_{i=1}^{n-1}P^i\phi_1  \quad \text{for $n\geq 1$}.
\end{equation}
This process satisfies that $F_n\in L^1(\mu)$, $F_n\geq 0$ and $F_n+P^nF_m \leq F_{n+m}$. Since $P$ is Markov on $L^1(\mu)$ and $(F_n)_{n\geq 1}$  $P$-superadditive (i.e, $(-F_n)_{n\geq 1}$  $P$-subadditive),  then  applying Fekete's lemma,
\begin{equation*}
   \Lambda_{F}(\mu)\eqdef   \lim_{n\to\infty} \frac{1}{n} \int F_n \, d\mu = \sup_{n\geq 1} \frac{1}{n} \int F_n \, d\mu.
\end{equation*}
Additionally,
$$\Lambda_F(\mu)=-\Lambda(\mu)+\int \phi_1\,d\mu <\infty.$$
According to~\cite[Thm.~2.1]{AS:78}, $(F_n)_{n\geq 1}$  admits an  exact dominant, \index{subadditive processes!exact dominant} that is, there exists $G\in L^1(\mu)$ such that for $\mu$-almost everywhere,
$$
  F_n \leq \sum_{i=0}^{n-1} P^iG \quad \text{and} \quad  \int G \, d\mu = \Lambda_F(\mu).
$$
Invoking~\cite[Thm.~3.1]{AS:78} and given $P1_X=1_X$, we use  the individual and mean ergodic theorems~(see~\cite[Thms.~2.3.4 and~2.3.5]{HL:12}) to deduce the existence of $G^*\in L^1(\mu)$ satisfying
$$
  \lim_{n\to\infty} \frac{1}{n}F_n = \lim_{n\to\infty} \frac{1}{n} \sum_{i=0}^{n-1} P^iG = G^*,  \quad  \int G^*\,d\mu = \int G\, d\mu \quad \text{and} \quad PG^*=G^*.
$$
This implies that the limit
$$
  \lim_{n\to\infty} \frac{1}{n}\phi_n = \phi_1^* - G^* \eqdef g \in L^1(\mu), \quad \text{exists $\mu$-almost everywhere}
$$
and
$$
       \int g \, d\mu = \int \left(\phi_1- G \right)\,d\mu = \int \phi_1\,d\mu - \Lambda_F(\mu) =\Lambda(\mu) \quad \text{and} \quad Pg=g
$$
where $\phi^*_1$ is the $P$-invariant limit 
of the additive process $\phi_1+P\phi_1+\dots+P^{n-1}\phi_1$.
This proves the first part of the theorem for the above mentioned assumptions. To conclude the theorem in this case, let us assume that $\mu$ is ergodic. Then, according to~\cite[Lemma~I.2.4]{kifer2012ergodic},\footnote{Strictly speaking, Kifer's lemma is not stated directly for operators on $L^1(\mu)$. Nonetheless, the proof of such a result can be applied literally for Markov operators on $L^1(\mu)$ satisfying that $P1_X=1_X$.} any $P$-invariant function is constant $\mu$-almost everywhere. Consequently, $g$ is constant $\mu$-almost everywhere and since $\Lambda(\mu)=\int g \, d\mu$, we get that $g(x)=\Lambda(\mu)$ for $\mu$-a.e.~$x\in X$.

Now, we will prove the theorem for the general case: $\Lambda(\mu) \geq -\infty$ and $\phi_1^+\in L^1(\mu)$.
For each $N\geq 1$, consider  the processes \( (\Phi^{N}_n)_{n\geq 1} \) defined by
$$
\Phi^{N}_n \eqdef \max \{ \phi_{n}, \, - n  N  \} \quad \text{for $n\geq 1$}.
$$
We note that $\Phi^N_1\in L^1(\mu)$ and the sequence \( (\Phi^{N}_n)_{n\geq 1} \) is $P$-subadditive. This is because  $P$ preserves constant functions (i.e., $P1_X=1_X$) leading to
\begin{align*}
   \Phi^N_{n+m} = \max\{\phi_{n+m},-(n+m)N\} &\leq \max\{\phi_n+P^n\phi_m, -nN - mN\} \\ &\leq \Phi^N_n + \max\{P^n\phi_m,-mN\}\leq \Phi^N_n + P^n\Phi^N_m.
\end{align*}
Furthermore, we have
$$
\Lambda_{\Phi^N}(\mu) \eqdef \lim_{n\to\infty}  \frac{1}{n}\int \Phi^N_n\, d\mu = \inf_{n\geq 1} \frac{1}{n}\int \Phi^N_n\, d\mu  >-\infty.
$$
Therefore, we can apply the previous case to obtain that the pointwise limits \( g_N = \lim_{n \to \infty} \frac{1}{n} \Phi^{N}_n \) exist $\mu$-almost everywhere, $g_N\in L^1(\mu)$, $\int g_N\, d\mu = \Lambda_{\Phi^N}(\mu)$ and $Pg_N=g_N$. Moreover, if $\mu$ is ergodic, $g_N(x)=\Lambda_{\Phi^N}(\mu)$ for $\mu$-a.e.~$x\in X$.

Since $(g_N)_{N\geq 1}$ is decreasing, $g\eqdef\inf_{N\geq 1} g_N =\lim_{N\to\infty} g_N$. Hence $g\leq g_N$ and, since $g_N\in L^1(\mu)$, then  $g^+\in L^1(\mu)$. From the $P$-invariance of $g_N$, it is easy to check that $Pg=g$.
Moreover, considering the function $m_N(t)=\max\{t,-N\}$, we have
\begin{align*}
g= \lim_{N\to\infty} \lim_{n\to \infty} \frac{1}{n}\Phi^N_n = \lim_{N\to\infty}   \lim_{n\to \infty} m_N\left(\frac{1}{n}\phi_n\right).
\end{align*}
Since $m_N$ is a continuous and  monotone increasing map,
we have that
\begin{align*}
&\liminf_{n\to \infty} m_N\left(\frac{1}{n}\phi_n\right) =m_N\left(\liminf_{n\to \infty} \frac{1}{n}\phi_n\right)
\quad \text{and}  \\ 
&\limsup_{n\to \infty} m_N\left(\frac{1}{n}\phi_n\right)=m_N\left(\limsup_{n\to \infty} \frac{1}{n}\phi_n\right).
\end{align*}
Taking limit above as $N\to\infty$ and observing that $m_N(t)\to t$ for any $t\in \mathbb{R}$, we obtain that
$\liminf_{n\to\infty} \frac{1}{n}\phi_n= g = \limsup_{n\to\infty} \frac{1}{n}\phi_n$ and thus
$g=\lim_{n\to\infty} \frac{1}{n}\phi_n$. Moreover,
$$
g_N=\lim_{n\to\infty} m_N\left(\frac{1}{n}\phi_n\right) = m_N\left(\lim_{n\to\infty}\frac{1}{n}\phi_n\right)=m_N(g).
$$
From here, since $m_N(g)=g^+-m_N(g)^-$ and $(m_N(g)^-)_{N\geq 1}$ is a non-negative increasing sequence, letting $N\to\infty$ and using the monotone convergence  theorem and that $m_N(g)^-\to g^-$,
\begin{equation}\label{eq:N1}
   \lim_{N\to \infty} \int g_N \, d\mu = \lim_{N\to\infty} \int m_N(g) \, d\mu =  \int g\, d\mu.
\end{equation}
Similarly, for every $n\geq 1$, it holds
$$
\lim_{N\to\infty} \Lambda_{\Phi^N}(\mu) \leq  \lim_{N\to\infty} \int m_N\left(\frac{1}{n}\phi_n\right)\, d\mu   =\int \frac{1}{n}\phi_n \, d\mu.
$$
This leads us to the inequality $\lim_{N\to\infty} \Lambda_{\Phi^N}(\mu)\leq \Lambda(\mu)$. But also, for every $N\geq 1$, we have that
$$
\Lambda(\mu)=\inf_{n\geq 1}  \frac{1}{n} \int \phi_n \, d\mu \leq \inf_{n\geq} \int m_N\left(\frac{1}{n}\phi_n\right)\, d\mu =\Lambda_{\Phi^N}(\mu)
$$
culminating in the equality
\begin{equation}\label{eq:N2}
  \lim_{N\to\infty} \Lambda_{\Phi^N}(\mu) =\Lambda(\mu).
\end{equation}
Since $\int g_N\,d\mu = \Lambda_{\Phi^N}(\mu)$,  putting together~\eqref{eq:N1} and~\eqref{eq:N2}, we arrive to
$\int g \, d\mu= \Lambda(\mu)$.

Finally, if $\mu$ is ergodic, we recall that, for every $N\ge 1$, we have that $g_N$ is constant $\mu$-almost everywhere. Then $g_N(x)=\int g_N\, d\mu$ for $\mu$-a.e.~$x\in X$ and thus,
$$g=\lim_{N\to \infty} g_N = \lim_{N\to\infty} \int g_N\, d\mu = \int g \, = \Lambda(\mu) \quad \text{$\mu$-almost everywhere}.$$
This completes the proof of the theorem.
\end{proof}

In the following remark, we explain that the above theorem can be applied under the assumptions of Theorem~\ref{Kingmanuniform} and to the Koopman operator.

\begin{rem}\label{rem:kingman}
Let \(P:C(X)\to C(X)\) be a Markov operator. As explained in~\cite[Thm.~I.6 and~I.7]{foguel1973ergodic}, since the adjoint operator \(P^*\) preserves the set of probability measures, the formula
\[
P(x,A)\eqdef P^*\delta_x(A)
\]
defines a transition probability:\index{transition probability@\(P(x,A)\), transition probability} for each fixed \(x\in X\), the map \(A\mapsto P(x,A)\) is a probability measure, and for each fixed Borel set \(A\subset X\), the map \(x\mapsto P(x,A)\) is measurable. Hence, \(P\) extends to the space \(B(X)\) of bounded measurable functions by
\[
Pf(x)\eqdef \int f(y)\,P(x,dy).
\]
Now let \(\mu\) be a \(P^*\)-invariant probability measure. Then \(P\) naturally induces an operator on \(L^\infty(\mu)\) by \(P[g]=[Pg]\), where \([g]\) denotes the equivalence class of \(g\in B(X)\). In particular, this definition is independent of the chosen representative. Moreover, \(P\) remains a Markov operator on \(L^\infty(\mu)\), so that \(P1_X=1_X\) \(\mu\)-almost everywhere. As is well known (see~\cite[Proposition~V.4.2]{Neveu1967existence}), a Markov operator on \(L^\infty(\mu)\) that also satisfies
\[
\int P\varphi\,d\mu=\int \varphi\,d\mu
\qquad\text{for all }\varphi\in L^\infty(\mu)
\]
extends uniquely to a Markov operator on \(L^1(\mu)\), still denoted by \(P\). Therefore, Theorem~\ref{Kingman} applies under the assumptions of Theorem~\ref{Kingmanuniform}.
\end{rem}

\begin{rem}
Let \(f:X\to X\) be a measurable map preserving a probability measure \(\mu\), and let \((\phi_n)_{n\geq 1}\) be a subadditive sequence of measurable extended real-valued functions on \(X\) such that \(\phi_1^+\in L^1(\mu)\). That is,
\[
\phi_{n+m}\leq \phi_n+\phi_m\circ f^n=\phi_n+P^n\phi_m
\qquad\text{for all } n,m\geq 1,
\]
where \(P\) is the Koopman operator associated with \(f\), defined by \(P\varphi=\varphi\circ f\). By the previous remark, Theorem~\ref{Kingman} applies in this setting as well. This recovers the generalized version of Kingman's subadditive ergodic theorem given by Ruelle in~\cite[Thm.~I.1]{ruelle1979ergodic}. In particular, it yields the extension of Birkhoff's ergodic theorem to quasi-integrable extended real-valued functions; see~\cite[Thm.~1]{hess2010ergodic}.
\end{rem}

{
\section{Asymptotically subadditive processes}\label{ss:asymp-P-subadd}

In \S\ref{sec:kigmanuniform} and \S\ref{sec:kingman} we established uniform and almost-everywhere
Kingman theorems (Theorem~\ref{Kingmanuniform} and~\ref{Kingman}) for \(P\)-subadditive processes.
In applications, the subadditivity relation is often available only up to a sublinear error, or only for
convenient representatives of \(L^\infty\)-classes (hence merely \(\mu\)-a.e.). This subsection aims to record that the main conclusions of Theorem~\ref{Kingmanuniform} and~\ref{Kingman} are stable under such $o(n)$-perturbations.

The following definition is the natural analogue, in our Markov-operator setting, of the notion of asymptotically subadditive process introduced by Feng-Huang~\cite{feng2010lyapunov}. Since it is formulated in terms of the norm \(\|\cdot\|_\infty\), throughout this subsection we restrict attention to processes $(\phi_n)_{n\geq 1}$ in the \emph{bounded version} of the two settings considered above. Namely, in the compact setting of Theorem~\ref{Kingmanuniform}, $(\phi_n)_{n\geq 1}$ are assumed to be real-valued upper semicontinuous functions on \(X\), endowed with the uniform norm, whereas in the \(L^1(\mu)\)-setting of Theorem~\ref{Kingman},  $(\phi_n)_{n\geq 1}$ are assumed to belong to \(L^\infty(\mu)\), endowed with the essential-sup norm.

\begin{defi}\label{def:asymp-P-subadd}
A sequence \((\phi_n)_{n\ge1}\) is called \emph{asymptotically \(P\)-subadditive}\index{subadditive processes!asymptotically \(P\)-subadditive sequence} if for every \(\varepsilon>0\) there exists a \(P\)-subadditive process \((\psi_n)_{n\ge1}\) {such that}
\begin{equation}\label{eq:asymp-def}
\limsup_{n\to\infty}\frac1n\|\phi_n-\psi_n\|_\infty\le \varepsilon.
\end{equation}
Here \(\|\cdot\|_\infty\) denotes the uniform norm in the compact setting, and the essential-sup norm in the
\(L^\infty(\mu)\) setting.
\end{defi}

The following corollary shows that, within the bounded versions fixed above, the conclusions of Theorem~\ref{Kingmanuniform} and~\ref{Kingman} remain stable under asymptotic \(P\)-subadditivity.

\begin{cor}\label{cor:asymp-CD}
If the process \((\phi_n)_{n\ge1}\) is asymptotically \(P\)-subadditive in the sense of Definition~\ref{def:asymp-P-subadd}, then
Theorem~\ref{Kingmanuniform} and~\ref{Kingman} remain valid, except for the conclusion that the limit may be replaced by the infimum.
\end{cor}

The basic mechanism is the following stability estimate.

\begin{lem}\label{lem:asymp-stability}
Let $(\phi_n)_{n\ge1}$ and $(\psi_n)_{n\ge1}$ be  processes such  that for some~$\varepsilon>0$,
\begin{equation}\label{eq:asymp-stability}
\limsup_{n\to\infty}\frac1n\|\phi_n-\psi_n\|_\infty\le \varepsilon.
\end{equation}
Then
\begin{enumerate}
\item For every probability measure $\mu$ for which the integrals are defined,
\begin{align*}
\Bigl|\limsup_{n\to\infty}\frac1n\!\int\!\phi_n\,d\mu-\limsup_{n\to\infty}\frac1n\!\int\!\psi_n\,d\mu\Bigr| &\le \varepsilon, \\
\Bigl|\liminf_{n\to\infty}\frac1n\!\int\!\phi_n\,d\mu-\liminf_{n\to\infty}\frac1n\!\int\!\psi_n\,d\mu\Bigr| &\le \varepsilon,
\end{align*}
and likewise with $\int(\cdot)\,d\mu$ replaced by $\max_X(\cdot)$ in the compact setting.
\item In the $L^1(\mu)$ setting, for $\mu$-a.e.\ $x$,
\begin{align*}
\liminf_{n\to\infty}\frac1n\psi_n(x)-\varepsilon &\leq \liminf_{n\to\infty}\frac1n\phi_n(x)  \\ &\leq \limsup_{n\to\infty}\frac1n\phi_n(x)\le \limsup_{n\to\infty}\frac1n\psi_n(x)+\varepsilon. 
\end{align*}
\end{enumerate}
\end{lem}

\begin{proof}
Fix \(\delta>0\). By \eqref{eq:asymp-stability}, there exists \(N\geq 1\) such that
\begin{equation}\label{eq:asymp-stability-delta}
\|\phi_n-\psi_n\|_\infty\leq (\varepsilon+\delta)n
\qquad\text{for all }n\geq N.
\end{equation}

We first prove item~(1). Let \(\mu\) be a probability measure for which the integrals are defined. For every \(n\geq N\), by the definition of the sup norm,
\[
\left|\frac1n\int \phi_n\,d\mu-\frac1n\int \psi_n\,d\mu\right|
\leq \frac1n\int |\phi_n-\psi_n|\,d\mu
\leq \frac1n\|\phi_n-\psi_n\|_\infty
\leq \varepsilon+\delta.
\]
Hence
\[
\frac1n\int \phi_n\,d\mu
\leq
\frac1n\int \psi_n\,d\mu+\varepsilon+\delta
\qquad\text{for all }n\geq N.
\]
Taking \(\limsup\) as \(n\to\infty\), we get
\[
\limsup_{n\to\infty}\frac1n\int \phi_n\,d\mu
\leq
\limsup_{n\to\infty}\frac1n\int \psi_n\,d\mu+\varepsilon+\delta.
\]
Interchanging the roles of \((\phi_n)\) and \((\psi_n)\) yields an analogous expresion and 
therefore, since $\delta>0$ is arbitrary, we obtain
\[
\Bigl|\limsup_{n\to\infty}\frac1n\!\int\!\phi_n\,d\mu
-
\limsup_{n\to\infty}\frac1n\!\int\!\psi_n\,d\mu\Bigr|
\leq \varepsilon.
\]
The proof for the \(\liminf\) is identical. 

In the compact setting, for every \(n\geq N\) and every \(x\in X\),
\[
\left|\frac1n\phi_n(x)-\frac1n\psi_n(x)\right|
\leq \frac1n\|\phi_n-\psi_n\|_\infty
\leq \varepsilon+\delta.
\]
Taking maxima over \(x\in X\), we get
\[
\left|\max_X\frac1n\phi_n-\max_X\frac1n\psi_n\right|
\leq \varepsilon+\delta
\qquad\text{for all }n\geq N.
\]
The same argument as above then gives the corresponding estimates for \(\limsup\) and \(\liminf\) with \(\int(\cdot)\,d\mu\) replaced by \(\max_X(\cdot)\).

We now prove item~(2). Assume that we are in the \(L^1(\mu)\)-setting. By \eqref{eq:asymp-stability-delta}, for each \(n\geq N\) there exists a \(\mu\)-null set \(E_n\) such that
\[
|\phi_n(x)-\psi_n(x)|\leq (\varepsilon+\delta)n
\qquad\text{for all }x\in X\setminus E_n.
\]
Since \(\bigcup_{n\geq N}E_n\) is still \(\mu\)-null, there exists a full \(\mu\)-measure set \(F\subset X\) such that
\[
|\phi_n(x)-\psi_n(x)|\leq (\varepsilon+\delta)n
\qquad\text{for all }x\in F \text{ and all }n\geq N.
\]
Fix \(x\in F\). Then, for every \(n\geq N\),
\[
\frac1n\psi_n(x)-(\varepsilon+\delta)
\leq
\frac1n\phi_n(x)
\leq
\frac1n\psi_n(x)+(\varepsilon+\delta).
\]
Taking \(\liminf\) in the left inequality and \(\limsup\) in the right inequality, we obtain
\[
\liminf_{n\to\infty}\frac1n\psi_n(x)-(\varepsilon+\delta)
\leq
\liminf_{n\to\infty}\frac1n\phi_n(x),
\]
and
\[
\limsup_{n\to\infty}\frac1n\phi_n(x)
\leq
\limsup_{n\to\infty}\frac1n\psi_n(x)+(\varepsilon+\delta).
\]
Thus, we conclude that for every \(x\in F\),
\begin{align*}
\liminf_{n\to\infty}\frac1n\psi_n(x)-(\varepsilon+\delta)
&\leq
\liminf_{n\to\infty}\frac1n\phi_n(x) \\
&\leq
\limsup_{n\to\infty}\frac1n\phi_n(x)
\leq
\limsup_{n\to\infty}\frac1n\psi_n(x)+(\varepsilon+\delta).
\end{align*}
Finally, letting \(\delta\downarrow0\), we obtain the desired inequalities for \(\mu\)-a.e.\ \(x\).
\end{proof}


\begin{proof}[Proof of Corollary~\ref{cor:asymp-CD}]
Fix $\varepsilon_k\downarrow0$. For each $k$ choose a $P$-subadditive
$\varepsilon_k$-approximation $\Psi^{(k)}=(\psi_n^{(k)})_{n\ge1}$ of $\Phi=(\phi_n)_{n\geq 1}$ so that
\begin{equation}\label{eq:asymp-approx-k}
\limsup_{n\to\infty}\frac1n\|\phi_n-\psi_n^{(k)}\|_\infty\le \varepsilon_k .
\end{equation}
Then, by the triangle inequality,
\[
\limsup_{n\to\infty}\frac1n\|\psi_n^{(k)}-\psi_n^{(\ell)}\|_\infty
\le
\varepsilon_k+\varepsilon_\ell
\qquad\text{for every }k,\ell .
\]
Hence, Lemma~\ref{lem:asymp-stability} applied to the pair
$(\Psi^{(k)},\Psi^{(\ell)})$ yields the uniform estimate
\begin{equation}\label{eq:Cauchy-transfer}
\Bigl|\lim_{n\to\infty}A_n(\Psi^{(k)})-\lim_{n\to\infty}A_n(\Psi^{(\ell)})\Bigr|
\le \varepsilon_k+\varepsilon_\ell,
\end{equation}
whenever $A_n(\cdot)$ is one of the expressions appearing in Theorems~\ref{Kingmanuniform} and~\ref{Kingman}
{(e.g.\ $A_n(\Theta)=\frac1n\int\theta_n\,d\mu$ or $A_n(\Theta)=\max\frac1n\theta_n$),}
and where the limits for $\Psi^{(k)}$ exist by
Theorems~\ref{Kingmanuniform} and~\ref{Kingman}.
{In the compact setting, applying Theorem~\ref{Kingmanuniform} to each $\Psi^{(k)}$ and using~\eqref{eq:Cauchy-transfer} exactly as above gives all the conclusions of Theorem~\ref{Kingmanuniform} for $\Phi$.}
{In the $L^1(\mu)$ setting, Theorem~\ref{Kingman} provides $P$-invariant limits $g_k\in L^1(\mu)$ for $\Psi^{(k)}$ such that
\[
\frac1n\psi_n^{(k)}(x)\longrightarrow g_k(x)
\qquad\text{for }\mu\text{-a.e. }x,
\]
and
\[
\int g_k\,d\mu=\lim_{n\to\infty}\frac1n\int\psi_n^{(k)}\,d\mu.
\]
By Lemma~\ref{lem:asymp-stability} applied to the pair $(\Psi^{(k)},\Psi^{(\ell)})$ and also to the reversed pair $(\Psi^{(\ell)},\Psi^{(k)})$, we get for $\mu$-a.e.~that
$|g_k-g_\ell|\le \varepsilon_k+\varepsilon_\ell$.
Hence, for $\mu$-a.e.~$x$, the sequence $(g_k(x))_{k\ge1}$ is Cauchy. Define
\[
g(x)\eqdef \lim_{k\to\infty}g_k(x)
\qquad\text{for }\mu\text{-a.e. }x.
\]
Passing to the limit in the previous inequality gives for $\mu$-a.e.~that
$|g-g_k|\le \varepsilon_k$.
In particular, $g\in L^1(\mu)$. Since each $g_k$ is $P$-invariant and $P$ is a contraction on $L^1(\mu)$,
\[
\|Pg-g\|_{L^1}\le \|P(g-g_k)\|_{L^1}+\|g_k-g\|_{L^1}\le 2\|g-g_k\|_{L^1}\le 2\varepsilon_k.
\]
Letting $k\to\infty$, we obtain $Pg=g$.
Now apply Lemma~\ref{lem:asymp-stability} to the pair $(\Phi,\Psi^{(k)})$. Since $\frac1n\psi_n^{(k)}\to g_k$ $\mu$-a.e., we get
\[
g_k-\varepsilon_k
\le
\liminf_{n\to\infty}\frac1n\phi_n
\le
\limsup_{n\to\infty}\frac1n\phi_n
\le
g_k+\varepsilon_k
\qquad \mu\text{-a.e.}
\]
Using $|g-g_k|\le \varepsilon_k$ $\mu$-a.e., this yields
\[
g-2\varepsilon_k
\le
\liminf_{n\to\infty}\frac1n\phi_n
\le
\limsup_{n\to\infty}\frac1n\phi_n
\le
g+2\varepsilon_k
\qquad \mu\text{-a.e.}
\]
Letting $k\to\infty$, we conclude that
\[
\frac1n\phi_n(x)\longrightarrow g(x)
\qquad\text{for }\mu\text{-a.e. }x.
\]
Finally, applying Lemma~\ref{lem:asymp-stability} to the integral functionals and using
\[
\left|\int g\,d\mu-\int g_k\,d\mu\right|\le \|g-g_k\|_{L^1}\le \varepsilon_k,
\]
we obtain
\[
\int g\,d\mu=\lim_{n\to\infty}\frac1n\int\phi_n\,d\mu.
\]
If, in addition, $\mu$ is ergodic, then the $P$-invariant function $g$ is constant $\mu$-almost everywhere, exactly as in Theorem~\ref{Kingman}.} {This concludes the proof.}
\end{proof}

\begin{rem}\label{rem:ae-classes}
In the $L^1(\mu)$ framework, functions are identified $\mu$-a.e.\ and $P$ acts canonically on $L^\infty(\mu)$-classes.
Thus, in many situations, the inequality $\phi_{n+m}\le \phi_n+P^n\phi_m$ is only verified $\mu$-a.e.\ for convenient
representatives.  The asymptotic framework above is robust under such modifications: changing finitely many terms or changing
$\phi_n$ on $\mu$-null sets does not affect~\eqref{eq:asymp-def}. In particular, Corollary~\ref{cor:asymp-CD} applies
whenever an $L^\infty$-class admits a $P$-subadditive approximation in the sense of~\eqref{eq:asymp-def}.
\end{rem}

A convenient source of asymptotic $P$-subadditivity is when the subadditivity defect is
\emph{removable} after adding a sublinear correction (a coboundary) and, possibly, a residual
\emph{block error}. This mechanism is commonly referred to  
\emph{almost subadditive}~\cite{Derriennic, Schurger}. We first provide the simplest form of a sufficiently useful criterion for asymptotic $P$-subadditivity.
\begin{prop}\label{prop:suff-asymp-P-subadd}
{Let $\Phi=(\phi_n)_{n\ge1}$ be a process.}
If there exists a sequence of constants  $c_n\geq 0$ with  $c_n /n \to 0$ such that
\begin{equation} \label{eq:block-defect-constant}
       \phi_{n+m}\le \phi_n + P^n\phi_m  + c_m\qquad\text{for all }n,m\ge1,
\end{equation}
then $\Phi$ is asymptotically $P$-subadditive.
In particular, this sufficient condition holds if there exists $C\ge0$ such that
\[
   \phi_{n+m}\le \phi_n + P^n\phi_m + C \qquad\text{for all }n,m\ge1.
\]
\end{prop} 

\begin{proof}
Define a new process $\Psi=(\psi_n)_{n\ge1}$ by
\begin{equation}\label{eq:eta-def}
\psi_n
:=
\inf\Bigl\{
\sum_{j=1}^r P^{s_{j-1}}\phi_{n_j}
\;+\; \sum_{j=1}^r c_{n_j}
:\ r\ge1,\ \sum_{j=1}^r n_j=n
\Bigr\}
\end{equation}
where $s_0=0$ and $s_j = n_1 + \dots + n_j$. Concatenating a decomposition of $n$ with a decomposition of $m$ produces a decomposition of $n+m$.
Since each $c_k$ is a constant, $P^n c_k=c_k$, and therefore concatenation yields
\(
\psi_{n+m}\le \psi_n+P^n\psi_m
\)
for all $n,m\ge1$. Thus, $\Psi$ is $P$-subadditive.

On the other hand, for every $n$ we have $\psi_n\le \phi_n+c_n$ by taking the one-block decomposition ($r=1$) in~\eqref{eq:eta-def}. Hence
\begin{equation}\label{eq:eta-upper}
\psi_n-\phi_n\ \le\ c_n.
\end{equation}
Next, fix a decomposition $n=n_1+\cdots+n_r$. Iterating~\eqref{eq:block-defect-constant} along this decomposition gives
\[
\phi_n
\le
\phi_{n_1}+P^{n_1}\phi_{n_2}+\cdots+P^{n_1+\cdots+n_{r-1}}\phi_{n_r}
\;+\;c_{n_2}+\cdots+c_{n_r}.
\]
Since all $c_{n_i}\ge0$, the right-hand side is bounded above by the corresponding candidate
in~\eqref{eq:eta-def} (which additionally includes $c_{n_1}$). Therefore $\phi_n$ is bounded above
by every candidate in the infimum \eqref{eq:eta-def}, and hence
\begin{equation}\label{eq:eta-lower}
\phi_n\ \le\ \psi_n.
\end{equation}
Combining \eqref{eq:eta-upper}--\eqref{eq:eta-lower} yields
$0\ \le\ \|\psi_n-\phi_n\|_\infty\ \le\ c_n$,
so
\[
\limsup_{n\to\infty}\frac1n\|\psi_n-\phi_n\|_\infty
\le
\lim_{n\to\infty}\frac1n c_n
=0.
\]
In the compact setting, if each $\phi_n$ is upper semicontinuous, then every candidate appearing in~\eqref{eq:eta-def} is upper semicontinuous, and so is their infimum. Hence $\Psi$ belongs to the same class as $\Phi$.
Since $\Psi$ is $P$-subadditive, this is exactly~\eqref{eq:asymp-def} for $\Phi$.
\end{proof}

{Now we extend the previous criterion in a slightly more general framework.}

\begin{cor}\label{prop:removable-defect}
{Let $\Phi=(\phi_n)_{n\ge1}$ be a process. Assume there exist bounded measurable functions
$g_n,\rho_n$ on $X$ such that}
\begin{equation}\label{eq:removable-defect-sublinear}
\lim_{n\to\infty}\frac1n\|g_n\|_\infty=0,
\qquad
\lim_{n\to\infty}\frac1n\|\rho_n\|_\infty=0,
\end{equation}
and for all $n,m\ge1$,
\begin{equation}\label{eq:removable-defect-ineq}
\phi_{n+m}\ \le\ \phi_n+P^n\phi_m\ +\ \bigl(g_n+P^n g_m-g_{n+m}\bigr)\ +\ P^n\rho_m.
\end{equation}
Then $\Phi$ is asymptotically $P$-subadditive.
\end{cor}

\begin{proof}  Define the process $\Psi=(\psi_n)_{n\ge1}$ by $\psi_n:=\phi_n+g_n$.
Using~\eqref{eq:removable-defect-ineq} we obtain, for all $n,m\ge1$,
\begin{equation*}\label{eq:block-defect}
\begin{aligned}
\psi_{n+m}
=\phi_{n+m}+g_{n+m}
&\le
\phi_n+P^n\phi_m+g_n+P^n g_m+P^n\rho_m
\\ &=
\psi_n+P^n\psi_m+P^n\rho_m.
\end{aligned}
\end{equation*}
Using that $P^n\rho_m \leq c_m:=\|\rho_m\|_\infty$,  the above inequality becomes 
$$
  \psi_{n+m} \leq \psi_n+P^n\psi_m+c_m.
$$
Hence, by Proposition~\ref{prop:suff-asymp-P-subadd} we obtain that $\Psi$ is asymptotically $P$-subadditive.  
{Returning to $\Phi$, since $\psi_n-\phi_n=g_n$ and $\|g_n\|_\infty/n\to 0$, by the triangle inequality any asymptotic \(P\)-subadditive approximation of $\Psi$ also serves as an asymptotic \(P\)-subadditive approximation of $\Phi$.}
Thus $\Phi$ is asymptotically $P$-subadditive.
\end{proof}

}

\chapter{Mostly contracting, local contraction on average and quasi-compactness}
\label{s:mostly-contracting}

\abstract{{
This chapter establishes the main structural consequences of the mostly contracting
condition. We first study the subadditivity of the logarithm of local Lipschitz constants and apply Kingman's ergodic theorems for Markov operators to obtain several
equivalent formulations of the maximal Lyapunov exponent. In particular, mostly
contracting is characterized by uniform negativity of finite-time annealed exponents. We then show that this negativity implies local
contraction on average. This yields a
Lasota-Yorke type estimate for the annealed Koopman operator on
\(C^\alpha(X)\), and hence quasi-compactness. The chapter also describes the
ergodic and spectral consequences of quasi-compactness: finiteness of ergodic
stationary measures, identification of their supports with minimal invariant
sets, and criteria for simplicity of the eigenvalue \(1\), absence of other
peripheral eigenvalues, and spectral gap in terms of unique ergodicity,
aperiodicity, and mingling. The final part relates global contraction on
average to synchronization, proximality, and unique ergodicity of the associated
two-point motion.
}}

{
\section{Local Lipschitz constants and subadditivity}
\label{sec:local-lipschitz-subadditivity}

We begin with the elementary properties of local Lipschitz constants that will
be used to construct the subadditive processes associated with the Lyapunov exponents for
random maps.} Recall the definition of the local Lipschitz constant \(Lg(x)\)
from~\S\ref{sec:lyapunov-def}: {
\[
  Lg(x)
  \eqdef
  \lim_{r\to 0^+} \operatorname{Lip}(g|_{B(x,r)})
  =
  \limsup_{\substack{(y,z)\to(x,x)\\ y\ne z}}
  \frac{d(g(y),g(z))}{d(y,z)}.
\]
}
\begin{lem} \label{lem1} The map $Lg: X\to [0,\infty)$ is upper semicontinuous for any  Lipschitz map $g$ on $X$.  Moreover, for any pair of Lipschitz maps $g$ and $h$ on $X$, it satisfies the chain rule inequality
\[L(g \circ h)(x) \leq Lg(h(x))\cdot Lh(x), \quad \text{for every $x\in X$}.\]
\end{lem}
\begin{proof} Let us prove that $Lg$ is upper semicontinuous, that is, the set $E_a\eqdef \{x\in X :  Lg(x)<a\}$ is an open set for any $a>0$. To do this, for a fixed $z\in E_a$, by definition of  $Lg(z)$, there exists $r>0$ such that
$$a>L_rg(z) \eqdef \mathrm{Lip}(g|_{B(z,r)})
\geq \sup_{x\in B(z,r)} Lg(x) .$$
Thus, $B(z,r) \subset E_a$ and consequently, $E_a$ is open.

We now prove the chain rule inequality. Since $h$ is a Lipschitz map and hence uniformly continuous,  given $\rho > 0$, there exists $r > 0$ such that if $d(x, y) < r$, then $d(h(x), h(y)) < \rho$. Hence,
\[
\frac{d(g(h(x)), g(h(y)))}{d(x, y)} = \frac{d(g(h(x)), g(h(y)))}{d(h(x), h(y))} \cdot \frac{d(h(x), h(y))}{d(x, y)} \leq L_\rho g(h(x)) \cdot L_r h(x).
\]
Taking $\rho$ and $r$ tending to $0^+$, we get the desired chain rule.
\end{proof}

\section{Consequences of Kingman's subadditive theorems}
{First, we} recall the notation of the maximal  Lyapunov exponents $\lambda(\mu)$,  $\lambda(f)$ and $\lambda(x)$ introduced in~\S\ref{ss:mostlycontracting}:
{\begin{align*}
\lambda(\mu) &\eqdef \lim_{n\to\infty} \frac{1}{n} \int \log Lf^n_\omega(x) \, d\mathbb{P}d\mu \quad \text{and} \quad 
\lambda(x) \eqdef
\limsup_{n\to\infty} \
\frac{1}{n} \int \log Lf^n_\omega(x) \, d\mathbb{P}
\end{align*}
and
\begin{align*}
\lambda(f) &\eqdef \sup \{ \lambda(\mu): \mu \text{\ $f$-stationary}\}. 
\end{align*}
}
We also define the \emph{global maximal annealed Lyapunov exponent} by
$$
\lambda_{\max} = \lim_{n\to\infty}  \max_{x\in X} \frac{1}{n} \int \log Lf^n_\omega(x) \, d\mathbb{P}=\inf_{n\geq 1} \max_{x\in X} \frac{1}{n} \int \log Lf^n_\omega(x) \, d\mathbb{P}.
$$
\index{maximal Lyapunov exponents!\(\lambda_{\max}\), global maximal annealed exponent}
The existence of this limit is justified by Fekete's subadditive lemma 
since the sequence $(\max \int \log Lf^n_\omega \, d\mathbb{P})_{n\geq 1}$ is subadditive.  From Theorems~\ref{Kingmanuniform} and~\ref{Kingman}, we obtain: 
\begin{thm} \label{prop:equivalence}
Let $(X,d)$ and $(\Omega,\mathscr{F},\mathbb{P})$ be, respectively, a compact metric space and a Bernoulli probability space. Consider a Lipschitz random map $f:\Omega\times X\to X$ satisfying the integrability condition~\eqref{eq:integrability}. Then,
\begin{equation} \label{eq:exponent1}
    \lambda(f)=\max \left\{\lambda(\mu): \  \mu \in \mathcal{I} \ \text{ergodic}\right\}  =   \lambda_{\max}  = \sup_{x\in X} \lambda(x)= \sup_{\mu\in\mathcal{I}} \int \lambda(x) \, d\mu
\end{equation}
where $\mathcal{I}$\index{space of measures!\(\mathcal I\), $f$-stationary measures} is the set of stationary probability measures  for $f$. Moreover,
\begin{equation} \label{eq:exponent2}
   \lambda(f)  = \lim_{n\to \infty} \sup_{\mu \in \mathcal{I}} \frac{1}{n} \int \log Lf^n_\omega(x) \, d\mathbb{P}d\mu = \inf_{n\geq 1} \sup_{\mu \in \mathcal{I}} \frac{1}{n} \int \log Lf^n_\omega(x) \, d\mathbb{P}d\mu.
\end{equation}
Also, for every $\mu\in \mathcal{I}$, the following holds:
\begin{enumerate}[leftmargin=1cm, label=(\roman*), itemsep=0.1cm]
\item $\displaystyle \lambda(\mu) = \int \lambda(x)\, d\mu$  and $\lambda^+ \in L^1(\mu)$,
\item $\displaystyle \lambda(x)=\lim_{n\to\infty} \frac{1}{n} \int \log Lf^n_\omega(x) \, d\mathbb{P}$  for $\mu$-a.e.~$x\in X$, \\[-0.2cm]
\item $\lambda(x)=\lambda(\mu)$ for $\mu$-a.e.~$x$, provided $\mu$ is ergodic.
\end{enumerate}
{Furthermore, 
\[
\lambda(x) =\lambda(\omega,x)\eqdef \lim_{n\to\infty}\frac1n\log Lf_\omega^n(x)
\quad\text{for }(\mathbb P\times\mu)\text{-a.e.~}(\omega,x)\in \Omega\times X.
\]}
\end{thm}

{

\begin{proof}
For each $n \in \mathbb{N}$, by the integrability condition~\eqref{eq:integrability}, we define the extended real-valued function
$$
\phi_n(x) = \int \log Lf^{n}_\omega(x) \, d\mathbb{P} \quad \text{for } x \in X.
$$
 From Lemma~\ref{lem1} and the reverse Fatou lemma, $\phi_n$ is upper semicontinuous. 
Moreover, since the $m$-th iterate of the annealed Koopman operator $P:C(X)\to C(X)$ introduced in~\eqref{eq:koopman0} is given by
$$
P^m\varphi(x) = \int \varphi \circ f^m_\omega(x) \, d\mathbb{P}\quad \text{for } x \in X,
$$
we have, by the chain rule, and since $\mathbb{P} = p^\mathbb{N}$ is a Bernoulli probability measure, that
\begin{align*}
  \phi_{n+m}(x) &\leq \int \log Lf^n_{\omega}(x) \, d\mathbb{P} + \int \log Lf^m_{\sigma^n\omega}(f^n_\omega(x))\,d\mathbb{P} \\
                 &= \phi_n(x) +  \int \int \log Lf^m_{\omega'}(f^n_\omega(x))\,d\mathbb{P}(\omega') d\mathbb{P}(\omega) = \phi_n(x) + P^n\phi_m(x).
\end{align*}
This implies that $(\phi_n)_{n\geq 1}$ is a $P$-subadditive sequence. Furthermore, as is well known,  the $P^*$-invariant probability measures $\mu$ are exactly the $f$-stationary measures.
Hence, in view of~\eqref{eq:Kingman2}, \eqref{eq:Kingman1}, and~\eqref{eq:Lambdamax}, we have,  we have that for every $x\in X$ and $f$-stationary measure $\mu$,
\begin{align*}
   \lambda(x) &= \limsup_{n\to \infty} \frac{1}{n} \phi_n(x), &
\lambda_{\max}&=\lim_{n\to\infty}  \max_{x\in X}   \frac{1}{n} \phi_n(x)=\Lambda_{\max},
 \\ \lambda(\mu)&=\lim_{n\to \infty}  \frac{1}{n} \int \phi_n \, d\mu= \Lambda(\mu) & \quad  \lambda(f)&=\Lambda.
\end{align*}
With the above notation, from Theorem~\ref{Kingmanuniform}, we 
obtain~\eqref{eq:exponent1} and~\eqref{eq:exponent2}.

Now, fixing an $f$-stationary probability measure $\mu$, again by condition~\eqref{eq:integrability}, we have that $\phi^+_1\in L^1(\mu)$, and thus by Remark~\ref{rem:kingman}, we can apply Theorem~\ref{Kingman} to the operator $P$ and the $P$-subadditive sequence $(\phi_n)_{n\geq 1}$. Then the function $g$ in Theorem~\ref{Kingman} satisfies  $g(x)=\lambda(x)$ for $\mu$-a.e.~$x\in X$, and from this, one easily concludes (i), (ii) and (iii).

{
Finally, we prove the last part of the theorem. First, let $\mu$ be an ergodic $f$-stationary measure.  As  introduced in~\S\ref{sec:lyapunov-def}, by Kingman's subadditive ergodic theorem, the pointwise Lyapunov exponent 
\[
\lambda(\omega,x)\eqdef \lim_{n\to\infty}\frac1n\log Lf_\omega^n(x) 
\]
exists for \((\mathbb{P}\times \mu)\)-a.e.~\((\omega,x)\). Moreover, since \(\bar\mu=\mathbb{P}\times \mu\) is ergodic, it is equal, \(\bar\mu\)-almost everywhere, to the constant 
\[
\lambda(\mu) \eqdef \lambda(\bar \mu) =\lim_{n\to\infty}\frac1n\int\log Lf_\omega^n(x)\,d\bar\mu.
\]
On the other hand, item~\textup{(iii)} gives
$\lambda(x)=\lambda(\mu)$ for $\mu$-a.e.~$x\in X$. Therefore,
\[
\lambda(\omega,x)=\lambda(x)
\qquad\text{for }(\mathbb P\times\mu)\text{-a.e.~}(\omega,x)\in \Omega\times X
\]
whenever \(\mu\) is ergodic. Now let \(\mu\) be an arbitrary \(f\)-stationary probability measure. Define
\[
G\eqdef \Bigl\{(\omega,x)\in \Omega\times X:
\lambda(\omega,x)=\lambda(x) \ \ \text{and the limits exist}\Bigr\}. 
\]
Let $G_x$ be the section of $G$ at $x$ and write $F\eqdef \{x\in X:\mathbb P(G_x)=1\}$.
Then \(F\) is measurable. In the ergodic case already proved, every ergodic \(f\)-stationary probability measure \(\nu\) satisfies $(\mathbb P\times \nu)(G)=1$,
and therefore, by Fubini's theorem, $\nu(F)=1$. By the ergodic decomposition of \(\mu\), since \(\nu(F)=1\) for every ergodic component \(\nu\), we obtain
$\mu(F)=1$.  Applying Fubini's theorem once more, we conclude that
$(\mathbb P\times\mu)(G) \geq \int_F \mathbb P(G_x)\,d\mu(x)=1$ and hence
\[
\lambda(\omega,x)=\lambda(x)
\qquad\text{for }(\mathbb P\times\mu)\text{-a.e.~}(\omega,x)\in \Omega\times X.
\]
This proves the last part of the theorem.}
\end{proof}

\begin{cor} \label{cor:equivalencia}
Under the assumption of Theorem~\ref{prop:equivalence}, the following assertions are equivalent:
\begin{enumerate}[leftmargin=1cm,label=(\roman*)]
    \item $f$ is mostly contracting (i.e., $\lambda(f)<0$);
    \item $\lambda(x) < 0$ for every $x \in X$;
    \item $\lambda_{\max}<0$;
    \item There exists  $n\geq 1$ such that $\int \log Lf^{n}_\omega (x) \, d\mathbb{P} <0$ for every $x \in X$.
\end{enumerate}
\end{cor}
\begin{proof}
In view of~\eqref{eq:exponent1}, it is evident that (i), (ii), and (iii) are equivalent. Moreover, given~that
$$
\lambda_{\max} = \inf_{n\geq 1}  \max_{x\in X} \frac{1}{n} \int \log Lf^n_\omega(x) \, d\mathbb{P},
$$
(iii) directly implies (iv). Finally, by assuming (iv), we deduce (i) from the last representation of $\lambda(f)$ in~\eqref{eq:exponent2}.
\end{proof}

\begin{rem} \label{rem:equivalencia}
{
Since $x \mapsto \int \log Lf^{n}_\omega (x) \, d\mathbb{P}$ is upper semicontinuous and \(X\) is compact, condition~\textup{(iv)} implies that this function attains a negative maximum. Hence there exists \(\eta>0\) such that
\[
\int \log Lf^{n}_\omega (x) \, d\mathbb{P} \le -\eta
\qquad\text{for every }x\in X.
\]
In particular, condition~\textup{(iv)} provides a uniform bound away from zero.}
\end{rem}

\section{Local contraction on average} In what follows, we prove two key propositions. The first proposition asserts that, under the assumptions of Theorem~\ref{thmA}, a mostly contracting random map is locally contracting on average (recall Definition~\ref{def:local-contracting-average}).

\begin{prop} \label{prop:alpha0}  
Let $f$ be a mostly contracting random map such that~\eqref{eq:integral_condition} holds. Then, 
there exist $\alpha_0>0$, $r>0$ and $n\in\mathbb{N}$  such that for any $0<\alpha \leq \alpha_0$  there is  $q<1$  satisfying
$$
\int d(f^n_\omega(x),f^n_\omega(y))^\alpha \, d\mathbb{P} \leq q \, d(x,y)^\alpha \quad \text{for any $x,y\in X$ with $d(x,y)<r$.}
$$
\end{prop}
\begin{proof}
Fix \(x\in X\) and let \(n\in \mathbb{N}\) be as specified in item (iv) of Corollary~\ref{cor:equivalencia}. For any \(\alpha > 0\), consider the functions
\[
\psi(\omega) \eqdef \log Lf_\omega^n(x) \ \text{and} \ \psi_\alpha(\omega) \eqdef \frac{Lf^n_\omega(x)^\alpha-1}{\alpha} = \frac{e^{\alpha \log Lf^n_\omega(x)}-1}{\alpha}= \psi(\omega) + O(\alpha)
\]
defined for \(\omega \in \Omega\). We apply Lemma~\ref{Fatou} to the sequence \((\psi_\alpha)_\alpha\) on the probability space \((\Omega,\mathscr{F},\mathbb{P})\). To do this, we need to bound $\psi_\alpha$ by a $\mathbb{P}$-integrable function $g$ for every $\alpha>0$ small enough. We consider two cases.

First, \(Lf^n_\omega(x) \leq 1\). Then \(\psi_\alpha \leq 0\) for any \(\alpha > 0\), and we can take \(g=0\) as the required integrable function. Second, \(Lf^n_\omega(x) > 1\). In this case,  since \((e^t-1)/t \leq e^t\) for  \(t>0\), we have
\[
\frac{e^{\alpha \log Lf^n_\omega(x)}-1}{\alpha \log Lf^n_\omega(x)} \leq e^{\alpha \log Lf^n_\omega(x)}.
\]
Thus, for any \(0<\alpha \leq \beta/2\) with \(\beta>0\) is given in the integrability condition~\eqref{eq:integral_condition}, and using that $\log t \leq t^{\alpha} / \alpha$ for $t\geq 1$ (and $\alpha=\beta/2$), it holds 
\[
\psi_\alpha (\omega) \leq Lf^n_\omega(x)^\alpha \log Lf^n_\omega(x) \leq \frac{2}{\beta} \, Lf^n_\omega(x)^{\alpha+\frac{\beta}{2}} \leq \frac{2}{\beta} \, \mathrm{Lip}(f^n_\omega)^\beta \eqdef g(\omega).
\]
By the chain rule and taking into account that the random variables $X_n(\omega)=\mathrm{Lip}(f_{\omega_n})^\beta$ are i.i.d., we have
\[
\int g(\omega) \, d\mathbb{P} \leq 
\frac{2}{\beta} \int \mathrm{Lip}(f_{\omega_{n-1}})^\beta \cdot \ldots \cdot \mathrm{Lip}(f_{\omega_{0}})^\beta \, d\mathbb{P} = \frac{2}{\beta} \left(\int \mathrm{Lip}(f_\omega)^{\beta} \, d\mathbb{P}\right)^n <\infty.
\]

Therefore, in both cases above, we can apply Lemma~\ref{Fatou}, and according to (iv) of Corollary~\ref{cor:equivalencia}, we obtain
\[
\limsup_{\alpha \to 0} \int \psi_\alpha(\omega) \, d\mathbb{P} \leq \int \psi(\omega) \, d\mathbb{P} =\int \log Lf_\omega^n(x) \, d\mathbb{P}< 0.
\]
Thus, there exists \(0<\alpha_0(x)\leq \beta/2 \) small enough such that
\[
\int \frac{Lf^n_\omega(x)^{\alpha_0(x)}-1}{\alpha_0(x)}\, d\mathbb{P} = \int \psi_{\alpha_0}(\omega) \, d\mathbb{P} < 0,
\]
which implies $\int Lf^n_\omega(x)^{\alpha_0(x)}\,  d\mathbb{P}<1$.
Then, there exist \(r(x)>0\) and \(q(x)<1\) such that
\begin{equation}\label{eq:qmenor1}
\int L_{r(x)}f^n_\omega(x)^{\alpha_0(x)}  d\mathbb{P}<q(x).
\end{equation}
Since \(X\) is compact, we can find finitely many points \(x_1,\dots,x_k\) such that 
\[
X=B_1\cup \dots \cup B_k, \quad  \text{where \ \  \(B_i=B(x_i,r_i)\)}
\]  
with \(r_i=r(x_i)\) for \(i=1,\dots, k\).
  Let \(r\) be the Lebesgue number of this open cover, and set \(q_i=q(x_i)\), \(\alpha_i=\alpha_0(x_i)\), \(i=1,\dots,r\), and \(\alpha_0=\min\{ \alpha_i: i=1,\dots, k\}\).
Now, for every \(0<\alpha\leq \alpha_0\), denote 
$p_i=\alpha_i/\alpha>1$ for \(i=1,\dots,k\) 
and 
$$
q=\max\{ q_i^{1/p_i}: i=1,\dots, k\}<1.
$$ 
Then, for \(x,y\in X\) with \(d(x,y)<r\), we have that \(x,y \in B_i\) for some \(i\in \{1,\dots,k\}\). By H\"older's inequality for \(p_i\) and~\eqref{eq:qmenor1}, we get
\begin{align*} 
    \int L_{r_i}f^n_\omega(x_i)^\alpha \, d\mathbb{P} &\leq \left(\int L_{r_i}f^n_\omega(x_i)^{p_i\alpha} \, d\mathbb{P}\right)^{1/p_i} \\ &=\left(\int L_{r_i}f^n_\omega(x_i)^{\alpha_i} \, d\mathbb{P}\right)^{1/p_i }<q_i^{1/p_i}\leq q.
\end{align*}
This implies that
\[
\int d(f^n_\omega(u),f^n_\omega(v))^\alpha\, d\mathbb{P} \leq q \, d(u,v)^\alpha \quad \text{for all $u,v\in B(x_i,r_i)=B_i$.}
\]
In particular, taking \(u=x\) and \(v=y\), we obtain the desired inequality.
\end{proof}


\begin{prop} \label{prop:quasi-compact}
Let $f:\Omega\times X \to X$ be a locally contracting on average random map and let $r>0$, $\alpha>0$, $q<1$ and $n\geq 1$ be the constants in Definition~\ref{def:local-contracting-average}. Moreover,  assume that $\mathrm{Lip}(f_\omega)^\beta$  is $\mathbb{P}$-integrable for some $\beta \geq \alpha$. Then   
\begin{align} \label{eq:quasi-compact}
P(C^\alpha(X))\subset C^\alpha(X) \quad  \text{and} \quad
|P^n\phi|_{\alpha} \leq q |\phi|_\alpha + C \|\phi\|_{\infty} \quad \text{for all } \phi\in C^\alpha(X)
\end{align}
where $C=\frac{2}{r^\alpha}>0$. 

\end{prop}
\begin{proof}
We aim to prove that $P(C^\alpha(X))\subset C^\alpha(X)$. To do this, for $x,y\in X$,
\begin{align*}
 \left|P\phi(x)-P\phi(y)\right|
 \leq |\phi|_\alpha \int d(f_\omega(x),f_\omega(y))^\alpha \, d\mathbb{P}  \leq |\phi|_\alpha \int \mathrm{Lip}(f_\omega)^\alpha\, d\mathbb{P} \cdot d(x,y)^\alpha.
\end{align*}
Given that $\alpha \leq \beta$ and $\mathrm{Lip}(f_\omega)^\beta$ is $\mathbb{P}$-integrable, then $\mathrm{Lip}(f_\omega)^\alpha$ is also $\mathbb{P}$-integrable. 
Thus, we deduce $\int \mathrm{Lip}(f_\omega)^\alpha\, d\mathbb{P}<\infty$. Hence, $P\phi\in C^\alpha(X)$.

Now, let us prove the inequality in~\eqref{eq:quasi-compact}. Consider any $x, y\in X$ and $\phi\in C^\alpha(X)$. 
If $d(x,y)<r$, then Definition~\ref{def:local-contracting-average} implies
\begin{equation}\label{eq:1} \begin{aligned}
  \left|P^n\phi(x)-P^n\phi(y)\right| &\leq \int \left|\phi ( f^n_\omega(x)) - \phi (f^n_\omega(y))\right|\, d\mathbb{P}\\
  &\leq   |\phi|_\alpha  \int  d(f^n_\omega(x),f^n_\omega(y))^\alpha \, d\mathbb{P} \leq q \, |\phi|_\alpha d(x,y)^\alpha.
\end{aligned}\end{equation}
On the other hand, if $d(x,y)\geq r$, and considering that $\|P\varphi\|_{\infty}\leq \|\varphi\|_{\infty}$ for every continuous function $\varphi:X\to \mathbb{R}$, we obtain
\begin{equation} \label{eq:2}
  \left|P^n\phi(x)-P^n\phi(y)\right| \leq 2\|P^n\phi\|_{\infty} \leq 2 \|\phi\|_\infty \leq C \,\|\phi\|_\infty d(x,y)^\alpha,
\end{equation}
where $C\eqdef\frac{2}{r^\alpha}>0$. Combining~\eqref{eq:1} and~\eqref{eq:2}, we deduce that
$$
\left|P^n\phi(x)-P^n\phi(y)\right| \leq \left( q |\phi|_\alpha + C\|\phi\|_\infty\right) d(x,y)^\alpha \quad \text{for all } x,y\in X.
$$
This confirms~\eqref{eq:quasi-compact} as required.
  \end{proof}

  \section{Proof of Theorem~\ref{thmA}} We apply \cite[Thm.~II.5]{HH:01} which states that a Markov operator $P$ of a Banach space $(E,\|\cdot\|)$ is quasi-compact provided 
  \begin{enumerate}[leftmargin=0.8cm,label=(\roman*)]
      \item $P(\{\varphi\in E: \|\varphi\|\leq 1\})$ is relatively compact in $(E,|\cdot |)$
      \item  $|P\varphi|\leq M \|\varphi\|$ for all $\varphi \in E$, 
      \item there exist $n\in \mathbb{N}$, $q<1$ and $K>0$ such that $\|P^n\phi\|\leq q \|\phi\| + K |\phi|$,
  \end{enumerate}
  where $|\cdot|$ denotes a continuous seminorm on $E$.
  To verify these assumptions,   
  we consider $E=C^\alpha(X)$, $\|\cdot \|=\|\cdot\|_\alpha$ and $|\cdot|=\|\cdot\|_\infty$ where $0<\alpha \leq \alpha_0$ given in Proposition~\ref{prop:alpha0}. By this proposition, we know that  $f$ is locally contracting on average with constant $\alpha\leq \beta$ where $\beta>0$ is given in~\eqref{eq:integral_condition}, and thus, we are under the assumptions of Proposition~\ref{prop:quasi-compact}. 
  
  \begin{lem} \label{lem:local-contracting-average-quasi-compact} Let $f$ be a random map under the assumptions of Proposition~\ref{prop:quasi-compact}. Then, the Koopman operator $P$ associated with $f$ is quasi-compact on $C^\alpha(X)$.
  \end{lem}
  \begin{proof}
  By Proposition~\ref{prop:quasi-compact},   
  $P$ is a Markov operator acting on $C^\alpha(X)$. By the Arzela-Ascoli theorem, the closed and equicontinuous subset $B=\{\varphi\in C^\alpha(X): \|\varphi\|_\alpha \leq 1\}$ of the space of continuous functions of $X$ is compact.
  Thus, we have that $P(B)$ is compact in $(C^\alpha(X), \|\cdot\|_\infty)$. Moreover, clearly $\|P\phi\|_\infty \leq \|\phi\|_\infty \leq \|\phi\|_\alpha$.  Finally, for $n\in \mathbb{N}$, $q<1$ and $C>0$ mentioned in the assumptions of  Proposition~\ref{prop:quasi-compact},  by equation~\eqref{eq:quasi-compact}, we have that
  $$\|P^n\phi\|_\alpha =\|P^n\phi\|_\infty + |P^n\phi|_\alpha  \leq \|\phi\|_\infty + q \, |\phi|_\alpha + C\, \|\phi\|_\infty \leq    q\|\phi\|_\alpha + 2C \, \|\phi\|_\infty.$$
  Consequently, $P$ is quasi-compact. 
  \end{proof}


The above lemma establishes the first part of Theorem~\ref{thmA}.  
The final part of this theorem follows from the subsequent result. Let $\delta_{ij}$ denote the Kronecker delta and recall that \( f \) is 
\begin{enumerate}[leftmargin=0.8cm,itemsep=0.05cm,label=$\bullet$]    
\item \emph{aperiodic} if there do not exist \( p \geq 2 \) pairwise disjoint closed subsets \( F_1, \ldots, F_p \) of \( X \) such that \( f_\omega(F_i) \subset F_{i+1} \) for \( i = 1, \ldots, p-1 \) and \( f_\omega(F_p) \subset F_1 \) for \(\mathbb{P}\)-a.e.~\(\omega \in \Omega\).
\item \emph{mingled} if there are no two disjoint closed subsets $f^k$-invariant for some positive integer $k$.
\end{enumerate}
\index{spectral properties!peripheral eigenvalues}
\begin{thm} \label{thm:Herver-improved}
Let $(X,d)$ and $(\Omega,\mathscr{F},\mathbb{P})$ be a compact metric space and a Bernoulli product probability space.
Consider the Koopman operator {\( P:C(X) \to C(X) \)} associated with a random map \( f:\Omega \times X \to X \).  
Suppose that for some \( \alpha \geq 0 \), \( P \) is quasi-compact on \( C^\alpha(X) \). Then:
\begin{enumerate}[leftmargin=0.8cm,label=(\roman*)]
    \item The multiplicity of the eigenvalue \( 1 \) is equal to the number of minimal \( f \)-invariant closed subsets of \( X \). This number coincides with the number of ergodic \( f \)-stationary probability measures. Furthermore, the topological supports of these measures are pairwise disjoint and are exactly the minimal \( f \)-invariant closed subsets of \( X \). 
    
    In particular, \( 1 \) is a simple eigenvalue of \( P \) if and only if there is a unique minimal \( f \)-invariant set, which is equivalent to \( f \) being uniquely ergodic. \\[-0.3cm]
    
    More specifically, if \( r \geq 1 \) denotes the dimension of \( \mathrm{Ker}(P-1) \), then there exist minimal \( f \)-invariant closed subsets \( F_1,\dots,F_r \) of \( X \), a basis \( \{g_1,\dots,g_r\} \) of \( \mathrm{Ker}(P-1) \) formed by non-negative functions in \( C^\alpha(X) \), and ergodic \( f \)-stationary probability measures \( \mu_1,\dots,\mu_r \) such that:
    $$
    g_i|_{F_j} = \delta_{ij}, \quad F_j = \mathrm{supp}(\mu_j), \quad \text{and} \quad \sum_{i=1}^r g_i = 1.
    $$
    \item \( 1 \) is the unique eigenvalue of \( P \) with modulus \( 1 \) if and only if \( f \) is aperiodic.
    \item \( P \) has a spectral gap if and only if \( f \) is mingled.
\end{enumerate}
\end{thm}

\begin{proof} 
The proof of the first item follows the argument of~\cite[Theorems~2.3 and~2.4]{Her:94}. 

\emph{$\bullet$ Proof of (i):} By quasi-compactness, the sequence of operators \( A_n = \frac{1}{n}\sum_{k=0}^{n-1}P^{k} \) converges in the strong operator topology to the spectral projection \( \Pi \) onto \( \mathrm{Ker}(P-1) \), which satisfies \( \Pi P = P\Pi=\Pi \), cf.~\cite[Corollary~2 and Theorem~5]{Na:18}. Consequently, every non-empty closed \( P \)-invariant subspace \( \mathcal{E} \) of \( C^\alpha(X) \) contains a \( P \)-invariant function.  

\begin{claim} \label{claim:delta} 
Let \( F_1, \dots, F_r \) be pairwise disjoint \( f \)-invariant closed subsets of \( X \). Then, for each \( i = 1, \dots, r \), there exists a \( P \)-invariant function \( g_i \) in \( C^\alpha(X) \) such that \( g_i|_{F_j} = \delta_{ij} \) for \( j = 1, \dots, r \).  
\end{claim}

\begin{proof}
For each \( i = 1, \ldots, r \), consider 
\[
\mathcal{E}_i = \{\phi \in C^\alpha(X) : \phi|_{F_j} = \delta_{ij} \ \text{for} \ j = 1, \ldots, r \}.
\] 
By Lemma~\ref{lem:Urysohn}, there exists \( \phi \in C^\alpha(X) \) such that \( \phi|_{F_i} = 1 \) and \( \phi|_{F_j} = 0 \) for \( j \neq i \); thus, \( \mathcal{E}_i \) is non-empty. Moreover, if \( \phi \in \mathcal{E}_i \), since \( F_j \) is \( f \)-invariant, 
\[
P\phi(x) = \int \phi \circ f_\omega(x) \, d\mathbb{P} = \delta_{ij} \quad \text{for } x \in F_j \text{ and } j = 1, \dots, r.
\]
Consequently, \( \mathcal{E}_i \) is also a \( P \)-invariant closed subset of \( C^\alpha(X) \). Thus, it contains a \( P \)-invariant function \( g_i \). 
\end{proof}

Let \( F_1, \ldots, F_r \) be some minimal \( f \)-invariant closed subsets of \( X \). In particular, they are pairwise disjoint. By Claim~\ref{claim:delta}, there exist \( P \)-invariant functions \( g_1, \ldots, g_r \) in \( C^\alpha(X) \) such that \( g_i|_{F_j} = \delta_{ij} \). Hence, these functions are linearly independent, implying their number is less than or equal to the dimension of \( \mathrm{Ker}(P-1) \). Thus, the number of minimal $f$-invariant closed subsets is finite and less than this dimension.

From now on, assume \( F_1, \ldots, F_r \) are all the minimal \( f \)-invariant closed subsets. 

\begin{claim} \label{claim:constant} 
For any  \( P \)-invariant \( \phi \in C^\alpha(X) \), \( \phi|_{F_i} \) is constant for every \( i = 1, \dots, r \). 
\end{claim}

\begin{proof}
Let \( M_i = \max \{ \phi(x) : x \in F_i \} \) and \(\tilde{F_i} =\{x \in F_i : \phi(x) = M_i \} \). Since, 
\[
M_i = \phi(x) = P\phi(x) = \int \phi \circ f_\omega(x) \, d\mathbb{P} \quad \text{for every } x \in \tilde{F}_i,
\]
and \( \phi \circ f_\omega(x) \leq M_i \) for \( \mathbb{P} \)-a.e.~\( \omega \in \Omega \), 
we have \( \phi \circ f_\omega(x) = M_i \) almost surely. This implies that the set $\tilde{F}_i$ is a non-empty \( f \)-invariant closed subset of \( F_i \), so by minimality, it must equal \( F_i \). 
\end{proof}

\begin{claim} \label{claim:zero}  
For any eigenvector \( \phi \in C^\alpha(X) \) associated with an eigenvalue \( \rho \) of modulus \( 1 \), if \( \phi|_{F_i} = 0 \) for every \( i = 1, \dots, r \), then \( \phi = 0 \). 
\end{claim}

\begin{proof}
Let \( M = \max \{ |\phi(x)| : x \in X \} \) and $\tilde{F}=\{x\in X\colon |\phi(x)|=M\}$. Similarly as before, since
\begin{equation*}
        \begin{split} 
    M =|\phi(x)|=|\rho \phi(x)| = |P\phi(x)|\leq \int \big|\phi\circ f_\omega(x) \big| \, d\mathbb{P} \quad \text{and} \\ 
    |\phi \circ f_\omega(x)| \leq M  \quad 
 \text{for every $x\in \tilde{F}$},
 \end{split}
\end{equation*}
 we have $|\phi \circ f_\omega(x)|=M$ almost surely. 
This implies that the set $\tilde{F}$ is a non-empty $f$-invariant closed subset of $X$, so it contains at least one set $F_i$, and thus $M=0$.
\end{proof}

Let \( g_1, \ldots, g_r \) be defined as above. By Claim~\ref{claim:constant}, for any \( P \)-invariant function \( \phi \in C^\alpha(X) \), \( c_i = \phi|_{F_i} \) is constant. Then \( \phi - (c_1g_1 + \dots + c_r g_r) \) is \( P \)-invariant and null on every \( F_i \), \( i = 1, \dots, r \). By Claim~\ref{claim:zero}, \( \phi = c_1g_1 + \dots + c_r g_r \). Thus, \( \{g_1, \ldots, g_r\} \) forms a basis of \( \mathrm{Ker}(P-1) \). In particular, the multiplicity of \( 1 \), i.e., the dimension of \( \mathrm{Ker}(P-1) \), is equal to the number of minimal \( f \)-invariant subsets of \( X \). Moreover, since \( 1_X \) is \( P \)-invariant and \( c_i = 1_X|_{F_i} = 1 \) for \( i = 1, \dots, r \), we also have \( 1 = g_1 + \dots + g_r \).

To conclude the first item of the theorem, we now describe the ergodic \( f \)-stationary probability measures. Let us consider the space \( \mathcal{S}_P \) of linear functionals 
\( \mu\colon C(X) \to \mathbb{R} \) such that \( \mu(\phi) = \mu(P\phi) \) for 
every \( \phi \in C(X) \). By the Riesz-Markov theorem, this vector space can be identified with the affine space of complex-valued Radon \( f \)-stationary measures. 

\begin{claim} \label{claim:SP} 
\( \mathcal{S}_P \) has dimension \( r = \dim \mathrm{Ker}(P-1) \). Moreover, there exists a basis \( \{\mu_1, \ldots, \mu_r\} \) of \( \mathcal{S}_P \), where each \( \mu_i \) can be identified with an \( f \)-stationary probability measure.  
\end{claim}

\begin{proof}
For any \( \mu \in \mathcal{S}_P \), \( \phi \in C(X) \), and \( n \geq 1 \), we have \( \mu(\phi) = \mu(A_n\phi) =  \mu(\Pi\phi) \). 
Recall that \( \Pi \phi \) is \( P \)-invariant. Hence, \( \Pi \phi = c_1 g_1 + \dots + c_r g_r \), where \( c_i = (\Pi \phi)|_{F_i} \). Picking \( x_i \in F_i \) for \( i = 1, \dots, r \), we have \( c_i = \mu_i(\phi) \), where \( \mu_i = \delta_{x_i} \circ \Pi \). 

Moreover, since
\[
\mu_i(\phi) =  \delta_{x_i}(\Pi\phi) = \delta_{x_i}(P\Pi\phi) = \delta_{x_i}(\Pi P \phi) = \mu_i(P\phi),
\]
it follows that \( \mu_i \in \mathcal{S}_P \). Furthermore, the family \( \{\mu_1, \ldots, \mu_r\} \) is linearly independent. For any \( \mu \in \mathcal{S}_P \), we have 
\[
\mu(\phi) =  \mu(\Pi\phi) = \sum_{i=1}^r c_i \mu(g_i) = \sum_{i=1}^r \mu(g_i) \mu_i(\phi) \quad \text{for every \( \phi \in C(X) \)},
\]
which proves that the family \( \{\mu_1, \ldots, \mu_r\} \) forms a basis of \( \mathcal{S}_P \). Since each \( \mu_i \) can be identified with an \( f \)-stationary probability measure, the claim is proved. 
\end{proof}

Since the ergodic \( f \)-stationary probability measures are linearly independent, Claim~\ref{claim:SP} implies that the number of such measures is less than or equal to \( r \). Let us now show that this number is exactly \( r \). 

The following claim is well known in the folklore of random dynamical systems, though a reference in this generality is not easily found. For completeness, we provide a detailed proof.

\begin{claim} \label{claim:invariant-support}
Let \( \mu \) be an \( f \)-stationary measure associated with a continuous random map. Then the topological support \( \mathrm{supp}\, \mu \) is \( f \)-invariant.
\end{claim}

\begin{proof}
Define \( S = \mathrm{supp}\, \mu \), which is, by definition, the smallest closed subset of \( X \) satisfying \( \mu(S) = 1 \). Consequently, its complement \( S^c = X \setminus S \) is open and satisfies \( \mu(S^c) = 0 \). Using the stationarity of \( \mu \), we obtain
\[
0 = \mu(S^c) = \int \mu(f_\omega^{-1}(S^c)) \, d\mathbb{P}.
\]
It follows that \( \mu(f_\omega^{-1}(S^c)) = 0 \) for \( \mathbb{P} \)-a.e.~\( \omega \in \Omega \). Let \( \Omega_0 \subseteq \Omega \) be the set of full \( \mathbb P \)-measure where \( f_\omega \) is continuous and \( \mu(f_\omega^{-1}(S^c)) = 0 \). Fix an arbitrary \( \omega \in \Omega_0 \). We aim to show that \( f_\omega(S) \subset S \).

Suppose, for contradiction, that \( f_\omega(S) \not\subseteq S \). Then there exists \( x_0 \in S \) such that \( f_\omega(x_0) \in S^c \). Since \( S^c \) is open and \( f_\omega \) is continuous, the preimage \( U = f_\omega^{-1}(S^c) \) is an open subset of \( X \) containing \( x_0 \), i.e., \( x_0 \in U \cap S \). By the definition of \( S \), every open set intersecting \( S \) must have positive \( \mu \)-measure. However, \( U = f_\omega^{-1}(S^c) \), and for \( \omega \in \Omega_0 \), we established that \( \mu(f_\omega^{-1}(S^c)) = 0 \). 
This yields \( \mu(U) = 0 \), which contradicts that \( \mu(U) > 0 \).  This contradiction shows that \( f_\omega(S) \subset S \) for all \( \omega \in \Omega_0 \), completing the proof.
\end{proof}

If \( F \) is a minimal \( f \)-invariant closed subset of \( X \), then there exists an ergodic \( f \)-stationary probability measure \( \mu \) supported in \( F \). By Claim~\ref{claim:invariant-support}, \( \mathrm{supp}\, \mu \) is an \( f \)-invariant closed set. By the minimality of \( F \), we conclude \( F = \mathrm{supp}\, \mu \). Since the number of disjoint minimal \( f \)-invariant closed subsets of \( X \) is exactly \( r \), the number of ergodic \( f \)-stationary probability measures is greater than or equal to \( r \). Combining this with the previous inequality, we conclude that this number is exactly \( r \). This completes the proof of~(i).

\emph{$\bullet$ Proof of (iii):} Assume first that \( f \) is not mingled. Hence, there exists an integer \( k \) such that the number of minimal \( f^k \)-invariant subsets of \( X \) is at least \( 2 \). By (i), we deduce that the dimension of \( \mathrm{Ker}(P^k-1) \) is at least \( 2 \). Consequently, \( P \) does not have a spectral gap, since otherwise \( \mathrm{Ker}(P^n-1) \) would be generated by \( 1_X \) for any \( n \geq 1 \). 

Conversely, assume that \( f \) is mingled. According to~\cite[Corollary~4]{Na:18}, there exists an integer \( n \) such that all the eigenvalues of modulus one are \( n \)-th roots of unity. Hence, all the corresponding eigenspaces are included in \( \mathrm{Ker}(P^n-1) \). Since \( f \) is mingled, there is exactly one minimal \( f^n \)-invariant closed subset of \( X \). By (i), it follows that \( \mathrm{Ker}(P^n-1) \) has dimension~\( 1 \). Therefore, \( P \) has a spectral gap. This completes the proof of (iii). 

\emph{$\bullet$ Proof of (ii):} If \( f \) is not aperiodic, there exists a sequence of disjoint non-empty closed subsets \( F_1', \ldots, F_m' \) such that \( f_\omega(F_i') \subset F_{i+1}' \) and \( f_\omega(F_m') \subset F_1' \) for \( \mathbb{P} \)-a.e.~\( \omega \). These sets are \( f^m \)-invariant. By Claim~\ref{claim:delta}, we can find a \( P^m \)-invariant function \( g \) such that \( g = 1 \) on \( F_1' \) and \( g = 0 \) on \( F_i' \) for \( i \geq 2 \). Since \( (Pg)|_{F_1'} = 0 \), this function cannot be \( P \)-invariant. Hence, there exists an eigenvalue of modulus one of \( P \) other than \( 1 \).

Now assume \( f \) is aperiodic. Denote by \( F_1, \ldots, F_r \) the minimal \( f \)-invariant closed subsets of \( X \). The restriction of \( f \) to \( F_i \) is mingled (since the restriction is minimal and aperiodic) for every \( i = 1, \ldots, r \). Thus, by item (iii), the restriction \( P_{F_i} \) of \( P \) on \( C^\alpha(F_i) \), given by 
\[ 
P_{F_i} \phi (x) = \int \phi \circ f_\omega(x) \, d\mathbb{P}, \quad \text{for } x \in F_i,
\]
has a spectral gap. Consequently, \( P_{F_i} \) has no unitary eigenvalues except \( 1 \). 

On the other hand, suppose \( \phi \in C^\alpha(X) \) satisfies \( P\phi = \lambda \phi \) with \( |\lambda| = 1 \) and \( \lambda \neq 1 \). Then, for every \( i = 1, \ldots, r \),
\[
P_{F_i}(\phi|_{F_i}) = \int \phi|_{F_i} \circ f_\omega \, d\mathbb{P} = (P\phi)|_{F_i} = \lambda \phi|_{F_i}.
\]
This implies \( \phi|_{F_i} = 0 \) for every \( i = 1, \ldots, r \). By Claim~\ref{claim:zero}, \( \phi = 0 \). Hence, \( \lambda \) is not an eigenvalue of \( P \). Therefore, \( 1 \) is the unique eigenvalue of \( P \) with modulus \( 1 \), completing the proof of item~(ii). 
\end{proof}

\section{Global contraction on average, proximality and unique ergodicity} \label{ss:global_contraction}
Recall that a random map $f:\Omega\times X \to X$ is said to be \emph{(global) contracting on average} if there exist $n\geq 1$, $0<\alpha\leq1$ and $q<1$  such that 
\begin{equation} \label{eq:contracting-avarage}
    \int d(f^{n}_\omega(x),f^{n}_\omega(y))^\alpha\, d\mathbb{P} \leq q \, d(x,y)^\alpha \quad \text{for every $x,y\in X$}.
\end{equation}
We also say that $f$ is
\begin{enumerate}[leftmargin=0.8cm,itemsep=0.05cm,label=$\bullet$]
\item \emph{synchronizing}\index{random maps!synchronizing random map} if for every $x, y \in X$, 
$$\displaystyle
\lim_{n\to\infty} d(f^n_\omega(x), f^n_\omega(y)) =0 \quad  \text{for $\mathbb P$-a.e.~$\omega\in \Omega$;}
$$
\item \emph{proximal} if for every $x, y$ in $X$, there exists $\omega\in \Omega$ such that 
$$\displaystyle\liminf_{n\to\infty} d(f^n_\omega(x), f^n_\omega(y)) =0.$$ 
\end{enumerate}
{ We will also consider the associated \emph{two-point random map}
\[
\tilde f:\Omega\times X\times X\to X\times X,
\qquad
\tilde f_\omega(x,y)\eqdef (f_\omega(x),f_\omega(y)),
\]
also called the \emph{two-point motion}\index{random maps!two-point random map} in the literature.}

Clearly, synchronizing implies proximal. {In fact, under the local contraction property, one has a sharper equivalence, due to Malicet~\cite[Proposition~4.18]{Mal:17}, which we state here in our notation for the reader's convenience.

\begin{prop}\label{prop:sync-prox-two-point}
Assume that \(f:\Omega\times X\to X\) is a continuous random map with the local contraction property. Then the following assertions are equivalent:
\begin{enumerate}[label=\textup{(\roman*)},leftmargin=0.8cm,itemsep=0.05cm]
\item \(f\) is synchronizing;
\item the two-point random map \(\tilde f\) is uniquely ergodic;
\item \(f\) is proximal.
\end{enumerate}
In particular, these equivalences hold for every mostly contracting random map.
\end{prop}

\begin{proof}
Let $D\eqdef \{(x,x):x\in X\}$ be the diagonal of \(X\times X\) endowed with the metric
\[
\tilde d((x,y),(x',y'))\eqdef \max\{d(x,x'),d(y,y')\}.
\]
We first note that if \(f\) has the local contraction property, then so does \(\tilde f\). Indeed, let  $(x,y)\in X\times X$. If for \(\mathbb P\)-a.e.~\(\omega\in\Omega\), \(B_x\) and \(B_y\) are neighborhoods of \(x\) and \(y\) such that
\[
\diam f_\omega^n(B_x)\to 0
\qquad\text{and}\qquad
\diam f_\omega^n(B_y)\to 0,
\] 
then for \(B\eqdef B_x\times B_y\) we have
\[
\diam \tilde f_\omega^n(B)
=
\max\bigl\{\diam f_\omega^n(B_x),\,\diam f_\omega^n(B_y)\bigr\}\longrightarrow 0.
\]

\smallskip

\noindent\emph{\((i)\Rightarrow(iii)\).} This implication is immediate.

\smallskip

\noindent\emph{\((iii)\Rightarrow(ii)\).}
Assume that \(f\) is proximal. In the present Bernoulli setting, this means equivalently that for every \(x,y\in X\) there exists a sequence \(g_n\) in the semigroup generated by the fiber maps of \(f\) such that $d(g_n(x),g_n(y))\to 0$.

We claim that \(\tilde f\) has a unique stationary probability measure. Suppose, by contradiction, that \(\tilde\mu_1\) and \(\tilde\mu_2\) are two distinct ergodic stationary probability measures for \(\tilde f\). Since \(\tilde f\) has the local contraction property, Corollary~\ref{rem:measure-mean-quasicompactness} applied to \(\tilde f\) implies that their supports \(F_1\) and \(F_2\) are disjoint non-empty closed \(\tilde f\)-invariant subsets of \(X\times X\). Fix $i\in \{1,2\}$ and pick \((x_i,y_i)\in F_i\). By proximality, there exists a sequence \(g_n\) in the semigroup generated by the fiber maps such that $d(g_n(x_i),g_n(y_i))\to  0$. Since \(F_i\) is $\tilde{f}$-invariant, 
$(g_n(x_i),g_n(y_i))$ 
belongs to \(F_i\). Passing to a cluster value, we obtain a point of \(F_i\cap D\). Hence
\[
F_1\cap D\neq\varnothing
\qquad\text{and}\qquad
F_2\cap D\neq\varnothing.
\]
Choose \((z_1,z_1)\in F_1\cap D\) and \((z_2,z_2)\in F_2\cap D\). Applying proximality again to \(z_1\) and \(z_2\), we find a sequence \(h_n\) in the semigroup generated by the fiber maps such that
$d(h_n(z_1),h_n(z_2))\to 0$. Then any cluster value of
$(h_n(z_1),h_n(z_1))$ coincides with a cluster value of
$(h_n(z_2),h_n(z_2))$, and belongs simultaneously to \(F_1\) and \(F_2\), contradicting the disjointness of \(F_1\) and \(F_2\). Thus \(\tilde f\) is uniquely ergodic.

\smallskip

\noindent\emph{\((ii)\Rightarrow(i)\).}
Assume that \(\tilde f\) is uniquely ergodic, and let \(\tilde\mu\) be its unique stationary probability measure. Since \(\tilde f\) has the local contraction property, item~\textup{(a)} of Corollary~\ref{rem:measure-mean-quasicompactness} applied to \(\tilde f\) shows that the support
$F= \mathrm{supp}(\tilde\mu)$ 
is the unique minimal non-empty closed \(\tilde f\)-invariant subset of \(X\times X\). Since the diagonal \(D\) is itself a non-empty closed \(\tilde f\)-invariant subset, it contains a minimal non-empty closed \(\tilde f\)-invariant subset, and therefore, by uniqueness, \(F\subset D\).

Fix now \((x,y)\in X\times X\). By item~\textup{(a)} of Corollary~\ref{rem:measure-mean-quasicompactness} applied to \(\tilde f\), for $\mathbb P$-a.e.~$\omega\in \Omega$, the sequence \( \{\tilde{f}_\omega(x,y)\}_{n\geq 0}\) accumulates on \(F\subset D\). Therefore, every cluster value also belongs to \(F\subset D\). Since \(D\) is closed, this implies $\tilde d(\tilde f_\omega^n(x,y),D)\to 0$.
Equivalently,
\[
d\bigl(f_\omega^n(x),f_\omega^n(y)\bigr)\longrightarrow 0
\qquad\text{for }\mathbb P\text{-a.e.~}\omega,
\]
so \(f\) is synchronizing.

The last claim follows by Theorem~\ref{cor:local-contraction}.
\end{proof}
}

Moreover, from~\cite[Proposition~1 and Remark~7]{stenflo2012survey} it follows that synchronizing implies uniquely ergodic for the original random map (see also~\cite[Proposition~X.2]{HH:01}). For completeness, we include the proof.

\begin{prop} \label{prop:syncro}
    Let $f: \Omega \times X \to X$ be a synchronizing random map. Then, $f$ is uniquely ergodic.  
\end{prop} 
\begin{proof}
Let $P$ be the annealed Koopman operator associated with $f$. Then, for every $x,y\in X$ and Lipschitz function $\phi: X\to \mathbb{R}$, we have 
\begin{align*}
\big|P^n \phi(x) - P^n \phi(y)\big| & \leq
\int \big|\phi(f^n_\omega(x)) - \phi(f^n_\omega(y))\big|\,d\mathbb{P}  \leq |\phi|_1 \int d(f^n_\omega(x), f^n_\omega(y))\, d\mathbb{P}.
\end{align*}
By synchronization and dominated convergence theorem, 
$$
\lim_{n\to \infty} \int d(f^n_\omega(x), f^n_\omega(y))\, d\mathbb{P}=0.
$$
Thus, $|P^n \phi(x) - P^n \phi(y)|\to 0$ as $n\to\infty$. Hence, for any pair $\mu$ and $\nu$ of $f$-stationary probability measures on $X$, by the $P^*$-invariance of $\mu$ and $\nu$ and the dominated convergence theorem,   
\begin{align*}
\big|\int \phi\, d\mu - \int \phi \, d\nu\big| &=
\big|\int P^n \phi(x) \, d\mu - \int P^n \phi(y) d\nu \big| \\ &\leq \int  \big|P^n \phi(x) - P^n \phi(y)\big| d\mu(x) d\nu(y) \rightarrow 0
\end{align*}
as $n \rightarrow \infty$. This implies that $\int \phi \, d\mu= \int \phi \, d\nu$. Since from the Weierstrass-Stone theorem, every real-valued continuous function on a compact metric space can be approximated in the uniform norm by Lipschitz functions, we get by  dominated convergence theorem that, actually, $\int \phi \, d\mu= \int \phi \, d\nu$ holds for every continuous function $\phi$. This implies that $\mu=\nu$. 
\end{proof}

As mentioned in item (a) of 
Corollary~\ref{rem:measure-mean-quasicompactness},  if $f$ is proximal and has the local contraction property, then it is uniquely ergodic. This follows from the previous proposition together with the preceding discussion.  
The following result follows~\cite[Rem.~5]{peigne1993iterated}.

\begin{prop} \label{prop:contration-proximal}
Let $f: \Omega \times X \to X$ be a contracting on average random map satisfying the integral condition~\eqref{eq:integral_condition}.  Then, $f$ is synchronizing. Therefore, it is proximal and uniquely ergodic.
\end{prop}

\begin{proof}
{Let $\beta>$, $\alpha>0$, $n\ge1$, and $0<q<1$ be given by~\eqref{eq:integral_condition}
and~\eqref{eq:contracting-avarage}, and set
\[
\bar\alpha\eqdef \min\{\alpha,\beta\}.
\]
Since $\bar\alpha\le \alpha$, by Lyapunov's inequality we obtain, for every $x,y\in X$,
\[
\int \left(\frac{d(f^n_\omega(x),f^n_\omega(y))}{d(x,y)}\right)^{\bar\alpha}\,d\mathbb{P}
\le
\left(
\int \left(\frac{d(f^n_\omega(x),f^n_\omega(y))}{d(x,y)}\right)^{\alpha}\,d\mathbb{P}
\right)^{\bar\alpha/\alpha}
\le q^{\bar\alpha/\alpha}.
\]
Hence \(f\) is also contracting on average with exponent \(\bar\alpha\) and contraction factor 
$\bar q\eqdef q^{\bar\alpha/\alpha}\in(0,1)$.} Iterating in~\eqref{eq:contracting-avarage}, we are led to the inequality 
\begin{equation} \label{eq:iteration-contracting_average}
        \int d(f^{kn+\ell}_\omega(x),f^{kn+\ell}_\omega(y))^{{\bar{\alpha}}}\, d\mathbb{P} \leq 
 C {\bar{q}}^k d(x,y)^{{\bar\alpha}}
\end{equation}
 for every $x,y\in X$, $k\geq 1$ and 
 $\ell=1,\dots,n-1$ 
 where $$C=\max_{\ell=1,\dots,n-1} \int \mathrm{Lip}(f^\ell_\omega)^{{\bar\alpha}}\, d\mathbb{P}<\infty$$ by virtue of~\eqref{eq:integral_condition} with  {$\beta \geq \bar\alpha$}.   From this, 
$$
\int \sum_{i=1}^\infty d(f^i_\omega(x),f^i_\omega(y))^{{\bar\alpha}}\, d\mathbb{P}<\infty
$$
and so, for $\mathbb{P}$-a.e.~$\omega \in\Omega$, $d(f^i_\omega(x),f^i_\omega(y)) \to 0$ as $i\to \infty$, and thus $f$ is synchronizing. Therefore, $f$ is proximal and, by Proposition~\ref{prop:syncro},  it is also uniquely ergodic.   
\end{proof}

\chapter{Robustness and  statistical stability of mostly contracting}  \label{s:statistical-stability}

\abstract{
{This chapter studies robustness and statistical stability for mostly contracting
random maps. We first introduce topologies on spaces of Lipschitz random maps
which control finite iterates, local Lipschitz constants, and global Lipschitz
constants in integral form. These topologies are compared with the usual
\(C^1\)-topology for random maps on compact Riemannian manifolds. We then prove
that the mostly contracting property is open in the Lipschitz topology and that
local contraction on average and the associated Lasota-Yorke estimates persist
under small perturbations. Using the Keller-Liverani stability theory for
quasi-compact operators, we obtain upper semicontinuity of the number of
ergodic stationary measures. On
each stratum where this number is locally constant, the ergodic stationary
measures vary weak\(^*\) continuously and H\"older continuously in the dual H\"older distance on probability measures. As a consequence, statistical stability holds on an open and dense
subset of the mostly contracting region, which contains the uniquely ergodic random maps. The chapter concludes with a weak
statistical stability result for continuous random maps in the \(C^0\)-topology.}
}

\section{Topologies for Lipschitz random maps} \label{ss:topologies}

In what follows, $(X,d)$ denotes a compact metric space.  
Let $\EE(X)$ be the space of Lipschitz random maps $f:\Omega \times X \to X$, up to $\mathbb{P}$-almost identification of fiber maps. That is, $f=g$ if and only if $f_\omega=g_\omega$ for $\mathbb{P}$-a.e.~$\omega\in \Omega$. For each $n\geq 1 $ and $0<\beta \leq 1$, we consider the metric 
\begin{equation*} \label{eq:distacia-random-maps}
 \Dbeta(f,g)=\max_{i=1,\dots,n} \int d_{\mathrm{L}}(f^i_\omega,g^i_\omega)^{\beta_i} \, d\mathbb{P}, \qquad \text{for $f, g \in \EE(X)$} \index{metric structures!Lipschitz random-map distance@\(\D_{\mathrm L}^{n,\beta}\), distance on \(\EE(X)\)}
  \end{equation*}
  where  $\beta_i=\beta/2^{i-1}$ and 
  \begin{align*} \label{eq:distancia-Lip}
      d_{\mathrm{L}}(h_1,h_2) &=  \sup_{x\in X} \big(d(h_1(x) ,h_2(x)) + |Lh_1(x)-Lh_2(x)|\big) + \big|\mathrm{Lip}(h_1)-\mathrm{Lip}(h_2)\big| \\
      &= d_{C^0}(h_1 ,h_2) + \|Lh_1-Lh_2\|_\infty + \big|\mathrm{Lip}(h_1)-\mathrm{Lip}(h_2)\big| \index{metric structures!Lipschitz distance@\(d_{\mathrm L}\), distance on \(\operatorname{Lip}(X)\)}
  \end{align*}
  is the distance between the  maps 
  $h_1$ and $h_2$ 
  in the space $\mathrm{Lip}(X)$ of Lipschitz transformations of $X$. 

\begin{lem} \label{claim:lipbeta} For any $f,g\in \mathrm{Lip}_\mathbb{P}(X)$,  it holds 
$$
  \int \mathrm{Lip}(g_\omega)^\beta \, d\mathbb{P}   \leq \int d_{\mathrm{L}}(f_\omega,g_\omega)^\beta \, d\mathbb{P}  + \int \mathrm{Lip}(f_\omega)^\beta \, d\mathbb{P} \quad \text{for any $0<\beta\leq 1$}.
$$
\end{lem}
\begin{proof}
To prove this, note that $\mathrm{Lip}(g_\omega) \leq |\mathrm{Lip}(g_\omega)-\mathrm{Lip}(f_\omega)|+ \mathrm{Lip}(f_\omega) \leq  d_{\mathrm{L}}(f_\omega,g_\omega) + \mathrm{Lip}(f_\omega)$. From here, the claim easily follows.
\end{proof}

  Recall that when $X$ is a compact Riemannian manifold, we also introduced in~\S\ref{ss:statistical-stability} the space $\CE(X)$ of $C^1$ random maps of $X$ endowed with the metrics 
  $$
  \D_{C^r}(f,g)= \int d_{C^r }(f_\omega,g_\omega) \, d\mathbb{P}, \qquad \text{for $f, g \in \CE(X)$, \ $r=0,1$.}
  $$
  We have $\CE(X)\subset \EE(X)$ and $\D_{C^0}\leq \D_{\mathrm{L}}\eqdef \D_{\smash{\mathrm{L}}}^{1,1}\leq \DD$. \index{metric structures!Lipschitz2 distance random-map distance@\(\D_{\mathrm L}=\D_{\smash{\mathrm{L}}}^{1,1}\), distance on \(\EE(X)\)} 

\begin{lem} \label{rem:topology}
    Let  $X$ be a compact Riemannian manifold and fix $n\geq 1$.  If 
    $f \in \CE(X)$ satisfies  
    \begin{equation} \label{eq:momento-dif}
 \int (\|Df_\omega\|^{}_\infty)^{\beta} \, d\mathbb{P} <\infty  \quad \text{for some $0<\beta\leq 1$},
    \end{equation}
then  there is $K>0$ such that $\Dbeta(f,g)\leq K \D_{\smash{C^1}}(f,g)^{\beta_n}$ for any $g \in \CE(X)$ close enough to $f$ in the $\DD$-distance. That is, the inclusion 
$(\CE(X),\DD)\hookrightarrow (\EE(X),\Dbeta)$ is continuous at $f$. 
\end{lem}
\begin{proof} Let us first prove the lemma when $X$ is connected. In this case,  for each $i\geq 1$, we have  $Lf^i_\omega(x)=\|Df^i_\omega(x)\|$ and
$\mathrm{Lip}(f^i_\omega) 
=\|Df^i_\omega\|_\infty$. Observe that 
\begin{align*}
|\mathrm{Lip}(f^i_\omega)-\mathrm{Lip}(g^i_\omega)|&\leq  \|Df^i_\omega-Dg^i_\omega\|_\infty \ \ \text{and} \\  
    |Lf^i_\omega(x)-Lg^i_\omega(x)|&\leq \|Df^i_\omega(x)-Dg^i_\omega(x)\|.  
\end{align*} 
This implies that 
$d_{\mathrm{L}}(f_\omega^i,g^i_\omega)\leq d_{C^1}(f^i_\omega,g^i_\omega)$. 
Then \begin{equation} \label{eq:Dbeta}
    \Dbeta(f,g)\leq \max_{i=1,\dots,n} \,\int d_{C^1}(f^i_\omega,g^i_\omega)^{\beta_i}\, d\mathbb{P}.
\end{equation} To complete the proof, we show the following claim: 

\begin{claim} \label{claim:induction} For each $n\geq 1$, there exists $K>0$ such that $$\int d_{C^r}(f^i_\omega,g^i_\omega)^{\beta_i} \, d\mathbb{P}\leq K \D_{C^r}(f,g)^{\beta_i} \quad \text{for all $i=1,\dots,n$ 
and $r=0,1$.}
$$
\end{claim}
\begin{proof}  The case $n=1$ follows from Jensen's inequality (since $\beta_n \le 1$).  Assume by induction that the statement holds for $n-1$ and consider 
\begin{equation} \label{eq:des-indution}
    \int  d_{C^1}(f^n_\omega,g^n_\omega)^{\beta_n} \, d\mathbb{P} \leq \int  d_{C^0}(f^n_\omega,g^n_\omega)^{\beta_n} \, d\mathbb{P} + \int \big(\|Df^n_\omega-Dg^n_\omega\|_\infty\big)^{\beta_n} \, d\mathbb{P}.
\end{equation}
First notice that 
\begin{align*}
    \int  d_{C^0}(f^n_\omega, \, & g^n_\omega)^{\beta_n} \, d\mathbb{P} \leq  \\ 
     & \leq \int 
\big(\|Df_{\omega_{n-1}}\|_\infty\big)^{\beta_n} \cdot  d_{C^0}(f_\omega^{n-1},g^{n-1}_\omega)^{\beta_n} \, d\mathbb{P} + \int  d_{C^0}(f_{\omega},g_{\omega})^{\beta_n} \, d\mathbb{P}.
\end{align*}
Using 
the H\"older inequality, we obtain 
\begin{align*}
\int  d_{C^0}( &f^n_\omega,  g^n_\omega)^{\beta_n} \, d\mathbb{P} \leq \\
&\leq \big(\int 
\big(\|Df_{\omega_{n-1}}\|_\infty\big)^{2\beta_n} \, d\mathbb{P}\big)^{\frac{1}{2}} \big(\int d_{C^0}(f_\omega^{n-1},g^{n-1}_\omega)^{2\beta_n} \, d\mathbb{P}\big)^{\frac{1}{2}} +  \D_{C^0}(f,g)^{\beta_n}.
\end{align*}
Since $2\beta_n=\beta_{n-1} \leq \beta$, using the shift-invariance of $\mathbb{P}$, the exponential moment assumption, and the induction hypothesis, we obtain  
$$\int  d_{C^0}(f^n_\omega,g^n_\omega)^{\beta_n} \, d\mathbb{P} \leq  K \D_{C^0}(f,g)^{\beta_n}$$ for some $K>0$. 
Similarly, we consider now 
$$
\|Dg^{n}_\omega-Df^{n}_\omega\|_\infty 
\leq 
2 
\|Dg_\omega\|_\infty \cdot \|Dg^{n-1}_\omega-Df^{n-1}_\omega\|_\infty +   2 \|Df^{n-1}_\omega\|_\infty \cdot \|Dg_\omega-Df_\omega\|_\infty. 
$$
Hence, as before 
\begin{align*}
\int \big(
\| Dg^{n}_\omega- & Df^{n}_\omega\|_\infty \big)^{\beta_n} \, d\mathbb{P}
\leq \\ \leq \  &2^{\beta_n} \big(\int\big(\|Dg_\omega\|_\infty\big)^{2\beta_n}\, d \mathbb{P}\big)^\frac{1}{2}  \big(\int\big(\|Dg^{n-1}_\omega-Df^{n-1}_\omega\|_\infty \big)^{2\beta_n}\, d\mathbb{P} \big)^{\frac{1}{2}} \\
&+ 2^{\beta_n} \big(\int\big(\|Df_\omega\|_\infty\big)^{(n-1)\beta_n}\, d \mathbb{P}\big)^\frac{1}{2}  \big(\int \big(\|Dg_\omega-Df_\omega\|_\infty \big)^{2\beta_n}\, d\mathbb{P} \big)^{\frac{1}{2}}.
\end{align*}
Similarly, since $2\beta_n=\beta_{n-1} \leq (n-1)\beta_n\leq \beta$, using the assumption on the exponential moment, the hypothesis of induction, and the Jensen inequality, we have that 
$$\int (
\|Dg^{n}_\omega-Df^{n}_\omega\|_\infty)^{\beta_n} \, d\mathbb{P} \leq K \DD(f,g)^{\beta_n}$$ for some $K>0$ which depends on the proximity of $g$ to $f$ according to Lemma~\ref{claim:lipbeta}. Combining the two inequalities obtained above for the right-hand side of~\eqref{eq:des-indution}, we conclude   that $\int d_{C^1}(f^n_\omega,g^n_\omega)^{\beta_n} \, d\mathbb{P} \leq K \DD(f,g)^{\beta_n}$ for some $K>0$ and for any $g$ close enough to $f$ in the $\DD$-distance. This completes the proof of the claim.
\end{proof}

Equation~\eqref{eq:Dbeta} and Claim~\ref{claim:induction}  imply the lemma in the case that $X$ is connected. On the other hand, every compact Riemannian manifold is a disjoint union of finitely many compact connected Riemannian manifolds. Thus, we can easily control the distance between the (global) Lipschitz constants with the $C^1$ distance between fiber maps.  Therefore, the same conclusion holds in general.    
\end{proof}

In the proof of the above result, we used that for a $C^1$ map $g$ on a compact connected Riemannian manifold $X$, it holds that   $\mathrm{Lip}(g)=\|Dg\|_\infty=\sup\{ Lg(x): x\in X\}$. However, this property does not generally hold in compact metric spaces. To ensure its validity, we restrict to the class of geodesic spaces or, more generally, to locally length metric spaces.

A space $(X,d)$ is called a \emph{locally length space}\index{metric spaces!locally length space}
 if every point in $X$ has a neighborhood $U$ such that the distance between any two points $x,y\in U$ is the infimum of the lengths of rectifiable paths from $x$ to $y$. Examples of locally length spaces include Euclidean spaces, Riemannian manifolds, Finsler manifolds, and all normed vector spaces. The following proposition from~\cite[Proposition~3.11]{BC:23} establishes this property for that class of metric spaces.

\begin{prop} \label{prop:lengthspace} Let $X$ be a compact locally length metric space. Then there exists $R>0$ such that for any $0<r\leq R$ and any Lipschitz map $f:X\to X$, 
$$
    L_rf(z)\eqdef \mathrm{Lip}(f|_{B(z,r)}) = \sup_{x\in B(z,r)} Lf(x) \quad \text{for all $z\in X$}.
$$
\end{prop}

\section{Robust properties of mostly contracting random maps}
The following theorem shows the robustness of mostly contracting random maps.
\begin{thm} \label{prop:open-mostly} Let $f\in \EE(X)$ be 
a mostly contracting random map. Then, there exist $n\geq 1$ and  an open neighborhood $\mathcal{B}$ of $f$ in $(\EE(X),\Dbeta)$ for every $0<\beta \leq 1$ such that every $h\in \mathcal{B}$ is mostly contracting. In particular, if $X$ is a compact Riemannian manifold, the set of 
mostly contracting random maps satisfying a finite exponential moment condition 
is open in $(\CE(X),\DD)$.  
\end{thm}
\begin{proof} As noted in Remark~\ref{rem:equivalencia} and Corollary~\ref{cor:equivalencia},  a Lipschitz random map $h$ 
is mostly contracting  if and only if
\begin{enumerate}[label=(\roman*),leftmargin=0.8cm]
    \item  $\int \log^+ \mathrm{Lip}(h_\omega)\, d\mathbb{P}<\infty$, and
    \item there are  $a>0$ and $n\in \mathbb{N}$ such that $\int \log Lh^n_\omega(x) \, d\mathbb{P} < -a$ for any $x\in X$. 
\end{enumerate}
For each $M>0$, define $\log_M t = \max\{  \log t  , -M\}$. 

\begin{claim}  Condition (ii) above is equivalent to 
\begin{enumerate}[leftmargin=0.8cm]
    \item[(ii')]  there are  $a>0$, $n\in \mathbb{N}$ and $M>0$ such that 
    $$\int \log_M Lh^n_\omega(x) \, d\mathbb{P} < -a \quad \text{for any $x\in X$.}
    $$
\end{enumerate}
\end{claim}
\begin{proof} Assume that (ii) holds. For each $x\in X$, there is $r>0$ such that
$Lh^n_\omega(x) \leq L_rh^n_\omega(x)$. Moreover, we can take $r>0$ small enough such that  $\int \log L_rh^n_\omega(x)\, d\mathbb{P}<-a$. By compactness of $X$, we get points $x_1,\dots x_k$ and radius $r_1>0, \dots, r_k>0$ such that 
$X=B(x_1,r_1)\cup \dots \cup B(x_k,r_k)$ 
and $$\int \log L_{r_i}h^n_\omega(x_i)\, d\mathbb{P}<-a.$$ Then, by finiteness and the monotone convergence theorem, we can find $M>0$ such that for each $i=1,\dots,k$, we have $\int \log_M L_{r_i}h^n_\omega(x_i) \, d\mathbb{P} < -a$. Finally, since $Lh^n_\omega(x) \leq L_{r_i}h^n_\omega(x_i)$ for every $x\in B(x_i,r_i)$ and $\{B(x_i,r_i)\}_{i=1,\dots,k}$ form a cover of $X$, we conclude that for each $x\in X$, there is $i\in \{1,\dots,k\}$ such that 
$$
   \int \log_M Lh^n_\omega(x) \, d\mathbb{P} \leq  \int \log_M L_{r_i}h^n_\omega(x_i) \, d\mathbb{P} < -a.
$$
This proves (ii'). Since (ii') also clearly implies (i), which is equivalent to (ii), we complete the proof of the claim.
\end{proof}
Now, let $f\in \EE(X)$ be a mostly contracting random map. Then, there are $a>0$, $n\in\mathbb{N}$ and $M>0$ satisfying the corresponding item (ii') above. Let $\epsilon>0$ such that  $$\int \log_M Lf^n_\omega(x) \, d\mathbb{P} + \epsilon < -a.$$ 
For $0<\beta\leq 1$, take a neighborhood $\mathcal{B}$ of $f$ in $(\EE(X),\Dbeta)$ such that $\Dbeta(h,f)< \epsilon M\beta_n$ for every $h\in \mathcal{B}$, where  $\beta_n=\beta/2^{n-1}$. Then,  
 using that $\log_M t$ has Lipschitz constant $1/M$ and that $t^\alpha- s^\alpha \leq |t-s|^\alpha$ for any $0<\alpha \leq 1$, we get 
\begin{align*}
  \big|\log_M Lh^n_\omega(x)^{\beta_n} - \log_M Lf^n_\omega(x)^{\beta_n}\big| &\leq \frac{1}{M} \big|Lh^n_\omega(x)^{\beta_n} -  Lf^n_\omega(x)^{\beta_n}\big| \\ &\leq \frac{1}{M} \big|Lh^n_\omega(x) -  Lf^n_\omega(x)\big|^{\beta_n}.
\end{align*}

Hence, 
$\log_M Lh^n_\omega(x)^{\beta_n} \leq \frac{1}{M} d_{\mathrm{L}}(f^n_\omega,h^n_\omega)^{\beta_n} + \log_M Lf^n_\omega(x)^{\beta_n}$ and after integrating 
\begin{align*}
\int \log_M Lh^n_\omega(x)^{\beta_n} \, d\mathbb{P} &\leq \frac{1}{M} \Dbeta(f,h) + \int \log_M Lf^n_\omega(x)^{\beta_n} \, d\mathbb{P} \\ 
&< \epsilon \beta_n + \beta_n \int \log_M Lf^n_\omega(x) \, d\mathbb{P} <  -\beta_n a. 
\end{align*}
This concludes (ii') above. Thus, to prove that $h$ is mostly contracting, it is enough to show that $\log^+ \mathrm{Lip}(h_\omega)$ is $\mathbb{P}$-integrable. For this, it is enough to consider the case where $\mathrm{Lip}(h_\omega)>1$ and $\mathrm{Lip}(h_\omega)$ is close enough to $\mathrm{Lip}(f_\omega)$. In particular, we also have $\mathrm{Lip}(f_\omega)>1$. In this case, since $\frac{a}{b}\leq e^{2|a-b|}$ if $b>1$, it follows that $$\log^+ \mathrm{Lip}(h_\omega) \leq 2|\mathrm{Lip}(h_\omega)-\mathrm{Lip}(f_\omega)|+ \log^+ \mathrm{Lip}(f_\omega).$$ Integrating, we obtain 
$$ 
\int \log^+  \mathrm{Lip}(h_{\omega})  \, d\mathbb{P} \leq 2 \D_{\mathrm{L}}(f,h) + \int \log \mathrm{Lip}^+(f_{\omega}) \, d\mathbb{P}(\omega) < \infty.
$$ 
This completes the proof of the first part of the theorem. The case where $X$ is a compact Riemannian manifold follows from Lemma~\ref{rem:topology}.  
\end{proof}

We now improve Proposition~\ref{prop:alpha0} to establish that all nearby Lipschitz random maps also exhibit local contraction on average, preserving the same constants $\alpha$, $r$, $q$ and $n$. This robustness result will be useful later in showing the statistical stability of mostly contracting random maps.

\begin{prop} \label{prop:lasota-york-uniforme}  
Let $(X,d)$ be a compact locally length metric space and consider a mostly contracting random map $f$ such that the integrability condition~\eqref{eq:integral_condition}  
holds for some $0<\beta\leq 1$. Then, 
there exist $\alpha_0>0$, $r>0$,  $n\in\mathbb{N}$ and a neighborhood $\mathcal{B}$ of $f$ in $(\EE(X),\Dbeta)$ such that for any $0<\alpha \leq \alpha_0$,  there is  $q<1$  satisfying 
$$
\int d(h^n_\omega(x),h^n_\omega(y))^\alpha \, d\mathbb{P} \leq q \, d(x,y)^\alpha \ \text{for any $x,y\in X$ with $d(x,y)<r$ and $h\in \mathcal{B}$.}
$$
Moreover, every $h\in \mathcal{B}$ satisfies the corresponding exponential moment~\eqref{eq:integral_condition} for the same $\beta>0$ and 
\begin{align*} \label{eq:lasota-york-uniform}
P_h(C^\alpha(X))\subset C^\alpha(X) \quad  \text{and} \quad
|P^n_h\phi|_{\alpha} \leq q |\phi|_\alpha + C \|\phi\|_{\infty} \quad \text{for all } \phi\in C^\alpha(X)
\end{align*}
where $C=\frac{2}{r^\alpha}>0$ 
and $P_h$ denotes the Koopman operator associated with $h$.
\end{prop}
\begin{proof} The proof follows the argument used for Proposition~\ref{prop:alpha0} up to equation~\eqref{eq:qmenor1}. Furthermore, we can assume that $\alpha_0(x)\leq \beta_n$ where $\beta_n=\beta/2^{n-1}$ and $r(x) \leq R$ where $R$ is given in Proposition~\ref{prop:lengthspace}. Now, 
choose $\varepsilon(x)>0$ such that 
\begin{equation}   \label{eq:q}
  \int L_{r(x)}f^n_\omega(x)^{\alpha_0(x)} \, d\mathbb{P} + \varepsilon(x) < q(x). 
\end{equation}
Let $\mathcal{B}(x)$ be a neighborhood of $f$ in $(\EE(X),\Dbeta)$ such that   $$\int d_{\mathrm{L}}(f^n_\omega,h^n_\omega)^{\beta_n} \, d\mathbb{P} \leq \varepsilon(x)^{\beta_n/\alpha_0(x)}$$ for every $h\in \mathcal{B}(x)$. By Proposition~\ref{prop:lengthspace},
\begin{align*}
   L_{r(x)}h^n_\omega(x) 
   &\leq  \sup_{z\in X} |Lh^n_\omega(z)-Lf^n_\omega(z)|  +  \sup_{z\in B(x,r(x))} Lf^n_\omega(z)  \\ &
   \leq d_{\mathrm{L}}(f^n_\omega,h^n_\omega) + L_{r(x)}f^n_\omega(x).   
\end{align*}
Using the Jensen inequality and~\eqref{eq:q}, we deduce that for every $h\in \mathcal{B}(x)$
\begin{align*}
 \int L_{r(x)}h^n_\omega(x)^{\alpha_0(x)} \, d\mathbb{P} &\leq \int d_{\mathrm{L}}(f^n_\omega,h^n_\omega)^{\alpha_0(x)} \, d\mathbb{P} + \int L_{r(x)}f^n_\omega(x)^{\alpha_0(x)} \, d\mathbb{P} \\ 
 &\leq 
 \left( \int d_{\mathrm{L}}(f^n_\omega,h^n_\omega)^{\beta_n} \, d\mathbb{P}\right)^{\frac{\alpha_0(x)}{\beta_n}} + \int L_{r(x)}f^n_\omega(x)^{\alpha_0(x)} \, d\mathbb{P} \\ &< q(x). 
\end{align*}
Now, the rest of the proof follows line by line as in Proposition~\ref{prop:alpha0}
after equation~\eqref{eq:qmenor1}. It is enough to consider the neighborhood in $(\EE(X),\Dbeta)$ of $f$ given by $\mathcal{B}=\mathcal{B}(x_1)\cap \dots \cap \mathcal{B}(x_k)$ and apply the above uniform inequality for $h\in \mathcal{B}$ instead of~\eqref{eq:qmenor1}.  This completes the proof of the first part of the proposition.  The second part of the proposition follows immediately from Proposition~\ref{prop:quasi-compact} observing that $\mathrm{Lip}(h_\omega)^\beta$ is $\mathbb{P}$-integrable for every $h\in \mathcal{B}$ by Lemma~\ref{claim:lipbeta}. 
\end{proof}


\section{Spectral stability of quasi-compact operators}


Suppose that $(\mathcal{E}, \| \cdot \|)$ is a Banach space equipped with a second norm $|\cdot| \leq \| \cdot \|$ with respect to which $\mathcal{E}$ need not be complete. For simplicity, we assume that the inclusion $\iota: (\mathcal{E}, \| \cdot \|) \hookrightarrow (\mathcal{E}, |\cdot|)$ is compact and denote the norm on $\mathcal{L}((\mathcal{E}, \|\cdot \|),(\mathcal{E},|\cdot|))$ by $\tnorm{\cdot}$, that is, 
$$\tnorm{Q} = \sup \left\{ |Qu| : u \in \mathcal{E}, \|u\| \leq 1 \right\}.$$

\index{Keller-Liverani theorem}

\begin{thm}[{\cite{KL:99}}] \label{thm:keller-liverani}
Let $Q_{\epsilon}: \mathcal{E} \rightarrow \mathcal{E}$ be a bounded linear operator for every $\epsilon\geq 0$. Suppose that
\begin{enumerate}[label=(\roman*)]
\item  there exist $n\in\mathbb{N}$, $q < 1$ and $C >0$ such that
\[
\|Q_{\varepsilon}^{n} u\| \leq q \|u\| + C |u| \quad \text{for every $u \in \mathcal{E}$ and $\varepsilon \geq 0$,}
\]
\item there is a monotone upper semicontinuous function $\epsilon \mapsto \tau(\epsilon)$  such that 
$$\tnorm{Q_{\epsilon} - Q_0} \leq \tau(\epsilon) \to 0 \quad  \text{as $\epsilon \to 0$}.$$
\end{enumerate}
Let $\lambda \in \mathbb{C}$ be an isolated eigenvalue of $Q_0$ with $|\lambda|>q$ and  consider the total spectral projection  
$\Pi_{\epsilon}^{(\lambda,\delta)}$ 
of $Q_\epsilon$ associated with $B_\delta(\lambda)\cap \sigma(Q_\epsilon)$
where $B_\delta(\lambda)$ is the open ball of radius $\delta$ around $\lambda$ and
$\sigma(Q_\epsilon)$ is the spectrum of $Q_\epsilon$. Then, for $\delta>0$ small enough, there exist $\eta>0$, $K>0$ and $\epsilon_0>0$ such that 
$$
    \tnorm{ \Pi_\epsilon^{(\lambda,\delta)}-\Pi_0^{(\lambda,\delta)} } \leq K \tau(\epsilon)^\eta \quad \text{for all \ $0\leq \epsilon\leq \epsilon_0$.}
$$
Moreover, the  {ranks} of $\Pi_\epsilon^{(\lambda,\delta)}$ and $\Pi_0^{(\lambda,\delta)}$  {coincide}. In particular, $B_\delta(\lambda)\cap \sigma(Q_\epsilon)$  consists of eigenvalues $\lambda^{(\epsilon)}_j $ such that $\lambda_j^{(\epsilon)}\to \lambda$ for all $j$, and the total multiplicity of $\lambda_j^{(\epsilon)}$ equals the multiplicity of $\lambda$.
\end{thm}

\begin{rem} \label{rem:liverani} The constant $\eta$ depends on a parameter $r \in (q,1)$. Moreover, as $r\to q$, we have $\eta\to 0$. On the other hand, the constants $K$ and $\epsilon_0$ also depend on $\delta$ and $Q_0$. Namely, these constants depend on the resolvent norm of $Q_0$ on the complement of $V_{\delta,r}(Q_0) = \{z \in \mathbb{C}: |z| \leq r \ \text{or} \ \mathrm{dist}(z, \sigma(Q_0)) \leq \delta \}$. That is, they depend on $$H_{\delta,r}(Q_0) = \sup \{ \|(z-Q_0)^{-1}\|: z \in V_{\delta,r}(Q_0) \}.$$ See~\cite[Sec.~3]{liverani2001rigorous}. This dependency might present a challenge when applying the above result to interpolate between two operators $Q_{\epsilon_1}$ and $Q_{\epsilon_2}$. However, Liverani~\cite[Lemma~4.2]{liverani2001rigorous} showed that for each $\delta > 0$ small enough, there exist $\epsilon_1 = \epsilon_1(Q_0, \delta/2, r) > 0$ and $c_* = c_*(Q_0, \delta/2, r) > 0$ such that $H_{\delta,r}(Q_{\bar{\epsilon}}) \leq c_*$ for any $0 \leq \tnorm{Q_{\bar{\epsilon}}-Q_{0}} \leq \epsilon_1$. From this, the constants $\epsilon_0 = \epsilon_0(Q_{\bar{\epsilon}}, \delta, r)$ and $K = K(Q_{\bar{\epsilon}}, \delta, r)$ provided by the above theorem when applied to $Q_{\epsilon}$ as a perturbation of $Q_{\bar{\epsilon}}$ instead of $Q_0$ can be uniformly bounded from below and above, respectively. Namely, there exist $\bar{\epsilon}_0 = \bar{\epsilon}_0(Q_0, \delta,r) > 0$ and $\bar{K} = \bar{K}(Q_0,  \delta,r) > 0$ such that 
$$
\tnorm{\Pi_{\epsilon_1}^{(\lambda,\delta)} - \Pi_{\epsilon_2}^{(\lambda,\delta)}} \leq \bar{K} \, \tau^\eta \quad \text{for any $0 \leq \tau=\tnorm{Q_{\epsilon_1} - Q_{\epsilon_2}} \leq \bar{\epsilon}_0$}.
$$
\end{rem}

\section{Proof of Proposition~\ref{prop:statistical-stability}} Now we are ready to prove the following theorem. First, we define the metric
\begin{equation*}
    \D^\alpha_{\smash{C^0}}(f,g)=\int d_{C^0}(f_\omega,g_\omega)^\alpha \, d\mathbb{P}, \quad \text{for $f,g\in C^0_{\smash{\mathbb{P}}}(X)$} \index{metric structures!Cr distance@\(\D_{C^0}^{\alpha}\), averaged \(\alpha\)-\(C^0\) distance on \(C^0_{\smash{\mathbb P}}(X)\)}
\end{equation*}
where $0<\alpha \leq 1$ and $C^0_{\smash{\mathbb{P}}}(X)$ denotes the set of continuous random maps up to $\mathbb{P}$-almost identification of fiber maps. For $\alpha=1$, we have that $\D^\alpha_{\smash{C^0}}=\D_{C^0}$ where the metric $\D_{C^0}$ was introduced in~\S\ref{ss:topologies}. Moreover, by the Jensen inequality,  $\D_{\smash{C^0}}^\alpha(\cdot,\cdot) \leq \D_{{C^0}}(\cdot,\cdot)^\alpha$.

\index{stationary measures!statistical stability}
\begin{thm} \label{prop:statistical-stability-length} Let $(X,d)$ be a locally length metric space and consider a mostly contracting random map $f\in \EE(X)$  satisfying~\eqref{eq:integral_condition} for some $0<\beta\leq 1$. Then, there exist $n\geq 1$ 
and a neighborhood $\mathcal{B}$ of $f$ in $(\EE(X),\Dbeta)$ such that the number of ergodic stationary probability measures is upper semicontinuous in $\mathcal{B}$.

{In addition,  if $\mathcal{C}$ denotes the subset of random maps $g$ in $\mathcal{B}$ having exactly the same number $r\in \mathbb{N}$ of ergodic stationary measure as $f$, then the $g$-stationary measures can be labeled as $\mu_{g,1},\dots,\mu_{g,r}$ in such a way that  for every $i=1,\dots,r$
\begin{enumerate}[leftmargin=1cm]
   \item  $\mu_{g,i}$ depends continuously in the weak$^*$ topology on $g\in \mathcal C$, and 
    \item  there is $\alpha_0>0$ such that  for any $0<\alpha\leq \alpha_0$, there exist a neighborhood $\mathcal{B}_\alpha \subset  \mathcal B$ of $f$ and constants $K>0$ and $\eta \in (0,1]$  satisfying 
\[
   \left|
      \int\phi\,d\mu_{g,i}
      -
      \int\phi\,d\mu_{h,i}
   \right|
   \leq
   K\|\phi\|_\alpha
   \D^\alpha_{\smash{C^0}}(g,h)^\eta  \quad \text{for every $g,h\in \mathcal{C}\cap\mathcal{B}_\alpha$.}
\]
\end{enumerate} 
}
\end{thm}


\begin{proof}
For each $h \in \EE(X)$, denote by $P_h$ the corresponding Koopman operator given by $P_h\phi=\int \phi\circ h_\omega \, d\mathbb{P}$.  Let $f\in \EE(X)$ be a mostly contracting random map satisfying~\eqref{eq:integral_condition}. Since $X$ is a compact locally length metric space, by Proposition~\ref{prop:lasota-york-uniforme}, there exist $0<\alpha_0<\beta$, $r>0$ and $n\geq 1$ such that every random map $h$ sufficiently close to $f$ in $(\EE(X),\Dbeta)$  satisfies the assumption~\eqref{eq:integral_condition} for $\beta$ and 
it is locally contracting on average for any $0<\alpha\leq \alpha_0$ and some $q=q(\alpha)<1$. By Lemma~\ref{lem:local-contracting-average-quasi-compact}, $P_h$ is quasi-compact. This implies that $\lambda=1$ is an isolated eigenvalue of $P_h$.  Moreover, we also find a neighborhood $\mathcal{B}_0$ of $f$ in $(\EE(X),\Dbeta)$
such that 
$|P^n_h\phi|_{\alpha} \leq q |\phi|_\alpha + C \|\phi\|_{\infty}$ for all $\phi\in C^\alpha(X)$ and $h\in\mathcal{B}_0$. In particular,
$$
 \|P^n_h\phi\|_{\alpha} \leq q \|\phi\|_\alpha + 2C \|\phi\|_{\infty}.
$$
Consider $\mathcal{E}=C^\alpha(X)$,  $|\cdot|=\|\cdot\|_{\infty}$ and $\|\cdot\|=\|\cdot\|_\alpha$. 
With this notation, the above estimate is exactly assumption~(i) in Theorem~\ref{thm:keller-liverani}.
Since the inclusion $C^\alpha(X) \hookrightarrow C(X)$ is also compact, we are in a position to apply that theorem and Remark~\ref{rem:liverani} provided that assumption (ii) holds. To show (ii), for any 
$g,h\in\mathcal{B}_0$,  $\phi\in C^\alpha(X)$ and $x\in X$,
\begin{align*}
\big|P_g\phi(x)-P_h\phi(x) \big| &=\left|\int (\phi\circ g_\omega(x)-\phi\circ h_\omega(x)) d\mathbb{P}\,\right| \\ 
&\leq |\phi|_{\alpha} \int d_{C^0}(g_\omega,h_\omega)^\alpha \,d\mathbb{P} 
= \|\phi\|_{\alpha} \D^\alpha_{\smash{C^0}}(g,f).  
\end{align*}
From here, we get that
$\tnorm{P_g-P_h}\leq \D^\alpha_{\smash{C^0}}(g,h)$
where $\tnorm{\cdot}$ is the  {operator norm} on the space of bounded operators $\mathcal{L}(C^\alpha(X),C(X))$. Thus, by Theorem~\ref{thm:keller-liverani} and Remark~\ref{rem:liverani} applied at $\lambda=1$ and $\delta>0$ small enough, we have a neighborhood ${\mathcal{B}_\alpha} \subset \mathcal{B}_0$ of $f$ in $(\EE(X),\Dbeta)$ such that 
\begin{equation}\label{continuityProj}
\tnorm{\Pi^{(1,\delta)}_g-\Pi^{(1,\delta)}_h}\leq \bar{K}  \D^\alpha_{\smash{C^0}}(g,h)^{\eta}
\end{equation}
for any $g,h$ in ${\mathcal{B}_\alpha}$, where $\bar{K}>0$ and $\eta>0$ are positive constants and $\Pi^{(1,\delta)}_g$ and $\Pi^{(1,\delta)}_h$ are the corresponding total spectral projections associated with $B_\delta(1)\cap \sigma(P_g)$  and $B_\delta(1)\cap \sigma(P_h)$ respectively. Moreover, the rank of these spectral projections coincides with the rank of $\Pi^{\smash{(1,\delta)}}_{f}$. Since $\Pi^{\smash{(1,\delta)}}_{f}$ projects onto $\mathrm{Ker}(1-P_f)$ for $\delta>0$ small enough, we get that 
$$
\dim \mathrm{Ker}(P_h-1) \leq \dim  \mathrm{Ker}(P_f-1)  \quad \text{for every $h\in{\mathcal{B}_\alpha}$.}
$$
 Since by Theorem~\ref{thm:Herver-improved}\textup{(i)} the dimension of these kernels coincides with the number of ergodic stationary measures, we get the first assertion of the theorem.

{
If \(g\in\mathcal C\cap \mathcal{B}_\alpha\), then \(g\) has
exactly \(r\) ergodic stationary measures, and again by
Theorem~\ref{thm:Herver-improved}\textup{(i)}, 
$$
\dim\mathrm{Ker}(P_g-1)=r =\dim\mathrm{Ker}(P_f-1)= \dim \operatorname{Img}(\Pi_g^{(1,\delta)}) \qquad g\in\mathcal C\cap \mathcal{B}_\alpha.
$$ 
Hence
\begin{equation}
\label{eq:projection-equals-stationary-projection}
   \operatorname{Img}(\Pi_g^{{(1,\delta)}})
   =
  \mathrm{Ker}(1-P_g),
   \qquad g\in\mathcal C\cap \mathcal{B}_\alpha .
\end{equation}
Thus, \(\Pi_g^{\smash{(1,\delta)}}\) is the projection onto $\mathrm{Ker}(P_g-1)$ for every \(g\in\mathcal C\cap \mathcal{B}_\alpha\). By quasi-compactness of $P_g$ on $C^\alpha(X)$, this spectral projection satisfies \[ \Pi_g^{\smash{(1,\delta)}} P_g = P_g\Pi_g^{\smash{(1,\delta)}}=\Pi_g^{\smash{(1,\delta)}}\] and  the sequence of operators 
\[ 
A^g_n = \frac{1}{n}\sum_{k=0}^{n-1}P^{k}_g \to  \Pi_g^{{(1,\delta)}} 
\]  
in the strong operator topology, cf.~\cite[Corollary~2 and Theorem~5]{Na:18}. Therefore, denoting by \(\mathcal I_g\) the  compact convex set of \(g\)-stationary probability
measures, we have the following:
\begin{claim} \label{claim:invariante-pii}
For every probability measure $\nu$ on $X$ it holds that $(\Pi_g^{\smash{(1,\delta)}})^*\nu \in \mathcal{I}_g$. Moreover, $(\Pi_g^{\smash{(1,\delta)}})^*\nu=\nu$ provided $\nu \in \mathcal{I}_g$.  
\end{claim}
\begin{proof}
For every \(\phi\in C^\alpha(X)\),
\[
\begin{aligned}
   \int P_g\phi\,d\bigl((\Pi_g^{(1,\delta)})^*\nu\bigr)
   =
   \int \Pi_g^{(1,\delta)}(P_g\phi)\,d\nu  
   =
   \int \Pi_g^{(1,\delta)}\phi\,d\nu       
   =
   \int \phi\,d\bigl((\Pi_g^{(1,\delta)})^*\nu\bigr).
\end{aligned}
\]
Thus \((\Pi_g^{(1,\delta)})^*\nu\) is \(g\)-stationary.  Moreover, since \(A_n^g\to\Pi_g^{(1,\delta)}\) strongly and each \(A_n^g\) is a Markov operator, \(\Pi_g^{(1,\delta)}\) is positive and satisfies \(\Pi_g^{(1,\delta)}1_X=1_X\). Hence, 
the functional
\[
   \phi\mapsto \int \Pi_g^{(1,\delta)}\phi\,d\nu
\]
is positive and sends \(1_X\) to \(1\). Therefore
\((\Pi_g^{(1,\delta)})^*\nu\) is also a probability measure.

On the other hand, if \(P^*_g\nu=\nu\), then for every \(\phi\in C^\alpha(X)\),
\[
   \int A_n^g\phi\,d\nu
   =
   \frac1n\sum_{k=0}^{n-1}\int P_g^k\phi\,d\nu
   =
   \int\phi\,d\nu .
\]
Letting \(n\to\infty\), we obtain
\[
   \int_X \Pi_g^{(1,\delta)}\phi\,d\nu
   =
   \int_X\phi\,d\nu .
\]
Hence \((\Pi_g^{(1,\delta)})^*\nu=\nu\).
\end{proof}

Consider $g,h\in\mathcal{C}$ and \(\mu\in\mathcal  I_g\). By Claim~\ref{claim:invariante-pii} \((\Pi_g^{\smash{(1,\delta)}})^*\mu=\mu\) and
\((\Pi_h^{\smash{(1,\delta)}})^*\mu\in\mathcal I_h\). Hence, for every
\(\phi\in C^\alpha(X)\) with \(\|\phi\|_\alpha\leq1\),
\[
\begin{aligned}
\left|
   \int\phi\,d\mu
   -
   \int\phi\,d\big((\Pi_h^{(1,\delta)})^*\mu\big)
\right|
=
\left|
   \int
   \left(
      \Pi_g^{(1,\delta)}\phi
      -
      \Pi_h^{(1,\delta)}\phi
   \right)
   d\mu
\right| \leq
   \tnorm{\Pi_g^{(1,\delta)}-\Pi_h^{(1,\delta)}} .
\end{aligned}
\]
Therefore,
$\operatorname{dist}_{d_\alpha^*}(\mu,\mathcal I_h)
   \leq
   \tnorm{\Pi_g^{(1,\delta)}-\Pi_h^{(1,\delta)}}$ where \(d_\alpha^*\) denotes the dual $\alpha$-H\"older distance on
probability measures given by
\[
   d_\alpha^*(\nu,\eta)
   \eqdef
   \sup_{\|\phi\|_\alpha\leq1}
   \left|
      \int_X\phi\,d\nu-\int_X\phi\,d\eta
   \right|.
\]
Interchanging \(g\) and \(h\), and using
\eqref{continuityProj}, we obtain
\begin{equation}
\label{eq:stationary-simplex-holder-hausdorff}
   d_{H,\alpha}(\mathcal I_g,\mathcal I_h)
   \leq
   \bar{K}\,\D^\alpha_{\smash{C^0}}(g,h)^\eta,
   \qquad g,h\in\mathcal C,
\end{equation}
where \(d_{H,\alpha}\) denotes the Hausdorff distance associated with
\(d_\alpha^*\). 

Let $E$ denote the vector space of finite signed Borel measures on
\(X\), endowed with the dual Hölder norm
\[
   \|\rho\|_\alpha^*
   \eqdef
   \sup_{\|\phi\|_\alpha\leq1}
   \left|
      \int_X\phi\,d\rho
   \right|.
\]
Thus, for probability measures \(\nu,\eta\), one has
$d_\alpha^*(\nu,\eta)=\|\nu-\eta\|_\alpha^*$. 
Since the ergodic \(f\)-stationary measures
\(\mu_{f,1},\dots,\mu_{f,r}\) are the vertices of the simplex
\(\mathcal I_f\), we can write
\[
   \mathcal I_f
   =
   \operatorname{conv}\{\mu_{f,1},\dots,\mu_{f,r}\}.
\]
Applying Lemma~\ref{lem:vertex-stability-simplex} in the normed space
\((E,\|\cdot\|_\alpha^*)\), and shrinking
\(\mathcal B_\alpha\) if necessary, the vertices of \(\mathcal I_g\), namely the ergodic \(g\)-stationary measures, can be labelled
$
   \mu_{g,1},\dots,\mu_{g,r}$,
for $g\in\mathcal C\cap\mathcal B_\alpha$,
so that $$ d_\alpha^*(\mu_{g,i},\mu_{h,i})
   =
   \|\mu_{g,i}-\mu_{h,i}\|_\alpha^*
   \leq
   C_0\,d_{H,\alpha}(\mathcal I_g,\mathcal I_h)\quad \text{for every \(i=1,\dots,r\).}$$ 
 Combining this with
\eqref{eq:stationary-simplex-holder-hausdorff}, we conclude that
 \[
   d_\alpha^*(\mu_{g,i},\mu_{h,i})
   \leq
   K\,\D^\alpha_{\smash{C^0}}(g,h)^\eta,
    \qquad
   \text{ for $g,h\in\mathcal{C}\cap \mathcal{B}_\alpha$ }  
   i=1,\dots,r.
 \]
Equivalently, for every
\(\phi\in C^\alpha(X)\),
 \[
   \left|
      \int\phi\,d\mu_{g,i}
      -
      \int\phi\,d\mu_{h,i}
   \right|
   \leq
   K\|\phi\|_\alpha
   \D^\alpha_{\smash{C^0}}(g,h)^\eta .
 \]}
 As $\alpha$-H\"older functions are uniformly dense in the set of continuous functions, we have the convergence of $\mu_h$ to $\mu_g$ in the weak$^*$ topology for every $g\in \mathcal{B}_\alpha$.  Taking $\mathcal{B}$ as the union of $\mathcal{B}_\alpha$ for $0<\alpha\leq \alpha_0$, we complete the proof of the proposition.
\end{proof}

Given $g\in \mathcal{C}$ and $\phi \in C(X)$, denoting 
\[
   \Lambda_g(\phi)   
   \eqdef  \sup\left\{\int \phi: \mu \ \text{$g$-stationary} \right\} 
   =
   \max_{1\leq i\leq r}
   \int\phi\,d\mu_{g,i},
\]
and  using the elementary inequality $|
      \max_{1\leq i\leq r} a_i
      -
      \max_{1\leq i\leq r} b_i|
   \leq
   \max_{1\leq i\leq r}|a_i-b_i|$, 
we obtain the following consequence:
\begin{cor}
\label{cor:holder-maximal-average}
Let \(f\in\EE(X)\) be under the assumptions of
Theorem~\ref{prop:statistical-stability-length}.  Then, 
\[
   \big|\Lambda_g(\phi)-\Lambda_h(\phi)\big|
   \le
   K\|\phi\|_\alpha\,
   \D^\alpha_{\smash{C^0}}(g,h)^\eta  \quad \text{for every \(g,h\in\mathcal C\cap \mathcal{B}_\alpha\) and  $\phi \in C^\alpha(X)$.}
\]
\end{cor}

We then conclude  Proposition~\ref{prop:statistical-stability} from Theorem~\ref{prop:statistical-stability-length}. 

\begin{proof}[Proof of Proposition~\ref{prop:statistical-stability}]
{Theorem~\ref{prop:statistical-stability-length} and
Lemma~\ref{rem:topology} gives a
neighborhood \(\mathcal B\) of \(f\) in \((\CE(X),\D_{C^1})\) where the number of
ergodic stationary measures is upper semicontinuous. Since \(f\) is uniquely ergodic, every \(g\in\mathcal B\) has at most one ergodic stationary measure.
As the set of stationary probability measures is non-empty, every
\(g\in\mathcal B\) is uniquely ergodic. Denote its stationary measure by
\(\mu_g\). The weak\(^*\)-continuity of \(g\mapsto\mu_g\) follows from the
continuity conclusion in Theorem~\ref{prop:statistical-stability-length}, since
in the uniquely ergodic case the constant-rank stratum \(\mathcal C\) coincides
with \(\mathcal B\).}

{It remains only to justify item \textup{(iii)} for an arbitrary H\"older
observable. Let \(\phi:X\to\mathbb R\) be H\"older continuous. Choose an exponent
\(\theta=\theta(\phi)\in (0,1]\) such that \(\phi\in C^\theta(X)\). By replacing the Riemannian
distance \(d\) by the equivalent normalized metric
\[
   d'=\frac{d}{\operatorname{diam}(X)},
\]
we may assume that \(\operatorname{diam}(X)=1\). This does not affect the
mostly contracting property, because local Lipschitz constants and Lyapunov
exponents are unchanged by multiplying the metric by a constant. It only
changes the metrics \(\D_{C^0}\) and the H\"older norms by multiplicative
constants.}

{Let \(\alpha_0>0\) be the exponent given by
Theorem~\ref{prop:statistical-stability-length}, and choose
$\alpha= \min\{\theta,\alpha_0\}$. 
Since \(\operatorname{diam}(X)=1\), the function \(\phi\) is also
\(\alpha\)-H\"older. Applying
Theorem~\ref{prop:statistical-stability-length} with this exponent \(\alpha\), using again that the uniquely ergodic stratum is the whole neighborhood, and observing that $\D^\alpha_{\smash{C^0}}(\cdot,\cdot) \leq \D_{\smash{C^0}}(\cdot,\cdot)^\alpha$,  
there exist a neighborhood \(\mathcal B_\alpha\subset\mathcal B\) of \(f\), and
constants \(K>0\) and \(\gamma \in(0,1]\), such that
\[
   \left|
      \int \phi\,d\mu_g-\int \phi\,d\mu_h
   \right|
   \leq
   K\|\phi\|_\alpha
   \D_{\smash{C^0}}(g,h)^\gamma
   \qquad
   \text{for every }g,h\in\mathcal B_\alpha.
\]
This proves item~\textup{(iii)} and completes the proof.}
\end{proof}

}

{
Since the number $r=r(f)$ of $f$-stationary measures is upper semicontinuous on the set of mostly contracting random maps $g$, the stability result on the stratum $\mathcal{C}$, where $r(f)$ is constant, yields the following generic result. 

\begin{thm}
\label{prop:generic-statistical-stability}
Let \(X\) be a compact Riemannian manifold, and consider 
  the open subset \(\mathcal{MC}\) of mostly contracting random maps in $(\CE(X),\D_{C^1})$ satisfying~\eqref{eq:integral_condition}.
Then, the set $\mathcal{SS}$ of statistically stable random maps is open and dense in \(\mathcal{MC}\). 

More precisely,  for every \(f\in\mathcal{SS}\) there exists a neighborhood $\mathcal{U}$ where the number $r\in\mathbb{N}$ of ergodic stationary probability measures is constant and admits a labelling
$\mu_{g,1},\dots,\mu_{g,r}$ such that the map $\mu_{g,i}$ varies weak\(^*\)-continuous for $g\in \mathcal{U}$ for $i=1,\dots,r$.  Moreover, for every sufficiently small \(\alpha>0\), after possibly shrinking \(\mathcal U\), the
$\alpha$-H\"older estimates of Theorem~\ref{prop:statistical-stability-length} hold on
\(\mathcal U\).  
\end{thm}

\begin{proof}
By Theorem~\ref{prop:statistical-stability-length}, the function 
$r:\mathcal{MC}\to\mathbb N$ of the number of ergodic stationary measures is upper semicontinuous. Let
$$\mathcal{SS}\eqdef
   \{f\in\mathcal{MC}: r \text{ is locally constant at } f\}.$$
Then \(\mathcal{SS}\) is open by definition. We prove that it is dense. Let
\(\mathcal O\subset\mathcal{MC}\) be a non-empty open set. Since \(r\) takes values in \(\mathbb N\), there exists \(f_0\in\mathcal O\) such that $r(f_0)=\min\{r(g):g\in\mathcal O\}$. By upper semicontinuity, there is a neighborhood
\(\mathcal V\subset\mathcal O\) of \(f_0\) such that
$r(g)\le r(f_0)$ for every $g\in\mathcal V$. By the minimality of \(r(f_0)\) in \(\mathcal O\), we also have
$ r(g)\ge r(f_0)$ for every $g\in\mathcal V$. Thus \(r(g)=r(f_0)\) for every \(g\in\mathcal V\), so \(r\) is locally constant
on \(\mathcal V\). Hence \(\mathcal O\cap\mathcal{SS}\neq\emptyset\), proving
that \(\mathcal{SS}\) is dense in \(\mathcal{MC}\).

Finally, let \(f\in\mathcal{SS}\). By definition, there is a neighborhood
\(\mathcal U\subset\mathcal{MC}\) of \(f\) on which \(r(g)=r(f)\) is constant.
The stratum part of Theorem~\ref{prop:statistical-stability-length} then gives
a labelling
\[
   \mu_{g,1},\dots,\mu_{g,r},
   \qquad g\in\mathcal U,
\]
of the ergodic \(g\)-stationary probability measures such that each
\(g\mapsto\mu_{g,i}\) is weak\(^*\)-continuous. The same theorem gives the
corresponding H\"older estimates, after possibly shrinking \(\mathcal U\) for
the chosen H\"older exponent. This proves the theorem.
\end{proof}
}

\begin{rem}[Metastability]
\label{rem:metastability-loss-stationary-regime}
Although the number of ergodic stationary measures is upper semicontinuous under mostly contracting perturbations, it need not be locally constant, and may drop
under arbitrarily small perturbations. For instance, consider a mostly contracting random map $f$ from an IFS generated by three maps $f_1$, $f_2$ and $f_3=\mathrm{id}$ 
on \(\mathbb S^1\), with equal probabilities, and suppose that there are two
disjoint closed intervals \(I_1,I_2\) which are invariant under all generators. 
Assume that the restrictions of $f$ to \(I_1\) and \(I_2\) are uniquely ergodic. Then
\(f\) has two ergodic stationary measures. If the identity generator is replaced
by a small irrational rotation \(R_\varepsilon\), the two invariant components
are destroyed and, by minimality, the perturbed random map $h_\varepsilon$ has a unique stationary measure.  This loss of a stationary regime has a spectral interpretation. For \(f\), the
eigenvalue \(1\) of the annealed Koopman operator \(P\) has multiplicity two, corresponding to the two stationary components. For \(h_\varepsilon\), only one
copy of the eigenvalue \(1\) remains, corresponding to the unique stationary probability measure. The other part of the former stationary eigenspace becomes
a near-invariant spectral mode with spectral value close to \(1\). In a simple
two-component metastable situation, this spectral value is a real eigenvalue
\(\xi_\varepsilon<1\), with \(\xi_\varepsilon\to1\) as
\(\varepsilon\to0\). The associated dual eigenvector is not a probability measure. Since
\(P_\varepsilon 1_{\mathbb{S}^1}=1_{\mathbb{S}^1}\) 
where $P_\varepsilon$ is the annealed Koopman operator of
$h_\varepsilon$, any finite signed measure \(\eta_\varepsilon\)
satisfying $P_{\varepsilon}^*\eta_\varepsilon
   =
   \xi_\varepsilon\eta_\varepsilon$ with $\xi_\varepsilon\neq1$,
must have total mass zero. Thus \(\eta_\varepsilon\) should be interpreted as a
signed measure measuring the imbalance between the two former stationary
components. Dynamically, it describes the long transient time during which the
system behaves as if the two old stationary regimes still existed, before the
slow exchange produced by the perturbation becomes visible.
\end{rem}

\section{Weak statistical stability}  To conclude the section, we prove the following result. Compare with~\cite[Prop.~2.3]{andersson2020statistical}. First, recall that $C^0_{\smash{\mathbb{P}}}(X)$ denotes the set of continuous random maps up to $\mathbb{P}$-almost identification of fiber maps and the metric $\D_{C^0}$ was introduced in~\S\ref{ss:topologies}.

\index{stationary measures!weak statistical stability}
\begin{thm} \label{thm:Weak statistically stability} The set of $f$-stationary measures varies upper semicontinuously in the weak$^*$ topology with respect to $f$ in $(C^0_{\smash{\mathbb{P}}}(X),\D_{C^0})$. That is,  for every sequence $(f_{n})_{n\geq 1}$
 of continuous random maps with $\D_{C^0}(f_n,f)\to 0$ and every sequence $(\mu_n)_{n\ge 1}$ of $f_n$-stationary measures, it holds that any accumulation point $\mu$ in the weak$^*$ topology of
 $(\mu_n)_{n\geq 1}$ is an $f$-stationary measure. 
\end{thm}

\begin{proof}
Let $\alpha\in (0,1)$ and consider an $\alpha$-H\"older continuous map $\phi: X \to \mathbb{R}$. By taking a subsequence if necessary, assume that  $\mu_n$ converges to $\mu$ in the weak$^*$ topology.  Then $\int \phi \, d\mu_n \to \int \phi \, d\mu$.  Since 
$\mu_n$ is $f_n$-stationary, we have that $\int P_n \phi \, d\mu_n = \int \phi \,d\mu_n$ where $P_n$ denotes the Koopman operator associated with $f_n$.
Hence
\begin{align*}
\big|\int \phi \, d\mu_n - \int P\phi \,d\mu\big| &\leq 
\int |P_n\phi-P\phi| \, d\mu_n 
+ \big|\int P\phi \, d\mu_n - \int P\phi \, d\mu\big| \\ 
&\leq \|P_n\phi-P\phi\|_{\infty} +  \big|\int P\phi \, d\mu_n - \int P\phi \, d\mu\big|. 
\end{align*}
Since $P\phi$ is also continuous, then the second term on the right-hand side of the above estimate goes to zero as $n\to \infty$. The first term can be estimated as follows: for every $x\in X$,
$$
   |P_n\phi(x)-P\phi(x)| \leq \int |\phi(f_{n,\omega}(x))-\phi(f_{\omega}(x))|\, d\mathbb{P} \leq |\phi|_\alpha \, \D_{C^0}(f_n,f)^\alpha. 
$$
Hence, since $\D_{C^0}(f_n,f)\to 0$, we get that $\|P_n\phi-P\phi\|_\infty \to 0$ as $n\to \infty$. Consequently, we get that $\int \phi d\mu_n \to \int P\phi \, d\mu$ as $n\to\infty$. Since the limit is unique, then $\int P\phi \, d\mu= \int \phi \, d\mu$ for every $\alpha$-H\"older continuous function $\phi$. By approximation, this implies $\int P\varphi \, d\mu= \int \varphi \, d\mu$ for every continuous function $\varphi:X\to \mathbb{R}$. Consequently, $\mu$ is $f$-stationary. 
\end{proof}

\chapter{Locally constant linear cocycles}
\label{s:linear}

\abstract{
{This chapter studies locally constant linear cocycles through their induced
projective random maps. We establish a variational representation of the
first Lyapunov exponent in terms of stationary measures on projective space and
annealed growth along individual projective directions. We then
characterize quasi-irreducibility through the equality of these projective averages and introduce the equator as the maximal invariant subspace where the first exponent is not attained. 
Maximizing stationary measures assigns zero mass to the equator, and vectors outside this subspace satisfy a strong law of large
numbers for the top exponent. The last part of the chapter studies the projective dynamics when the first Lyapunov
exponent is simple.  In this regime, we establish the equivalence between quasi-irreducibility of the cocycle, the unique ergodicity, the global contraction on average, being mostly contracting of the induced projective random map, and the spectral gap of the associated Koopman operator. } 
}


\section{{Setting and notation}}
Let $(\Omega,\mathscr{F},\mathbb{P})$ be a Bernoulli product probability space, and consider a measurable, locally constant linear cocycle \( A:\Omega \to \mathrm{GL}(m) \) with \( m \geq 2 \), such that \( \log^+ \|A\| \) is \(\mathbb{P}\)-integrable. Define the associated projective random map \( f_A:\Omega \times P(\mathbb{R}^m) \to P(\mathbb{R}^m) \) and the upper Lyapunov exponent \(\lambda_1(A)\) by 
\[
(f_A)_\omega(\hat{x}) \eqdef A(\omega)\hat{x}, \quad \text{and} \quad
\lambda_1(A) \eqdef \lim_{n\to\infty} \frac{1}{n} \int \log \|A^n(\omega)\| \, d\mathbb{P}.
\]

For a projective class \(\hat{x} \in P(\mathbb{R}^m)\), we denote by \(x \in \mathbb{R}^m\) any unit vector representative of \(\hat{x}\). Using this representation, we introduce the following quantities:
\begin{align*} \index{linear cocycles!\(\phi_A\), projective potential}
    \lambda(\hat{x}) \eqdef \limsup_{n\to\infty} \frac{1}{n} \int \log \|A^n(\omega)x\|\, d\mathbb{P}, \quad \text{and} \quad 
    \phi_A(\hat{x}) \eqdef \int \log \|A(\omega)x\|\, d\mathbb{P}, 
\end{align*} 
for $\hat{x} \in P(\mathbb{R}^m)$. By the integrability assumption on \(\log^+ \|A\|\), the function \(\phi_A\) is upper semicontinuous and satisfies \(\phi_A < \infty\). Moreover, if \(\log^+ \|A^{-1}\|\) is also \(\mathbb{P}\)-integrable, then \(\phi_A\) becomes a continuous real-valued function.

We denote by \(\LP\) the space of all locally constant linear cocycles \(A:\Omega \to \mathrm{GL}(m)\) satisfying that  \(\log^+ \|A^{\pm 1}\|\) are \(\mathbb{P}\)-integrable, up to \(\mathbb{P}\)-almost sure identification. That is, two cocycles \(A\) and \(B\) are identified if \(A(\omega) = B(\omega)\) for \(\mathbb{P}\)-a.e.~\(\omega \in \Omega\).  
Sometimes, we further require the exponential moment condition~\eqref{eq:exponential-moment-condition}, that is,
\[
\int  (\max\{\|A(\omega)\|,\, \|A(\omega)^{-1}\|\})^\beta \, d\mathbb{P} < \infty, \quad \text{for some } \beta > 0.
\]

\begin{rem} \label{rem:beta}
Since \(\ell_A \eqdef  \max\{\|A\|,\, \|A^{-1}\|\} \geq 0\), the above condition can be 
written as \(\ell_A \in L^\beta(\mathbb{P})\). Moreover, as \(\mathbb{P}\) is a probability measure, \(\ell_A \in L^{\beta'}(\mathbb{P})\) for any \(0 < \beta' \leq \beta\). Therefore, without loss of generality, we may assume \(0 < \beta \leq 1\) in~\eqref{eq:exponential-moment-condition} if necessary.
\end{rem}

\vspace{-0.5cm}

\section{Representation of Lyapunov exponents}
According to Furstenberg-Kesten~\cite{FK:60}, if $\log^+ \|A\|\in L^1(\mathbb{P})$, then the Lyapunov exponent of $A$ equals the pointwise Lyapunov exponent for $\mathbb{P}$-a.e.~$\omega$. Namely,
\begin{equation} \label{eq:limsup4} 
 \lambda_1(A)= 
\lim_{n\to\infty} \frac{1}{n}  \log  \|A^n(\omega)\|   \qquad  \text{for $\mathbb{P}$-a.e.~$\omega\in \Omega$}.   \index{linear cocycles!\(\lambda_1(A)\), first Lyapunov exponent}
\end{equation}
This follows from Kingman's subadditive ergodic theorem applied to the subadditive sequence $(F_n)_{n\geq 1}$ given by $F_n(\omega)=\log \|A^n(\omega)\|$ and the invariance of the ergodic measure~$\mathbb{P}$. 
{The following theorem gives a complementary description of $\lambda_1(A)$ in terms of stationary measures and annealed growth rates along projective directions.
}
\index{Lyapunov exponents!variational representation}
\index{Lyapunov exponents!asymptotic representation}
\begin{thm} \label{cor:Furstenberg}  
 Let $A\in \LP$ 
 be a locally constant linear cocycle.  
 Then
\begin{align} 
    \label{eq:Ledrapier-formula}  \lambda_1(A) &=  \max\big\{\int \phi_A \, d\mu: \ \  \mu \in \mathcal{I} \ \text{ergodic}\big\} 
 = \sup_{\hat{x}\in P(\mathbb{R}^m)} \lambda(\hat{x})
\end{align}
where $\mathcal{I}$ is the set of $f_A$-stationary probability measures. Furthermore, 
\begin{align*} 
   \lambda_1(A)  
    = 
 \lim_{n\to\infty} \, \max_{\hat{x} \in P(\mathbb{R}^m)} 
 \frac{1}{n}\int \log \|A^n(\omega)x\|\, d\mathbb{P} = \lim_{n\to\infty} \sup_{\mu\in\mathcal{I}}\frac{1}{n} 
 \int \log \|A^n(\omega)x\|\, d\mathbb{P} d\mu
\end{align*}
where the limit may be replaced by the infimum.
In particular,
\begin{equation}  \label{eq:Furstenberg-formula}
    \lambda_1(A)=\int \phi_A  \,d\mu  
\end{equation}
provided $\int \phi_A \, d\mu$ takes the same value for every $\mu \in \mathcal{I}$. Moreover, in this case, for every $\hat{x}\in P(\mathbb{R}^m)$ and $\mathbb{P}$-a.e.~$\tilde{\omega}\in \Omega$, 
\begin{equation} \label{eq:Kifer-Furstenberg}
\index{strong law of large numbers}
 \lambda_1(A)= \lim_{n\to\infty} \frac{1}{n} 
 \int \log \|A^n(\omega)x\|\, d\mathbb{P}=
 \lim_{n\to\infty} \frac{1}{n} \log \|A^n(\tilde{\omega})x\|.
 \end{equation}
 \end{thm}

Furstenberg's theorem~\cite[Thm.~8.5]{Fur:63} asserts the integral formula~\eqref{eq:Furstenberg-formula} in the case of irreducible random walks on  $\mathrm{GL}(m)$. 
This also follows from the previous  {theorem}. Indeed, if $A$ is irreducible, then $A$ is quasi-irreducible. Moreover, according to Proposition~\ref{prop:quasi-irreducible} below, $A$ is quasi-irreducible if and only if $\int \phi_A \, d\mu$ takes the same value for all $f_A$-stationary measures~$\mu$. Thus, the previous theorem applies to obtain the mentioned integral formula. 

On the other hand, the first equality in~\eqref{eq:Ledrapier-formula} was also independently obtained  by Furstenberg and Kifer~\cite[Thm.~2.2]{furstenberg1983random} and Hennion~\cite[Cor.~2]{hennion1984loi} (see also~\cite{Led:84} and \cite[Sec.~6.1 and~6.2]{viana2014lectures}). Furthermore, the so-called \emph{strong law of large numbers} for the maximal Lyapunov exponent of linear cocycles, i.e.,~\eqref{eq:Kifer-Furstenberg}, was also previously achieved in~\cite[Thm.~2.1 and Cor.~3.11]{furstenberg1983random} and \cite[Thm~1]{hennion1984loi}. {The novelty of the theorem lies in the representation
\[
\lambda_1(A)=\sup_{\hat x\in P(\mathbb R^m)}\lambda(\hat x)
=\lim_{n\to\infty}\max_{\hat x\in P(\mathbb R^m)}
\frac1n\int \log \|A^n(\omega)x\|\,d\mathbb P,
\]
which allows one to recover the maximal Lyapunov exponent from the annealed growth of individual projective directions.  Moreover, our proof of Theorem~\ref{cor:Furstenberg} is conceptually different from the classical arguments: it arises as a direct application of the uniform Kingman theorem for Markov operators (Theorem~\ref{Kingmanuniform}) to the projective random map associated with the cocycle.}

\begin{proof}[Proof of Theorem~\ref{cor:Furstenberg}]
Let $X=P(\mathbb{R}^m)$. Consider the operator $P:C(X)\to C(X)$ associated with $f_A$ given by
$$
P\varphi(\hat{x}) = \int \varphi(A(\omega)\hat{x}) \, d\mathbb{P},  \quad \varphi\in C(X), \ \ \hat{x}\in X.  
$$
Let $(\phi_n)_{n\geq 1}$ be the $P$-additive sequence of real-valued continuous functions
$$
\phi_1 =\phi_A \quad \text{and} \quad \phi_n= \sum_{i=0}^{n-1} P^i\phi_1 \ \  \text{for $n>1$}.
$$
\begin{claim}  \label{claim:aditiva} $\phi_n(\hat{x})=\int \log \|A^n(\omega)x\|\, d\mathbb{P}$ for every $\hat{x}\in X$. 
\end{claim}
\begin{proof} For each $i=0,\dots,n-1$ and $\hat{x}\in X$, it holds that
\begin{align*}    
P^i\phi_1(\hat{x})&=\int \phi_1(A^i(\omega)\hat{x})\, d\mathbb{P}= 
 \int \log \bigg\| A(\tilde{\omega})\frac{A^i(\omega)x}{\|A^i(\omega)x\|} \bigg\| \, d\mathbb{P}(\tilde{\omega})d\mathbb{P}(\omega) \\
 &= \int \log \|A^{i+1}(\omega)x\|\, d\mathbb{P} - \int \log \|A^i(\omega)x\|\, d\mathbb{P}.
\end{align*}
From here, noting that $\|x\|=1$, we obtain that
\begin{equation*}
    \begin{aligned}
    \phi_n(\hat{x}) &=\sum_{i=0}^{n-1} P^i\phi_1(\hat{x})= \int \log \|A^n(\omega)x\|\, d\mathbb{P}- \int \log \|x\| \,d\mathbb{P} \\ &= \int \log \|A^n(\omega)x\|\, d\mathbb{P}. 
\end{aligned}
\end{equation*}
This concludes the proof of the claim.
\end{proof}

In view of the above observations, we can apply Theorem~\ref{Kingmanuniform} and Corollary~\ref{cor:Kingman-uniform}  to  $(\phi_n)_{n\geq 1}$. Moreover, since the $f_A$-stationary measures are exactly the $P^*$-invariant measures, for each $f_A$-stationary measure $\mu$, we have 
$$
\Lambda(\mu)\eqdef \lim_{n\to \infty } \frac{1}{n} \int \phi_n \, d\mu = \int \phi_1\, d\mu. 
$$
From Theorem~\ref{Kingmanuniform} (see~\eqref{eq:formula-Furstenberg} and (ii)), recalling the definition of the pointwise annealed Lyapunov exponent $\lambda(\hat{x})$ and Claim~\ref{claim:aditiva}, we obtain that
$$
\Lambda \eqdef \sup_{\mu\in \mathcal{I}}\Lambda(\mu)  = \max \big\{\int \phi_1\, d\mu:  \ \mu \ \text{ergodic $f_A$-stationary}\big\}= \sup_{\hat{x}\in X} \lambda(\hat{x}).$$
To conclude the proof of~\eqref{eq:Ledrapier-formula}, we will prove $\Lambda=\lambda_1(A)$.

First, from~\eqref{eq:Kingman1}, we have that 
\begin{equation*} 
    \Lambda =\lim_{n\to\infty} \max_{\hat{x}\in X} \frac{1}{n}  \phi_n(\hat{x})  
\leq 
\lim_{n\to\infty}  \frac{1}{n}\int \log \sup_{\hat{x}\in X}  \|A^n(\omega)x\| \, d\mathbb{P} \eqdef \lambda_1(A).
\end{equation*}
Secondly, let $\phi(\omega,\hat x)=\log \|A(\omega)x\|$. Clearly $\phi(\omega,\cdot) \in C(X)$ for all $\omega\in \Omega$ and $\int \|\phi(\omega,\cdot)\|_\infty\, d\mathbb{P} \leq \int \big|\log \|A(\omega)\|\big| \,d\mathbb{P}<\infty$ since $\log^+ \|A^{\pm 1}\|$ are $\mathbb{P}$-integrable. Hence, we can apply Corollary~\ref{thm:Bierman-Furstenberg} to $\phi$, $f_A$ and its associated skew-product $F$. Moreover, since $$\phi(F^i(\omega,\hat x))=\log \|A^{i+1}(\omega)x\|-\log \|A^i(\omega)x\| \ \ \text{for all} \ i\ge 0,$$ 
we have, similarly as before, that $\sum_{i=0}^{n-1} \phi(F^i(\omega,\hat x))=\log \|A^n(\omega)x\|$, and  thus, 
from Corollary~\ref{thm:Bierman-Furstenberg} we obtain that for every $\hat x\in X$, 
\begin{equation} \label{eq:lamabdamenor0}
    \limsup_{n\to\infty} \frac{1}{n} \log \|A^n(\omega)x\| \leq \max_{\mu \ \text{$f_A$-stationary}} \ \int \phi_1 \, d\mu = \Lambda \ \ \ \text{for $\mathbb{P}$-a.e.~$\omega\in \Omega$}. 
\end{equation}
On the other hand, given a basis $v_1,\dots,v_m$ of unit vectors in $\mathbb{R}^m$, we can identify $\|M\|$  with $\max_i\|Mv_i\|$ for every matrix $M$. 
Hence, by~\eqref{eq:limsup4}, for $\mathbb{P}$-a.e.~$\omega\in \Omega$,
\begin{align*}
    \lambda_1(A)&=\lim_{n\to\infty} \frac{1}{n} \log \|A^n(\omega)\| =\lim_{n\to\infty}  \frac{1}{n} \log \max_{i=1,\dots,m} \|A^n(\omega)v_i\| \\
&\leq \max_{i=1,\dots,m}  \limsup_{n\to\infty}  \frac{1}{n} \log \|A^n(\omega)v_i\|.
\end{align*}
Combining this with~\eqref{eq:lamabdamenor0}, we have $\lambda_1(A) \leq \Lambda$. This implies that $\Lambda=\lambda_1(A)$ as required and completes the proof of~\eqref{eq:Ledrapier-formula}. Moreover, this and~(iii) in Theorem~\ref{Kingmanuniform} also conclude the representation of $\lambda_1(A)$ as a uniform limit.  

Finally, \eqref{eq:Furstenberg-formula} follows immediately from~\eqref{eq:Ledrapier-formula} provided $\int \phi_A \, d\mu$ takes the same value for every $\mu\in \mathcal{I}$. Moreover, also in this case, 
for every $\hat{x}\in X$, from~\eqref{eq:uniform-Furstenberg-Kifer},
\begin{align*}
\lambda_1(A) =\lim_{n\to\infty} \frac{1}{n} \int \log \|A^n(\omega)x\|\, d\mathbb{P} 
\end{align*}
and from Corollary~\ref{thm:Bierman-Furstenberg},
$$\lambda_1(A)=\lim_{n\to\infty} \frac{1}{n}  \log \|A^n(\omega)x\| \quad \text{for $\mathbb{P}$-a.e.~$\omega\in\Omega$}
$$
concluding~\eqref{eq:Kifer-Furstenberg} and completing the proof of the theorem.
\end{proof}

We also obtain as a consequence of Kingman's ergodic theorem for Markov  operators (Theorem~\ref{Kingman}) the following:

\begin{cor} \label{cor:kigman-fusternberg}
Let $A :\Omega \to \mathrm{GL}(m)$ be a locally constant linear cocycle such that $\log^+ \|A\|$ is $\mathbb{P}$-integrable. Then, for every $f_A$-invariant measure $\mu$, 
\begin{enumerate}[label=(\roman*), leftmargin=1.2cm]
\item $\int \lambda(\hat{x})\, d\mu = \int \phi_A\, d\mu$ and $\lambda^+ \in L^1(\mu)$, \\[-0.2cm]
\item $\lambda(\hat{x}) = {\displaystyle\lim_{n \to \infty}} \,\frac{1}{n} \int \log \|A^n(\omega)x\|\, d\mathbb{P}$ for $\mu$-a.e.~$\hat{x} \in P(\mathbb{R}^m)$, \\[-0.2cm]
\item $\lambda(\hat{x}) = \int \phi_A\, d\mu$ for $\mu$-a.e.~$\hat{x}$, provided $\mu$ is ergodic.
\end{enumerate}
\end{cor}

\begin{proof} 
By Remark~\ref{rem:kingman}, Theorem~\ref{Kingman} applies under the assumptions of Theorem~\ref{Kingmanuniform}. In particular, we apply it to the Koopman operator \(P\) associated with \(f_A\) and to the sequence \((\phi_n)_{n\geq 1}\) considered in the proof of Theorem~\ref{cor:Furstenberg}. The conclusions then follow from the identity
$\Lambda(\mu)=\int \phi_A \, d\mu$.
\end{proof}

\section{Quasi-irreducibility}
Here, we characterize quasi-irreducible cocycles; see~\eqref{eq:reducible-Lyapunov-exponent} for the definition. We first fix some notation. A subspace \(L\subset \mathbb{R}^m\) is \(A\)-invariant if \(A(\omega)L=L\) for \(\mathbb{P}\)-a.e.~\(\omega\in\Omega\). For notational simplicity, we use the same symbol, say \(V\), for a subspace of \(\mathbb{R}^m\) and for the corresponding subset of \(P(\mathbb{R}^m)\), that is, the set of projective directions represented by vectors in \(V\).

\begin{lem}\label{prop:lambdaLint} 
Let $A: \Omega \to \mathrm{GL}(m)$ be  a locally constant linear cocycle such that $\log^+ \|A\|$ is $\mathbb{P}$-integrable. Consider an ergodic $f_A$-stationary measure $\mu$. Then
\begin{equation} \label{eq:phi-KF}
    \lim_{n\to\infty} \frac{1}{n} \log \|A^n(\omega)x\| = \int \phi_A \, d\mu \quad \text{for $\bar{\mu}$-a.e.~$(\omega, \hat{x}) \in \Omega \times P(\mathbb{R}^m)$}
\end{equation}
and there is an $A$-invariant subspace $L$ of $\mathbb{R}^m$ (the vector space spanned by the support of $\mu$) such that  $\mu(L) = 1$ and 
\begin{equation} \label{eq:L-mu}
    \lambda_1(A|_L) = 
    \lim_{n\to\infty} \frac{1}{n} \log \big\|A^n(\omega) \big\arrowvert_{L}\,\big\| 
    = \int \phi_A \, d\mu \ \ \text{for $\mathbb{P}$-a.e.~$\omega \in \Omega$.}
\end{equation}
\end{lem}

\begin{proof} 
Let $\tilde{\phi}_A(\omega, \hat{x}) \eqdef \log \|A(\omega)x\|$ and consider the skew-product $F$ associated with the projective random map $f_A$. Take the sequence $(\phi_n)_{n \geq 1}$ of continuous functions,
$$
\phi_1 = \tilde{\phi}_A \quad \text{and} \quad \phi_n = \sum_{i=0}^{n-1} \phi_1 \circ F^i \ \ \text{for $n > 1$}.
$$
Similar to Claim~\ref{claim:aditiva}, we have that 
$$
\phi_n(\omega, \hat{x}) = \log \|A^n(\omega)x\| \quad  \text{for every $(\omega, \hat{x}) \in \Omega \times P(\mathbb{R}^m)$}.
$$
Moreover, since $\bar{\mu} = \mathbb{P} \times \mu$ is $F$-ergodic and $\phi_1^+ \leq \log^+ \|A\| \in L^1(\bar{\mu})$, by Kingman's subadditive ergodic theorem, we conclude~\eqref{eq:phi-KF}. 

To prove~\eqref{eq:L-mu}, consider $L$ as the vector subspace of $\mathbb{R}^m$ spanned by the support of $\mu$. Equivalently, $L$ is the minimal subspace for which $\mu(L)=1$. Since $\mu$ is $f_A$-stationary,
$$
1 = \mu(L) = \int (f_A)_\omega \mu(L) \, d\mathbb{P} = \int \mu(A(\omega)^{-1}L) \, d\mathbb{P}.
$$
Hence $\mu(A(\omega)^{-1}L) = 1$ for $\mathbb{P}$-a.e.~$\omega \in \Omega$. But then $A(\omega)^{-1}L \supset L$ and so $A(\omega)L = L$ almost surely. 

On the other hand, by Fubini's theorem and~\eqref{eq:phi-KF}, there exists a Borel set
\(G\subset P(\mathbb{R}^m)\) with \(\mu(G)=1\) such that for every \(\hat{x}\in G\),
\[
\lim_{n\to\infty}\frac1n\log \|A^n(\omega)x\|
=
\int \phi_A\,d\mu
\qquad\text{for $\mathbb{P}$-a.e.~$\omega\in \Omega$}.
\]
Since \(G\cap \operatorname{supp}\mu\) is dense in \(\operatorname{supp}\mu\), and the latter spans \(L\), we can choose
projective directions \(\hat v_1,\dots,\hat v_k\in G\cap \operatorname{supp}\mu\) such that the corresponding unit vectors
\(v_1,\dots,v_k\) form a basis of \(L\).

Because \(L\) is finite-dimensional, there exists \(C>0\) such that for every linear map \(M:L\to L\),
\[
C^{-1}\|M\arrowvert_L\|
\leq
\max_{1\leq i\leq k}\|Mv_i\|
\leq
\|M\arrowvert_L\|.
\]
Moreover, since \(\log^+\|A(\omega)|_L\|\leq \log^+\|A(\omega)\|\in L^1(\mathbb{P})\), the Furstenberg-Kesten theorem applies to the restricted cocycle \(A|_L\). Hence, there exists a full \(\mathbb{P}\)-measure set \(\Omega_0\subset \Omega\) such that for every \(\omega\in \Omega_0\),
\[
\lim_{n\to\infty}\frac1n\log \|A^n(\omega)v_i\|
=
\int \phi_A\,d\mu
\quad\text{for all }i=1,\dots,k,
\]
and
\[
\lambda_1(A|_L)=\lim_{n\to\infty}\frac1n\log \big\|A^n(\omega)\big\arrowvert_L\big\|.
\]
Fix \(\omega\in \Omega_0\). Using the norm equivalence above, and since the maximum of finitely many sequences having the same limit has the same limit again, we obtain
\[
\lambda_1(A|_L)=\lim_{n\to\infty}\frac1n\log \big\|A^n(\omega)\big\arrowvert_L\big\|
= \lim_{n\to\infty}\frac1n\log \max_{1\leq i\leq k}\|A^n(\omega)v_i\|
= 
\int \phi_A\,d\mu.
\]
This proves~\eqref{eq:L-mu}.
\end{proof}

Finally, we obtain the announced characterization: 

\begin{prop} \label{prop:quasi-irreducible}
    Let $A \in \LP$ 
    be a locally constant linear cocycle. Then, $A$ is quasi-irreducible if and only if $\int \phi_A \, d\mu$ takes the same value for every $f_A$-stationary measure $\mu$, namely $\int \phi_A \, d\mu = \lambda_1(A)$. 
\end{prop}

\begin{proof}  
First, we show that $f_A$ is quasi-irreducible provided $\int \phi_A \, d\mu$ takes the same value for every $f_A$-stationary measure $\mu$. To see this, assume that there exists an $A$-invariant subspace $L$ of $\mathbb{R}^m$. Otherwise, $A$ is irreducible and immediately quasi-irreducible. Now, as $L$ is $A$-invariant, we can consider the cocycle given by the restriction $A|_{L}$ and apply Theorem~\ref{cor:Furstenberg}. This implies 
$$
\lambda_{1}(A|_L) \eqdef \lim_{n\to \infty} \frac{1}{n} \int \log \big\|A^n(\omega)\big\arrowvert_{L}\,\big\| \,d\mathbb{P} = \max\big\{\int \phi_{A|_{L}} \, d\mu :  \mu \in \mathcal{I}_L \, \text{ergodic} \big\} 
$$
where $\mathcal{I}_L$ is the set of stationary measures of the projective random map associated with $A|_{L}$. This projective random map coincides with the restriction $(f_A)|_{L}: \Omega \times P(L) \to P(L)$ of the projective random map $f_A$ associated with $A$. Hence, every $\mu \in \mathcal{I}_L$ is an $f_A$-stationary measure. Moreover, since we also have $\phi_{A|_{L}} = (\phi_A)|_{L}$, by assumption, we find that $\int \phi_{A|_{L}} \, d\mu = \int \phi_A \, d\mu$ takes the same value for every $\mu \in \mathcal{I}_L$. This value needs to be $\lambda_1(A)$ because of~\eqref{eq:Furstenberg-formula}. Thus, we get that $\lambda_1(A|_L) = \lambda_1(A)$ and therefore $A$ is quasi-irreducible. 

Now, we prove the converse. To do this, assume that there exist two ergodic $f_A$-stationary measures $\nu$ and $\mu$ such that $\int \phi_A \, d\nu > \int \phi_A \, d\mu$. By Lemma~\ref{prop:lambdaLint}, we find an $A$-invariant subspace $L$ of $\mathbb{R}^m$ such that $\lambda_1(A|_L) = \int \phi_A \, d\mu$. Since by Theorem~\ref{cor:Furstenberg}, $\lambda_1(A) \geq  \int \phi_A \, d\nu > \int \phi_A \, d\mu = \lambda_1(A|_L)$, we get that $A$ is not quasi-irreducible, concluding the proof. 
\end{proof}

{
If \(f_A\) is uniquely ergodic, then \(\int \phi_A\,d\mu\) takes the same value for every
\(f_A\)-stationary measure \(\mu\), and Proposition~\ref{prop:quasi-irreducible} yields that \(A\) is
quasi-irreducible. Conversely, under the assumption
\(\lambda_1(A)>\lambda_2(A)\), Aoun and Guivarc'h~\cite[Thm.~2.4(a)]{aoun2020random} show that there exists a unique \(f_A\)-stationary measure $\mu$ with $\mu(E)=0$, where \(E\) is the maximal \(A\)-invariant subspace on which the
top Lyapunov exponent is strictly smaller than \(\lambda_1(A)\); in our notation, this is precisely
the equator (see Definition~\ref{def:equator}). In 
Proposition~\ref{thm:Aoun-a}, we revisit and simplify the proof of such a result. 
Thus, if \(A\) is quasi-irreducible, then \(E\) is 
trivial and \(f_A\) is uniquely ergodic. }


\section{The equator and localization of maximizing stationary measures}  We first introduce the notion of a maximizing measure,  which plays a central role in understanding the dynamics of linear cocycles:

\begin{defi} \index{Lyapunov exponents!maximizing stationary measure}
An $f_A$-stationary measure $\mu$ on $P(\mathbb{R}^m)$ is \emph{maximizing} if $$\lambda_1(A) = \int \phi_A \, d\mu.$$
\end{defi}

Next, we revisit and deepen the concept of the equator, introduced earlier in Definition~\ref{def:equator}.
Consider any subspace \( L \) of \( \mathbb{R}^m \) that is invariant under \( A \). The maximal Lyapunov exponent \( \lambda_1(A|_L) \) of \( A \) restricted to \( L \), as introduced in~\eqref{eq:reducible-Lyapunov-exponent}, is well defined. Furthermore, for two \( A \)-invariant subspaces \( L_1 \) and \( L_2 \), the sum \( L_1 + L_2 \) is also \( A \)-invariant, and the maximal Lyapunov exponent satisfies 
\(
\lambda_1(A|_{L_1 + L_2}) = \max\{\lambda_1(A|_{L_1}), \lambda_1(A|_{L_2})\}.
\)
This implies a key property: if \( \lambda_1(A|_{L_i}) < \lambda_1(A) \) for \( i = 1, 2 \), then the combined subspace \( L_1 + L_2 \) also satisfies \( \lambda_1(A|_{L_1 + L_2}) < \lambda_1(A) \). Now, suppose \( A \) is not quasi-irreducible. In this case, there exists at least one \( A \)-invariant subspace \( L \) with \( \lambda_1(A|_L) < \lambda_1(A) \). By iteratively considering the sum of such subspaces, the above property ensures that there is a unique maximal \( A \)-invariant subspace \( E \) such that \( \lambda_1(A|_E) < \lambda_1(A) \). This subspace \( E \) is called the \emph{equator} of \( A \).

\begin{cor} \label{cor:localization}
    Let $A :\Omega \to \mathrm{GL}(m)$ 
    be a locally constant linear cocycle 
    such that $\log^+ \|A\|$ is $\mathbb{P}$-integrable. Assume that the equator $E$ is non-trivial and consider an $f_A$-stationary measure $\mu$. Then $\mu$ is maximizing if and only if $\mu(E)=0$. 
\end{cor}

\begin{proof}
First, note that since any $f_A$-stationary measure is a convex combination of ergodic ones, it suffices to prove the corollary for ergodic $f_A$-stationary measures.

If an ergodic $f_A$-stationary measure $\mu$ satisfies that $\mu(E)=0$, then it is maximizing. Otherwise, by~\eqref{eq:L-mu}, we find a vector space $L$ with $\mu(L)=1$ such that $\lambda_1(A|_L)=\int \phi_A \, d\mu < \lambda_1(A)$. This implies that $L \subset E$ (because the equator is the maximal subspace such that $\lambda_1(A|_E)<\lambda_1(A)$). However, $\mu(E)=0$ and consequently $\mu(L)$ cannot be one.

Conversely, let $\mu$ be an ergodic maximizing $f_A$-stationary measure and assume  $\mu(E)>0$. Moreover, by item (iii) in Corollary~\ref{cor:kigman-fusternberg} and since $\mu$ is a maximizing ergodic measure, we have $\lambda(\hat{x})=\lambda_1(A)$ for $\mu$-a.e.~$\hat{x} \in P(\mathbb{R}^m)$. Then, we can find $\hat{x}$ in $E$ satisfying this condition. For this point, we have $\lambda_1(A)=\lambda(\hat{x}) \leq \lambda_1(A|_E)<\lambda_1(A)$, which is a contradiction.
\end{proof}

\begin{rem} \label{rem:maximazing}
If the equator $E$ is trivial, then $A$ is quasi-irreducible, and thus $\mu(E) = 0$. Consequently, Proposition~\ref{prop:quasi-irreducible} and Corollary~\ref{cor:localization} show that maximizing measures are characterized by their property of assigning zero measure to the equator, regardless of whether it is trivial or not.
\end{rem}

Recall the notion of almost-irreducible cocycle in Definition~\ref{def:almost}. 

\begin{prop} \label{rem:almost}
Let $A \in \LP$ be a locally constant linear cocycle, and denote its equator by $E$. Then, $A$ is almost-irreducible if and only if there exists a compact set $X \subset P(\mathbb{R}^m)$ satisfying:
\begin{enumerate}[leftmargin=1cm]
    \item[(i)] 
    there exists an open set $V \subset X$ such that $A(\omega)X \subset V$ for $\mathbb{P}$-a.e.~$\omega \in \Omega$;
    \item[(ii)] 
    every $f_A$-stationary measure $\mu$ supported on $X$ is maximizing. 
\end{enumerate}
\end{prop}

\begin{proof}
Let $X$ be the compact set from the definition of an almost-irreducible cocycle. Since $X \subset P(\mathbb{R}^m) \setminus E$, every $f_A$-stationary measure $\mu$ supported on $X$ satisfies $\mu(E) = 0$. By Corollary~\ref{cor:localization}, $\mu$ must therefore be maximizing, proving~(ii). Additionally, condition (i) follows directly from the definition of almost-irreducibility. This completes the proof of the forward implication.

For the reverse implication, assume that conditions (i) and (ii) hold. We need to show that $X \subset P(\mathbb{R}^m) \setminus E$. Suppose, for contradiction, that $X \cap E \neq \emptyset$. In this case, $X \cap E$ would form a non-empty, {forward} $A$-invariant compact set. Consequently, there would exist an $f_A$-stationary measure $\mu$ supported on this set. However, such a measure would not be maximizing since it is supported on $E$, contradicting condition (ii), which requires all $f_A$-stationary measures supported on $X$ to be maximizing. Thus, $X \subset P(\mathbb{R}^m) \setminus E$, completing the proof.
\end{proof}

\section{Strong law of large numbers} \index{strong law of large numbers}
Here, we revisit~\cite[Prop.~3.8]{furstenberg1983random} to understand the asymptotic behavior of the vector norms $\|A^n(\omega)x\|$ for unit vectors $x$ far from the equator.

\begin{prop} \label{prop:far-equator}
    Let $A \in \LP$ 
    be a locally constant linear cocycle. If the equator $E$ is non-trivial, then for every unit vector $x \in \mathbb{R}^m \setminus E$,  
    \[
    \lambda_1(A) = \lim_{n \to \infty} \frac{1}{n} \int \log \|A^n(\omega)x\| \, d\mathbb{P} = \lim_{n \to \infty} \frac{1}{n} \log \|A^n(\tilde{\omega})x\| \quad \text{for $\mathbb{P}$-a.e.~$\tilde{\omega} \in \Omega$}.
    \]
\end{prop}

\begin{proof}
According to~\cite[Lem.~3.6]{furstenberg1983random}, for any $A$-invariant subspace $L$ of $\mathbb{R}^m$, 
$$
\lambda_1(A)=\max\big\{\lambda_{1}(A|_L),\,\lambda_{1}(A|_{\mathbb{R}^m/L})\big\}$$ 
{where \(A|_L:\Omega\to \mathrm{GL}(L)\) denotes the restriction cocycle, and
\[
A|_{\mathbb{R}^m/L}:\Omega\to \mathrm{GL}(\mathbb{R}^m/L),
\qquad
A|_{\mathbb{R}^m/L}(\omega)[v]\eqdef [A(\omega)v],
\]
denotes the cocycle induced by \(A\) on the quotient space \(\mathbb{R}^m/L\). Here \([v]\) stands for the class of \(v\in \mathbb{R}^m\) modulo \(L\).}
In particular, since $E$ is the equator of $A$, we have $\lambda_1(A|_E)<\lambda_1(A)$ and thus $\lambda_1(A)=\lambda_{1}(A|_{\mathbb{R}^m/ E})$. We can identify the quotient space $\mathbb{R}^m/E$ with the orthogonal subspace $E^\bot$ to $E$ and the action of $A$ on $\mathbb{R}^m/E$ with the action of $A^\bot$ on $E^\bot$ where 
\[
A=
\begin{pmatrix}
A|_E &  B \\ 0 & A^\bot 
\end{pmatrix} 
\quad \text{with $A|_E: \Omega \to  \mathrm{GL}(E)$ and $A^\bot :\Omega \to  \mathrm{GL}(E^\bot)$}.
\]
Moreover, we have $\lambda_{1}(A|_{\mathbb{R}^m/E})=\lambda_{1}(A^\bot)$ and thus $\lambda_1(A)=\lambda_{1}(A^\bot)$. Consider the projective random map $f_{A^\bot}: \Omega \times P(E^\bot) \to P(E^\bot)$. 

\begin{claim} \label{claim:quasi-irreducible}
     $\lambda_1(A)=\int \phi_{A^\bot} \, d\mu$ for every $f_{A^\bot}$-stationary measure $\mu$. In particular,  ${A^\bot}$ is quasi-irreducible.
\end{claim}

\begin{proof}
By contradiction, we assume that the claim does not hold. Since these integrals can never be greater than $\lambda_{1}(A^\bot) = \lambda_1(A)$, we can assume that there exists an $f_{A^\bot}$-stationary measure $\mu$ such that $\int \phi_{A^\bot} \, d\mu < \lambda_1(A)$. Hence, according to Lemma~\ref{prop:lambdaLint}, there is an $A^\bot$ invariant subspace $L$ of $E^\bot$ such that $\lambda_{1}(A^\bot|_L)= \int \phi_{A^\bot} \, d\mu$. In particular, $\lambda_1(A^\bot|_L)<\lambda_1(A)$.  

Consider $E' = E \oplus L$. Notice that $E'$ is $A$ invariant and since $\lambda_{1}(A|_{E'}) \leq \max\{\lambda_1(A|_E), \lambda_{1}(A^\bot|_L)\}$, we get $\lambda_{1}(A|_{E'}) < \lambda_1(A)$. This contradicts the definition of the equator, since $E'$ contains $E$ strictly.    
\end{proof}

By the above claim, the Law of Large Numbers~\eqref{eq:Kifer-Furstenberg} in Theorem~\ref{cor:Furstenberg} applied to  \(A^\bot: \Omega \to \mathrm{GL}(E^\bot)\) implying that, for every non-zero \(v\in E^\bot\),
\begin{equation*}
\lambda_1(A^\bot)
=\lim_{n\to\infty}\frac1n\int \log \|(A^\bot)^n(\omega)v\|\,d\mathbb{P}
=\lim_{n\to\infty}\frac1n\log \|(A^\bot)^n(\omega)v\|
\end{equation*}
for $\mathbb{P}$-a.e.~$\omega\in\Omega$. 
Since \(\lambda_1(A)=\lambda_1(A^\bot)\), it remains to compare \(A^n(\omega)x\) with \((A^\bot)^n(\omega)v\). Let \(x\in \mathbb{R}^m\setminus E\) be a unit vector and write
\[
x=u+v,
\qquad u\in E,\ \ v\in E^\bot.
\]
Because \(x\notin E\), we have \(v\neq 0\). Moreover, from the block decomposition of \(A\), the lower-right block of \(A^n(\omega)\) is \((A^\bot)^n(\omega)\), and therefore
\[
A^n(\omega)x=\bigl(*,\,(A^\bot)^n(\omega)v\bigr)
\qquad\text{for every }n\ge1.
\]
In particular,
\[
\|(A^\bot)^n(\omega)v\|\le \|A^n(\omega)x\|\le \|A^n(\omega)\|.
\]
Thus, both pointwise and after integration, we have
\[
\frac1n\log \|(A^\bot)^n(\omega)v\|
\le
\frac1n\log \|A^n(\omega)x\|
\le
\frac1n\log \|A^n(\omega)\|
\]
and
\[
\frac1n\int \log \|(A^\bot)^n(\omega)v\|\,d\mathbb{P}
\le
\frac1n\int \log \|A^n(\omega)x\|\,d\mathbb{P}
\le
\frac1n\int \log \|A^n(\omega)\|\,d\mathbb{P}.
\]
The left-hand side converges to \(\lambda_1(A^\bot)=\lambda_1(A)\) by the previous paragraph, while the right-hand side converges to \(\lambda_1(A)\) by~\eqref{eq:limsup4} and the definition of \(\lambda_1(A)\). Therefore, 
\[
\lambda_1(A)=\lim_{n\to\infty}\frac1n\log \|A^n(\omega)x\|
\quad\text{for $\mathbb{P}$-a.e.~$\omega\in\Omega$},
\]
and also
\[
\lambda_1(A)=\lim_{n\to\infty}\frac1n\int \log \|A^n(\omega)x\|\,d\mathbb{P}.
\]
This completes the proof.
\end{proof}

As noted earlier in  Remark~\ref{rem:maximazing}, we can write a unified result for the asymptotic behavior of the iterations of $\|A^n(\omega)x\|$ as follows:  
\begin{rem} \label{rem:law-large-number}  If the equator $E$ is trivial, $A$ is quasi-irreducible and from Proposition~\ref{prop:quasi-irreducible}, all the $f_A$-stationary measure are maximizing. Therefore,~\eqref{eq:Kifer-Furstenberg} and Proposition~\ref{prop:far-equator} show that 
$$
\lambda_1(A)=\lim_{n\to\infty} \frac{1}{n} \int \log \|A^n(\omega)x\| \, d\mathbb{P}= \lim_{n \to \infty} \frac{1}{n} \log \|A^n(\tilde{\omega})x\|
$$
for every $\hat{x}\in P(\mathbb{R}^m)\setminus E$ and $\mathbb{P}$-a.e.~$\tilde{\omega} \in \Omega$
whether the equator $E$ is trivial or not.  
\end{rem}

\section{Angular distance} \label{ss:angular}
Let $d$ be the natural angular distance in $P(\mathbb{R}^m)$, i.e.,
$$
d(\hat{x},\hat{y}) = |\sin(\theta)|, \quad \text{for $\hat{x}, \hat{y} \in P(\mathbb{R}^m)$ and $\theta = \mathrm{angle}(\hat{x}, \hat{y})$.}
\index{metric structures!00angular distance@angular distance on \(P(\mathbb R^m)\)}
$$
Note that this distance does not depend on how the angle $\theta$ between $\hat{x}$ and $\hat{y}$ is measured. That is, $\theta$ can be calculated as the angle between any two non-zero vectors $v$ and $w$ in $\hat{x}$ and $\hat{y}$ respectively.
Thus, taking unit vectors $x$ and $y$ in $\mathbb{R}^m$ representing $\hat{x}$ and $\hat{y}$ respectively, since $\cos(\theta)= \langle x, y \rangle$, one can write 
$$
d(\hat{x}, \hat{y}) = \left(1 - \langle x, y \rangle^2 \right)^{\frac{1}{2}} = \|x \exterior{} y\|.
$$
Also, since $\|x-y\|^2 = 2(1-\cos\theta) \ge 1-\cos^2\theta = d(\hat{x},\hat{y})^2$, it follows that
\[
d(\hat{x},\hat{y})\le \|x-y\|.
\]
Moreover, if \(x\) and \(y\) are chosen such that \(\theta=\mathrm{angle}(\hat x, \hat y)\in[0,\frac{\pi}{2}]\), then
\[
d(\hat{x},\hat{y})=\sin\theta \ge \frac{2}{\pi}\theta \ge \frac{2}{\pi}\|x-y\|. 
\]

A useful representation of the angular distance is given by the orthogonal projection. To do this, for each $\hat{x}\in P(\mathbb{R}^m)$, we define the projection on the orthogonal complement of $\hat{x}$ as the linear map 
$$
 \Pi_{\hat{x}}(v)= v-\langle v,x\rangle x  \quad \text{for $v\in\mathbb{R}^m.$}
$$ 
Then, $$d(\hat{x},\hat{y})=\left\|\Pi_{\hat x}(y)\right\|=\left\|\Pi_{\hat x}-\Pi_{\hat y}\right\|.$$
Now, we can calculate the locally Lipschitz constant for a projective random map $f_A$.

\begin{lem}
    \label{lem:lip-fA}
    Let $A\in \LP$ be a locally constant linear cocycle. Then,
    $$
     L(f_A)^n_\omega(\hat{x})\leq \frac{\|\exterior{2} A^n(\omega)\|}{\|A^n(\omega)x\|^2} \quad \text{for any $(\omega,\hat{x}) \in \Omega\times P(\mathbb{R}^m)$ and $n\geq 1$.}
    $$
    Moreover, 
    \begin{equation} \label{eq:limitation-Lipchitz}
    \mathrm{Lip}((f_A)_\omega) \leq \|A(\omega)\|^2 \|A^{-1}(\omega)\|^2  \leq \max\big\{\|A(\omega)\|, \, \|A^{-1}(\omega)\|\big\}^4.
\end{equation}
\end{lem}

\begin{proof}
Given a matrix $M\in \mathrm{GL}(m)$ and unit vectors $y,z\in\mathbb{R}^m$, we have 
$$
\frac{d(M\hat{y},M\hat{z})}{d(\hat{y},\hat{z})} = \frac{\|My\exterior{} Mz\|}{\|My\|\,\|Mz\|\, \|y\exterior{} z\|} \leq \frac{\|\exterior{2} M\|}{\|My\|\,\|Mz\| } \quad \text{for $\hat{y},\hat{z}\in P(\mathbb{R}^m)$}.
$$
Then, 
\begin{equation} \label{eq:LfA}
    L(f_A)^n_\omega(\hat{x})
= \limsup_{\hat{y},\hat{z}\to \hat{x}} \frac{d(A^n(\omega)\hat{y},A^n(\omega)\hat{z})}{d(\hat{y},\hat{z})} 
\leq  \frac{\|\exterior{2} A^n(\omega)\|}{\|A^n(\omega)x\|^2}
\end{equation}
for any $(\omega,\hat{x}) \in \Omega\times P(\mathbb{R}^m)$ and $n\geq 1$.
Moreover, since $\|\exterior{2} M\| \leq \|M\|^2$ and $\|Mx\|\geq \|M^{-1}\|^{-1}$ for any unitary $x\in \mathbb{R}^m$ and matrix $M$ in $\mathrm{GL}(m)$,  from~\eqref{eq:LfA}  we find~\eqref{eq:limitation-Lipchitz}.
\end{proof}

{Lemma~\ref{lem:lip-fA} gives pointwise upper bounds for the local Lipschitz constants of the projective action. The next lemma complements this by showing how these local quantities control the distortion between two projective directions. This estimate is essentially in~\cite[Prop.~4]{baraviera2019approximating}.

\begin{lem}\label{lem:projective-distance-local}
Let \(A\in \LP\) and \(\alpha>0\). Then, for every \(n\geq 1\), \(\omega\in \Omega\), and distinct
\(\hat{x},\hat{y}\in P(\mathbb{R}^m)\),
\[
\frac{d\big(A^n(\omega)\hat{x},A^n(\omega)\hat{y}\big)^\alpha}{d(\hat{x},\hat{y})^\alpha}
\leq \frac{1}{2}\Big( L(f_A)^n_\omega(\hat{x})^\alpha
+ L(f_A)^n_\omega(\hat{y})^\alpha \Big).
\]
\end{lem}

\begin{proof}
Fix \(n\geq 1\), \(\omega\in \Omega\), and write \(M=A^n(\omega)\). Let \(g=\widehat{M}=(f_A)^n_\omega\).
Take unit representatives \(x,y\in \mathbb{R}^m\) of \(\hat{x},\hat{y}\). By the definition of the angular distance,
\[
\frac{d(g(\hat{x}),g(\hat{y}))^\alpha}{d(\hat{x},\hat{y})^\alpha}
=
\left(\frac{\|Mx\exterior{} My\|}{\|x\exterior{} y\|}\right)^\alpha
\frac{1}{\|Mx\|^\alpha \|My\|^\alpha}.
\]
Using \(ab\leq \frac12(a^2+b^2)\) with \(a=\|Mx\|^{-\alpha}\) and \(b=\|My\|^{-\alpha}\), we get
\[
\frac{d(g(\hat{x}),g(\hat{y}))^\alpha}{d(\hat{x},\hat{y})^\alpha}
\leq \frac12\left[
\left(\frac{\|Mx\exterior{} My\|}{\|x\exterior{} y\|\,\|Mx\|^2}\right)^\alpha
+
\left(\frac{\|Mx\exterior{} My\|}{\|x\exterior{} y\|\,\|My\|^2}\right)^\alpha
\right].
\]
Now set
\[
u_x\eqdef \frac{\Pi_{\hat{x}}(y)}{\|\Pi_{\hat{x}}(y)\|}\in T_{\hat{x}}P(\mathbb{R}^m),
\qquad
u_y\eqdef \frac{\Pi_{\hat{y}}(x)}{\|\Pi_{\hat{y}}(x)\|}\in T_{\hat{y}}P(\mathbb{R}^m).
\]
Since \(\|\Pi_{\hat{x}}(y)\|=\|x\exterior{} y\|\), the derivative formula for the projective action gives
\[
\|Dg(\hat{x})u_x\|
=
\frac{\|Mx\exterior{} My\|}{\|x\exterior{} y\|\,\|Mx\|^2},
\qquad
\|Dg(\hat{y})u_y\|
=
\frac{\|Mx\exterior{} My\|}{\|x\exterior{} y\|\,\|My\|^2}.
\]
Therefore,
\[
\frac{d(g(\hat{x}),g(\hat{y}))^\alpha}{d(\hat{x},\hat{y})^\alpha}
\leq \frac12\Big(\|Dg(\hat{x})\|^\alpha+\|Dg(\hat{y})\|^\alpha\Big).
\]
Finally, since \(g=(f_A)^n_\omega\) is a smooth map on the projective manifold,
\[
\|Dg(\hat{z})\|=L(f_A)^n_\omega(\hat{z})
\qquad\text{for every }\hat{z}\in P(\mathbb{R}^m),
\]
which concludes the proof.
\end{proof}
}

{
\section{Mostly contracting and unique ergodicity} Here, we study  locally constant cocycles $A \in \LP$ under the assumption that $\lambda_1(A)>\lambda_2(A)$. The following result was first established by Aoun and Guivarc'h~\cite[Theorem~2.4(a)]{aoun2020random}. We provide a different simplified proof. 

\begin{prop} \label{thm:Aoun-a} Let $A: \Omega \to \mathrm{GL}(m)$ be a locally constant linear cocycle and denote by  $E$ the (perhaps trivial) equator of $A$. Assume that
\begin{enumerate}[leftmargin=1cm,label=(\roman*)]
    \item  $\log^+\|A^{\pm1}\|$ 
    is $\mathbb{P}$-integrable, and
    \item  $\lambda_1(A)>\lambda_2(A)$.
    \end{enumerate} 
Then there exists a unique $f_A$-stationary measure such that $\mu(E)=0$. 
\end{prop}

\begin{proof}
Let \(\mathcal I\) denote the set of \(f_A\)-stationary probability measures on \(P(\mathbb{R}^m)\).

We first prove existence. Since \(P(\mathbb{R}^m)\) is compact, \(\mathcal I\neq\emptyset\). By Theorem~\ref{cor:Furstenberg}, there exists an ergodic measure \(\mu\in\mathcal I\) such that
$\int \phi_A\,d\mu=\lambda_1(A)$.
If the equator \(E\) is trivial, then \(\mu(E)=0\) is automatic. If \(E\) is non-trivial, then Corollary~\ref{cor:localization} shows that every maximizing stationary measure satisfies \(\mu(E)=0\). This proves existence.

We now prove uniqueness. Let \(\mu,\nu\in\mathcal I\) satisfy \(\mu(E)=\nu(E)=0\). 

\begin{claim} \label{claim:mean-proximality} For every
\(\hat{x},\hat{y}\in P(\mathbb{R}^m)\setminus E\),
\begin{equation}\label{eq:contraction-outside-equator}
\lim_{n\to\infty}\int d(A^n(\omega)\hat{x},A^n(\omega)\hat{y})\,d\mathbb{P}=0.
\end{equation}
\end{claim}
\begin{proof} If \(x,y\) are unit representatives of \(\hat{x},\hat{y}\), then
\[
d(A^n(\omega)\hat{x},A^n(\omega)\hat{y})
\leq
\frac{\|\exterior{2} A^n(\omega)\|}{\|A^n(\omega)x\|\,\|A^n(\omega)y\|}.
\]
By definition, $\lambda_1(A)+\lambda_2(A)$ is the upper Lyapunov exponent associated with $\exterior{2} A^n(\omega)$. Hence, by the Furstenberg-Kesten theorem (see~\eqref{eq:limsup4}), 
\[
\lim_{n\to\infty}\frac1n\log \|\exterior{2} A^n(\omega)\|
=\lambda_1(A)+\lambda_2(A)
\qquad\text{for $\mathbb{P}$-a.e.~\(\omega\in\Omega\)}.
\]
Now, Proposition~\ref{prop:far-equator} gives
\[
\lim_{n\to\infty}\frac1n\log \|A^n(\omega)x\|
=
\lim_{n\to\infty}\frac1n\log \|A^n(\omega)y\|
=
\lambda_1(A)
\qquad\text{for $\mathbb{P}$-a.e.~\(\omega\in\Omega\)}.
\]
Therefore, for \(\mathbb{P}\)-a.e.~\(\omega\in\Omega\),
\[
\limsup_{n\to\infty}\frac1n
\log d(A^n(\omega)\hat{x},A^n(\omega)\hat{y})
\leq
\lambda_1(A)+\lambda_2(A)-2\lambda_1(A)
=
\lambda_2(A)-\lambda_1(A)
<0.
\]
Hence $d(A^n(\omega)\hat{x},A^n(\omega)\hat{y})\to 0$ for $\mathbb{P}$-a.e.~\(\omega\in\Omega\).
Therefore, since the metric \(0\leq d(\cdot,\cdot)\leq 1\), the dominated convergence theorem yields~\eqref{eq:contraction-outside-equator}.
\end{proof}

Let \(\varphi:P(\mathbb{R}^m)\to \mathbb{R}\) be a Lipschitz function, and let \(P\) be the annealed Koopman operator associated with \(f_A\). 
Then, for every \(n\geq 1\) and \(\hat{x},\hat{y}\in P(\mathbb{R}^m)\setminus E\), we have
\begin{align*}
|P^n\varphi(\hat{x})-P^n\varphi(\hat{y})|
&
\leq
\mathrm{Lip}(\varphi)\int d(A^n(\omega)\hat{x},A^n(\omega)\hat{y})\,d\mathbb{P}.
\end{align*}
By~Claim~\ref{claim:mean-proximality}, the last term converges to \(0\) as \(n\to\infty\). Therefore,
\[
P^n\varphi(\hat{x})-P^n\varphi(\hat{y})\longrightarrow 0
\qquad\text{for every }\hat{x},\hat{y}\in P(\mathbb{R}^m)\setminus E.
\]

Fix \(\hat{x}_0\in P(\mathbb{R}^m)\setminus E\). Since \(\mu(E)=0\), we may write
\[
\int P^n\varphi\,d\mu - P^n\varphi(\hat{x}_0)
=
\int_{P(\mathbb{R}^m)\setminus E}
\Big(P^n\varphi(\hat{x})-P^n\varphi(\hat{x}_0)\Big)\,d\mu(\hat{x}).
\]
Since for each \(\hat{x}\in P(\mathbb{R}^m)\setminus E\), the above last integrand converges to \(0\) and $|P^n\varphi(\hat{x})-P^n\varphi(\hat{x}_0)| \leq 2\|\varphi\|_\infty$,
by the dominated convergence theorem,
\[
\int P^n\varphi\,d\mu - P^n\varphi(\hat{x}_0)\longrightarrow 0.
\]
Exactly the same argument, using \(\nu(E)=0\), gives
$\int P^n\varphi\,d\nu - P^n\varphi(\hat{x}_0)\to 0$.

On the other hand, since \(\mu\) and \(\nu\) are \(f_A\)-stationary, 
\begin{align*}
    \int \varphi\,d\mu-\int \varphi\,d\nu &= \int P^n\varphi\,d\mu - \int P^n\varphi\,d\nu \\ &= \big( \int P^n\varphi\,d\mu - P^n\varphi(\hat{x}_0) \big) - \big(\int P^n\varphi\,d\nu - P^n\varphi(\hat{x}_0) \big) \longrightarrow 0.
\end{align*}
We obtain $\int \varphi\,d\mu=\int \varphi\,d\nu$.
Hence \(\mu=\nu\), because Lipschitz functions are uniformly dense in \(C(P(\mathbb{R}^m))\). 
\end{proof}

The following proposition is the counterpart, in our notation, of~\cite[Thm.~2.4(b) and Cor.~2.10]{aoun2020random}. For completeness, we include a proof based on the variational characterization of \(\lambda_1(A)\) established above. Recall the notion of pinnacle in Definition~\ref{def:equator} and that $\Gamma_A$ denotes the semigroup in $\mathrm{GL}(m)$ generated by the support of the distribution~$\nu_A = A_*\mathbb{P}$. 
}
\begin{prop} \label{claim:F-invariant} 
Let \(A\in \LP\) be a cocycle with equator $E$ (perhaps trivial)  such that \(\lambda_1(A)>\lambda_2(A)\). Denote by $S$ the support of the unique \(f_A\)-stationary measure $\mu$ satisfying \(\mu(E)=0\). Then, the following conditions are equivalent: 
\begin{enumerate}[leftmargin=0.8cm]
    \item $S \subset P(\mathbb{R}^m) \setminus E$,
    \item $(f_A)|_S$ is minimal,
    \item there exists $\hat{x}$ such that the closure of its $\Gamma_A$-orbit is contained in $P(\mathbb{R}^m) \setminus E$,
    \item there exists an $f_A$-invariant compact set $X \subset P(\mathbb{R}^m) \setminus E$. 
\end{enumerate}
Moreover, the pinnacle subspace $F$ of $A$ coincides with the subspace of $\mathbb{R}^m$ spanned by $S$. Furthermore, for any set $X$ satisfying (iv), the subspace $L = \langle x \in \mathbb{R}^m : \hat{x} \in X \rangle$ is $A$-invariant, and $F \subset L$.
\end{prop}

\begin{proof}
{We begin with two elementary facts.

\smallskip

\noindent\emph{Fact 1.} If \(Y\subset P(\mathbb{R}^m)\) is a non-empty compact \(f_A\)-invariant set, then there exists an \(f_A\)-stationary probability measure supported on \(Y\).

\smallskip

\noindent\emph{Fact 2.} If \(Y\subset P(\mathbb{R}^m)\) is a closed \(f_A\)-invariant set, then \(Y\) is \(\Gamma_A\)-invariant. Indeed, \(A(\omega)Y\subset Y\) for \(\mathbb{P}\)-a.e.~\(\omega\). Hence, if \(g\in \operatorname{supp}\nu_A\), we can choose \(\omega_n\) in a full \(\mathbb{P}\)-measure set such that \(A(\omega_n)\to g\). Since the projective action is continuous and \(Y\) is closed, we obtain \(gY\subset Y\). Therefore \(Y\) is invariant under the semigroup \(\Gamma_A\) generated by \(\operatorname{supp}\nu_A\).

\smallskip

We now prove the equivalence of \emph{(i)}--\emph{(iv)}.

\smallskip

\noindent\emph{\((i)\Rightarrow(ii)\).}
Assume that \(S\subset P(\mathbb{R}^m)\setminus E\), and let \(Y\subset S\) be a non-empty closed \(f_A\)-invariant set. By Fact~1, there exists an \(f_A\)-stationary probability measure \(\eta\) supported on \(Y\). Since \(Y\cap E=\emptyset\), we have \(\eta(E)=0\). By the uniqueness of the \(f_A\)-stationary measure assigning zero mass to \(E\), it follows that \(\eta=\mu\). Hence
$S=\operatorname{supp}\mu=\operatorname{supp}\eta\subset Y$.
Since \(Y\subset S\), we conclude that \(Y=S\). Thus \((f_A)|_S\) is minimal.

\smallskip

\noindent\emph{\((ii)\Rightarrow(i)\).}
Assume that \((f_A)|_S\) is minimal. Since \(E\) is a closed \(f_A\)-invariant set, the intersection \(S\cap E\) is a closed \(f_A\)-invariant subset of \(S\). If \(S\cap E\neq\emptyset\), minimality gives \(S\subset E\), hence \(\mu(E)=1\), contradicting the choice of \(\mu\). Therefore \(S\cap E=\emptyset\), that is,
$S\subset P(\mathbb{R}^m)\setminus E$.

\smallskip

\noindent\emph{\((i)\Rightarrow(iii)\).}
Fix \(\hat{x}\in S\), and let
$Y\eqdef \overline{\Gamma_A\cdot \hat{x}}$.
Since \(S\) is closed and \(f_A\)-invariant, Fact~2 implies that \(S\) is \(\Gamma_A\)-invariant. As \(\hat{x}\in S\), it follows that $\Gamma_A\cdot \hat{x}\subset S$, and therefore \(Y\subset S\subset P(\mathbb{R}^m)\setminus E\) by $(i)$.

\smallskip

\noindent\emph{\((iii)\Rightarrow(iv)\).}
Set $X\eqdef \overline{\Gamma_A\cdot \hat{x}}$.
Then \(X\) is compact, \(\Gamma_A\)-invariant, and contained in \(P(\mathbb{R}^m)\setminus E\). Since \(A(\omega)\in \operatorname{supp}\nu_A\subset \Gamma_A\) for \(\mathbb{P}\)-a.e.~\(\omega\), the set \(X\) is \(f_A\)-invariant. This proves \emph{(iv)}.

\smallskip

\noindent\emph{\((iv)\Rightarrow(i)\).}
Assume that there exists a compact \(f_A\)-invariant set \(X\subset P(\mathbb{R}^m)\setminus E\). By Fact~1, there exists an \(f_A\)-stationary probability measure \(\eta\) supported on \(X\). Since \(X\cap E=\emptyset\), we have \(\eta(E)=0\), and by uniqueness \(\eta=\mu\). Hence
$S=\operatorname{supp}\mu=\operatorname{supp}\eta\subset X\subset P(\mathbb{R}^m)\setminus E$,
which proves \emph{(i)}.

\smallskip

We now prove the final claims. Let $L_S\eqdef \langle x\in \mathbb{R}^m:\hat{x}\in S\rangle$.
Since \(S\) is compact and \(f_A\)-invariant, the subspace \(L_S\) is \(A\)-invariant. We first prove that \(F\subset L_S\). Since \(\mu\) is supported on \(P(L_S)\), it is an \(f_{A|_{L_S}}\)-stationary measure. Moreover, \(\mu(E)=0\), so \(\mu\) is maximizing; indeed, this is automatic if \(E=\{0\}\), while if \(E\neq\{0\}\), it follows from Corollary~\ref{cor:localization}. Hence
\[
\int \phi_A\,d\mu=\lambda_1(A).
\]
Thus, since \(\phi_{A|_{L_S}}=\phi_A\) on \(P({L_S})\), applying Theorem~\ref{cor:Furstenberg}  to   \(A|_{L_S}\), we obtain
\[
\lambda_1(A|_{L_S})\ge \int \phi_{A|_{L_S}}\,d\mu =\int \phi_A\,d\mu=\lambda_1(A).
\]
Since always \(\lambda_1(A|_{L_S})\le \lambda_1(A)\), we conclude that
$\lambda_1(A|_{L_S})=\lambda_1(A)$ .
By the definition of the pinnacle subspace, this implies \(F\subset {L_S}\).

We now prove the converse inclusion \({L_S}\subset F\). Let \(G\subset \mathbb{R}^m\) be any \(A\)-invariant subspace such that
$\lambda_1(A|_G)=\lambda_1(A)$.

\begin{claim} \(\mu(P(G))=1\). \end{claim} 
\begin{proof} Let \(A|_{\mathbb{R}^m/G}\) be the  cocycle induced by $A$ on quotient space \(\mathbb{R}^m/G\). Since
\(\lambda_1(A)>\lambda_2(A)\) and \(G\) carries the top Lyapunov exponent, it holds $\lambda_1(A|_{\mathbb{R}^m/G})<\lambda_1(A)$.
Choose the quotient norm on \(\mathbb{R}^m/G\) induced by the Euclidean norm, that is, $\|\,[z]\,\|_G \eqdef \inf_{w\in G}\|z-w\|$. Identifying \(\mathbb{R}^m/G\) with the orthogonal complement \(G^\perp\), one has
$\|\,[z]\,\|_G=\|\pi_{G^\perp}(z)\|$,
where \(\pi_{G^\perp}\) denotes the orthogonal projection onto \(G^\perp\). Let \(z\in \mathbb{R}^m\setminus\{0\}\), and write
$z=u+v$ with $u\in G$ and  $v\in G^\perp$.
Then $\|\,[z]\,\|_G=\|v\|$. 
On the other hand, by the definition of the angular distance,
\[
d(\hat z,P(G))
=
\inf_{\hat y\in P(G)} d(\hat z,\hat y) \leq d(\hat z, \hat u) = \left\|\Pi_{\hat u}\!\left(\frac{z}{\|z\|}\right)\right\|.
\]
Observing that \(\Pi_{\hat u}(z/\|z\|)\) is just the component of \(z/\|z\|\) orthogonal to the line \(\mathbb Ru\subset G\), we obtain
\begin{equation} \label{eq:dz}
    d(\hat z,P(G))\le \frac{\|v\|}{\|z\|}
=\frac{\|\,[z]\,\|_G}{\|z\|}.
\end{equation}
Fix \(\hat{x}\in P(\mathbb{R}^m)\setminus E\), and let \(x\) be a unit representative. Applying~\eqref{eq:dz} with \(z=A^n(\omega)x\), and using that the cocycle induced on the quotient satisfies
$[A^n(\omega)x]=(A|_{\mathbb{R}^m/G})^n(\omega)[x]$,
we obtain
\[
d(A^n(\omega)\hat{x},P(G))
\le \frac{\|\,(A|_{\mathbb{R}^m/G})^n(\omega)[x]\,\|_G}{\|A^n(\omega)x\|}.
\]
Taking logarithms, dividing by \(n\), and using Proposition~\ref{prop:far-equator} together with the Furstenberg--Kesten theorem for the quotient cocycle, we get
\[
\limsup_{n\to\infty}\frac1n
\log d(A^n(\omega)\hat{x},P(G))
\le \lambda_1(A|_{\mathbb{R}^m/G})-\lambda_1(A)<0 \quad \text{for \(\mathbb{P}\)-a.e.~\(\omega\).}
\]
Hence
\begin{equation} \label{eq:lim-0-PG}
\lim_{n\to\infty} d(A^n(\omega)\hat{x},P(G))=0 \quad \text{for \(\mathbb{P}\)-a.e.~\(\omega\) and 
for every \(\hat{x}\in P(\mathbb{R}^m)\setminus E\)}.
\end{equation}

Fix \(\varepsilon>0\), and set
$U_\varepsilon \eqdef \{\hat{x}\in P(\mathbb{R}^m): d(\hat{x},P(G))>\varepsilon\}$. Since \(G\) is \(A\)-invariant, the projective subspace \(P(G)\) is \(f_A\)-invariant. Thus, for every \(n\ge1\),
$$
P^n 1_{U_\varepsilon}(\hat{x}) = \int  1_{U_\varepsilon}(A^n(\omega)\hat x)\, d\mathbb{P}
=
\mathbb{P}\big(\{\omega\in \Omega: d(A^n(\omega)\hat{x},P(G))>\varepsilon\}\big).
$$
Moreover, since \(\mu\) is \(f_A\)-stationary, it is \(P^*\)-invariant, and hence
$\mu(U_\varepsilon)= \int  1_{U_\varepsilon}\,d\mu
= \int P^n 1_{U_\varepsilon}\,d\mu$. Using also that \(\mu(E)=0\), we can write
\[
\mu(U_\varepsilon)
=
\int_{P(\mathbb{R}^m)\setminus E}
\mathbb{P}\big(\{\omega\in \Omega: d(A^n(\omega)\hat{x},P(G))>\varepsilon\}\big)\,d\mu(\hat{x}).
\]
By~\eqref{eq:lim-0-PG}, for each \(\hat{x}\in P(\mathbb{R}^m)\setminus E\), the integrand converges to \(0\) as \(n\to\infty\), and it is bounded by \(1\). Hence, by the dominated convergence theorem,
$\mu(U_\varepsilon)=0$. Since this holds for every \(\varepsilon>0\), we conclude that $\mu(P(G))=1$.
\end{proof}

Since \(P(G)\) is closed, every point
\(\hat{x}\notin P(G)\) admits an open neighborhood \(U\subset P(\mathbb{R}^m)\setminus P(G)\). Thus, as  \(\mu(P(G))=1\) by the previous claim, 
$\mu(U)\le \mu(P(\mathbb{R}^m)\setminus P(G))=0$. Hence \(\hat{x}\notin \operatorname{supp}\mu\). Thus
$ S=\operatorname{supp}\mu\subset P(G)$. This  implies \({L_S}\subset G\). Since the above inclusion holds for every \(A\)-invariant subspace \(G\) satisfying
\(\lambda_1(A|_G)=\lambda_1(A)\), it follows from the definition of \(F\) that \({L_S}\subset F\).

Combining both inclusions,} we have that the pinnacle subspace $F$ of $A$ coincides with the subspace $L_S$ spanned by $\mu$. Moreover, since $\mu$ is the unique stationary measure assigning zero measure to the equator, any set $X$ satisfying \emph{(iv)} must contain $S$, and thus $F \subset L = \langle x \in \mathbb{R}^m : \hat{x} \in X \rangle$. Finally, the projectivization $\hat{L}$ of $L$ satisfies $X \subset \hat{L}$, and since $X$ is $f_A$-invariant, it follows that $A(\omega)L \subset L$ for $\mathbb{P}$-a.e.~$\omega\in\Omega$. Since $L$ and $A(\omega)L$ have the same dimension, then $A(\omega)L = L$ for $\mathbb{P}$-a.e.~$\omega \in \Omega$. This completes the last claim of the statement.
\end{proof}

\begin{prop} \label{prop:projective2}
Let $A: \Omega \to \mathrm{GL}(m)$ be a locally constant linear cocycle and denote by  $E$ the (perhaps trivial) equator of $A$. 
{Assume that} 
\begin{enumerate}[leftmargin=1cm,label=(\roman*)]
    \item  $\log^+\|A^{\pm1}\|$ 
    is $\mathbb{P}$-integrable, 
    \item  $\lambda_1(A)>\lambda_2(A)$, and
    \item  there is an $f_A$-invariant compact set $X$ in $P(\mathbb{R}^m)\setminus E$.
\end{enumerate} 
Then the random map $(f_A)|_X:\Omega \times X \to X$ is mostly contracting {and uniquely ergodic}. 
Moreover, under the exponential moment condition~\eqref{eq:exponential-moment-condition} 
it follows that $\mathrm{Lip}((f_A)|_X)^\beta$ is $\mathbb{P}$-integrable for some $\beta>0$ and $(f_A)|_X$ is contracting on average where the exponent $\alpha>0$ in~\eqref{eq:contracting-avarage} can be taken arbitrarily small. Consequently, $(f_A)|_X$ is also proximal and  its associated Koopman operator has spectral gap on $C^\alpha(X)$ for any $\alpha>0$ small enough.  
\end{prop}


\begin{proof} 
By (iii), 
$(f_A)|_X$ is a well-defined random map. 
{Moreover, since $X$ is compact, the set of $(f_A)|_X$-stationary measures $\mu$ is non-empty. Since $\mu$ needs to be supported on $X \subset {P}(\mathbb{R}^m)\setminus E$, we have that $\mu(E)=0$. Hence, Proposition~\ref{thm:Aoun-a} implies that $\mu$ is unique. This proves that $(f_A)|_X$ is uniquely ergodic. }

\begin{claim} \label{claim:mostly-contracting}
$(f_A)|_X$ is mostly contracting.    
\end{claim}
\begin{proof} By Lemma~\ref{lem:lip-fA}, 
$L(f_A)^n_\omega(\hat{x})
\leq  \|\exterior{2} A^n(\omega)\|\cdot \|A^n(\omega)x\|^{-2}$
for any $(\omega,\hat{x}) \in \Omega\times X$ and $n\geq 1$ and hence, 
\begin{equation} \label{eq:log}
     \frac{1}{n}  \log L(f_A)^n_\omega(\hat{x})  
\leq \frac{1}{n} \log \|\exterior{2} A^n(\omega)\| -   \frac{2}{n}  \log \|A^n(\omega)x\|.
\end{equation}
By definition, $\lambda_1(A)+\lambda_2(A)$ is the upper Lyapunov exponent associated with $\exterior{2} A^n(\omega)$, i.e.,   
\begin{equation} \label{eq:L1plusL2}
 \lim_{n\to \infty} \frac{1}{n} \int \log \|\exterior{2} A^n(\omega)\| \, d\mathbb{P} =\lambda_1(A)+\lambda_2(A). 
\end{equation}
Since 
$X\subset P(\mathbb{R}^m)\setminus E$, where  $E$ is the (perhaps trivial) equator of $A$,  by Remark~\ref{rem:law-large-number},  
\begin{equation} \label{eq:Furstenberg}
  \lim_{n\to\infty} \frac{1}{n}  \int \log \|A^n(\omega)x\| \, d\mathbb{P} =\lambda_1(A) \quad \text{for every $\hat{x}\in X$}.
\end{equation}
Thus, integrating and taking upper limit in~\eqref{eq:log}, we get in view of~\eqref{eq:L1plusL2} and~\eqref{eq:Furstenberg} that
\begin{align*}
\lambda(\hat{x}) &\eqdef  
\limsup_{n\to\infty} \frac{1}{n} \int  \log L(f_A)^n_\omega(\hat{x}) \, d\mathbb{P} 
\\ &\leq \lambda_1(A)+\lambda_2(A)-2\lambda_1(A)=\lambda_2(A)-\lambda_1(A)
\end{align*}
for any $\hat{x}\in X$. Since $\lambda_1(A)>\lambda_2(A)$ by assumption, we get that $\lambda(\hat{x})<0$ for every $\hat{x}\in X$ and thus, from Corollary~\ref{cor:equivalencia}, we get that $f_A$ is mostly contracting provided $\log^+ \mathrm{Lip}(f_A)$ will be $\mathbb{P}$-integrable.   
To see this, again by Lemma~\ref{lem:lip-fA}, we have 
$\mathrm{Lip}((f_A)_\omega) \leq \|A(\omega)\|^2 \|A^{-1}(\omega)\|^2$ and hence $$\log^+ \mathrm{Lip}((f_A)_\omega) \leq 2 \log \|A(\omega)\| + 2 \log \|A^{-1}(\omega)\|.$$ Then, since $\log^+ \|A^{\pm 1}\|$ is $\mathbb{P}$-integrable by assumption, we also get the required integrability of $\log^+ \mathrm{Lip}(f_A)$.   \end{proof}

Now we prove the second part of the proposition.  To do this, we additionally assume~\eqref{eq:exponential-moment-condition} in the sequel. By Lemma~\ref{lem:lip-fA} and the exponential moment condition~\eqref{eq:exponential-moment-condition}, $\mathrm{Lip}(f_A)^{\beta/4}$ is $\mathbb{P}$-integrable and~\eqref{eq:integral_condition} holds. 
Thus, $(f_A)|_X$ satisfies the conclusions of Theorem~\ref{thmA}.

We can also easily obtain the following (cf.~\cite[Prop.~1]{page1989regularite}).

\begin{claim} \label{lem:quasi-irreducible-contracting-average}
      $(f_A)|_X$ is contracting on average where  $\alpha>0$ in~\eqref{eq:contracting-avarage} can be taken arbitrarily small.
\end{claim}
\begin{proof}  From  Claim~\ref{claim:mostly-contracting} and the above observation, $(f_A)|_X$ is mostly contracting and the integrability condition~\eqref{eq:integral_condition} in Theorem~\ref{thmA} holds. Thus, 
arguing as in the proof of Proposition~\ref{prop:alpha0} to get~\eqref{eq:qmenor1}, we obtain an integer $n\geq 1$ such that for each $\hat{x}\in P(\mathbb{R}^m)$, there are $\alpha_0(x)>0$ and $q(x)<1$ satisfying
$\int  L(f_A)^n_\omega(\hat{x})^{\alpha_0(x)} \, d\mathbb{P}<q(x)$. By continuity and since \(X\) is compact, we can find 
a finite open cover $\{B_1,\dots,B_k\}$ of $X$, and numbers $\alpha_i>0$ and $0<q_i<1$ for $i=1,\dots,k$ 
such that for each $x\in B_i$
\begin{equation} \label{eq:qmenor2}
    \int  L(f_A)^n_\omega(\hat{x})^{\alpha_i} \, d\mathbb{P}<q_i.
\end{equation}
Set \(\alpha_0=\min\{ \alpha_i: i=1,\dots, k\}\).
Now, for every \(0<\alpha\leq \alpha_0\), denote \(p_i=\alpha_i/\alpha>1\) for \(i=1,\dots,k\), and \(q=\max\{ q_i^{1/p_i}: i=1,\dots, k\}<1\). Then, for \(x\in B_i\), by H\"older's inequality for \(p_i\) and~\eqref{eq:qmenor2}, we get
\begin{equation} \label{eq:qmemor3}
\begin{aligned}
    \int L(f_A)^n_\omega(x)^\alpha \, d\mathbb{P} &\leq \left(\int L(f_A)^n_\omega(x)^{p_i\alpha} \, d\mathbb{P}\right)^{1/p_i} \\ &=\left(\int L(f_A)^n_\omega(x_i)^{\alpha_i} \, d\mathbb{P}\right)^{1/p_i }<q_i^{1/p_i}\leq q.
\end{aligned}
\end{equation}
{Applying Lemma~\ref{lem:projective-distance-local} to \(A^n(\omega)\), we get}
$$
\frac{d\big(A^n(\omega)\hat{x},A^n(\omega)\hat{y}\big)^\alpha}{d(\hat{x},\hat{y})^\alpha}  
\leq \frac{1}{2} \big( L(f_A)^n_\omega(\hat{x})^\alpha 
+ L(f_A)^n_\omega(\hat{y})^\alpha \big)  \quad \text{for $\hat{x},\hat{y}\in P(\mathbb{R}^m)$.} 
$$
Integrating, taking supremum and using~\eqref{eq:qmemor3}, we get that 
$$
\sup_{\hat{x}\not = \hat{y}}\int \frac{d\big(A^n(\omega)\hat{x},A^n(\omega)\hat{y}\big)^\alpha}{d(\hat{x},\hat{y})^\alpha}  \, d\mathbb{P}  \leq \sup_{\hat{x}\in P(\mathbb{R}^m)} \int L(f_A)^n_\omega(\hat{x})^\alpha \, d\mathbb{P} \leq q. 
$$
This concludes that $(f_A)|_X$ is contracting on average. 
\end{proof} 

Finally, Claim~\ref{lem:quasi-irreducible-contracting-average} and  Propositions~\ref{prop:contration-proximal}  imply that $(f_A)|_X$ is also proximal,   and in particular mingled. {Moreover, Theorem~\ref{thmA} shows that the Koopman operator $P$ associated with $f_A$ also has a spectral gap on $C^\alpha(X)$ for any $\alpha>0$ small enough}, completing the proof of the proposition. 
\end{proof}

\section{Proof of Proposition~\ref{prop:projective}}
{Let $A\in \LP$ with $\lambda_1(A)>\lambda_2(A)$. If $A$ is a quasi-irreducible, then the equator $E$ is trivial. Thus, we can take $X = P(\mathbb{R}^m)$ in  Proposition~\ref{prop:projective2}. Then, we get that~(1) implies~(3) and also~(1) implies~(2). Clearly~(2) implies (1) from Proposition~\ref{prop:quasi-irreducible}. The following lemma concludes (3) implies (1).

\begin{lem} Let $A \in \LP$ with $\lambda_1(A)>\lambda_2(A)$.  If $f_A$ is mostly contracting, then $A$ is quasi-irreducible.
\end{lem}

\begin{proof}
Arguing by contradiction, assume that \(A\) is not quasi-irreducible, and let \(E\) be its equator. Then \(E\) is a non-trivial \(A\)-invariant subspace and $\lambda_1(A|_E)<\lambda_1(A)$.
Choose the orthogonal decomposition
$\mathbb{R}^m=E\oplus E^\perp$,
and write \(A\) in block form as
\[
A(\omega)=
\begin{pmatrix}
A|_E(\omega) & B(\omega)\\
0 & A^\perp(\omega)
\end{pmatrix}
\qquad\text{for \(\mathbb{P}\)-a.e.~\(\omega\in\Omega\)}.
\]
As in the proof of Proposition~\ref{prop:far-equator}, the cocycle \(A^\perp\) represents the action of \(A\) on the quotient \(\mathbb{R}^m/E\), and therefore $\lambda_1(A)=\max\{\lambda_1(A|_E),\,\lambda_1(A^\perp)\}$.
Since \(\lambda_1(A|_E)<\lambda_1(A)\), it follows that $\lambda_1(A^\perp)=\lambda_1(A)$. Moreover, by Claim~\ref{claim:quasi-irreducible}, the cocycle \(A^\perp\) is quasi-irreducible. Hence, by Proposition~\ref{prop:quasi-irreducible} and the last part of Theorem~\ref{cor:Furstenberg}, for every unit vector \(v\in E^\perp\),
\[
\lim_{n\to\infty}\frac1n\int \log \|(A^\perp)^n(\omega)v\|\,d\mathbb{P}
=
\lambda_1(A^\perp)
=
\lambda_1(A).
\]
Fix unit vectors \(x\in E\) and \(v\in E^\perp\), and let \(\hat x\in P(E)\subset P(\mathbb{R}^m)\) be the projective point determined by \(x\). 

\begin{claim} \label{eq:lowerbound-local-lip-clean} $L(f_A)^n_\omega(\hat x)
\ge
\frac{\|(A^\perp)^n(\omega)v\|}{\|A^n(\omega)x\|}$ for every \(n\ge 1\) and \(\mathbb{P}\)-a.e.~\(\omega\in\Omega\).
\end{claim}
\begin{proof} Since the projective action is smooth,
$
L(f_A)^n_\omega(\hat x)=\|D(f_A)^n_\omega(\hat x)\|$. Identifying \(T_{\hat x}P(\mathbb{R}^m)\) with \(x^\perp\), the derivative formula gives
\[
D(f_A)^n_\omega(\hat x)u
=
\frac{\Pi_{A^n(\omega)\hat x}\big(A^n(\omega)u\big)}{\|A^n(\omega)x\|}
\qquad\text{for }u\in x^\perp.
\]
Since \(x\in E\) and \(v\in E^\perp\), we have \(v\in x^\perp\). Also, \(A^n(\omega)x\in E\), while the \(E^\perp\)-component of \(A^n(\omega)v\) is \((A^\perp)^n(\omega)v\). Therefore
\[
\Big\|\Pi_{A^n(\omega)\hat x}\big(A^n(\omega)v\big)\Big\|
\ge
\|(A^\perp)^n(\omega)v\|.
\]
Substituting \(u=v\) into the derivative formula yields the claim.
\end{proof}
Taking logarithms in Claim~\ref{eq:lowerbound-local-lip-clean}, integrating, and dividing by \(n\), 
\[
\frac1n\int \log L(f_A)^n_\omega(\hat x)\,d\mathbb{P}
\ge
\frac1n\int \log \|(A^\perp)^n(\omega)v\|\,d\mathbb{P}
-
\frac1n\int \log \|A^n(\omega)x\|\,d\mathbb{P}.
\]
Passing to the upper limit and using
\[
\limsup_{n\to\infty}\frac1n\int \log \|A^n(\omega)x\|\,d\mathbb{P}
\le
\lambda_1(A|_E),
\]
we get
\begin{align*}
\lambda(\hat x)
&\eqdef
\limsup_{n\to\infty}\frac1n\int \log L(f_A)^n_\omega(\hat x)\,d\mathbb{P} \\ &\ge
\lambda_1(A^\perp)-\lambda_1(A|_E)  =
\lambda_1(A)-\lambda_1(A|_E)
>0.
\end{align*}
Thus, the maximal annealed Lyapunov exponent of \(f_A\) at the point \(\hat x\in P(E)\) is positive. On the other hand, by Lemma~\ref{lem:lip-fA},
\[
\mathrm{Lip}((f_A)_\omega)\le \|A(\omega)\|^2\|A(\omega)^{-1}\|^2,
\]
so the integrability condition in the definition of mostly contracting is satisfied. Therefore, by Corollary~\ref{cor:equivalencia}, if \(f_A\) were mostly contracting, then
\[
\lambda(\hat y)<0
\qquad\text{for every }\hat y\in P(\mathbb{R}^m).
\]
This contradicts \(\lambda(\hat x)>0\). Hence \(A\) must be quasi-irreducible.
\end{proof}

Now, we suppose that the exponential moment condition~\eqref{eq:exponential-moment-condition} holds. As shown in Proposition~\ref{prop:projective2}, the integral condition~\eqref{eq:integral_condition} follows from this exponential moment condition.   Assuming quasi-irreducibility, Proposition~\ref{prop:projective2} concludes  by taking $X=P(\mathbb{R}^m)$ that (1) implies~(4). 
Moreover, Proposition~\ref{prop:contration-proximal} shows that if $f_A$ is contracting on average and satisfies the integral condition~\eqref{eq:integral_condition}, then $f_A$ is also proximal and uniquely ergodic. That is,~(4) implies~(2). This concludes the equivalence between~(1)--(4). Moreover, under any of these equivalent conditions, $f_A$ is proximal and thus, mingled. Hence, Theorem~\ref{thmA} concludes that~(3) implies~(5). Also by Theorem~\ref{thm:Herver-improved} we have that~(5) implies~(2), completing the equivalence between~(1)--(5) and thus the proof of Proposition~\ref{prop:projective}. }

\begin{rem}
\label{rem:mingled-proximal-not-enough-linear}
In Proposition~\ref{prop:projective}, the properties of proximality or mingledness are consequences of the quasi-irreducibility and its equivalent conditions. However, even in the presence of a
simple top Lyapunov exponent and an exponential moment, these properties do not imply quasi-irreducibility. To see this, let \(0<r<1\), and let \(B:\Omega\to\mathbb R\) be a one-step  random variable with
unbounded essential support such that
\[
   \int (1+|B(\omega)|)^\beta\,d\mathbb P<\infty \quad \text{for some \(\beta>0\)}.
\]
Consider the locally constant linear cocycle $A:\Omega \to \mathrm{GL}(2)$ given by
\[
   A(\omega)=
   \begin{pmatrix}
      r & B(\omega)\\
      0 & 1
   \end{pmatrix}.
\]
Then, possibly replacing \(\beta\) by a smaller positive number, \(A\) satisfies
the exponential moment condition
\[
   \int
   \left(
      \max\{\|A(\omega)\|,\|A(\omega)^{-1}\|\}
   \right)^\beta
   d\mathbb P(\omega)<\infty .
\]
Moreover,
\[
   \lambda_1(A)=0,
   \qquad
   \lambda_2(A)=\log r<0.
\]
The line \(E\) spanned  by the vector $e_1=(1,0)$ is \(A\)-invariant and carries the exponent
\(\log r\). Hence, the equator is non-trivial, and \(A\) is not
quasi-irreducible. Let \(f_A\) be the induced projective random map. In the affine chart
\(u=x/y\) of \(P(\mathbb R^2)\setminus E\), the action is
\[
   u\longmapsto r u+B(\omega),
\]
while \(E\) is the point at infinity. Moreover, for two finite points \(u,v\in\mathbb R\), the affine
distance after \(n\) iterates is
\[
   |(f_A)^n_\omega(u)-(f_A)^n_\omega(v)|
   =
   r^n |u-v|,
\]
and therefore the corresponding projective distance tends to zero. If one of
the two points is \(E\), the unbounded essential support of \(B\) allows one
to choose symbols for which the finite point is sent arbitrarily far in the
affine chart, hence arbitrarily close to \(E\) in projective space. Thus
\(f_A\) is proximal, and consequently, it is mingled by
Remark~\ref{rem:topological-notions}(ii).
\end{rem}

{
\section{Characterization of mostly contracting projective cocycles}

We now complete the characterization of those locally constant linear cocycles
whose projective action is mostly contracting. Proposition~\ref{prop:projective}
gives one implication under the simplicity assumption
\(\lambda_1(A)>\lambda_2(A)\). The next result shows that this assumption is
not an additional technical hypothesis: it is forced by the mostly contracting
property of the projective random map.

\begin{prop} \label{prop:mostly-simplicity}
Let \(A\in\LP\) be a locally constant linear cocycle. If $f_A$   is mostly contracting, then $\lambda_1(A)>\lambda_2(A)$.
\end{prop}

\begin{proof}
We argue by contradiction and assume that
$\lambda_1(A)=\lambda_2(A)\eqdef \lambda$. 
Since \(f_A\) is mostly contracting, Theorem~\ref{cor:local-contraction}
applied to the projective random map gives \(q<1\) with the following property:
for every \(\hat u\in P(\mathbb R^m)\), for \(\mathbb P\)-a.e.
\(\omega\), there exist a neighborhood \(B=B(\omega,\hat u)\) of \(\hat u\)
and a constant \(C=C(\omega,\hat u)>0\) such that
\begin{equation}
\label{eq:projective-local-contraction}
   d\big(A^n(\omega)\hat u,A^n(\omega)\hat v\big)
   \le Cq^n
   \qquad
   \text{for every }\hat v\in B \text{ and every }n\ge1.
\end{equation}
Here \(d\) denotes the angular distance on projective space.

Let \(E\subset\mathbb R^m\) be the equator of \(A\), with the convention
\(E=\{0\}\) when \(A\) is quasi-irreducible. Let also
$H\subset \exterior{2}\mathbb R^m$ be the equator of the exterior-square cocycle \(\exterior{2} A\). Since
\[
   \lambda_1(\exterior{2} A)=\lambda_1(A)+\lambda_2(A)=2\lambda,
\]
Proposition~\ref{prop:far-equator}, applied to \(A\) and to \(\exterior{2} A\),
gives
\begin{equation}
\label{eq:growth-outside-E}
   \lim_{n\to\infty}
   \frac1n\log \|A^n(\omega)x\|
   =
   \lambda
   \qquad
   \text{for every }x\notin E,\ \mathbb P\text{-a.e. }\omega,
\end{equation}
and
\begin{equation}
\label{eq:growth-outside-H}
   \lim_{n\to\infty}
   \frac1n
   \log\|(\exterior{2} A^n(\omega))\xi\|
   =
   2\lambda
   \qquad
   \text{for every }\xi\notin H,\ \mathbb P\text{-a.e. }\omega.
\end{equation}
If one of the equators is trivial, the corresponding statement is understood
for every non-zero vector, as in Remark~\ref{rem:law-large-number}.

We now choose the two projective directions. Since \(E\) is a proper subspace
of \(\mathbb R^m\) and \(H\) is a proper subspace of
\(\exterior{2}\mathbb R^m\), we may choose \(u\in\mathbb R^m\setminus E\) such that
\[
   u\exterior{}\mathbb R^m\not\subset H.
\]
Indeed, the set of directions \(\hat u\) for which
\(u\exterior{}\mathbb R^m\subset H\) is a proper algebraic subset of
\(P(\mathbb R^m)\). If it were all of \(P(\mathbb R^m)\), then \(H\) would
contain every simple bivector \(u\exterior{} v\), hence all of
\(\exterior{2}\mathbb R^m\), which is impossible.

Fix such a unit vector \(u\). Define
$V_u
   \eqdef
   \{v\in\mathbb R^m:\ v\notin E \text{ and } u\exterior{} v\notin H\}$. 
The projectivization \(P(V_u)\) is dense in \(P(\mathbb R^m)\). Indeed, its
complement is contained in the union of the two proper projective subspaces
\[
   P(E)
   \quad\text{and}\quad
   P\big(\{v\in\mathbb R^m:\ u\exterior{} v\in H\}\big).
\]
Choose a countable dense set $\mathcal D\subset P(V_u)$.

For the fixed direction \(\hat u\), the contraction estimate
\eqref{eq:projective-local-contraction} holds for \(\mathbb P\)-a.e.
\(\omega\). Moreover, by \eqref{eq:growth-outside-E} and
\eqref{eq:growth-outside-H}, and since \(\mathcal D\) is countable, there is a
full \(\mathbb P\)-measure set of \(\omega\)'s such that simultaneously
\[
   \lim_{n\to\infty}\frac1n\log\|A^n(\omega)u\|=\lambda,
\]
and, for every \(\hat v\in\mathcal D\),
\[
   \lim_{n\to\infty}\frac1n\log\|A^n(\omega)v\|=\lambda,
   \qquad
   \lim_{n\to\infty}
   \frac1n\log\|A^n(\omega)u\exterior{} A^n(\omega)v\|
   =
   2\lambda .
\]
Choose \(\omega\) in the intersection of these full-measure sets, and let
\(B=B(\omega,\hat u)\) be the neighborhood given by
\eqref{eq:projective-local-contraction}. Since \(\mathcal D\) is dense, we can
choose \(\hat v \in\mathcal D\cap B\).

Since the angular distance satisfies
\[
   d(\hat a,\hat b)
   =
   \frac{\|a\exterior{} b\|}{\|a\|\,\|b\|}
   \qquad
   (a,b\neq0),
\]
we have that
\[
\begin{aligned}
   \lim_{n\to\infty}
   \frac1n
   \log d\big(A^n(\omega)\hat u,A^n(\omega)\hat v\big)
   &=
   \lim_{n\to\infty}
   \frac1n
   \log
   \frac{\|A^n(\omega)u\exterior{} A^n(\omega)v\|}
        {\|A^n(\omega)u\|\,\|A^n(\omega)v\|}  \\
   &= 2\lambda-\lambda-\lambda =0.
\end{aligned}
\]
On the other hand, since \(\hat v\in B\), the contraction estimate
\eqref{eq:projective-local-contraction} gives
\[
   d\big(A^n(\omega)\hat u,A^n(\omega)\hat v\big)
   \le Cq^n
   \qquad\text{for every }n\ge1.
\]
Hence
\[
   \limsup_{n\to\infty}
   \frac1n
   \log d\big(A^n(\omega)\hat u,A^n(\omega)\hat v\big)
   \le \log q<0,
\]
which contradicts the previous limit. Thus
\(\lambda_1(A)=\lambda_2(A)\) is impossible, and the proof is complete.
\end{proof}

As an immediate consequence of Proposition~\ref{prop:projective} and Lemma~\ref{prop:mostly-simplicity}, we get the following corollary.

\begin{cor} Let $A\in \LP$ be a locally constant linear cocycle. Then, $f_A$ is mostly contracting if and only if $A$ is quasi-irreducible and $\lambda_1(A)>\lambda_2(A)$.  
\end{cor}
}

\section{{Quasi-irreducibility of the transposed cocycle}}

Here, we also provide the following result that allows us to get the results of continuity and H\"older continuity of Lyapunov exponents for any cocycle having at most one invariant subspace.  

\begin{cor} \label{cor:atmostone}
Let $A\in\LP$. If $A$ has at most one $A$-invariant space, then either $A$ or $A^T$ is quasi-irreducible.
\end{cor}
\begin{proof} Assume that $A$ is not quasi-irreducible and denote by  $E$ the equator. Then, $E^\bot$ is $A^T$-invariant {because
$\langle A^Tv,u\rangle=\langle v,Au\rangle=0$
for every \(v\in E^\perp\) and \(u\in E\). In fact, this subspace is the unique non-trivial proper \(A^T\)-invariant subspace. Indeed,  if \(F\) is a non-trivial proper \(A^T\)-invariant subspace, then arguing as above, \(F^\perp\) is \(A\)-invariant, and hence \(F^\perp=E\) by the assumption that \(A\) has at most one invariant subspace. Therefore \(F=E^\perp\).

Now choose the orthogonal decomposition \(\mathbb R^m=E\oplus E^\perp\). Relative to this splitting, we may write
\[
A=
\begin{pmatrix}
A|_E & B\\
0 & A^\perp
\end{pmatrix}
\quad \text{and} \quad 
A^T=
\begin{pmatrix}
(A|_E)^T & 0\\
B^T & (A^\perp)^T
\end{pmatrix}. \]
In particular, the restriction \(A^T|_{E^\perp}\) is exactly \((A^\perp)^T\).} As shown in the proof of Proposition~\ref{prop:far-equator}, \[\lambda_1(A)=\lambda_{1}(A|_{\mathbb{R}^m/E})=\lambda_{1}(A^\bot).\] Since $\lambda_1(A^T)=\lambda_1(A)$ and   $\lambda_{1}(A^\bot)=\lambda_1((A^\bot)^T)=\lambda_{1}(A^T|_{E^\bot})$, we get that $\lambda_1(A^T)=\lambda_{1}(A^T|_{E^\bot})$. Therefore, $A^T$ is quasi-irreducible.   
\end{proof}

A cocycle $A\in \LP$ is said to be \emph{diagonalizable} if for $\mathbb{P}$-a.e.~$\omega\in \Omega$, the matrices $A(\omega)$ are simultaneously diagonalizable over $\mathbb{R}$. That is, there is $P\in \mathrm{GL}(m)$ such that $D(\omega)=P^{-1}A(\omega)P$ is a diagonal matrix for $\mathbb{P}$-a.e.~$\omega\in \Omega$. For locally constant cocycles valued at $\mathrm{GL}(2)$, this is equivalent to the existence of two different transverse $A$-invariant lines.  
The following result provides a similar classification in the cases (not mutually exclusive) exhibited in~\cite[Lemma~2.3]{duarte2020large}. 
\enlargethispage{0.2cm}
\begin{cor} \label{prop:dichotomy}
    For a cocycle $A \in \mathrm{L}_{\mathbb{P}}(2)$ one of the  following conditions hold: 
    \begin{enumerate}[leftmargin=1cm]
        \item $A$ is diagonalizable,
        \item $A$ or $A^T$ is quasi-irreducible.
    \end{enumerate}
    Moreover, a diagonalizable cocycle $A$ is quasi-irreducible if and only if \mbox{$\lambda_1(A)=\lambda_2(A)$.}
\end{cor}
\begin{proof}
    Assume that $A$ is not diagonalizable.  This means $A$ has either exactly one invariant line, or no invariant lines. If $A$ has no invariant lines, then $A$ is quasi-irreducible by definition (see Equation~\eqref{eq:reducible-Lyapunov-exponent}). If $A$ has exactly one invariant line $L$, by Corollary~\ref{cor:atmostone} we get that $A$ or $A^T$ is quasi-irreducible.

    To show the last observation in the statement of the proposition, observe that $\lambda_1(A)=\lambda_2(A)$ immediately implies that $A$ is quasi-irreducible.  Conversely, if $A$ is diagonalizable and quasi-irreducible, then $\lambda_1(A)=\lambda_2(A)$ because otherwise $A$ has an invariant line $L$ where $\lambda_1(A)>\lambda_1(A|_L)=\lambda_2(A)$.  This concludes the proof.
\end{proof}


\chapter{Limit theorems for linear cocycles} \label{chap:CLT-linear}

\abstract{ {
This chapter proves limit theorems for locally constant linear cocycles by combining the projective contraction theory developed earlier with the quasi-compact perturbation method of Hennion and Herv\'e. We first explain how limit theorems for Lipschitz random maps with a quasi-compact annealed Koopman operator apply to one-step observables depending on both the random parameter and the phase point. We then apply this framework to the projective action of a linear cocycle. Quasi-irreducibility and simplicity of the first Lyapunov exponent correspond to the spectral gap of the projective action; from this, we obtain central limit theorems, large deviation estimates, and Berry-Esseen bounds for the logarithmic growth of random matrix products, both along fixed directions and in operator norm. We also discuss the non-degeneracy of the variance in terms of the pinnacle subspace. Finally, the chapter treats cases beyond the uniquely ergodic, or spectral-gap, regime. In particular, we obtain these results  for cocycles in \(\mathrm{GL}(2)\) with a simple top Lyapunov exponent.
}
}


\section{Limit theorems for random maps via quasi-compactness}
\label{ss:limit-theorems-HH}
The results in~\cite[\S X.4]{HH:01} are established in a framework characterized by the following elements: \begin{enumerate}[label=$\bullet$, leftmargin=0.5cm] 
\item a measurable space $(E,\mathscr{E})$,
\item a measurable function $\xi:E\to \mathbb{R}$,
\item an annealed Koopman operator $Q$ associated with a random map $f:\Omega\times E \to E$, 
\item a Banach subspace $\mathcal{B}$ of the space $B(E)$ of complex-valued bounded   functions on~$E$. 
\end{enumerate} 
The required conditions are  (see~\cite[Condition $\mathcal{H}{[2]}$ and $\mathcal{K}{[2,d]}$]{HH:01}):

\begin{enumerate}[label=(\alph*)]
    \item $1\in \mathcal{B}$, and $\mathcal{B}$ is both a Banach lattice and a Banach algebra;
    \item $Q$ is quasi-compact on $\mathcal{B}$;   
    \item There is a real neighborhood $I_0$ of $t=0$ such that the map $t\mapsto Q(t)$, defined by
    $$
     Q(t)\phi= Q(e^{it\xi}\phi), \quad \text{for any $\phi\in \mathcal{B}$}, 
    $$
    has continuous derivatives up to the second order, and the derivative operator satisfies  {$Q^{(k)}(0)\phi = Q\big((i\xi)^k\phi\big)$} for $k=1,2$.  
\end{enumerate}

By~\cite[Lemma~VIII.10]{HH:01}, if $\xi \in \mathcal{B}$, then the Fourier-twisted operator $Q(\cdot)$ is an entire function, and its derivatives at zero satisfy the required relation in (c). This significantly simplifies the hypotheses needed to apply the limit theorems (CLT and LD), reducing them essentially to having a suitable function space and the quasi-compactness of the annealed Koopman operator on that space.

In~\cite[Chapter X]{HH:01}, the authors also developed a technique to apply the limit theorems to Lipschitz random maps $f:\Omega \times X \to X$, where the associated annealed Koopman operator $P$ is quasi-compact on $C^\alpha(X)$, and the potential $\xi$ is defined on $T\times X$. Note that this potential is not in $C^\alpha(X)$ and does not meet the assumption (c) above for the Fourier-twisted operator associated with $P$, as $\xi$ is not defined on $X$ alone. For our purposes, $(X,d)$ is a compact metric space, but the theory developed by Hennion and Herv\'e also applies to the non-compact case. We will now describe the method outlined in~\cite{HH:01} using our notation.

We first extend the compact metric space $X$ to the measurable product space $E=T\times X$, where $(T,\mathscr{A},p)$ is the probability space defining the sample space $(\Omega,\mathscr{F},\mathbb{P})=(T^\mathbb{N},\mathscr{A}^\mathbb{N},p^{\mathbb{N}})$. The random map $f:\Omega\times X\to X$ is then extended to $\bar{f}:\Omega \times E \to E$ defined by
$$
\bar{f}_\omega(t,x)=(\omega_0,f_t(x))\quad \text{for $(t,x) \in E$ and $\omega=(\omega_i)_{i\geq 0} \in \Omega$.}
$$
The annealed Koopman operator associated with $\bar{f}$ is given by
$$
Q\phi(t,x) = \int \phi(\omega_0,{f}_t(x))\, d\mathbb{P} \quad \text{for $(t,x) \in E$ and  $\phi\in C_b(E)$}.  
$$
For $\phi\in C^\alpha(X)$, define $J(\phi)\in B(E)$ by 
$$
J(\phi)(\bar{x})=\phi(f_t(x)), \quad \text{for $\bar{x}=(t,x)\in E$}.
$$
The function $J$ is  {injective} from $C^\alpha(X)$ to $B(E)$. Setting $\mathcal{B}=J(C^\alpha(X))$ and $\|J(\phi)\|=\|\phi\|_\alpha$, $J:C^\alpha(X)\to \mathcal{B}$ is an isometry, making  $(\mathcal{B},\|\cdot\|)$  a Banach space that satisfies the assumption~(a) above. 
Moreover, as shown in~\cite[Proof of Corollary~X.8]{HH:01}, $J(P\phi)=QJ(\phi)$ for every $\phi\in C^\alpha(X)$. Therefore, the quasi-compactness of $P$ is transferred to $Q$, fulfilling condition~(b). It has also been demonstrated that    if $\mu$ is an $f$-stationary measure, then $\hat{\mu}=p\times \mu$ is $\bar{f}$-stationary and
    $\int J(\phi)\, d\hat{\mu}=\int \phi \, d\mu$.
Similarly, the Fourier-twisted operator $Q(\cdot)$ satisfies
\begin{equation} \label{eq:relation}
\begin{aligned}
     Q(t)J(\phi) &= J(\bar{P}(t)\phi), \ \ \ \text{where} \\
\bar{P}(t)\phi(x) &=\int e^{it\xi(t,x)}\phi(f_t(x))\, dp(t) \quad \text{for $\phi\in C^\alpha(X), \ x\in X$}. 
\end{aligned}
\end{equation}
See \cite[proof of Theorem X.9]{HH:01}. 
Under the finite moment assumption,
\begin{equation}
    \label{eq:moment-HH}
K_\xi(\theta) \eqdef \int e^{\theta \|\xi_t\|_\infty}(1+|\xi_t|_1+\mathrm{Lip}(f_t))\, dp(t) <\infty  \quad \text{for some $\theta> 0$},
\end{equation}
where $\xi_t$ denotes the function $\xi_t(x)=\xi(t,x)$, and $|\xi_t|_1$ is the Lipschitz constant of this function,
Hennion and Herv\'e~\cite[Proposition~X.10]{HH:01} showed that $\bar{P}(\cdot)$ is holomorphic near $z=0$ and computed its derivatives at zero. Combined with~\eqref{eq:relation}, this result implies that the Fourier-twisted operator $Q(\cdot)$ satisfies condition (c).
Therefore, the limit theorems apply 
to the Birkhoff sum
$$
S_n\xi(\omega,(t,x))\eqdef\sum_{k=1}^n \xi(\bar{f}^k_\omega(t,x)), \quad \text{for $(t,x)\in T\times X$ and $\omega\in \Omega$},
$$
under the assumption that $\xi$ is orthogonal to the space of  $\bar{f}$-stationary product measures. Moreover, we can assume, without loss of generality, that there exists an $e\in T$ such that $f_e=\mathrm{id}_X$. Hence, 
\begin{equation} \label{eq:Sn}
S_n\xi(\omega,(e,x))=\sum_{k=1}^n \xi(\omega_{k-1},f^{k-1}_\omega(x)) \quad \text{for all $x\in X$, $\omega\in \Omega$, and $n\geq 1$}. 
\end{equation}

In summary: 

\begin{rem} \label{rem:sum} All the conclusions in Propositions~\ref{maincor2},~\ref{maincor3} and~\ref{maincor4} hold for the sequence of Birkhoff sums~\eqref{eq:Sn} instead $S_n\phi(\omega,x)$ under the assumption: 
\begin{itemize}
    \item $f:\Omega\times X\to X$ is a Lipschitz random map whose associated annealed Koopman operator $P$ is quasi-compact on $C^\alpha(X)$; 
    \item the potential $\xi:T\times X\to\mathbb{R}$ satisfies~\eqref{eq:moment-HH} and has zero integral with respect to any measure $\hat{\mu}_i=p\times \mu_i$ where $\mu_i$  is $f$-stationary.  
\end{itemize} 
The unique point that must be rewritten is the following:  $\sigma_i=0$ if and only if there is $\psi \in C^\alpha(X)$ such that
$$
\xi(t, x)=\psi(x)-\psi(f_{t}(x)) \quad  \text{for \  $\hat{\mu}_i$-a.e.~$(t,x)\in T\times X$}.
$$
\end{rem}

\section{General limit theorem}  
Let $A \in \LP$ be a locally constant cocycle, and consider the equator subspace $E$ and the pinnacle subspace $F$ of $A$, as introduced in Definition~\ref{def:equator}.   We recall that given an $A$-invariant subspace $F \subset \mathbb{R}^m$, the restriction of $A$ to $F$ is identified with the locally constant cocycle $A|_F: \Omega \to \mathrm{GL}(F)$.

\begin{thm} \label{thm:CLT-linear-general}
Let $A \in \LP$ be a locally constant linear cocycle, and denote by $E$ its equator. Assume:
\begin{enumerate}
    \item $A$ satisfies the exponential moment condition~\eqref{eq:exponential-moment-condition},
    \item $\lambda_1(A) > \lambda_2(A)$, and 
    \item there exists an $f_A$-invariant compact set $X \subset P(\mathbb{R}^m) \setminus E$.
\end{enumerate}
Then, there exists $\sigma \geq 0$ such that for every $\hat{x} \in X$, 
\begin{equation} \label{eq:CLT-point}
\frac{\log \|A^n(\omega)x\| - n\lambda_1(A)}{\sqrt{n}}  
\xrightarrow[n\to +\infty]{\text{law}} \mathcal{N}(0, \sigma^2), \index{central limit theorem}
\end{equation}
and there are positive constants $C$ and $h$ (independent of $\hat{x}$) such that  for any $\epsilon>0$ small enough, one can find $n_0=n_0(\epsilon)\in \mathbb{N}$ for which
\begin{equation} \label{eq:LD-point} \index{large deviations}
  \mathbb{P}\left( \left|\frac{1}{n}\log \|A^n(\omega)x\|-\lambda_1(A)\right| > \epsilon \right) \leq C e^{-n h \epsilon^2} \quad \text{for all $n\geq n_0$}.
\end{equation}
Moreover,~\eqref{eq:CLT-point} and~\eqref{eq:LD-point} also hold for $\log \|A^n(\omega)|_L\|$ instead of $\log \|A^n(\omega)x\|$, where $L$ is the $A$-invariant subspace of $\mathbb{R}^m$ spanned by $X$. \\[-0.4cm]

\noindent Furthermore, 
{$$
 \sigma=0 \iff \dim \tilde{F}=1 \ \ \text{and} \ \ \|\tilde{A}(\omega)|_{\tilde{F}}\|  \ \ \text{is $\mathbb{P}$-a.s.~constant}
$$
where $\tilde{F}$ is the pinnacle subspace of $\tilde{A}=A|_{\mathbb{R}^m/E}$ on $\mathbb{R}^m/E$. In particular, if $\sigma>0$,
then} for any $\hat{x} \in X$, we have
\begin{equation} \label{eq:CLT2}
        \sup_{u \in \mathbb{R}} \left|\mathbb{P}\left( \frac{\log \|A^n(\omega)x\| - n\lambda_1(A)}{\sigma \sqrt{n}} \leq u  \right)- \frac{1}{\sqrt{2\pi}} \int_{-\infty}^u e^{-\frac{s^2}{2}}\, ds\right| = O\left(n^{-1/2}\right)
        \index{Berry-Esseen theorem}
\end{equation}
and
\begin{equation} \label{eq:LD2}
    \lim_{n \to \infty} \frac{1}{n} \log \mathbb{P}\left( \left|\frac{1}{n} \log \|A^n(\omega)x\|-\lambda_1(A)\right| > \epsilon \right) = -c(\epsilon)
\end{equation}
where $c(\cdot)$ is a non-negative, strictly convex function that vanishes only at $\epsilon = 0$. Moreover,~\eqref{eq:CLT2} also holds for $\log \|A^n(\omega)|_L\|$, and
\begin{equation} \label{eq:LD3}
    \lim_{n \to \infty} \frac{1}{n} \log \mathbb{P}\left( \frac{1}{n} \log \|A^n(\omega)|_L\|-\lambda_1(A) > \epsilon\right) = -c(\epsilon).
\end{equation}
\end{thm}

\begin{proof}
According to Proposition~\ref{prop:projective2}, under assumptions (i)--(iii), $\mathrm{Lip}((f_A)|_X)^\beta$ is $\mathbb{P}$-integrable for some $\beta > 0$, and $(f_A)|_X$ contracting on average where the exponent $\alpha > 0$ in~\eqref{eq:contracting-avarage} can be made arbitrarily small. Additionally, $(f_A)|_X$ is uniquely ergodic, and its associated annealed Koopman operator is quasi-compact on $C^\alpha(X)$. As noted in~\cite[pg.~76]{HH:01}, this is the key point of the theory.

\begin{claim} \label{claim:bound8}
For every $\hat{x}, \hat{y} \in X$ and $r > 0$, 
$$
\int \sup_{n \geq 1} \big| \log \|A^n(\omega)y\| - \log \|A^n(\omega)x\| \,\big|^r \, d\mathbb{P} < \infty.
$$    
\end{claim}

\begin{proof}
By Lemma~\ref{lem:Holder}, for any $\alpha > 0$, there exists $K_\alpha > 0$ such that for every $\hat{x}, \hat{y} \in X$,
\begin{align} \label{eq:intral}
\int  & \sup_{i \geq 0} \big|\log \|A^{i+1}(\omega)y\| - \log \|A^{i+1}(\omega)x\|\,\big|^r \, d\mathbb{P} \leq  \notag\\
& \leq K_\alpha \int \sum_{i=0}^\infty
  \|A(\sigma^{i}\omega)^{-1}\|^{r\alpha} \, \|A(\sigma^i\omega)\|^{r\alpha} \cdot d(A^{i}(\omega)\hat{x},A^{i}(\omega)\hat{y})^{r\alpha} \, d\mathbb{P}  \\
& \leq K_\alpha \left(\int \|A(\omega)^{-1}\|^{4r\alpha} \, d\mathbb{P}\right)^{1/4} \cdot \left(\int \|A(\omega)\|^{4r\alpha} \, d\mathbb{P}\right)^{1/4} \cdot \notag  \\ &\hspace{4.3cm} \cdot \sum_{i=0}^\infty \left(\int  d(A^{i}(\omega)\hat{x},A^{i}(\omega)\hat{y})^{2r\alpha} \, d\mathbb{P}\right)^{1/2}. \notag
\end{align}
By taking $\alpha > 0$ sufficiently small, the exponential moment~\eqref{eq:exponential-moment-condition} and Lemma~\ref{lemma:2exponents} ensure the finiteness of the first two integrals in the last line of~\eqref{eq:intral}. Furthermore, since $\mathrm{Lip}((f_A)|_X)^\beta$ is $\mathbb{P}$-integrable for some $\beta > 0$ and $(f_A)|_X$ is contracting on average, one can iterate the last integral in~\eqref{eq:intral} similarly to~\eqref{eq:iteration-contracting_average} to obtain
$$
\sum_{i=0}^\infty \left(\int  d(A^i(\omega)\hat{x},A^i(\omega)\hat{y})^{2r\alpha} \, d\mathbb{P}\right)^{1/2} < \infty.
$$
This concludes the proof of the claim.
\end{proof}

\begin{claim} \label{claim:bound9}
For any $\hat{x} \in X$ and for $\mathbb{P}$-a.e.~$\omega \in \Omega$, there exists $C = C(\omega) > 0$ such that
$$
1 \leq \frac{\|A^n(\omega)|_L\|}{\|A^n(\omega)x\|} \leq C \quad \text{for any } n \geq 1.
$$
\end{claim}

\begin{proof}
From Claim~\ref{claim:bound8}, for every $\hat{x}, \hat{y} \in X$, we have
$$
\int \sup_{n \geq 1} \log \frac{\|A^n(\omega)y\|}{\|A^n(\omega)x\|} \, d\mathbb{P} < \infty.
$$
This implies that for $\mathbb{P}$-a.e.~$\omega \in \Omega$, there exists $K > 0$ such that
\begin{equation} \label{eq:final9}
\frac{\|A^n(\omega)y\|}{\|A^n(\omega)x\|} < K \quad \text{for any } n \geq 1.
\end{equation}
Now, consider a (unitary) basis $\{{x}_1,\dots,{x}_d\}$ of $L = \langle x \in \mathbb{R}^m : \hat{x} \in X \rangle$, where $\hat{x}_i \in X$ for $i = 1,\dots,d$. By Proposition~\ref{claim:F-invariant}, $L$ is $A$-invariant, and we write $A(\omega)|_L$ as a matrix in $\mathrm{GL}(L)$, identifying $\|A^n(\omega)|_L\|$ with $\max_i \|A^n(\omega)x_i\|$. Hence, the desired estimate follows by taking $y = x_i$ for $i = 1,\dots,d$ in~\eqref{eq:final9}.
\end{proof}

{\bf \noindent $\bullet$ {CLT, Berry-Esseen and LD for $\boldsymbol{\log \|A^n(\omega)x\|$}:}}  
In~\cite[Thm.~X.14]{HH:01}, the authors proved that {condition (i) implies}  {the moment condition~\eqref{eq:moment-HH} with} 
$$
\xi(t,\hat{x}) = \log \|A(t)x\| - \lambda_1(A), \quad (t,\hat{x}) \in T \times X.
$$
{By Remark~\ref{rem:maximazing},}  $\lambda_1(A) = \int \phi_A d\mu$  where $\phi_A(\hat{x}) = \int \log \|A(\omega)x\|\, d\mathbb{P}$, and $\mu$ is the unique $f_A$-stationary measure supported on $X \subset P(\mathbb{R}^m) \setminus E$. Therefore, $\int \xi \, d\hat{\mu} = 0$, where $\hat{\mu} = p \times \mu$.  

By the results outlined in~\S\ref{ss:limit-theorems-HH}, we apply the theorems in~\cite[Theorems~A, B, and E]{HH:01} to the sequence of Birkhoff sums
\begin{align*}
    S_n\xi(\omega,(e,\hat{x})) &= \sum_{i=1}^n \xi(\omega_{i-1},A^{i-1}(\omega)\hat{x}) 
    = \sum_{i=1}^n \left(\log \frac{\|A^i(\omega)x\|}{\|A^{i-1}(\omega)x\|} - \lambda_1(A) \right) \\
    &= \log \|A^n(\omega)x\| - n\lambda_1(A)
\end{align*}
for every $\hat{x} \in X$. Thus, the conclusions from Propositions~\ref{maincor2},~\ref{maincor3}, and~\ref{maincor4} (with $r = 1$) apply to this sequence. This implies~\eqref{eq:CLT-point} and~\eqref{eq:LD-point} for 
$$Z^x_n(\omega) \eqdef \frac{1}{n} \log \|A^n(\omega)x\| - \lambda_1(A).$$ 
Furthermore,~\eqref{eq:CLT2} and~\eqref{eq:LD2} hold for $Z^x_n$ provided that $\sigma > 0$. 

To obtain the large deviation principle for $|Z^x_n|$, we first note that the results in~\cite{HH:01} can be applied to $-\xi$ as well, since $K_{-\xi} = K_{\xi}$ and $\int -\xi \, d\hat{\mu} = - \int \xi \, d\hat{\mu}$. Therefore, we also have
\begin{align*}
    \mathbb{P}\left( Z^x_n  < - \epsilon \right) \leq  C e^{-n h\epsilon^2}  \ \text{and} 
    \
    \lim_{n \to \infty} \frac{1}{n} \log \mathbb{P}\left( Z^x_n < - \epsilon\right) = -c(\epsilon) \ \text{provided $\sigma > 0$}.
\end{align*}
Since $\mathbb{P}(|Z^x_n| > \epsilon) \leq \mathbb{P}(Z^x_n > \epsilon) + \mathbb{P}(Z^x_n < -\epsilon)$, we conclude~\eqref{eq:LD-point}. To derive~\eqref{eq:LD2}, we apply 
the fact that when the rate function is continuous and strictly convex, the large deviation principle holds for any borel set in $\mathbb{R}$ if and only if the corresponding limit exists for tail events (see~\cite[Lemma~4.4]{huang2012moments}). 

{\bf \noindent  $\bullet$ {LD for $\boldsymbol{\log \|A^n(\omega)|_L\|}$:}}  
Let $\hat{x} \in X$ and define
$$
    W_n(\omega) \eqdef \frac{1}{n}\log \|A^n(\omega)|_L\| - \lambda_1(A).
$$
Since $\|A^n(\omega)x\| \leq \|A^n(\omega)\|$, we have $Z^x_n \leq W_n$. Using the large deviation principle for $Z^x_n$ and identifying $\|M\|$ with $\max_i \|Mx_i\|$ where $M \in \mathrm{GL}(L)$ and $\{x_1,\dots,x_d\}$ is a unitary basis of $L$ with $\hat{x}_i \in X$ for $i = 1,\dots,d$, we obtain
\begin{align*}
\mathbb{P}(|W_n| > \epsilon) &\leq \mathbb{P}(W_n < -\epsilon) + \mathbb{P}(W_n > \epsilon) \\&\leq \mathbb{P}(Z^x_n < -\epsilon) + \sum_{i=1}^{m} \mathbb{P}(Z^{x_i}_n > \epsilon) \leq C e^{-nh \epsilon^2}.
\end{align*}
For the case $\sigma > 0$, we now show~\eqref{eq:LD3}. Using that $Z^x_n \leq W_n$, we first derive
\begin{equation*} 
\liminf_{n\to \infty} \frac{1}{n}\log \mathbb{P}(W_n > \epsilon) \geq \liminf_{n\to \infty} \frac{1}{n}\log \mathbb{P}(Z^x_n > \epsilon) = - c(\epsilon).
\end{equation*}
Next, we prove the upper bound
\begin{equation} \label{eq:rr}
\limsup_{n\to \infty} \frac{1}{n} \log \mathbb{P}(W_n > \epsilon) \leq \limsup_{n\to \infty} \frac{1}{n} \log \sum_{i=1}^{d} \mathbb{P}(Z^{x_i}_n > \epsilon).
\end{equation}
Since $\limsup_{n\to \infty} \frac{1}{n} \log \mathbb{P}(Z^{x_i}_n > \epsilon) = -c(\epsilon)$ for $i = 1,\dots,d$, the sum in~\eqref{eq:rr} is bounded by $d e^{-n(c(\epsilon) - \delta)}$ for large $n$ and small $\delta > 0$. Thus
$$
\limsup_{n\to \infty} \frac{1}{n} \log \sum_{i=1}^{m} \mathbb{P}(Z^{x_i}_n > \epsilon) \leq - c(\epsilon) + \delta.
$$
Since $\delta > 0$ is arbitrary, we obtain the desired upper bound.

{\bf \noindent  $\bullet$ {CLT for $\boldsymbol{\log \|A^n(\omega)|_L\|}$:}}  
Let $\hat{x} \in X$ and consider the sequence of random variables:
\begin{align*}
    V_n(\omega) &\eqdef \frac{1}{\sqrt{n}} \big(\log \|A^n(\omega)|_L\| - n\lambda_1(A)\big), \\
    Y_n(\omega) &\eqdef \frac{1}{\sqrt{n}} \big(\log \|A^n(\omega)|_L\| - \log \|A^n(\omega)x\|\big),  \\
    Z_n(\omega) & \eqdef \frac{1}{\sqrt{n}} \big(\log \|A^n(\omega)x\| - n \lambda_1(A)\big).
\end{align*}
By Claim~\ref{claim:bound9}, $Y_n$ converges almost surely to zero. Hence,  since $V_n = Z_n + Y_n$, we obtain~\eqref{eq:CLT-point} for $\log \|A^n(\omega)|_L\|$ instead of $\log \|A^n(\omega)x\|$.

{\bf \noindent $\bullet$ {Berry-Esseen for $\boldsymbol{\log \|A^n(\omega)|_L\|}$:}}  
Let $Z: \Omega \to \mathbb{R}$ be a random variable. We define
$$
\delta(Z) \eqdef \sup_{u \in \mathbb{R}} \left|\mathbb{P}(Z \leq u) - \frac{1}{\sqrt{2\pi}} \int_{-\infty}^u e^{-\frac{s^2}{2}}\, ds\right|.
$$
According to~\cite[Lemma~1]{bolthausen1982exact}, for any pair of real-valued random variables $Y$ and $Z$, 
$$
\delta(Z+Y) \leq 2\delta(Z) + 3 \left( \int |Y|^2 \, d\mathbb{P} \right)^{1/2}.
$$
Thus, since $V_n = Z_n + Y_n$ and $\delta(Z_n) = O(n^{-1/2})$, we need only show that $\int |Y_n|^2 \, d\mathbb{P} = O(n^{-1})$. By Claim~\ref{claim:bound8}, we have
\begin{align*}
\int |Y_n|^2\, d\mathbb{P} &\leq \frac{1}{\sigma^2 n} \int \max_{j=1,\dots,d} \big|\log \|A^n(\omega)x_j\| - \log \|A^n(\omega)x\|\,\big|^2 \, d\mathbb{P} \\
&\leq \frac{1}{\sigma^2 n} \sum_{j=1}^{d} \int \big|\log \|A^n(\omega)x_j\| - \log \|A^n(\omega)x\|\,\big|^2 \, d\mathbb{P} \\
&\leq \frac{1}{\sigma^2 n} \sum_{j=1}^{d} \int \sup_{i \geq 1} \big|\log \|A^i(\omega)x_j\| - \log \|A^i(\omega)x\|\,\big|^2 \, d\mathbb{P} = O(n^{-1}).
\end{align*}
This concludes~\eqref{eq:CLT2} for $\log \|A^n(\omega)|_L\|$ instead of $\|\log A^n(\omega)x\|$. 

{
{\bf \noindent $\bullet$ {Positivity of the variance:}}
Now, under assumptions \emph{(i)}--\emph{(iii)}, we prove that
$\sigma=0$ if and only if  $\dim \tilde F=1$ and 
$\|\tilde A(\omega)|_{\tilde F}\|$ is $\mathbb P$-a.s.~constant,
where \(\tilde A=A|_{\mathbb R^m/E}\) is the quotient cocycle on
\(\tilde V=\mathbb R^m/E\), and \(\tilde F\) denotes its pinnacle subspace. Since we quotient by the equator \(E\), we have
\[
\lambda_1(\tilde A)=\lambda_1(A)>\lambda_2(A)\ge \lambda_2(\tilde A).
\]
Moreover, by definition of the equator, \(\tilde A\) has a trivial equator, that is, \(\tilde A\) is quasi-irreducible; see Claim~\ref{claim:quasi-irreducible}. Thus \(\tilde A\) satisfies assumptions \emph{(i)}--\emph{(iii)} with \(X=P(\tilde V)\). Furthermore:

\begin{claim}\label{claim:tildeA-tildeF-irreducible}
The restricted cocycle \(\tilde A|_{\tilde F}\) is irreducible.
\end{claim}

\begin{proof}
Let \(H\subset \tilde F\) be a non-trivial \(\tilde A\)-invariant subspace.    Since $\tilde{F}$ is the pinnacle of $\tilde{A}$, we have $\lambda_1(\tilde{A}|_H)<\lambda_1(\tilde{A})$. As $\tilde{A}|_H=(\tilde{A}|_{\tilde{F}})|_H$ and $\lambda_1(\tilde{A})=\lambda_1(\tilde{A}|_{\tilde{F}})$, we get $\lambda_1((\tilde{A}|_{\tilde{F}})|_H)<\lambda_1(\tilde{A}|_{\tilde F}))$, which is not possible since $\tilde{A}|_{\tilde{F}}$ is quasi-irreducible. Thus, there are no proper $\tilde{A}$-invariant subspaces of $\tilde{F}$. In other words, $\tilde{A}|_{\tilde{F}}$  is irreducible. 
\end{proof}

\smallskip

\noindent\emph{-- The reduction:} Let $\hat\pi:P(\mathbb R^m)\setminus E\longrightarrow P(\tilde V)$ be defined by $\hat\pi(\hat x)\eqdef \widehat{\pi(x)}$,
where \(\pi:\mathbb R^m\to \tilde V\) is the canonical projection.  Since \(X\subset P(\mathbb R^m)\setminus E\) is compact, after identifying \(\tilde V\) with \(E^\perp\), the map
$\hat x\longmapsto \|\pi(x)\|/\|x\|$
is continuous and strictly positive on \(X\). Thus, there exists \(c>0\) such that $\|\pi(x)\|\ge c \,\|x\|$ for every $\hat x\in X$. On the other hand, always
$\|\pi(x)\|\le \|x\|$. As \(X\) is \(f_A\)-invariant and by definition  $A(\omega)\pi(x)=\pi(A(\omega)x)$, the same estimates hold for every iterate \(A^n(\omega)x\). Therefore
$
c\,\|A^n(\omega)x\|
\le
\|\tilde A^n(\omega)\pi(x)\|
\le
\|A^n(\omega)x\|
$
for every \(\hat x\in X\), \(\mathbb P\)-a.e.~\(\omega\), and every \(n\ge1\). Hence
\[
\bigl|\log \|A^n(\omega)x\|-\log \|\tilde A^n(\omega)\pi(x)\|\bigr|
\le |\log c|.
\]
From this and since $\lambda_1(A)=\lambda_1(\tilde A)$, it follows that all limit theorems for \(\log\|A^n(\omega)x\|\)  and for
\(\log\|\tilde A^n(\omega)\pi(x)\|\)  have the same centering and the same variance. In particular, it is enough to prove the desired equivalence with the variance associated with the quotient cocycle \(\tilde A\)  on \(\hat\pi(X)\). 

We reduce further from \(\tilde A\) on $\pi(X)$ to its restriction $\tilde A|_{\tilde F}$ on $P(\tilde{F})$. Since  \(\tilde{A}|_{\tilde{F}}\) is irreducible and
 $\lambda_1(\tilde{A}|_{\tilde{F}})=\lambda_1(\tilde A)>\lambda_2(\tilde A)\geq \lambda_2(\tilde{A}|_{\tilde{F}})$,
 by Proposition~\ref{prop:projective}, the projective random map \(f_{\tilde{A}|_{\tilde{F}}}\) is uniquely ergodic. Let \(\mu\) be its unique stationary probability measure on \(P(\tilde F)\). By Proposition~\ref{claim:F-invariant} applied to \(\tilde{A}|_{\tilde{F}}\), the support of \(\mu\) spans \(\tilde F\). 
Choose \(\hat v\) in the support of $\mu$, and let \(v\in \tilde F\setminus\{0\}\) be a representative. Since $\tilde{A}^n(\omega)|_{\tilde{F}}v=\tilde A^n(\omega)v$,
\[
\log\|\tilde{A}^n(\omega)|_{\tilde{F}}v\|=\log\|\tilde A^n(\omega)v\|
\qquad\text{for every }n\ge1\text{ and every }\omega\in\Omega.
\]
Hence, as in addition $\lambda_1(\tilde{A}|_{\tilde{F}})=\lambda_1(\tilde{A})$, the asymptotic variance associated with \(\tilde{A}|_{\tilde{F}}\) at the point \(\hat v\) is the same as the asymptotic variance associated with \(\tilde A\) at \(\hat v\). Since the first part of the theorem applies to both cocycles, this variance is independent of the initial direction. Consequently, it suffices to prove the equivalence with the asymptotic variance of \(\tilde A|_{\tilde F}\) on $P(\tilde{F})$.

\smallskip

\noindent\emph{-- Direct implication:} First assume that $\dim \tilde{F}=1$ and there is $c\geq 0$ such that $\|\tilde{A}(\omega)|_{\tilde{F}}\|=c$ for $\mathbb{P}$-a.e.~$\omega\in \Omega$. Then $\|\tilde A^n(\omega)|_{\tilde F}\|=c^n$ for every $n\ge1$,
and so
\[
\lambda_1(\tilde A|_{\tilde F})
=
\lim_{n\to\infty}\frac1n\log \|\tilde A^n(\omega)|_{\tilde F}\|
=
\log c.
\]
Thus $c=e^{\lambda_1(\tilde{A}|_{\tilde{F}})}$ and hence  for every \(n\ge1\),
$\|\tilde A^n(\omega)|_{\tilde F}\|=e^{n\lambda_1(\tilde{A}|_{\tilde{F}})}$ for \(\mathbb P\)-a.e.~\(\omega\). 
Then
$\log\|\tilde A^n(\omega)|_{\tilde F}\|-n\lambda_1(\tilde{A}|_{\tilde{F}})=0$ for \(\mathbb P\)-a.e.~\(\omega\), 
so the asymptotic variance of $\tilde{A}|_{F}$ is zero. Therefore, $\sigma=0$.

\smallskip
\noindent\emph{-- Converse implication:}
Assume that the asymptotic variance of $\tilde{A}|_{\tilde{F}}$ is zero. To simplify the notation, in what follows we rewrite  $A=\tilde{A}|_{\tilde{F}}$. Notice that $A:\Omega \to \mathrm{GL}(\tilde{F})$ satisfies (i)-(ii), (iii) with $X=P(\tilde F)$ and, by Claim~\ref{claim:tildeA-tildeF-irreducible},  it is also irreducible. Thus, the random map $f_{A}:\Omega \times X \to X$ and the potential $\xi(t,x)=\log \|{A}(t)\|-\lambda_1({A})$, $(t,x)\in T\times X$, satisfy the requeriments to apply Remark~\ref{rem:sum}. 
Then, there exists a function $\psi \in C^\alpha(X)$   such that
\begin{equation} \label{eq:xitx}    
\xi(t, x) = \psi(\hat{x}) - \psi( A(t)\hat{x}) \quad  \text{for $(p\times \mu)$-a.e.~} (t,\hat{x}) \in T \times X.
\end{equation}
{If \(\dim \tilde F=1\), let \(e\) be a unit vector spanning \(\tilde F\). In this case, we have
 \(\mu=\delta_{\hat e}\) and $A(t)\hat e= \hat e$ for $p$-a.e.~$t\in T$. Evaluating the above identity at \(\hat x=\hat e\), 
\[
 \log \|{A}(t)e\| - \lambda_1({A}) = \xi(t,\hat e)=\psi(\hat e)-\psi({A}(t)\hat e)=0
\qquad
\text{for $p$-a.e.~\(t\in T\)}.
\]
Hence, 
$\|{A}(t)\|=e^{\lambda_1(A)}$ for \(p\)-a.e.~\(t\in T\). That is, recalling that here $A$ denotes the restricted cocycle $\tilde{A}|_{\tilde{F}}$, we have that  $\|\tilde{A}(\omega)|_{\tilde{F}}\|$ is $\mathbb{P}$-a.s.~constant as required. This concludes the implication after showing that the case that $\dim \tilde F \geq 2$ is impossible.  
}

Now, we assume that $\dim \tilde F \geq 2$.  
We consider the cocycle defined by $$B(\omega) = e^{-\lambda_1(A)} {A}(\omega).$$ Note that the projective random map $f_B = f_A$, and thus $\mu$ remains the unique $f_B$-stationary measure. Furthermore, $B$ is also irreducible since $A$ is. Once again, we can refine this conclusion:

\begin{claim} \label{claim:contracting222} 
$B:\Omega \to \mathrm{GL}(\tilde{F})$ is strongly irreducible and contracting in the sense that there exists a sequence $(g_n)_{n\geq 1}$ in the semigroup on $\mathrm{GL}(\tilde{F})$ generated by the support of $\nu_{B}=B_*\mathbb{P}$ such that $\|g_n\|^{-1}g_n$ converges to a rank-one endomorphism.     
\end{claim}
\begin{proof} Since $\|B^n(\omega)\|=e^{-n\lambda_1({A})}\|{A}^n(\omega)\|$, we have that $\lambda_1(B)=0$ and 
\begin{align*}
    \lambda_1(B)+\lambda_2(B)
&=
\lim_{n\to\infty}\frac1n\int
\log \big\|\wedge^2(B^n(\omega))\big\|\,d\mathbb{P} \\ 
&=
\lim_{n\to\infty}\frac1n\int
\log \big\|\wedge^2({A}^n(\omega))\big\|\,d\mathbb{P}
-2\lambda_1({A}) \\
&=
\lambda_1({A})+\lambda_2({A})-2\lambda_1({A}).
\end{align*}
We conclude that $\lambda_2(B)=\lambda_2({A})-\lambda_1({A}) <0=\lambda_1(B)$. Since $B$ is also irreducible, by the result of Guivarc'h and Raugi~\cite{GR85} (see also~\cite[Thm.~6.1, Chap.~III]{BouLac:85}) yields that \(B\) is strongly irreducible and contracting. 
\end{proof}

On the other hand, by~\eqref{eq:xitx} and since the projective random maps of $B$ and $A$ coincide, we have that 
$$
\log \|B^n(\omega)x\| = \sum_{i=0}^{n-1} \xi(\omega_i, B^{i-1}(\omega)\hat{x}) = \psi(\hat{x}) - \psi(B^n(\omega)\hat{x}).
$$
for $\mathbb{P}$-a.e.~$\omega = (\omega_i)_{i \geq 0} \in \Omega$, $\mu$-a.e.~$\hat x \in X$, and all $n \geq 1$. 
Since $\psi$ is a continuous function on the compact set $X$, this equation implies the existence of $K>0$ such that $K^{-1}\leq \|B^n(\omega)x\| \leq K$ for $(\mathbb{P}\times \mu)$-a.e.~$(\omega,\hat{x})$ and $n\geq 1$. Now, consider a basis $\{x_1, \dots, x_d\}$ of $\tilde{F}$ with $\hat{x}_1, \dots, \hat{x}_d$ in the support of $\mu$. We can then identify the norm $\|B^n(\omega)\|$ with $\max_i \|B^n(\omega)x_i\|$. Therefore, we have
\begin{equation} \label{eq:limitado}
K^{-1} \leq \|B^n(\omega)\| \leq K \quad \text{for $\mathbb{P}$-a.e.~}\omega \in \Omega \text{ and for all } n \geq 1.
\end{equation}

\begin{claim} \label{eq:annulus-T-correct} It holds that 
    $K^{-1}\le \|M\|\le K$ for every $M\in \Gamma_{B}$, where $\Gamma_{B}$ denotes the semigroup generated by the support of \(\nu_{B}=B_*\mathbb P\). 
\end{claim} 
\begin{proof} Let \(M=M_n\cdots M_1\) be a matrix in $\Gamma_{B}$  with \(M_i\in \operatorname{supp}\nu_{B}\). For each \(i\), choose an open neighborhood \(U_i\) of \(M_i\) with \(\nu_{B}(U_i)>0\). Then
$U=U_n\cdots U_1$ is an open neighborhood of \(M\), and the convolution
$(\nu_{B})^{*n}(U)>0$. Hence \(M\in \operatorname{supp}((\nu_{B})^{*n})\). Since \eqref{eq:limitado} holds for \(n\)-step products almost surely, the support of \((\nu_{B})^{*n}\) is contained in $\{N\in \mathrm{GL}(\tilde{F}): K^{-1}\le \|N\|\le K\}$. Thus, the claim follows.
\end{proof}

By rescaling the norm by $\|M\|_\ast\eqdef K\|M\|$, then Claim~\ref{eq:annulus-T-correct} becomes
\[
1\le \|M\|_\ast\le K^2
\qquad\text{for every }M\in \Gamma_B.
\]
Hence, according to~\cite[Lem.~8.2, Chap.~V]{BouLac:85}, the irreducibility of $B$ and the above estimate implies that \(\Gamma_B\) is contained in a compact subgroup \(G\subset \mathrm{GL}(\tilde F)\). However, this is impossible because \(\dim \tilde F\ge2\) and \(B\) is contracting by Claim~\ref{claim:contracting222} . Indeed, choose \(g_n\in \Gamma_B\) such that $g_n/\|g_n\|_\ast \to R$
with \(\operatorname{rank}R=1\). Since \(G\) is compact, after passing to a subsequence, we have \(g_n\to g\in G\). Because \(1\le \|g_n\|_\ast\le K^2\), we may also assume \(\|g_n\|_\ast^{-1}\to a\in [K^{-2},1]\). Thus
${g_n}/\|g_n\|_\ast\to ag$. Hence $R=ag$, which is a contraction since \(ag\) is invertible, whereas \(R\) has rank one. 

Now, the proof of the equivalence is complete.} 
\end{proof}

{
The following proposition studies the relation between the pinnacle of $A$ and the pinnacle of $\tilde{A}=A|_{\mathbb{R}^m/E}$. 

\begin{prop}\label{prop:old-implies-new}
Let \(E\) and $F$ be the equator and pinnacle of $A\in \LP$. Consider the quotient cocycle $\tilde A=A|_{\mathbb R^m/E}$ and denote by \(\tilde F\) its  pinnacle subspace.  Then
\[
\tilde F=(F+E)/E \cong F/(F\cap E).
\]
\end{prop}

\begin{proof}
The canonical projection $W \mapsto W/E$ is a bijection between $A$-invariant subspaces of $\mathbb{R}^m$ containing $E$ and $\tilde{A}$-invariant subspaces of $\mathbb{R}^m/E$. Since $\lambda_1(A|_E) < \lambda_1(A)$, the relation
$\lambda_1(A|_W) = \max\{\lambda_1(\tilde{A}|_{W/E}), \lambda_1(A|_E)\}$
ensures that $W/E$ carries the top exponent $\lambda_1(\tilde{A}) = \lambda_1(A)$ if and only if $W$ carries the top exponent $\lambda_1(A)$. By definition, the pinnacle $\tilde{F}$ is the intersection of all such $W/E$. Therefore, $\tilde{F} = W_0/E$, where $W_0$ is the intersection of all $A$-invariant subspaces $W \supset E$ carrying the top exponent. 

We claim $W_0 = F+E$. First, because the pinnacle $F$ is the intersection of {all} $A$-invariant subspaces carrying the top exponent, it is contained in $W_0$. Since $E \subset W_0$ trivially, we have $F+E \subset W_0$. Conversely, $F+E$ is an $A$-invariant subspace containing $E$, and it carries the top exponent since $\lambda_1(A|_{F+E}) = \max\{\lambda_1(A|_F), \lambda_1(A|_E)\} = \lambda_1(A)$. Thus, $F+E$ is one of the subspaces being intersected to form $W_0$, which forces $W_0 \subset F+E$. 

Hence $W_0 = F+E$, giving $\tilde{F} = (F+E)/E$. The second isomorphism theorem provides $(F+E)/E \cong F/(F\cap E)$, completing the proof.
\end{proof}

As a consequence, we get the following sufficient condition to ensure the positivity of the asymptotic variance $\sigma$.

\begin{cor} \label{cor:positive-sigma} Under the assumption (i)--(iii) of Theorem~\ref{thm:CLT-linear-general}, 
\begin{enumerate}[label=(\arabic*)]
    \item If $\dim F-\dim F\cap E\geq 2$, then $\sigma>0$; 
    \item If $E\cap F=\{0\}$, then
    $$
 \sigma=0 \iff \dim {F}=1 \ \ \text{and} \ \ \|A(\omega)|_F\|  \ \ \text{is $\mathbb{P}$-a.s.~constant}.
    $$
\end{enumerate}
\end{cor}
\begin{proof} By Proposition~\ref{prop:old-implies-new}, $\dim \tilde{F}=\dim F-\dim F\cap E$. Thus, the first item follows immediately by the equivalence shown in Theorem~\ref{thm:CLT-linear-general}. To see the second item, suppose $E\cap F=\{0\}$. Then $\dim \tilde{F} = \dim F$, and the canonical projection induces an isomorphism \(J:F\to \tilde F\). By the definition of the quotient cocycle, $\tilde A(\omega)|_{\tilde F}\circ J = J\circ A(\omega)|_F$ for \(\mathbb P\)-a.e.~\(\omega\).
Thus, the restricted cocycles \(A|_F\) and \(\tilde A|_{\tilde F}\) are conjugate. By Theorem~\ref{thm:CLT-linear-general}, $\sigma=0$ if and only if $\dim \tilde{F}=1$ and $\|\tilde A(\omega)|_{\tilde F}\|$ is $\mathbb{P}$-a.s.~constant. Since, in one dimension, the operator norm is simply the absolute value of the corresponding scalar multiplier, we have $\|\tilde A(\omega)|_{\tilde F}\| = \|A(\omega)|_F\|$ almost surely. This establishes the required equivalence and concludes the proof.
\end{proof}

In Theorem~\ref{thm:CLT-linear-general}, we get the CLT and LD limit theorems for $\log \|A^n(\omega)|_L\|$ where $L$ is the $A$-invariant subspace of $\mathbb{R}^m$ spanned by the $f_A$-invariant compact set $X\subset  P(\mathbb{R}^m)\setminus E$ in assumption (iii). When $L=\mathbb{R}^m$, then the limit theorems hold for $\log \|A^n(\omega)\|$. This is the case for instance when $A$ is quasi-irreducible because in such case $X=P(\mathbb{R}^m)$ and $L=\langle x\in \mathbb{R}^m: \hat x \in X\rangle = \mathbb{R}^m$. Moreover, in this quasi-irreducible case, the equator is trivial and thus $\sigma>0$ because of $\dim F-\dim F\cap E =\dim F =m\geq 2$ and Corollary~\ref{cor:positive-sigma}. In the following proposition, we provide a more general sufficient condition to get the CLT with a rate of convergence for $\log \|A^n(\omega)\|$. }

\begin{prop} \label{cor:CLT-geral} 
Let $A \in \LP$ be a locally constant linear cocycle, and denote by $E$ and $F$ its equator and pinnacle, respectively. Assume that
\begin{enumerate}
    \item $A$ satisfies the exponential moment condition~\eqref{eq:exponential-moment-condition},
     \item $\lambda_1(A) > \lambda_2(A)$, 
    \item there exists an $f_A$-invariant compact set $X\subset  P(\mathbb{R}^m) \setminus E$ such that \\ $\mathbb{R}^m = E \oplus L$, where $L$ is the $A$-invariant subspace of $\mathbb{R}^m$ spanned by $X$.
\end{enumerate}
Then, there exists $\sigma \geq 0$ such that
\begin{equation} \label{CLT-norm}
    \frac{\log \|A^n(\omega)\| - n\lambda_1(A)}{\sqrt{n}}  
    \xrightarrow[n\to +\infty]{\rm{law}} \mathcal{N}(0,\sigma^2)
\end{equation} 
and there are positive constants $C$ and $h$ such that  
 for any $\epsilon>0$ small enough, one can find $n_0=n_0(\epsilon)\in \mathbb{N}$ for which
\begin{equation} \label{LD-norm}
  \mathbb{P}\left( \left|\frac{1}{n}\log \|A^n(\omega)\|-\lambda_1(A)\right| > \epsilon \right) \leq C e^{-n h\epsilon^2}  \quad \text{for all $n\geq n_0$}. 
\end{equation} 
~\\[-0.1cm]
\noindent 
Furthermore, 
$$
{\sigma=0 \iff \dim {F}=1 \ \ \text{and} \ \ \|A(\omega)|_F\|  \ \ \text{is $\mathbb{P}$-a.s.~constant}.}
$$
where $F$ denotes the pinnacle of $A$. In particular, if $\sigma>0$, then
\begin{equation} \label{BE-norm} 
    \sup_{u \in \mathbb{R}} \left|\mathbb{P}\left(\frac{\log \|A^n(\omega)\| - n\lambda_1(A)}{\sigma \sqrt{n}} \leq u  \right) - \frac{1}{\sqrt{2\pi}} \int_{-\infty}^u e^{-\frac{s^2}{2}}\, ds\right| = O\left(n^{-1/2}\right).
\end{equation}
\end{prop}

{
\begin{proof}
Fix 
the adapted norm
\[
\|z\|_\ast\eqdef \max\{\|u\|,\|v\|\} \quad \text{where} \ \ z=u+v \in E\oplus L=\mathbb R^m.
\]
Since all norms on \(\mathbb R^m\) are equivalent, there exists \(K>0\) such that
\[
\bigl|\log\|A^n(\omega)\|-\log\|A^n(\omega)\|_\ast\bigr|\le K
\qquad\text{for all }n\ge1\text{ and }\omega\in\Omega.
\]
Thus it is enough to prove the proposition for \(\|\cdot\|_\ast\). With respect to the splitting \(\mathbb R^m=E\oplus L\), we have
\[
A(\omega)=
\begin{pmatrix}
A(\omega)|_E & 0\\
0 & A(\omega)|_L
\end{pmatrix},
\qquad
\|A^n(\omega)\|_\ast
=
\max\{\|A^n(\omega)|_E\|,\|A^n(\omega)|_L\|\}.
\]
By Proposition~\ref{prop:far-equator},
\[
\lambda_1(A|_L)=\lambda_1(A)
\qquad\text{and}\qquad
\lambda_1(A|_E)<\lambda_1(A).
\]


\begin{lem}\label{claim:E-upper-bound-split}
For every \(a\) with $\lambda_1(A|_E)<a<\lambda_1(A)$,
there exist constants \(C>0\) and \(c>0\) such that
\[
\mathbb P\bigl(\|A^n(\omega)|_E\|\ge e^{an}\bigr)\le Ce^{-cn}
\qquad\text{for all }n\ge1.
\]
\end{lem}

\begin{proof} Set $Y_n(\omega) \eqdef \log\|A^n(\omega)|_E\|$. Since $\lambda_1(A|_E)=\inf_{n \geq 1} \frac{1}{n} \int Y_n(\omega)\, d\mathbb{P}$, we can choose \(\ell\ge1\) such that
\begin{equation} \label{eq:menorquea}
    \frac1\ell\int Y_\ell(\omega)\,d\mathbb P<a.
\end{equation}
As in the proof of Proposition~\ref{prop:alpha0}, we have the following:
\begin{claim} It holds that
    \[
\lim_{\alpha\to 0}\frac1{\alpha\ell}\log \int e^{\alpha Y_\ell(\omega)}\,d\mathbb P
=
\frac1\ell\int Y_\ell(\omega)\,d\mathbb P.
\]
\end{claim} \label{claim:derivada}
\begin{proof} Let $M(\omega)\eqdef \max\{\|A(\omega)\|,\|A(\omega)^{-1}\|\}$.
By the exponential moment condition, there exists \(\beta>0\) such that $\int M(\omega)^\beta\,d\mathbb P<\infty$.
Moreover,
\[
|Y_\ell(\omega)|
\le \sum_{j=0}^{\ell-1}\log M(\sigma^j\omega).
\]
Choose \(\eta>0\) so small that \(2\eta\le \beta\). Then
\[
e^{2\eta |Y_\ell(\omega)|}
\le \prod_{j=0}^{\ell-1} M(\sigma^j\omega)^{2\eta},
\]
and therefore, since \(\mathbb P=p^{\mathbb N}\) is a Bernoulli product measure,
\[
\int e^{2\eta |Y_\ell(\omega)|}\,d\mathbb P
\le
\prod_{j=0}^{\ell-1}\int M(\sigma^j\omega)^{2\eta}\,d\mathbb P
=
\left(\int M(\omega)^{2\eta}\,d\mathbb P\right)^\ell
<\infty.
\]
Moreover, since \(u\le \eta^{-1}e^{\eta u}\) for every \(u\ge0\), we also get
\[
|Y_\ell|e^{\eta |Y_\ell|}
\le \eta^{-1}e^{2\eta |Y_\ell|}\in L^1(\mathbb P).
\]
Now define
\[
\phi(\alpha)\eqdef \int e^{\alpha Y_\ell(\omega)}\,d\mathbb P,
\qquad |\alpha|\le \eta.
\]
Then \(\phi(\alpha)<\infty\) for \(|\alpha|\le\eta\), and since $|e^t-1|\leq |t|e^{|t|}$,  we have that 
\[
\left|\frac{e^{\alpha Y_\ell}-1}{\alpha}\right|
\le |Y_\ell|e^{|\alpha||Y_\ell|}
\le |Y_\ell|e^{\eta |Y_\ell|}\in L^1(\mathbb P).
\]
Hence, by dominated convergence,
\[
\lim_{\alpha\to0}\frac{\phi(\alpha)-1}{\alpha}
= \int \lim_{\alpha\to0} \frac{e^{\alpha Y_\ell(\omega)}-1}{\alpha} \, d\mathbb P =
\int Y_\ell(\omega)\,d\mathbb P.
\]
Since \(\phi(0)=1\), it follows that
\[
\phi(\alpha)=1+\alpha\int Y_\ell\,d\mathbb P+o(\alpha).
\]
Setting $u=\alpha \int Y_\ell \, d\mathbb{P}+o(\alpha)$ we have $\phi(\alpha)=1+u$. Applying the expansion \(\log(1+u)=u+O(u^2)\) as $u\to 0$ and using that $u=O(\alpha)$, we have $O(u^2)=O(\alpha^2)$ and
\[
\log \phi(\alpha) = \alpha\int Y_\ell\,d\mathbb P+o(\alpha)+O(\alpha^2)=\alpha\int Y_\ell\,d\mathbb P + o(\alpha).
\]
Therefore,
\[
\lim_{\alpha\to 0}\frac1{\alpha\ell}\log \phi(\alpha)\,d\mathbb P
=
\frac1\ell\int Y_\ell(\omega)\,d\mathbb P,
\]
concluding the proof of the claim.
\end{proof}

Combining~\eqref{eq:menorquea} with Claim~\eqref{claim:derivada}, we get, for \(\alpha>0\) sufficiently small that
\[
0<q\eqdef e^{-\alpha a\ell}\int e^{\alpha Y_\ell(\omega)}\,d\mathbb P
=
e^{-\alpha a\ell}\int \|A^\ell(\omega)|_E\|^\alpha\,d\mathbb P <1
\]
Write \(n=k\ell+r\), where \(0\le r<\ell\). By submultiplicativity,
\[
\|A^n(\omega)|_E\|
\le
\|A^r(\sigma^{k\ell}\omega)|_E\|
\prod_{j=0}^{k-1}\|A^\ell(\sigma^{j\ell}\omega)|_E\|.
\]
Taking \(\alpha\)-th powers and expectations, and using independence, we obtain
\[
\int \|A^n(\omega)|_E\|^\alpha \, d\mathbb P
\le
C_r\left(\int \|A^\ell(\omega)|_E\|^\alpha\,d\mathbb P\right)^k
\]
where
\[
C_r\eqdef \int \|A^r(\omega)|_E\|^\alpha\,d\mathbb P<\infty.
\]
Markov's inequality now gives
\begin{align*}
\mathbb P\bigl(\|A^n(\omega)|\|\ge e^{an}\bigr)
&\le
e^{-\alpha an}\,\int \|A^n(\omega)|_E\|^\alpha \, d\mathbb P \\
&\le
e^{-\alpha a(k\ell+r)}\,C_r
\left(\int \|A^\ell(\omega)|_E\|^\alpha\,d\mathbb P\right)^k \\
&=
e^{-\alpha ar}C_r\,q^k.
\end{align*}
Since \(k\ge n/\ell-1\) and \(0<q<1\), this yields
\[
\mathbb P\bigl(\|A^n(\omega)|_E\|_E\ge e^{an}\bigr)\le Ce^{-cn}
\]
for suitable constants \(C,c>0\).
\end{proof}

We next establish the corresponding limit theorems for the cocycle \(A|_L\). If \(\dim L=1\), let \(x\) be a unit vector spanning \(L\). Since \(L\) is \(A\)-invariant, there exists a measurable map \(\rho:T\to\mathbb R\setminus\{0\}\) such that
\[
A(t)x=\rho(t)x
\qquad\text{for \(p\)-a.e.~\(t\in T\)}.
\]
Set
\[
X_j(\omega)\eqdef \log|\rho(\omega_j)|,
\qquad
S_n(\omega)\eqdef \sum_{j=0}^{n-1}X_j(\omega).
\]
Then \[
\log \|A^n(\omega)|_L\|=S_n(\omega).
\]
Since
\[
|\rho(t)|=\|A(t)|_L\|\le \|A(t)\|,
\qquad
|\rho(t)|^{-1}=\|A(t)|_L^{-1}\|\le \|A(t)^{-1}\|,
\]
the random variable \(X_0\) has finite exponential moments in a neighborhood of the origin. Moreover,
$\lambda_1(A)=\lambda_1(A|_L)=\mathbb E[X_0]$.
Hence the classical central limit theorem and Cram\'er's theorem apply to \((X_j)_{j\ge0}\). In particular, 
there exists \(\sigma\ge0\) such that
\begin{equation} \label{eq:CLT-ultimo}
    \frac{\log\|A^n(\omega)|_L\|-n\lambda_1(A)}{\sqrt n}
\xrightarrow[n\to\infty]{\rm law}\mathcal N(0,\sigma^2),
\end{equation}
and there are \(C_0,h_0>0\) such that, for every sufficiently small \(\varepsilon>0\),
\begin{equation}\label{eq:LD-L-unified}
\mathbb P\left(\left|\frac1n\log\|A^n(\omega)|_L\|-\lambda_1(A)\right|>\varepsilon\right)
\le C_0e^{-h_0n\varepsilon^2}
\end{equation}
for all large \(n\). If \(\sigma>0\), the classical Berry--Esseen theorem also gives
\begin{equation}\label{eq:BE-L-unified}
\sup_{u\in\mathbb R}
\left|
\mathbb P\left(
\frac{\log\|A^n(\omega)|_L\|-n\lambda_1(A)}{\sigma\sqrt n}\le u
\right)-\Phi(u)
\right|
=
O(n^{-1/2}),
\end{equation}
where \(\Phi\) denotes the standard Gaussian distribution function.

Assume now that \(\dim L>1\). Since \(X\subset P(\mathbb R^m)\) spans \(L\), we have \(X\subset P(L)\). We claim that \(A|_L\) is quasi-irreducible. Indeed, if \(H\subset L\) were a non-trivial \(A\)-invariant subspace such that
$\lambda_1(A|_H)<\lambda_1(A|_L)=\lambda_1(A)$,
then \(E\oplus H\) would be an \(A\)-invariant subspace strictly containing \(E\), and
\[
\lambda_1(A|_{E\oplus H})
=
\max\{\lambda_1(A|_E),\lambda_1(A|_H)\}
<
\lambda_1(A),
\]
contradicting the maximality of the equator \(E\). Thus \(A|_L\) is quasi-irreducible. Since
\[
\lambda_2(A|_L)\le \lambda_2(A)<\lambda_1(A)=\lambda_1(A|_L),
\]
Theorem~\ref{thm:CLT-linear-general} applies to \(A|_L\) and the compact \(f_{A|_L}\)-invariant set \(X\subset P(L)\). Therefore 
it holds~\eqref{eq:CLT-ultimo},~\eqref{eq:LD-L-unified} and~\eqref{eq:BE-L-unified}.

Choose \(a\) such that $\lambda_1(A|_E)<a<\lambda_1(A)$,
and so close to \(\lambda_1(A)\) that \(\delta\eqdef \lambda_1(A)-a\) lies in the small-deviation regime of~\eqref{eq:LD-L-unified}. Then
\[
\bigl\{\|A^n(\omega)|_L\|\le e^{an}\bigr\}
=
\left\{\frac1n\log\|A^n(\omega)|_L\|-\lambda_1(A)\le -\delta\right\},
\]
and hence~\eqref{eq:LD-L-unified} gives
\begin{equation}\label{eq:L-lower-tail-unified}
\mathbb P\bigl(\|A^n(\omega)|_L\|\le e^{an}\bigr)\le Ce^{-cn}
\end{equation}
for suitable constants \(C,c>0\) and all large \(n\).

Now define
\[
B_n\eqdef \bigl\{\|A^n(\omega)|_E\|\ge \|A^n(\omega)|_L\|\bigr\}.
\]
Then
\[
B_n\subset
\bigl\{\|A^n(\omega)|_E\|\ge e^{an}\bigr\}
\cup
\bigl\{\|A^n(\omega)|_L\|\le e^{an}\bigr\}.
\]
By Lemma~\ref{claim:E-upper-bound-split} and~\eqref{eq:L-lower-tail-unified}, there exist constants \(C,c>0\) such that
\begin{equation}\label{eq:bad-event-unified}
\mathbb P(B_n)\le Ce^{-cn}
\qquad\text{for all large }n.
\end{equation}
On the complement of \(B_n\), one has $\|A^n(\omega)\|_\ast=\|A^n(\omega)|_L\|$. Set
\[
U_n(\omega)\eqdef \frac{\log\|A^n(\omega)\|_\ast-n\lambda_1(A)}{\sqrt n},
\qquad
V_n(\omega)\eqdef \frac{\log\|A^n(\omega)|_L\|-n\lambda_1(A)}{\sqrt n}.
\]
Since \(U_n=V_n\) on \(\Omega\setminus B_n\), for every \(u\in\mathbb R\),
\[
\bigl|\mathbb P(U_n\le u)-\mathbb P(V_n\le u)\bigr|
\le \mathbb P(B_n).
\]
By~\eqref{eq:bad-event-unified}, the right-hand side tends to \(0\). Since \(V_n\) converges in law to \(\mathcal N(0,\sigma^2)\), the same is true of \(U_n\). This proves~\eqref{CLT-norm} for the adapted norm. Morevor, assuming \(\sigma>0\) and dividing the previous variables by \(\sigma\), the same comparison gives
\[
\sup_{u\in\mathbb R}\left|\mathbb P(U_n/\sigma\le u)-\mathbb P(V_n/\sigma\le u)\right|
\le \mathbb P(B_n).
\]
Combining this with~\eqref{eq:BE-L-unified} and~\eqref{eq:bad-event-unified}, we obtain~\eqref{BE-norm} for the adapted norm, and hence for the original norm.

Similarly, if
\[
\widetilde U_n(\omega)\eqdef \frac1n\log\|A^n(\omega)\|_\ast-\lambda_1(A),
\qquad
\widetilde V_n(\omega)\eqdef \frac1n\log\|A^n(\omega)|_L\|-\lambda_1(A),
\]
then \(\widetilde U_n=\widetilde V_n\) on \(\Omega\setminus B_n\), and therefore
\[
\mathbb P\bigl(|\widetilde U_n|>\varepsilon\bigr)
\le
\mathbb P\bigl(|\widetilde V_n|>\varepsilon\bigr)+\mathbb P(B_n).
\]
Using~\eqref{eq:LD-L-unified} and~\eqref{eq:bad-event-unified}, we obtain~\eqref{LD-norm} for the adapted norm.

Finally, since \(F\subset L\) by Proposition~\ref{claim:F-invariant} and \(E\cap L=\{0\}\), we have
$E\cap F=\{0\}$. Therefore item~(2) of Corollary~\ref{cor:positive-sigma} gives that $\sigma=0$ if and only if  $\dim F=1$ and $\|A(\omega)|_F\|$ is $\mathbb P$-a.s.~constant. This concludes the proof.
\end{proof}
}

{
\begin{exap}[A diagonal counterexample to LD with optimal rate] \label{exap:LD}
Let \(\Omega=\{1,2,3,4\}^{\mathbb N}\) be endowed with the uniform Bernoulli measure
\[
\mathbb P=\left(\tfrac14\delta_1+\tfrac14\delta_2+\tfrac14\delta_3+\tfrac14\delta_4\right)^{\mathbb N},
\]
and define the locally constant cocycle \(A:\Omega\to \mathrm{GL}(2)\) by \(A(\omega)=A_{\omega_0}\), where
\[
A_1= \begin{pmatrix} e^{-5} & 0\\ 0 & 1 \end{pmatrix}, \quad
A_2= \begin{pmatrix} e^{5} & 0\\ 0 & 1 \end{pmatrix}, \quad
A_3= \begin{pmatrix} e^{-5} & 0\\ 0 & e^2 \end{pmatrix}, \quad
A_4= \begin{pmatrix} e^{5} & 0\\ 0 & e^2 \end{pmatrix}.
\]
Let us define the random variables
\[
Y_i(\omega)\eqdef \log \bigl(A(\omega_i)_{11}\bigr)\in\{-5,5\},
\qquad
X_i(\omega)\eqdef \log \bigl(A(\omega_i)_{22}\bigr)\in\{0,2\}.
\]
By construction, \((X_i)_{i\ge0}\) and \((Y_i)_{i\ge0}\) are two independent sequences of i.i.d. random variables taking their respective values with equal probabilities of $1/2$. Their expectations are $\mathbb E[X_0]=1$ and $\mathbb E[Y_0]=0$.
Consequently, the Lyapunov exponents of the cocycle are \(\lambda_1(A)=1\) and \(\lambda_2(A)=0\). The equator is \(E=\mathbb R e_1\), and the pinnacle subspace is \(F=\mathbb R e_2\). Because \(\|A(\omega)|_F\|=e^{X_0}\) is not \(\mathbb P\)-a.s.~constant, the positivity condition for the asymptotic variance is satisfied. Now, fix the maximum norm on \(\mathbb R^2\). Since all matrices \(A_i\) are diagonal with positive entries, the operator norm of the product is simply the maximum of the products of the diagonal entries. Taking the logarithm and dividing by $n$ yields
\[
\frac1n\log \|A^n(\omega)\|
=
\max\left\{\overline Y_n, \overline X_n\right\},
\]
where we denote the empirical means by
\[
\overline X_n\eqdef \frac1n\sum_{i=0}^{n-1}X_i,
\qquad
\overline Y_n\eqdef \frac1n\sum_{i=0}^{n-1}Y_i.
\]
For the normalized logarithmic norm to exhibit a large deviation exceeding \(\epsilon \in (0,1)\), the maximum of the two empirical means must exceed \(1+\epsilon\). This corresponds to a union of two events:
\[
\left\{\frac1n\log\|A^n(\omega)\|-\lambda_1(A)>\epsilon\right\}
=
\{\overline X_n>1+\epsilon\}\cup \{\overline Y_n>1+\epsilon\}.
\]
Let \(a_n(\epsilon)\eqdef \mathbb P(\overline X_n>1+\epsilon)\) and \(b_n(\epsilon)\eqdef \mathbb P(\overline Y_n>1+\epsilon)\). By standard probability bounds for unions of events, we have
\[
\max\{a_n(\epsilon),b_n(\epsilon)\}
\le
\mathbb P\left(\frac1n\log\|A^n(\omega)\|-\lambda_1(A)>\epsilon\right)
\le
a_n(\epsilon)+b_n(\epsilon).
\]
By Cram\'er's theorem~\cite[Thm.~23.3]{Klenke:2020}, the tail probabilities decay exponentially at rates given by their respective rate functions
\[
\lim_{n\to\infty}\frac1n\log a_n(\epsilon)=-I_X(1+\epsilon),
\qquad
\lim_{n\to\infty}\frac1n\log b_n(\epsilon)=-I_Y(1+\epsilon),
\]
where \(I_X\) and \(I_Y\) are the Cram\'er rate functions of \(X_0\) and \(Y_0\). Because the sum of two exponentially decaying sequences is asymptotically dominated by the term with the slower decay (i.e., the smaller rate function), we deduce that
\[
\lim_{n\to\infty}\frac1n
\log \mathbb P\left(\frac1n\log\|A^n(\omega)\|-\lambda_1(A)>\epsilon\right)
=
-\min\{I_X(1+\epsilon),I_Y(1+\epsilon)\}
\]
for every \(\epsilon\in(0,1)\). We now compute these two rate functions explicitly. Since \(X_0\in\{0,2\}\) with equal probabilities, standard convex duality~\cite[Thm.~23.1]{Klenke:2020} yields
\[
I_X(x)=
\begin{cases}
\displaystyle \frac{x}{2}\log x+\left(1-\frac{x}{2}\right)\log(2-x), & x\in[0,2],\\[2mm]
+\infty, & \text{otherwise}.
\end{cases}
\]
Similarly, since \(Y_0\in\{-5,5\}\) with equal probabilities,
\[
I_Y(y)=
\begin{cases}
\displaystyle \frac{1+y/5}{2}\log(1+y/5)+\frac{1-y/5}{2}\log(1-y/5), & y\in[-5,5],\\[2mm]
+\infty, & \text{otherwise}.
\end{cases}
\]
Thus, the upper-tail rate function for the norm is given by the minimum
\[
c(\epsilon)=\min\{I_X(1+\epsilon),I_Y(1+\epsilon)\},
\qquad \epsilon\in(0,1).
\]
Crucially, this function is \emph{not convex}. To see this, we evaluate $c(\epsilon)$ at three points:
\[
c(0.1)=\min\{I_X(1.1), I_Y(1.1)\} = I_X(1.1)\approx 0.005,
\]
\[
c(0.5)=\min\{I_X(1.5), I_Y(1.5)\} = I_Y(1.5)\approx 0.045,
\]
\[
c(0.9)=\min\{I_X(1.9), I_Y(1.9)\} = I_Y(1.9)\approx 0.074.
\]
By evaluating the midpoint of the secant line, we find:
\[
c(0.5)\approx 0.045
>
\frac{c(0.1)+c(0.9)}{2}
\approx 0.039.
\]
Because the function value at the midpoint strictly exceeds the average of the endpoints, \(c(\epsilon)\) is not convex, and therefore certainly not strictly convex. 
\end{exap}
}

The following consequence of the previous results provides a clean sufficient condition to get the limit theorems with positive asymptotic variance for non-quasi-irreducible cocycles in $\mathrm{GL}(m)$ with $m\geq 3$.  

\begin{cor} \label{prop:suf-cond}
Let $A \in \LP$ be a locally constant linear cocycle with $m\geq 3$, and denote by $E$  its equator. Assume that
\begin{enumerate}[leftmargin=1cm]
    \item $A$ satisfies the exponential moment condition~\eqref{eq:exponential-moment-condition},
     \item $\lambda_1(A) > \lambda_2(A)$, 
    \item there exists an $f_A$-invariant compact set $X\subset  P(\mathbb{R}^m) \setminus E$, 
   \item {$\dim E = 1$ and $A|_{\mathbb{R}^m/E}$ is irreducible.}
\end{enumerate}
Then, there exists $\sigma > 0$ such that \eqref{LD-norm} and \eqref{BE-norm} hold. 
\end{cor}

\begin{proof}
 Since condition (ii) and (iii) hold, Proposition~\ref{claim:F-invariant} provides a unique \(f_A\)-stationary measure \(\mu\) such that \(\mu(E)=0\), whose support
$S\subset P(\mathbb R^m)\setminus E$
is compact and \(f_A\)-invariant. Moreover, the pinnacle subspace \(F\) of \(A\) is the subspace spanned by \(S\). Taking
$X=S$, we have $L=\langle x\in \mathbb R^m:\hat x\in X\rangle=F$.

{Denote by $\tilde{A}=A|_{\mathbb{R}^m/E}$ the quotient cocycle induced by $A$ on $\mathbb{R}^m/E$. Let \(\tilde F\) be the pinnacle of \(\tilde A\). Since \(\tilde A\) is irreducible, $\tilde F=\mathbb R^m/E$. On the other hand, by Proposition~\ref{prop:old-implies-new}, we have $\tilde F=(F+E)/E$. Therefore
$\mathbb R^m/E=(F+E)/E$, which implies $F+E=\mathbb R^m$. Since \(\dim E=1\), either \(E\subset F\) or \(E\cap F=\{0\}\), and 
$$m=\dim F+1-\dim E\cap F.$$ Hence, $\dim F-\dim E\cap F \geq m-1\geq 2$ and hence $\sigma>0$ by Proposition~\ref{cor:positive-sigma}.

If \(E\subset F\), then $\mathbb R^m=F+E=F$, 
so \(L=F=\mathbb R^m\). Thus, Theorem~\ref{thm:CLT-linear-general} implies \eqref{BE-norm} and \eqref{LD-norm} --actually, in this case we get the LD with optimal rate.
If \(E\cap F=\{0\}\), then $\mathbb R^m=F\oplus E=  L\oplus E$ and $\dim F= m-1\geq 2$. Thus, Proposition~\ref{cor:CLT-geral} is satisfied. This concludes the proof.}
\end{proof}

The next example illustrates Corollary~\ref{prop:suf-cond} in a genuinely non-quasi-irreducible setting.

{
\begin{exap}[A Heisenberg-type cocycle]\label{ex:heisenberg}
Consider the locally constant cocycle \(A:\Omega\to \mathrm{GL}(3)\) generated by
\[
g_1=
\begin{pmatrix}
\frac12 & 2 & 3\\
0 & 1 & 1\\
0 & 0 & 1
\end{pmatrix},
\qquad
g_2=
\begin{pmatrix}
\frac12 & 1 & 2\\
0 & 0 & -1\\
0 & 1 & 0
\end{pmatrix},
\]
that is,  $A(\omega)=\omega_0$ for $\omega=(\omega_i)_{i\geq 0}\in \Omega=\mathrm{GL}(3)^\mathbb{N}$ and Bernoulli law \(\mathbb{P}=p^\mathbb{N}\), where \(p=\frac12(\delta_{g_1}+\delta_{g_2})\). These matrices belong to the subgroup
\[
G=
\left\{
\begin{pmatrix}
a & b\\
0 & C
\end{pmatrix}
:\ a\in\mathbb R\setminus\{0\},\ b\in\mathbb R^2,\ C\in \mathrm{GL}(2,\mathbb R)
\right\}\subset \mathrm{GL}(3),
\]
which models the automorphism group of the Heisenberg group; see~\cite[Example~3]{aoun2020random}. In this framework, the one-dimensional subspace
$E=\mathbb Re_1$
is the equator. The quotient cocycle \(A|_{\mathbb R^3/E}\) is generated by the matrices
\[
C_1=
\begin{pmatrix}
1 & 1\\
0 & 1
\end{pmatrix},
\qquad
C_2=
\begin{pmatrix}
0 & -1\\
1 & 0
\end{pmatrix}.
\]
These are exactly the quotient matrices appearing in~\cite[Example~1 in \S5.2.2]{aoun2020random}, where Aoun and Guivarc'h prove that the corresponding semigroup is a non-compact strongly irreducible subsemigroup of \(\mathrm{SL}(2)\). Since Example~2 in~\cite[\S5.2.2]{aoun2020random} has the same projection on \(\mathrm{GL}(\mathbb R^3/E)\), it follows that \(A|_{\mathbb R^3/E}\) is irreducible. Moreover, for this pair \(g_1,g_2\), Aoun and Guivarc'h prove that $\lambda_1(A)>\lambda_2(A)$
and that the support of the unique stationary measure \(\mu\) on
$P(\mathbb R^3)\setminus E$
is compact; see~\cite[Example~2 in \S5.2.2]{aoun2020random}. Denote this support by
$X=\operatorname{supp}\mu$.
Then \(X\subset P(\mathbb R^3)\setminus E\) is a compact \(f_A\)-invariant set. Since \(\dim E=1\), all the hypotheses of Corollary~\ref{prop:suf-cond} are satisfied. Therefore this Heisenberg-type cocycle satisfies the Berry--Esseen estimate~\eqref{BE-norm} and the large deviation principle~\eqref{LD-norm} for \(\log\|A^n(\omega)\|\).
\end{exap}
}


\section{Proof of Propositions~\ref{mainprop:CLT-LD} and~\ref{mainprop:BE}}

Since quasi-irreducibility implies that $E$ is trivial and $L = \mathbb{R}^m$, Propositions~\ref{mainprop:CLT-LD} and~\ref{mainprop:BE} follow immediately from Theorem~\ref{thm:CLT-linear-general} and Corollary~\ref{cor:positive-sigma}. 

\section{Proof of Propositions~\ref{mainprop:CLT-LD2} and~\ref{mainprop:BE2}}
 Let $A\in \mathrm{L}_{\mathbb{P}}(2)$ be a locally constant cocycle with $\lambda_1(A)>\lambda_2(A)$ and satisfying the finite exponential moment~\eqref{eq:exponential-moment-condition}.  We analyze the cases provided by Corollary~\ref{prop:dichotomy}. If $A$ is quasi-irreducible, by~Propositions~\ref{mainprop:CLT-LD} and~\ref{mainprop:BE}, we obtain the CLT, LD and Berry-Esseen concluding the propositions in this case. {Moreover, since in this case the equator is trivial, Proposition~\ref{mainprop:BE} gives
 \[
\sigma=0
\iff
\dim F=1 \ \text{and}\ \|A(\omega)|_F\|\ \text{is $\mathbb{P}$-a.s.~constant}.
\]
Hence Proposition~\ref{mainprop:BE2} also follows in this case.}
 If $A^T$ is quasi-irreducible, observe that since for an adequate norm $\|M^T\|=\|M\|$ and $(M^{T})^{-1}=(M^{-1})^T$ for any $M\in\mathrm{GL}(2)$, the finite exponential moment condition~\eqref{eq:exponential-moment-condition} also holds for $A^T$. 
 Furthermore,  $\lambda_1(A^T)=\lambda_1(A)>\lambda_2(A)=\lambda_2(A^T)$, cf.~\cite{key1988lyapunov}, 
  Moreover, since $$\lambda_1(A^T)=\lambda_1(A) \quad  \text{and} \quad 
 \|(A^T)^n(\omega)\|=
 \|A(\omega_{0})\cdot\ldots\cdot A(\omega_{n-1})\|
 $$
 and taking into account that $\mathbb{P}$ is a Bernoulli measure, 
 we have 
 \begin{align*}
     \mathbb{P}\bigg(\frac{1}{\sqrt{n}}\big(\log &\|(A^T)^n(\omega)\|-n\lambda_1(A^T)\big)\leq \alpha\bigg)  \\ &= \mathbb{P}\bigg(\frac{1}{\sqrt{n}}\big(\log \|A(\omega_{0})\cdot\ldots\cdot A(\omega_{n-1})\|-n\lambda_1(A)\big)\leq \alpha\bigg) \\
     &=\mathbb{P}\bigg(\frac{1}{\sqrt{n}}\big(\log \|A^n(\omega)\|-n\lambda_1(A)\big)\leq \alpha\bigg).
  \end{align*}
 This implies that CLT and Berry-Esseen also hold for $A$ (i.e.~for $\log \|A^n(\omega)\|$). A similar argument works for LD for $\log \|A^n(\omega)\|$. {Furthermore, Proposition~\ref{mainprop:BE} applied to $A^T$ yields
\[
\sigma=0
\iff
\dim F(A^T)=1 \ \text{and}\ \|A^T(\omega)|_{F(A^T)}\|\ \text{is $\mathbb{P}$-a.s.~constant}.
\]
If \(L^\ast=F(A^T)\), then \(E=(L^\ast)^\perp\) is a one-dimensional \(A\)-invariant subspace. Since, in dimension one, the quotient cocycle \(A|_{\mathbb{R}^2/E}\) is dual to \(A^T|_{L^\ast}\), we have
\[
\|A^T(\omega)|_{L^\ast}\|=\|A(\omega)|_{\mathbb{R}^2/E}\|
\qquad\text{for $\mathbb{P}$-a.e.~$\omega$}.
\]
Hence
\[
\sigma=0
\iff
\dim E=1 \ \text{and}\ \|A(\omega)|_{\mathbb{R}^2/E}\|\ \text{is $\mathbb{P}$-a.s.~constant},
\]
which proves Proposition~\ref{mainprop:BE2} in this case.}

To complete the proof of Propositions~\ref{mainprop:CLT-LD2} and~\ref{mainprop:BE2} remains to analyze the diagonal case. The following result also concludes this case (cf.~\cite[Theorem~2.1]{duarte2020large} for LD of compactly supported diagonalizable cocycles).

 \begin{cor}
 Consider a locally constant linear cocycle $A\in \mathrm{L}_{\mathbb{P}}(2)$ and assume: 
 \begin{enumerate}
     \item $A$ satisfies an exponential moment condition~\eqref{eq:exponential-moment-condition},
     \item $\lambda_1(A) > \lambda_2(A)$, and 
     \item $A$ is diagonalizable. 
  \end{enumerate}
 Then~\eqref{CLT-norm} and~\eqref{LD-norm} hold. 
 {Moreover,
\[
\sigma=0
\iff
\|A(\omega)|_F\|\ \text{is $\mathbb{P}$-a.s.~constant}
\]
where $F$ is the pinnacle of $A$. 
 In particular, if \(\sigma>0\), then~\eqref{LD-norm} and~\eqref{BE-norm} hold.}
\end{cor}
\begin{proof}
{By assumption~(iii), there exists \(P\in \mathrm{GL}(2)\) such that
\[
P^{-1}A(\omega)P=\mathrm{diag}(a(\omega_0),b(\omega_0))
\qquad\text{for $\mathbb{P}$-a.e.~$\omega\in\Omega$}.
\]
Hence \(A\) admits two one-dimensional invariant subspaces. After possibly permuting the diagonal coordinates, we may assume that the first eigendirection \(F\) satisfies
\[
\lambda_1(A|_F)=\lambda_1(A)
\qquad\text{and}\qquad
\lambda_1(A|_E)=\lambda_2(A),
\]
where \(E\) is the other eigendirection. Since \(\lambda_1(A)>\lambda_2(A)\), it follows that \(F\) is the pinnacle of \(A\), \(E\) is the equator of \(A\), and
$\mathbb{R}^2=E\oplus F$. Now set
\[
X=P(F)\subset P(\mathbb{R}^2)\setminus E.
\]
Then \(X\) is compact and \(f_A\)-invariant, and the subspace \(L\) spanned by \(X\) is precisely \(F\). Therefore, all assumptions of Proposition~\ref{cor:CLT-geral} are satisfied. Consequently,~\eqref{CLT-norm} and~\eqref{LD-norm} hold.

Moreover, since \(\dim F=1\), the criterion for the positivity of the variance given in Proposition~\ref{cor:CLT-geral} becomes
$\sigma>0$ if and only if $\|A(\omega)|_F\|$ is not $\mathbb{P}$-a.s.~constant. Finally, if \(\sigma>0\), Proposition~\ref{cor:CLT-geral} also yields~\eqref{LD-norm} and~\eqref{BE-norm}.}
\end{proof}

  {This concludes the proof of Propositions~\ref{mainprop:CLT-LD2} and~\ref{mainprop:BE2}.}

\chapter{Continuity of Lyapunov exponents of linear cocycles}
\label{chap:continuity-linear}

\abstract{{
This chapter studies the regularity of the top Lyapunov exponent for locally
constant linear cocycles. We first introduce the topologies on spaces of
measurable cocycles used in the sequel and compare them with the Lipschitz
topologies of the induced projective random maps. This allows us to transfer
the robustness and statistical stability results for mostly contracting random
maps to the projective actions of linear cocycles. We then prove continuity of the first Lyapunov exponent  \(\lambda_1(\cdot)\) at cocycles $A$ for which either \(A\) or its
transpose is quasi-irreducible. Under the simplicity of the first Lyapunov exponent and
an exponential moment condition, we prove the robustness of quasi-irreducibility and establish local H\"older continuity of \(\lambda_1(\cdot)\) at quasi-irreducible or with transpose quasi-irreducible cocycles. In the compactly supported setting, the same method
also applies where the maximizing projective dynamics is localized on a compact attracting region away from the equator, that is, for almost-irreducible cocycles.
}}

\section{{Topology on the space of linear cocycles}}
We want to understand the regularity of the map $A \mapsto f_A$ or, more specifically, the inclusion of $\LP$ into $\EE(P(\mathbb{R}^m))$. To do this, we recall some previously introduced metrics. Given $A, B \in \LP$,
\begin{align*}
\Dp(A,B) &\eqdef \int \|A(\omega) - B(\omega)\| \, d\mathbb{P}  
\\ 
\Dpm(A,B) &\eqdef \int \|A(\omega) - B(\omega)\| + \|A(\omega)^{-1} - B(\omega)^{-1}\|\, d\mathbb{P}.
\end{align*}
Moreover, for each $n \geq 1$ and $0<\beta \leq 1 $, 
\begin{equation*} \label{eq:distancia-random-maps}
 \Dbeta(f_A, f_B) \eqdef \max_{i = 1, \dots, n} \int d_{\mathrm{L}}((f_A)^i_\omega, (f_B)^i_\omega)^{\beta_i} \, d\mathbb{P},
  \end{equation*}
  where $\beta_i = \beta/2^{i-1}$ and 
  \begin{equation*} 
      d_{\mathrm{L}}((f_A)^i_\omega, (f_B)^i_\omega) =  \sup_{\hat{x} \in P(\mathbb{R}^m)} \big(d(A^i(\omega)\hat{x}, B^i(\omega)\hat{x}) + |L(f_A)^i_\omega(\hat{x}) - L(f_B)^i_\omega(\hat{x})|\big).
  \end{equation*}
  
  \noindent{Here we do not include a separate term involving the global Lipschitz constants. Indeed, since \(P(\mathbb{R}^m)\) is a connected compact Riemannian manifold, it is a compact locally length space, and Proposition~\ref{prop:lengthspace} implies that $\mathrm{Lip}(g)=\sup_{\hat{x}\in P(\mathbb{R}^m)} Lg(\hat{x})$ for every Lipschitz map \(g:P(\mathbb{R}^m)\to P(\mathbb{R}^m)\). 
Consequently, for any two Lipschitz maps \(g,h\),
\[
|\mathrm{Lip}(g)-\mathrm{Lip}(h)|
\le \sup_{\hat{x}\in P(\mathbb{R}^m)} |Lg(\hat{x})-Lh(\hat{x})|,
\]
so the distance between the global Lipschitz constants is already controlled by the local term in \(d_{\mathrm L}\).}

\begin{lem} \label{lemma:2exponents}
For any $0<\alpha \leq \min\{\beta, 1\}$ and for $A,B\in \LP$, it holds that 
\begin{align*}
\int {\|B(\omega)\|^{\alpha}} \,d\mathbb{P} &\leq 1 +\int \|A(\omega)\|^\beta \, d\mathbb{P} +  \Dp(A,B)^{\alpha} 
\end{align*}
and
    \begin{align*}
\int {\|B(\omega)^{-1}\|^{\alpha}} \,d\mathbb{P} &\leq 1+ \int \|A(\omega)^{-1}\|^\beta \, d\mathbb{P} +  \D(A,B)^{\alpha}.
\end{align*}
\end{lem}
\begin{proof}
    Let us prove the first inequality. The second inequality follows similarly. To do this,  observe that, since $0<\alpha\leq 1$ and $t^\alpha - s^\alpha \leq |t-s|^\alpha$ for any $t,s\geq 0$, it holds that 
\begin{align*}
\int \|B(\omega)\|^\alpha  \, d\mathbb{P} 
\leq \int \|A(\omega)\|^\alpha \, d\mathbb{P} + 
\int \big| \|A(\omega)\|-\|B(\omega)\|\big|^\alpha \, d\mathbb{P}.
\end{align*}
From this, taking into account that $\alpha\leq \beta$,  $\big|\|A(\omega)\|- \|B(\omega)\|\big| \leq \|A(\omega) - B(\omega) \|$ and using the Jensen inequality, we conclude the required inequality.
\end{proof}

  The following result is a modification of Lemma~\ref{claim:lipbeta} in this context.
  
\begin{lem} \label{lem:convergencia}
Let $n \geq 1$ and consider $A \in \LP$ satisfying~\eqref{eq:exponential-moment-condition} for some $0<4\beta \leq 2$. Then there is $K > 0$ such that $$\Dbeta(f_A, f_B) \leq K \Dp(A, B)^{\beta_n}$$ for any $B$ close enough to $A$ in $(\LP, \Dpm)$. That is, the inclusion 
$$(\LP, \Dpm) \hookrightarrow (\EE(P(\mathbb{R}^m)), \Dbeta)$$ is continuous at $A$. 
\end{lem}
\begin{proof} For simplicity of notation, we denote $f=f_A$ and $g=f_B$. For each $i\geq 1$ and $\hat x\in P(\mathbb{R}^m)$, we have  $Lf^i_\omega(\hat x)=\|Df^i_\omega(\hat x)\|$. Hence,  $|Lf^i_\omega(\hat x)-Lg^i_\omega(\hat x)|\leq \|Df^i_\omega(\hat x)-Dg^i_\omega(\hat x)\|$. 
This implies 
$d_{\mathrm{L}}(f_\omega^i,g^i_\omega)\leq d_{C^1}(f^i_\omega,g^i_\omega)$. 
Then \begin{equation} \label{eq:Dbeta2}
    \Dbeta(f,g)\leq \max_{i=1,\dots,n} \,\int d_{C^1}(f^i_\omega,g^i_\omega)^{\beta_i}\, d\mathbb{P}.
\end{equation} 

\begin{claim} \label{claim:contas} For every $N,M\in \mathrm{GL}(m)$, $\hat{x} \in P(\mathbb{R}^m)$ and $v\in\mathbb{R}^m$ unitary, 
    $$\left\|\frac{Nv}{\|Nx\|}-\frac{Mv}{\|Mx\|}\right\|\leq \big(\|N^{-1}\|+ \|N^{-1}\| \, \|M^{-1}\|\, \|M\|\big)\cdot \|N-M\|
    $$
    and
    $$
    \Dp(N\hat{x},M\hat{x})\leq 2 \|N^{-1}\| \, \|N-M\|.
    $$
\end{claim}
\begin{proof}
To shorten notation, we write $a=\|Nx\|$ and 
$b=\|Mx\|$. Then, since $v\in\mathbb{R}^m$ is unitary, 
$$
\big\|\textstyle\frac{1}{a}Nv-\frac{1}{b} Mv\big\|\leq \frac{1}{a} \big\|N-M\big\| + \big|\frac{1}{a}-\frac{1}{b}\big|\cdot \big\|Mv\big\|.
$$
Moreover, since $x$ is also unitary, we have  $|\frac{1}{a}-\frac{1}{b}| =\frac{1}{ab}|a-b|\leq \frac{1}{ab}\|N-M\|$. Then, using that $a^{-1}\leq \|N^{-1}\|$ and $b^{-1}\leq \|M^{-1}\|$, we get the first inequality in the statement of the claim. If $v=x$, we can get a better estimate since  $|\frac{1}{a}-\frac{1}{b}|\cdot \|Mv\| \leq \frac{1}{a}\|N-M\|$. Namely, we get 
$$
\left\|\frac{Nx}{\|Nx\|}-\frac{Mx}{\|Mx\|}\right\|\leq 
    2\|N^{-1}\| \cdot \|N-M\| \quad \text{if $\hat{v}=\hat{x}$} 
$$ 
Since the left-hand side of the above inequality is greater than or equal to $d(N\hat{x},M\hat{x})$, we also conclude the second inequality of the claim.  
\end{proof}

\begin{claim} \label{claim:induction2} For each $n\geq 1$, there exists $K>0$ such that $$\int d_{C^1}(f^i_\omega,g^i_\omega)^{\beta_i} \, d\mathbb{P}\leq K \Dp(A,B)^{\beta_i} \quad \text{for all $i=1,\dots,n$.}
$$
\end{claim}
\begin{proof} The proof of this result is a straightforward modification of the induction in Claim~\ref{claim:induction}, but it requires checking some assumptions and steps of the proof. 

First, we have to prove that the integrability condition~\eqref{eq:momento-dif} is satisfied. To see this, by Lemma~\ref{lem:lip-fA}, we have $\mathrm{Lip}(f_\omega) \leq \max\{\|A(\omega)\|,\|A^{-1}(\omega)\|\}^4$  and since we are assuming the exponential moment condition~\eqref{eq:exponential-moment-condition} for the exponent $4\beta>0$, then  $(\|Df_\omega\|_\infty)^\beta=\mathrm{Lip}(f)^{\beta}$ is $\mathbb{P}$-integrable. This concludes~\eqref{eq:momento-dif}.

We also have to modify the induction assumption of Claim~\ref{claim:induction} in accordance with the corresponding inductive step here. Moreover, the proof also requires the inductive step for $n=1$, which does not hold by the same argument as in Claim~\ref{claim:induction}. In what follows, we will prove this step.

We first prove that $ \int d_{C^0}(f_\omega,g_\omega)^{\beta} \, d\mathbb{P}\leq K \Dp(A,B)^{\beta}$ for some $K>0$. To see this, observe that by Claim~\ref{claim:contas} we have that
\begin{align*}
\int d_{C^0}(f_\omega,g_\omega)^{\beta}  d\mathbb{P} &= \int \bigg(\sup_{\hat x \in P(\mathbb{R}^m)} d\big(A(\omega)\hat{x},B(\omega)\hat{x}\big) \bigg)^\beta d\mathbb{P} \\ &\leq 
\int \big(2\big\|A(\omega)^{-1}\big\| \cdot \big\|A(\omega)-B(\omega)\big\|\big)^\beta \, d\mathbb{P}.
\end{align*}
Hence, by H\"older inequality, the integrability assumption~\eqref{eq:exponential-moment-condition} and Jensen inequality ($2\beta \leq 1$), we obtain that  
\begin{align*}
\int d_{C^0}&(f_\omega,g_\omega)^{\beta} \, d\mathbb{P} \leq \\   &\leq 2^\beta \bigg(\int \|A(\omega)^{-1}\|^{2\beta}\,d\mathbb{P}\bigg)^{\frac{1}{2}} \cdot \bigg(\int \|A(\omega)-B(\omega)\|^{2\beta}\,d\mathbb{P}\bigg)^{\frac{1}{2}} \leq K \Dp(A,B)^\beta 
\end{align*}
for some constant $K>0$. 

Now, we will prove that $\int (\|Df_\omega-Dg_\omega\|_\infty)^\beta \, d\mathbb{P} \leq K \Dp(A,B)^\beta$ for some $K>0$. To see this, 
observe that
$$
\left\|Df_\omega(\hat{x})v-Dg_\omega(\hat{x})v\right\| =\left\|\Pi_{A(\omega)\hat{x}}(u)-\Pi_{B(\omega)\hat{x}}(w)\right\| 
$$ 
with 
$$
u=\frac{A(\omega)v}{\|A(\omega)x\|} \ \ \text{and} \ \ w=\frac{B(\omega)v}{\|B(\omega)x\|}
$$
where $v$ is a unit vector in $\mathbb{R}^m$, cf.~\cite[Prop.~2.28]{duarte2016lyapunov}. Hence,
$$
\left\|\Pi_{A(\omega)\hat{x}}(u)-\Pi_{B(\omega)\hat{x}}(w)\right\| \leq \left\|\Pi_{A(\omega)\hat{x}}\right\|\,\big\|u-w\big\| + \left\|\Pi_{A(\omega)\hat{x}}-\Pi_{B(\omega)\hat{x}}\right\| \, \big\|w\big\|.
$$
From Claim~\ref{claim:contas}, 
$$
\|u-w\|\leq \big(\|A(\omega)^{-1}\| + \|A(\omega)^{-1}\|\, \|B(\omega)^{-1}\|\, \|B(\omega)\| \big) \cdot \|A(\omega)-B(\omega)\|. 
$$
We also have that
$$\|\Pi_{A(\omega)\hat{x}}-\Pi_{B(\omega)\hat{x}}\|=d(A(\omega)\hat x, B(\omega)\hat x) \leq 2 \|A(\omega)^{-1}\| \cdot \|A(\omega)-B(\omega)\|$$ and $\|w\| \leq \|B(\omega)^{-1}\|\, \|B(\omega)\|$.  Moreover, $\|\Pi_{A(\omega)\hat{x}}\|\leq 1$ because it is a projection. 
Putting together all of them, we get 
\begin{align*}
\|Df_\omega(\hat{x})v -Dg_\omega(\hat{x})v\| \leq \|A(\omega)^{-1}\|\, (1+\, 3 \, \|B(\omega)^{-1}\|  \, \|B(\omega)\| )\cdot\|A(\omega)-B(\omega)\|.
\end{align*}
From here, using, as before, H\"older and Jensen inequalities, the assumption~\eqref{eq:exponential-moment-condition} and Lemma~\ref{lemma:2exponents} for $B$ close enough to $A$, we can easily get an upper bound of the expectation of $(\|Df_\omega-Dg_\omega\|_\infty)^\beta$ of the form $K\cdot \Dp(A,B)^\beta$ where $K>0$ is a constant.   

Both estimates above together imply that $\int d_{C^1}(f_\omega,g_\omega)^\beta \,d\mathbb{P} \leq K \Dp(A,B)^\beta$, concluding the first step of the induction, which concludes  the proof. \end{proof}

Finally,~\eqref{eq:Dbeta2} and Claim~\ref{claim:induction2}  conclude the proof of the lemma.
\end{proof}

\section{Continuity of maximal Lyapunov exponent}
\index{Lyapunov exponents!continuity}
Here we prove the following result.

\begin{thm} \label{thm:continuity-quasi}
Let $A \in \LP$ be a quasi-irreducible 
locally constant linear cocycle
satisfying~\eqref{eq:exponential-moment-condition}.  
Then, $\lambda_1: (\LP, \Dp) \to \mathbb{R}$ is continuous at $A$.  
\end{thm}
\begin{proof} 
    Let $(B_n)_{n \geq 1}$ be a sequence of locally constant linear cocycles converging to $A$ in $(\LP, \Dp)$. To prove the continuity of $\lambda_1(\cdot)$ at $A$, we need to see that $\lambda_1(B_n) \to \lambda_1(A)$ as $n \to \infty$. Actually, to prove this convergence, it suffices to show that every subsequence of $(\lambda_1(B_n))_{n\geq 1}$ has a further subsequence which converges to $\lambda_1(A)$. 
    Indeed, if $\lambda_1(B_n) \not\to \lambda_1(A)$, there exists an $\epsilon > 0$, such that for all $k$, there exists an $n_k > k$ satisfying $|\lambda_1(B_{n_k}) - \lambda_1(A)| \geq \epsilon$ (since if there is some $k$ which does not have such $n_k$, then we can take it as $N$, so $\lambda_1(B_n)$ converges to $\lambda_1(A)$). Thus, the subsequence $(\lambda_1(B_{n_k}))_{k\geq 1 }$ does not have any subsequence converging to $\lambda_1(A)$, contradicting the assumption.

    Let $\mu_n$ be a maximizing stationary measure of $f_{B_n}$. That is, $\int \phi_{B_n} \, d\mu_n = \lambda_1(B_n)$. By taking a subsequence if necessary, we can assume that we have a sequence of $f_{B_n}$-stationary measures $\mu_n$ converging to a measure $\mu$. Then, from Lemma~\ref{lem:convergencia} and Theorem~\ref{thm:Weak statistically stability}, we have that $\mu$ is an $f_A$-stationary measure. Since $A$ is quasi-irreducible, from Proposition~\ref{prop:quasi-irreducible} it follows that $\mu$ is maximizing, i.e., $\lambda_1(A) = \int \phi_A \, d\mu$. Then,
\begin{align} 
\big|\lambda_1(B_n) - &\lambda_1(A)\big| = \big|\int \phi_{B_n}\, d\mu_n - \int \phi_A \, d\mu\big |  \notag\\ 
&\leq \big|\int \phi_{B_n}\, d\mu_n - \int \phi_A \, d\mu_n\big| + \big|\int \phi_A \, d\mu_n - \int \phi_A \, d\mu\big|. \label{eq:notag}
\end{align}
Since $\log^+ \|A^{\pm1}\|$ is $\mathbb{P}$-integrable, $\phi_A$ is a continuous real-valued function. Thus,  since $\mu_n \to \mu$ in the weak$^*$ topology, the second term in~\eqref{eq:notag} tends to zero as $n \to \infty$. On the other hand,
\begin{align*}
     \lim_{n \to \infty}
     \bigg| \int \phi_{B_n}\, d\mu_n - \int \phi_A\, d\mu_n\bigg| &\leq
  \lim_{n \to \infty}\int \left|\log \frac{\|B_n(\omega)\|}{ \|A(\omega)\|} \right| \, d\mathbb{P}. 
\end{align*}
According to Remark~\ref{rem:beta},  fix $0<\beta \leq 1$ in~\eqref{eq:exponential-moment-condition} and consider $0 < \alpha \leq \beta/2$.  Using that $\log t \leq \frac{1}{\alpha} (t^\alpha - 1)$ for $t > 0$ and $|t|^\alpha - |s|^\alpha \leq |t - s|^\alpha$ for $t, s \in \mathbb{R}$, we obtain 
\begin{align*}
    \int \bigg| \log \frac{\|B_n(\omega)\|}{\|A(\omega)\|} \bigg| \, d\mathbb{P} 
    &\leq \frac{1}{\alpha} \int \frac{\big| \|B_n(\omega)\| - \|A(\omega)\|\big|^\alpha}
    {\|A(\omega)\|^\alpha} \, d\mathbb{P} \\
    &\leq \frac{1}{\alpha} \int \frac{\| B_n(\omega) - A(\omega) \|^\alpha}
    {\|A(\omega)\|^\alpha} \, d\mathbb{P}\\
     &\leq \frac{1}{\alpha} \left(\int \|A(\omega)\|^{-2\alpha} \, d\mathbb{P}\right)^{1/2} \left(\int \| B_n(\omega) - A(\omega) \|^{2\alpha} \, d\mathbb{P}\right)^{1/2}. 
\end{align*}
The last inequality follows from H\"older's inequality. Moreover, using that $2\alpha \leq \beta$, the Jensen inequality, and since $\|M^{-1}\|^{-1} \leq \|M\|$ for any matrix $M$, we arrive at
\begin{align*}
    \int \bigg| \log \frac{\|B_n(\omega)\|}{\|A(\omega)\|} \bigg| \, d\mathbb{P} 
         &\leq \frac{1}{\alpha} \left( 1 + \int \|A(\omega)^{-1}\|^{\beta} \, d\mathbb{P}\right)^{1/2} \Dp(B_n, A)^{\alpha}. 
\end{align*}
By the exponential moment condition~\eqref{eq:exponential-moment-condition}, the factor multiplying $\Dp(B_n, A)^{\alpha}$ is finite. Thus, since this distance tends to zero as $n \to \infty$, we get that the first term in~\eqref{eq:notag}  goes to zero. Therefore, $\lambda_1(B_n) \to \lambda_1(A)$. This concludes that $\lambda_1(\cdot)$ is continuous at $A$.
\end{proof}

\section{Robustness of quasi-irreducibility and almost-irreducibility} First the notion of almost-irreducible cocycle in Definition~\ref{def:almost}. 
We demonstrate the robustness of the properties of quasi-irreducibility and almost-irreducibility for linear cocycles under small perturbations in different topologies. 

\begin{prop} \label{prop:open-quasi}
Let $A \in \LP$ be a quasi-irreducible cocycle such that $\lambda_1(A) > \lambda_2(A)$ and~\eqref{eq:exponential-moment-condition} holds. Then, there exists a neighborhood $\mathcal{U}$ of $A$ in $(\LP, \Dpm)$ such that every $B \in \mathcal{U}$ is quasi-irreducible {and $\lambda_1(B)>\lambda_2(B)$. In particular,} $f_B$ is a uniquely ergodic mostly contracting random map. 
\end{prop}

\begin{proof}
By Proposition~\ref{prop:projective2}, $f_A$ is a uniquely ergodic mostly contracting random map and satisfies~\eqref{eq:integral_condition}. Thus, by Lemma~\ref{lem:convergencia} and Theorem~\ref{prop:statistical-stability-length}, we have a neighborhood $\mathcal{U}$ of $A$ in $(\LP, \Dpm)$ such that $f_B$ is a uniquely ergodic mostly contracting random map for any $B \in \mathcal{U}$. In particular, every $f_B$-stationary measure is maximizing, and we will conclude that $B$ is also quasi-irreducible from Proposition~\ref{prop:quasi-irreducible}. {We also have that $\lambda_1(B)>\lambda_2(B)$ from Corollary~\ref{prop:mostly-simplicity}.}
\end{proof}

Let $\LPC$ be the spaces of compactly supported locally constant linear measurable cocycles as introduced in~\S\ref{ss:continuity-intro}.   In the compactly supported case, when the cocycle is almost-irreducible but not quasi-irreducible, we can still get some robustness under perturbations on the topology induced by the metric 
$$
\Delta(A, B) = \Dpm(A, B) + \delta_H(\mathrm{supp}\, \nu_A, \mathrm{supp}\, \nu_B) \quad \text{with } A, B \in \LPC.
$$
Recall that $\delta_H$ is the Hausdorff distance induced by the metric $$\delta(g_1, g_2) = \|g_1 - g_2\|+\|g_1^{-1} - g_2^{-1}\|, \quad g_1,g_2\in \mathrm{GL}(m)$$  and $\mathrm{supp}\, \nu_C$ denotes the support of the distribution $\nu_C=C_*\mathbb{P}$ in $\mathrm{GL}(m)$  associated with a cocycle $C \in \LPC$. 

\begin{prop} \label{prop:open-almost}
Let $A \in \LPC$ be an almost-irreducible  cocycle such that $\lambda_1(A) > \lambda_2(A)$. Then, there exists a neighborhood $\mathcal{U}$ of $A$ in $(\LPC, \Delta)$ such that every $B \in \mathcal{U}$ is almost-irreducible (on the same compact set $X$). Moreover, $f_B$ has a unique ergodic measure supported in $X$ which is maximizing. 
\end{prop}

\begin{proof} 
Since condition (i) in Proposition~\ref{rem:almost} is open in $(\LPC, \Delta)$, we only need to prove that condition (ii) persists under perturbation. Namely, we need to show that any $f_B$-stationary measure supported on $X$ is maximizing for any cocycle $B$ close enough to $A$ in $(\LPC, \Delta)$.

Since by Definition~\ref{def:almost}, 
$X\subset P(\mathbb{R}^m)\setminus E$  where $E$ is the equator, then $A$ meets the assumptions of  
Proposition~\ref{prop:projective2}. Hence,  $(f_A)|_X$ is a uniquely ergodic mostly contracting random map {on $X$} and satisfies~\eqref{eq:integral_condition}. Thus, by Lemma~\ref{lem:convergencia} and Theorem~\ref{prop:statistical-stability-length}, there exists a neighborhood $\mathcal{U}$ of $A$ in $(\LPC, \Delta)$ such that $(f_B)|_X$ is a uniquely ergodic mostly contracting random map {on $X$} for any $B \in \mathcal{U}$. Denote by $\mu_B$ this measure, and let $E_B$ be the (perhaps trivial) equator of $B$. 

\begin{claim} $E_B \cap X = \emptyset$.
\end{claim}

\begin{proof} 
Assume that $E_B \cap X \neq \emptyset$. Since $E_B$ and $X$ are both $B$-invariant sets, then $E_B \cap X$ is also $B$-invariant. Thus, there exists an $f_B$-stationary measure in $E \cap X$. Since $(f_B)|_X$ is uniquely ergodic, $\mu_B$ must be supported in $E_B \cap X$. In particular, by Remark~\ref{rem:maximazing}, $\mu_B$ cannot be a maximizing measure.   

Moreover, since $(f_B)|_X$ is uniquely ergodic, by Claim~\ref{claim:aditiva} and~\eqref{eq:uniform-Furstenberg-Kifer} in Theorem~\ref{Kingmanuniform}, we have  
$$
\lambda_1(B) > \int \phi_B \, d\mu_B = \lim_{n \to \infty} \frac{1}{n} \int \log \|B^n(\omega)x\| \, d\mathbb{P} \quad \text{for every } \hat{x} \in X.
$$
However, by Remark~\ref{rem:law-large-number}, 
$$
\lambda_1(B) = \lim_{n \to \infty} \frac{1}{n} \int \log \|B^n(\omega)x\| \, d\mathbb{P} \quad \text{for every  $\hat{x} \in P(\mathbb{R}^m) \setminus E_B$.}
$$
Thus, $X \subset E_B$. Consequently, $X = E_B \cap X$. But, since there is an open set $V \subset X = E_B \cap X \subset E_B$, then $E_B = \mathbb{R}^m$. This is a contradiction because the equator has dimension $\dim E_B < m$. 
\end{proof}

The above claim implies that $\mu_B$ is supported in $P(\mathbb{R}^m) \setminus E_B$ and thus, Corollary~\ref{cor:localization} concludes that $\mu_B$ is maximizing. Therefore, $B$ is almost-irreducible.  
\end{proof}

\section{Proof of the H\"older continuity} \index{Lyapunov exponents!H\"older continuity}
We prove the H\"older continuity of~$\phi_A$. 

\begin{lem}
    \label{lem:Holder}
    Let $A:\Omega \to \mathrm{GL}(m)$ be a locally constant linear cocycle. Then for any $\alpha>0$, 
\begin{align*}
\big|\log \|A(\omega)x\| - \log \|A(\omega)y\|\big|  
\leq \frac{2^\alpha}{\alpha\pi^\alpha} 
 \, \|A(\omega)^{-1}\|^\alpha \, \|A(\omega)\|^\alpha \cdot d(\hat{x},\hat{y})^\alpha
\end{align*}
for every $\hat{x}, \hat{y} \in P(\mathbb{R}^m)$.
Moreover, if $A$ satisfies~\eqref{eq:exponential-moment-condition} for some $0<\beta \leq 1$, then $\phi_A$ is $\alpha$-H\"older continuous for any $0<\alpha \leq \beta/2$.   
\end{lem}
\begin{proof}
Let $ \alpha >0$. Using that $\log t \leq \frac{1}{\alpha} (t^\alpha-1)$ for $t>0$,  $|t|^\alpha-|s|^\alpha \leq |t-s|^\alpha$ for $t,s\in \mathbb{R}$, $\big|\|u\|-\|v\| \big|\leq \|u-v\|$ for $u,v\in\mathbb{R}^m$, and $\|M^{-1}\|^{-1} \leq \|Mz\|\leq \|M\|$ for any unitary vector $z\in\mathbb{R}^m$ and $M\in \mathrm{GL}(m)$,  we obtain that for every $\hat{x},\hat{y}\in P(\mathbb{R}^m)$,
  \begin{align*}
    \big|\log \|A(\omega)x\| - \log \|A(\omega)y\|\big|  &
    = \big|  \log \frac{\|A(\omega){x}\|}{\|A(\omega){y}\|}  \, \,\big| 
     \leq \frac{1}{\alpha} \,  \frac{\big| \|A(\omega){x}\|-\|A(\omega){y}\| \big|^\alpha}
    {\|A(\omega){y}\|^\alpha} \\
      &\leq \frac{1}{\alpha} \, \|A(\omega)^{-1}\|^\alpha  \|A(\omega)\|^\alpha  \|x-y\|^{\alpha}. 
\end{align*}
Choosing the unitary representative $x$ and $y$ of $\hat{x}$ and $\hat{y}$ such that $\theta=\mathrm{angle}(x,y)\in [0,\frac{\pi}{2}]$, as discussed at the beginning of~\S\ref{ss:angular}, we have $\|x-y\|\leq \frac{2}{\pi} d(\hat{x},\hat{y})$.  Putting together  this observation and the above inequality, we obtain the first part of the lemma.
   
   By assumption, $A$ satisfies~\eqref{eq:exponential-moment-condition}  for some $0<\beta \leq 1$. 
Now, let $0< \alpha \leq \beta/2$. Then, from the first part of the lemma, for every $\hat{x},\hat{y}\in P(\mathbb{R}^m)$,
\begin{align*}
    \big|\phi_A(\hat{x})-\phi_A(\hat{y})\big| &
    \leq  \int \big|\log \|A(\omega){x}\|- \log \|A(\omega){y}\| \big|  \, d\mathbb{P} 
     \\
      &\leq \frac{2^\alpha}{\alpha\pi^\alpha}  \int \|A(\omega)^{-1}\|^\alpha  \|A(\omega)\|^\alpha \, d\mathbb{P} \cdot d(\hat{x},\hat{y})^{\alpha} \\
     &\leq \frac{2^\alpha}{\alpha\pi^\alpha} \left(\int \|A(\omega)^{-1}\|^{2\alpha} \, d\mathbb{P}\right)^{1/2} \left(\int \|A(\omega)\|^{2\alpha} \, d\mathbb{P}\right)^{1/2} \cdot d(\hat{x},\hat{y})^{\alpha}. 
\end{align*}
The last inequality follows from H\"older's inequality. Moreover, since  $2\alpha \leq \beta$, the exponential moment~\eqref{eq:exponential-moment-condition} 
imply that the factors multiplying $d(\hat{x},\hat{y})^{\alpha}$  are finite. Thus, $\phi_A$ is H\"older continuous with H\"older exponent arbitrarily small.   
\end{proof}

Here we show the following result.  


\begin{thm} \label{thm:holder}
    Let $A$ be a quasi-irreducible (resp.~almost-irreducible)  cocycle in $\LP$ (resp.~$\LPC$) such that $\lambda_1(A) > \lambda_2(A)$ and~\eqref{eq:exponential-moment-condition} holds. Then, there exists a neighborhood $\mathcal{U}$ of $A$ in $(\mathrm{L}_{\mathbb{P}}(m), \Dpm)$ (resp.~$(\LPC, \Delta)$) such that the map $$B \in \mathcal{U} \mapsto \lambda_1(B)$$ is H\"older continuous. Moreover, there exist $C > 0$ and $\gamma \in (0,1]$ such that 
    \begin{equation*}
        |\lambda_1(B) - \lambda_1(B')| \leq C \cdot \Dp(B, B')^\gamma \quad \text{for all } B, B' \in \mathcal{U}.
    \end{equation*}
\end{thm}

\begin{proof}
In what follows, $X$ denotes either $P(\mathbb{R}^m)$ or a projective  compact subset satisfying Definition~\ref{def:almost}, depending on whether $A$ is quasi-irreducible or almost-irreducible, respectively.

By Proposition~\ref{prop:projective2}, $(f_A)|_X$ is a uniquely ergodic mostly contracting random map, and~\eqref{eq:integral_condition} holds for some $\beta>0$ which we can assume $\beta\leq 1/2$ and satisfying~\eqref{eq:exponential-moment-condition}. See Remark~\ref{rem:beta}. Hence, since $\Dpm\leq \Delta$, we can apply Lemma~\ref{lem:convergencia} and Theorem~\ref{prop:statistical-stability-length} and conclude the existence of a neighborhood $\mathcal{U}$ of $A$ in either $(\LP, \Dpm)$ or $(\LPC, \Delta)$ such that $f_B$ is close enough to $f_A$ in $(\EE(P(\mathbb{R}^m)), \Dbeta)$ for any $B \in \mathcal{U}$. Moreover, by Lemma~\ref{lemma:2exponents}, we can assume that any $B\in \mathcal{U}$ still satisfies the corresponding finite exponential moment~\eqref{eq:exponential-moment-condition} for $\beta$. Thus, by Lemma~\ref{lem:Holder}, $\phi_B$ is $\alpha$-H\"older continuous for any $0<\alpha\leq\beta/2$.  
Also, by shrinking this neighborhood $\mathcal{U}$ if necessary, we can ensure that either Proposition~\ref{prop:open-quasi} or Proposition~\ref{prop:open-almost} applies. Thus, $B$ is almost-irreducible for any $B \in \mathcal{U}$. Moreover, $(f_B)|_X$ is a mostly contracting random map having a unique ergodic stationary measure $\mu_B$ supported in $X$ which is maximizing. 

Now, for each $B, C \in \mathcal{U}$, 
    \begin{align*}
    \big|\lambda_1(B)-\lambda_1(C)\big| &=\big|\int \phi_B \, d\mu_B -  \int \phi_C \, d\mu_C \big| \\ &\leq 
    \int |\phi_B - \phi_C |\, d\mu_B +  \big|\int \phi_C \, d\mu_B-\int \phi_C \, d\mu_C\big|. 
    \end{align*}
Since $\phi_C$ is $\alpha$-H\"older continuous where the exponent $\alpha>0$ can be chosen arbitrarily small, from Theorem~\ref{prop:statistical-stability-length} there are uniform contacts $K>0$ and $\eta>0$ on $\mathcal{U}$ such that 
\begin{equation} \label{eq:dalphaC}
    \big|\lambda_1(B)-\lambda_1(C)\big|\leq 
    \|\phi_B - \phi_C \|_\infty +  K \, \|\phi_C\|_\alpha \, \D^\alpha_{\smash{C^0}}(f_B,f_C)^{\eta}.  
    \end{equation}
Fix $\alpha>0$ such that $\gamma=\eta\alpha\leq \beta/2<1$. Hence,  as before,  
we obtain 
\begin{align*}
    \big\|\phi_B-\phi_C\big\|_\infty &= \left| \, \sup_{\hat x\in P(\mathbb{R}^m)} \int \log \frac{\|B(\omega)x\|}{\|C(\omega)x\|}  \, d\mathbb{P} \, \right|
    \leq 
\frac{1}{\gamma}\sup_{\hat x\in P(\mathbb{R}^m)} \int \frac{\|B(\omega)-C(\omega)\|^\gamma}{\|C(\omega)x\|^\gamma} \,d\mathbb{P} \\ &\leq \frac{1}{\gamma} \left(\int {\|C(\omega)^{-1}\|^{2\gamma}} \,d\mathbb{P}\right)^{1/2} \cdot \left(\int \|B(\omega)-C(\omega)\|^{2\gamma} \,d\mathbb{P}\right)^{1/2} \\
&\leq \frac{1}{\gamma} \left(\int {\|C(\omega)^{-1}\|^{2\gamma}} \,d\mathbb{P}\right)^{1/2} \cdot \Dp(B,C)^\gamma
\end{align*}
The above inequality, the exponential moment~\eqref{eq:exponential-moment-circulo} and Lemma~\ref{lemma:2exponents}  
implies that $$\|\phi_B-\phi_C\|_\infty\leq K_\gamma \Dp(B,C)^\gamma \quad \text{for some  $K_\gamma>0$.}
$$
Moreover, by Lemma~\ref{lem:convergencia}, we have that $\D^\alpha_{\smash{C^0}}(f_B,f_C)\leq \Dp(B,C)^\alpha$ and hence, above inequality and~\eqref{eq:dalphaC} imply
$|\lambda(B)-\lambda(B)|\leq \bar{K} \Dp(B,C)^\gamma$  
where $\bar{K}>0$ depends on $\mathcal{U}$ and $\alpha$ (note that $\gamma=\eta\alpha$). This proves the H\"older continuity of $\lambda(\cdot)$ and completes the proof. 
\end{proof}

\begin{rem} \label{rem:remover-momento} If $A \in \LPC$, then~\eqref{eq:exponential-moment-condition} holds. Thus, the exponential moment condition in Theorem~\ref{thm:holder} is immediately fulfilled in the compactly supported case.  
\end{rem}

\section{Proof of Propositions~\ref{mainthm:continuity},~\ref{mainprop:linear-exponent-quasi} and~\ref{mainprop:linear-exponent-almost}}

Since in the appropriate norm $\|M\|=\|M^T\|$ for any matrix $M\in \mathrm{GL}(m)$, Furstenberg and Kifer~\cite{furstenberg1983random} (see also~\cite{key1988lyapunov}) remark that 
\begin{align*}
    \lambda_1(A^T)&=\lim_{n\to\infty} \frac{1}{n} \int \log \|A(\omega_{n-1})^T\cdot \ldots \cdot A(\omega_0)^T\|\, d\mathbb{P}  \\ &= \lim_{n\to\infty} \frac{1}{n} \int \log \|(A(\omega_{0})\cdot \ldots \cdot A(\omega_{n-1}))^T\|\, d\mathbb{P}  \\
    &= \lim_{n\to\infty} \frac{1}{n} \int \log \|A^n(\omega_{n-1})\cdot \ldots \cdot A(\omega_{0})\|\, d\mathbb{P} = \lambda_1(A).
\end{align*}
Moreover,  since also $(M^T)^{-1}=(M^{-1})^T$, the distance between any pair of cocycles $A$ and $B$ coincides with the distance between $A^T$ and $B^T$. Thus, Propositions~\ref{mainthm:continuity},~\ref{mainprop:linear-exponent-quasi} and~\ref{mainprop:linear-exponent-almost} follow from the corresponding Theorems~\ref{thm:continuity-quasi},~\ref{thm:holder} and Remark~\ref{rem:remover-momento}.

\section{Proof of Proposition~\ref{mainprop:regularity-distribution}}
\label{s:proof-regularity-distribution} \index{Lyapunov exponents!continuity}
\index{Lyapunov exponents!H\"older continuity}

We now prove Proposition~\ref{mainprop:regularity-distribution}.  The only
additional ingredient is the gluing lemma from optimal transport, which allows
one to realize several distributions as one-step laws of locally constant
cocycles over a common Bernoulli space.

Given two probability measures \(\mu\) and \(\nu\) on a measurable space \(Y\), we denote
by \(\Pi(\mu,\nu)\) the set of couplings of \(\mu\) and \(\nu\), that is, the
set of probability measures on \(Y\times Y\) whose first and second marginals
are \(\mu\) and \(\nu\), respectively. The following gluing lemma is standard in the literature of optimal transport, see~\cite[Lemma~7.6]{villani2003topics}.

\begin{lem}
\label{lem:gluing}
Let \(Y\) be a Polish space and consider probability measures \(\mu_1,\mu_2,\mu_3\) on $Y$.
If $\pi_{12}\in\Pi(\mu_1,\mu_2)$ and $\pi_{23}\in\Pi(\mu_2,\mu_3)$, then there exists
$\rho\in\mathcal P(Y^3)$ 
such that
\[
   (\mathrm{pr}_{12})_*\rho=\pi_{12},
   \qquad
   (\mathrm{pr}_{23})_*\rho=\pi_{23}.
\]
where $\mathrm{pr}_{12}(y_1,y_2,y_3)=(y_1,y_2)$ and $\mathrm{pr}_{23}(y_1,y_2,y_3)=(y_2,y_3)$. 
\end{lem}

Using this lemma, we get the following realization lemma. Recall from~\S\ref{ss:wasserstein-distributions-intro} the definitions of the metrics $\delta$, $\delta^+$ and the associated Wasserstein distances $\delta_W$, $\delta^+_{\smash{W}}$ on the space of probability measures $\mathcal P_1(m)=\mathcal{P}_1(\mathrm{GL}(m),\delta)$, $\mathcal P^+_1(m)=\mathcal{P}_1(\mathrm{GL}(m),\delta^+)$ with finite first moment.

\begin{lem}
\label{lem:wasserstein-realization}
Let \(c\in\{\delta^+,\delta\}\), and let \(c_W\) be the corresponding
Wasserstein distance. Fix $\nu_0\in \mathcal P_1(\mathrm{GL}(m),c)$ and consider
$$(\Omega,\mathscr{F},\mathbb P)=([0,1]\times \mathrm{GL}(m),\mathscr{A},\mathrm{Leb} \times \nu_0 )^{\mathbb{N}}$$ 
where $\mathscr{A}$ is the product Borel $\sigma$-algebra. 
Then there exists a locally constant cocycle $A_0:\Omega\to\mathrm{GL}(m)$
 with $(A_0)_*\mathbb P=\nu_0$ such that for every $ \nu,\nu'\in \mathcal P_1(\mathrm{GL}(m),c)$ 
and \(\varepsilon>0\), there are locally constant cocycles $A_\varepsilon,A'_\varepsilon:\Omega\to\mathrm{GL}(m)$ satisfying 
\[
   (A_\varepsilon)_*\mathbb P=\nu,
   \qquad
   (A'_\varepsilon)_*\mathbb P=\nu',
\]
and
\[
   \int c(A_0,A_\varepsilon)\,d\mathbb P
   \leq
   c_W(\nu_0,\nu)+\varepsilon,
   \qquad
   \int c(A_\varepsilon,A'_\varepsilon)\,d\mathbb P
   \leq
   c_W(\nu,\nu')+\varepsilon .
\]
\end{lem}

\begin{proof}
For \(\omega\in\Omega\), write $\omega_0=(u_0,g_0)\in [0,1]\times \mathrm{GL}(m)$ and define
$A_0(\omega)\eqdef g_0$. 
Then \((A_0)_*\mathbb P=\nu_0\). Now fix \(\nu,\nu'\in\mathcal P_1(\mathrm{GL}(m),c)\) and
\(\varepsilon>0\). By the primal Wasserstein definition (see also the Kantorovich--Rubinstein duality), we can write
$$
c_W(\mu,\mu')=\inf_{\pi \in \Pi(\mu,\mu')} \int c(g,g')\, d\pi(g,g').
$$
Hence, we can choose couplings
$\pi_{12}\in\Pi(\nu_0,\nu)$ and $\pi_{23}\in\Pi(\nu,\nu')$
such that
\[
   \int c(g_0,g)\,d\pi_{12}(g_0,g)
   \leq c_W(\nu_0,\nu)+\varepsilon, \quad 
   \int c(g,g')\,d\pi_{23}(g,g')
   \leq c_W(\nu,\nu')+\varepsilon .
\]
Applying Lemma~\ref{lem:gluing} with \(Y=\mathrm{GL}(m)\), we obtain $\rho_\epsilon\in\mathcal P(\mathrm{GL}(m)^3)$
such that
\[
   (\mathrm{pr}_{12})_*\rho_\varepsilon=\pi_{12},
   \qquad
   (\mathrm{pr}_{23})_*\rho_\varepsilon=\pi_{23}.
\]
In particular, the first marginal of \(\rho_\varepsilon\) is \(\nu_0\). Disintegrating \(\rho_\varepsilon\) with respect to the first coordinate, we
obtain a probability kernel $K_\varepsilon: \mathrm{GL}(m) \times \mathscr{B}(\mathrm{GL}(m)^2) \to [0,1]$ such that $d\rho_{\varepsilon}= K_\varepsilon(g_0,\cdot) d\nu_0(g_0)$. 
Since \(\mathrm{GL}(m)\) is Polish, by the  random mapping representation theorem (c.f.~\cite[Lemma~3.22]{kallenberg2002foundations} and~\cite[Theorem~1.1]{kifer2012ergodic}) there exists a measurable map
$h_\varepsilon : [0,1]\times 
   \mathrm{GL}(m)\to\mathrm{GL}(m)^2$ 
such that, 
\[
   (h_\varepsilon(\cdot,g_0))_*\mathrm{Leb}
   =
   K_\varepsilon(g_0,\cdot) \quad \text{for \(\nu_0\)-a.e. \(g_0\).}
\]
Write $h_\varepsilon(u_0,g_0) =
(g_\varepsilon(u_0,g_0),
         g'_\varepsilon(u_0,g_0))$
and define
\[
   A_\varepsilon(\omega)
   \eqdef
   g_\varepsilon(u_0,g_0),
   \qquad
   A'_\varepsilon(\omega)
   \eqdef
   g'_\varepsilon(u_0,g_0),
   \qquad
   \omega_0=(u_0,g_0).
\]
These are locally constant cocycles on the fixed Bernoulli space
\((\Omega,\mathscr{F},\mathbb P)\).

By construction, when \(\omega_0=(u_0,g_0)\)
is distributed according to \(p=\mathrm{Leb}\times\nu_0\), the triple
$(A_0(\omega),A_\varepsilon(\omega),A'_\varepsilon(\omega))$ has distribution \(\rho_\varepsilon\) on \(\mathrm{GL}(m)^3\). Equivalently,
for every bounded Borel function
\(\varphi:\mathrm{GL}(m)^3\to\mathbb R\),
\[
   \int \varphi(A_0,A_\varepsilon,A'_{\smash{\varepsilon}})\,d\mathbb P
   =
   \int
   \varphi \,d\rho_\varepsilon. 
\]
Taking \(\varphi(g_0,g,g')= 1_E(g)\) and
\(\varphi(g_0,g,g')= 1_E(g')\), respectively, shows that
\[
   (A_\varepsilon)_*\mathbb P=\nu,
   \qquad
   (A'_\varepsilon)_*\mathbb P=\nu'.
\]
Taking instead
 $\varphi(g_0,g,g')=c(g_0,g)$ and $\varphi(g_0,g,g')=c(g,g')$ 
and using the two marginal identities for \(\rho_\varepsilon\), we obtain
\[
\begin{aligned}
   \int c(A_0,A_\varepsilon)\,d\mathbb P
   &=
   \int c(g_0,g)\,d\rho_\varepsilon(g_0,g,g')  \\
   &=
   \int c(g_0,g)\,d\pi_{01}(g_0,g)  \leq c_W(\nu_0,\nu)+\varepsilon,
\end{aligned}
\]
and similarly
\[
\begin{aligned}
   \int c(A_\varepsilon,A'_\varepsilon)\,d\mathbb P
   &=
   \int c(g,g')\,d\rho_\varepsilon(g_0,g,g')  \\
   &=
   \int c(g,g')\,d\pi_{12}(g,g')  \leq c_W(\nu,\nu')+\varepsilon.
\end{aligned}
\]
This completes the proof. 
\end{proof}

We prove the three assertions of Proposition~\ref{mainprop:regularity-distribution} separately.

\medskip

\noindent\emph{Proof of item \textup{(1)}.}
Let \((\nu_n)_{n\geq 1}\) be a sequence in
\(\mathcal P_{\smash{1}}^+(m)\) such that
$\delta_{\smash{W}}^+(\nu_n,\nu_0)\to 0$.
We shall prove that \(\lambda_1(\nu_n)\to \lambda_1(\nu_0)\).
Apply Lemma~\ref{lem:wasserstein-realization} with
\(c=\delta^+\).  We obtain a Bernoulli probability space
\((\Omega,\mathscr{F},\mathbb P)\) and locally constant  cocycle $A_{0}:\Omega\to \mathrm{GL}(m)$ with $(A_0)_*\mathbb{P}=\nu_0$ such that for each $n\geq 1$ there is a locally constant linear cocycle $A_n:\Omega \to \mathrm{GL}(m)$ satisying $(A_n)_*\mathbb{P}=\nu_n$ and 
\[
   \Dp(A_{0},A_n)
   =
   \int \delta^+(A_{0},A_n) \,d\mathbb P
   \leq
   \delta_W^+(\nu_0,\nu_n)+\frac1n .
\]
Hence $\Dp(A_{0},A_n)\to 0$. Since \(A_{0}\) has  distribution \(\nu_0\), it satisfies the same
exponential moment condition as \(\nu_0\). Moreover, the assumption that
\(\nu_0\), or \(\nu_{\smash{0}}^T\), is quasi-irreducible is exactly the corresponding
assumption for \(A_{0}\), or \(A_{\smash{0}}^T\). Therefore
Proposition~\ref{mainthm:continuity} applies to \(A_{0}, A_n \in \mathrm{L}_{\mathbb{P}}(m)\). Consequently,
\[
   |\lambda_1(\nu_n)-\lambda_1(\nu_{0})|=|\lambda_1(A_n)-\lambda_1(A_{0})|
   \to 0  \quad \text{as $n\to \infty$}.
\]
This proves the continuity of 
\(\lambda_1:(\mathcal P_{\smash{1}}^+(m),\delta_{\smash{W}}^+) \to \mathbb{R}\) at \(\nu_0\). 

\medskip

\noindent\emph{Proof of item \textup{(2)}.}
Assume that \(\nu_0\in\mathcal P_1(m)\),
$\lambda_1(\nu_0)>\lambda_2(\nu_0)$, and that either \(\nu_0\) or \(\nu_{\smash{0}}^T\) is quasi-irreducible.
Apply Lemma~\ref{lem:wasserstein-realization} with \(c=\delta\). This gives
a Bernoulli probability space \((\Omega,\mathscr F,\mathbb P)\) and a locally constant cocycle
$A_0:\Omega\to\mathrm{GL}(m)$ with $(A_0)_*\mathbb P=\nu_0$.
Then $\lambda_i(A_0)=\lambda_i(\nu_0)$ for $i=1,2$ 
and \(A_0\), or \(A_{\smash{0}}^T\), is quasi-irreducible. Hence, by
Proposition~\ref{mainprop:linear-exponent-quasi}, there exist a neighborhood
\(\mathcal B\) of \(A_0\) in \((\mathrm L_{\mathbb P}(m),\Dpm)\),
constants \(C>0\) and \(\gamma\in(0,1]\), such that
\begin{equation} \label{eq:delta-B}
   |\lambda_1(B)-\lambda_1(B')|
   \leq
   C\,\Dpm(B,B')^\gamma
   \qquad
   \text{for every } B,B'\in\mathcal B .
\end{equation}
Choose \(r>0\) such that
$\{B:\Dpm(B,A_0)<2r\}\subset \mathcal B$ and consider 
\[
   \mathcal U
   \eqdef
   \{\nu\in\mathcal P_1(m):\delta_W(\nu,\nu_0)<r/4\}.
\]
Let \(\nu,\nu'\in\mathcal U\), and let \(\varepsilon>0\). Also by
Lemma~\ref{lem:wasserstein-realization}, there exist locally constant cocycles
$A_\varepsilon,A'_\varepsilon:\Omega\to\mathrm{GL}(m)$ 
such that
$(A_\varepsilon)_*\mathbb P=\nu$, $(A'_\varepsilon)_*\mathbb P=\nu'$ 
and
\begin{align*}
       \Dpm(A_0,A_\varepsilon)
   &=
   \int \delta(A_0,A_\varepsilon)\,d\mathbb P
   \leq
   \delta_W(\nu_0,\nu)+\varepsilon, \\
   \Dpm(A_\varepsilon,A'_\varepsilon)
   &=
   \int \delta(A_\varepsilon,A'_\varepsilon)\,d\mathbb P
   \leq
   \delta_W(\nu,\nu')+\varepsilon .
\end{align*}
Taking \(\varepsilon>0\) sufficiently small, the first inequality gives $\Dpm(A_0,A_\varepsilon)<r$ and thus $A_\varepsilon \in \mathcal B$. 
Moreover, since \(\nu,\nu'\in\mathcal U\),
\[
   \delta_W(\nu,\nu')
   \leq
   \delta_W(\nu,\nu_0)+\delta_W(\nu_0,\nu')
   <\frac r2.
\]
Thus, for \(\varepsilon>0\) sufficiently small,
$\Dpm(A_\varepsilon,A'_\varepsilon)<r$ and consequently we have that  $\Dpm(A_0,A'_\varepsilon)
   \leq
   \Dpm(A_0,A_\varepsilon)
   +
   \Dpm(A_\varepsilon,A'_\varepsilon)
   <2r$. Hence also $A'_\varepsilon\in\mathcal B$. 
Applying~\eqref{eq:delta-B},
\[
\begin{aligned}
   |\lambda_1(\nu)-\lambda_1(\nu')|
   =
   |\lambda_1(A_\varepsilon)-\lambda_1(A'_\varepsilon)|   \leq
   C\,\Dpm(A_\varepsilon,A'_\varepsilon)^\gamma  
   \leq
   C\bigl(\delta_W(\nu,\nu')+\varepsilon\bigr)^\gamma .
\end{aligned}
\]
Letting \(\varepsilon\to0\), we get that
$ \lambda_1:
   (\mathcal P_1(m),\delta_W)\to\mathbb R$
is locally H\"older continuous at \(\nu_0\).

\medskip

\noindent\emph{Proof of item \textup{(3)}.}
Assume that $\nu_0\in\mathcal P_{\hspace{-0.7mm}\scriptscriptstyle\rm cpt}(m)$, $\lambda_1(\nu_0)>\lambda_2(\nu_0)$,
and that either \(\nu_0\) or \(\nu_{\smash{0}}^T\) is almost-irreducible. As before, applying Lemma~\ref{lem:wasserstein-realization} with \(c=\delta\), one has a fixed Bernoulli probability space \((\Omega,\mathscr F,\mathbb P)\) and a cocycle $A_0\in \LPC$ with $(A_0)_*\mathbb P=\nu_0$ such that  \(A_0\), or \(A_{\smash{0}}^T\), is almost-irreducible, and $\lambda_i(A_0)=\lambda_i(\nu_0)$ for $i=1,2$.  Hence Proposition~\ref{mainprop:linear-exponent-almost} gives a neighborhood
\(\mathcal B\) of \(A_0\) in \((\LPC,\Delta)\), constants \(C>0\) and \(\gamma\in(0,1]\) such that
\begin{equation} \label{eq:Delta-B}
   |\lambda_1(B)-\lambda_1(B')|
   \leq
   C\,\Delta(B,B')^\gamma
   \qquad
   \text{for every } B,B'\in\mathcal B .
\end{equation}
Similar as above, we choose \(r>0\) such that $\{B\in\LPC:\Delta(B,A_0)<2r\}\subset \mathcal B$ and consider
$\mathcal U
   \eqdef
   \{
   \nu\in\mathcal P_{\hspace{-0.7mm}\scriptscriptstyle\rm cpt}(m):
   \delta_{\mathcal T}(\nu,\nu_0)<r/4
   \}$ .
Let \(\nu,\nu'\in\mathcal U\), and let \(\varepsilon>0\). By
Lemma~\ref{lem:wasserstein-realization}, there exist also locally constant cocycles $ A_\varepsilon,A'_\varepsilon:\Omega\to\mathrm{GL}(m)$
such that
$(A_\varepsilon)_*\mathbb P=\nu$, $(A'_\varepsilon)_*\mathbb P=\nu'$, and
\[
   \Dpm(A_0,A_\varepsilon)
   \leq
   \delta_W(\nu_0,\nu)+\varepsilon,
   \qquad
   \Dpm(A_\varepsilon,A'_\varepsilon)
   \leq
   \delta_W(\nu,\nu')+\varepsilon .
\]
Since \(\nu\) and \(\nu'\) are compactly supported, the cocycles \(A_\varepsilon\) and \(A'_\varepsilon\) are compactly supported and $\operatorname{supp}(A_\varepsilon{}_*\mathbb P)=\operatorname{supp}\nu$, $   \operatorname{supp}(A'_\varepsilon{}_*\mathbb P)=\operatorname{supp}\nu'$. 
Therefore
\[
\begin{aligned}
   \Delta(A_0,A_\varepsilon)
   &=
   \Dpm(A_0,A_\varepsilon)
   +
   \delta_H(\operatorname{supp}\nu_0,\operatorname{supp}\nu)  \\
   &\leq
   \delta_W(\nu_0,\nu)
   +
   \delta_H(\operatorname{supp}\nu_0,\operatorname{supp}\nu)
   +
   \varepsilon  =
   \delta_{\mathcal T}(\nu_0,\nu)+\varepsilon,
\end{aligned}
\]
and similarly $\Delta(A_\varepsilon,A'_\varepsilon)
   \leq
   \delta_{\mathcal T}(\nu,\nu')+\varepsilon$. 
Taking \(\varepsilon>0\) sufficiently small, the first estimate gives $\Delta(A_0,A_\varepsilon)<r$ and hence $A_\varepsilon \in \mathcal{B}$. 
Moreover, since \(\nu,\nu'\in\mathcal U\),
\[
   \delta_{\mathcal T}(\nu,\nu')
   \leq
   \delta_{\mathcal T}(\nu,\nu_0)
   +
   \delta_{\mathcal T}(\nu_0,\nu')
   <\frac r2.
\]
Thus, for \(\varepsilon>0\) sufficiently small,
$\Delta(A_\varepsilon,A'_{\smash{\varepsilon}})<r$ and consequently we have that 
$\Delta(A_0,A'_{\smash{\varepsilon}}) \leq
   \Delta(A_0,A_{\smash{\varepsilon}})
   +
   \Delta(A_\varepsilon,A'_\varepsilon)
   <2r$.
Hence $A'_\varepsilon$ also belongs to $\mathcal{B}$. 
By~\eqref{eq:Delta-B}, 
\[
   |\lambda_1(\nu)-\lambda_1(\nu')|
   =
   |\lambda_1(A_\varepsilon)-\lambda_1(A'_\varepsilon)|  
   \leq
   C\,\Delta(A_\varepsilon,A'_\varepsilon)^\gamma  
   \leq
   C\bigl(\delta_{\mathcal T}(\nu,\nu')+\varepsilon\bigr)^\gamma .
\]
Letting \(\varepsilon\to0\), we get that
$\lambda_1:
   \bigl(
   \mathcal P_{\hspace{-0.7mm}\scriptscriptstyle\rm cpt}(m),
   \delta_{\mathcal T}
   \bigr)
   \to \mathbb R
$   
is locally H\"older continuous at \(\nu_0\).

\begin{rem} Realization arguments above show that if $\lambda_1(\cdot)$ is continuous, repectively locally H\"older continuous, at a locally constant cocycle $A:(\Omega,\mathbb{P})\to \mathrm{GL}(m)$ with respect to the metrics \(\Dp\), \(\Dpm\), or \(\Delta\), then $\lambda_1(\cdot)$ have the same continuity property at the distribution $\nu=A_*\mathbb{P}$ with respect to \(\delta_W^+\), \(\delta_W\), or \(\delta_{\mathcal T}\). The converse passage, from distributional regularity to the corresponding
fixed-base cocycle regularity, also follows directly from the fact that the law map
is non-expanding for the relevant metrics. Indeed, if \(A,B\in\LP\) and
\(\nu_A=A_*\mathbb P\), \(\nu_B=B_*\mathbb P\), then \((A,B)_*\mathbb P\) is a
coupling of \(\nu_A\) and \(\nu_B\). Hence
$\delta_W^+(\nu_A,\nu_B)\leq \Dp(A,B)$ and $ \delta_W(\nu_A,\nu_B)\leq \Dpm(A,B)$. 
Moreover, if \(A,B\in\LPC\), then
$\delta_{\mathcal T}(\nu_A,\nu_B)
   =
   \delta_W(\nu_A,\nu_B)
   +
   \delta_H(\operatorname{supp}\nu_A,\operatorname{supp}\nu_B)
   \leq
   \Delta(A,B)$. 
Consequently, continuity, respectively local H\"older continuity, of
\(\lambda_1(\cdot)\) with respect to a distribution $\nu$ in $\mathrm{GL}(m)$ in the metrics 
\(\delta_W^+\), \(\delta_W\), or \(\delta_{\mathcal T}\) is equivalent to the same property for any locally constant realization \(A:(\Omega,\mathbb{P})\to \mathrm{GL}(m)\) of \(\nu\), with
respect to \(\Dp\), \(\Dpm\), or \(\Delta\), respectively. 
\end{rem}

\chapter{Mostly contracting random maps of one-dimensional diffeomorphisms} \label{s:examples}

\abstract{ {
This chapter develops the one-dimensional examples of the mostly contracting
theory. We begin with the exponent of contraction for random circle maps and
compare it with the usual Lyapunov exponent in the \(C^1\)-diffeomorphism
setting. This comparison allows us to use Malicet's invariance principle to
obtain a criterion for mostly contracting random circle diffeomorphisms: in the
absence of a common invariant probability measure, every stationary regime has
negative Lyapunov exponent. We then extend this criterion to random
diffeomorphisms of compact intervals onto their images. The chapter
concludes with random diffeomorphisms of Cantor sets, where the absence of a
common invariant measure again forces uniform contraction away from finite
exceptional sets and hence the mostly contracting property. In all these
settings, the abstract results of the previous chapters yield the corresponding
ergodic, spectral, and statistical consequences.}
}


\section{Exponent of contraction} 
\label{ss:Exponent_contraction}
 Let $(\Omega, \mathscr{F})$ be a standard Borel space. Consider a  measurable skew-product transformation
$$
F\colon\Omega\times \mathbb{S}^1 \to \Omega\times \mathbb{S}^1, \quad F(\omega,x) = (T(\omega), f_\omega(x)).
$$ 
As usual, we denote by $f^n_\omega(x)$ the second coordinate of $F^n(\omega,x)$ for $n\geq 0$. The maximal Lyapunov exponent $\lambda(\omega,x)$ measures the exponential rate of contraction of $f^n_\omega$ in a neighborhood of $x$ but requires at least Lipschitz regularity of the fiber maps to be well-defined. In~\cite[Def.~3.2]{Mal:17}, Malicet introduces the \emph{exponent of contraction} \index{Lyapunov exponents!\(\lambda_{con}(\omega,x)\), exponent of contraction}
$$
\lambda_{con}(\omega,x)\eqdef \limsup_{y\to x}\limsup_{n\to \infty} \frac{1}{n}\log d(f_\omega^n(x),f_\omega^n(y))$$
where $d$ is the standard legth-metric on $\mathbb{S}^1$.  This exponent provides analogous information without assuming any regularity, {i.e., assuming just that the fiber maps $f_\omega$ are measurable functions.} As usual, for any $F$-invariant probability measure $\bar{\mu}$, the exponent of contraction of $\bar{\mu}$ is defined as
$$
\lambda_{con}(\bar{\mu})=\int \lambda_{con}(\omega,x)\, d\bar{\mu}.
$$
{Normalizing the metric so that \(d\le 1\), we have $\lambda_{con}(\omega,x)\leq 0$  and $\lambda_{con}(\bar\mu)\leq0$. 
{For each fixed $\omega\in \Omega$, we also have the following  properties:

{ 
\begin{lem} \label{lem:semicontoflambdacon}
    The function
\(x\mapsto \lambda_{con}(\omega,x)\) is upper semicontinuous. 
\end{lem}
\begin{proof} If
\(\lambda_{con}(\omega,x)<a\), choose \(\lambda_{con}(\omega,x)<b<a\) and a neighborhood \(U\) of \(x\)
such that, for every \(y\in U\),
\[
\limsup_{n\to\infty}
\frac1n\log d(f_\omega^n(x),f_\omega^n(y))<b.
\]
If \(z\) is sufficiently close to \(x\) and \(w\) is sufficiently close to \(z\),
then both \(z\) and \(w\) belong to \(U\). Hence
\[
d(f_\omega^n(z),f_\omega^n(w))
\le
d(f_\omega^n(z),f_\omega^n(x))
+
d(f_\omega^n(x),f_\omega^n(w)),
\]
and the exponential rate of the right-hand side is at most \(b\). Indeed, if two non-negative sequences \(A_n,B_n\) satisfy
\[
\limsup_{n\to\infty}\frac1n\log A_n\le b,
\qquad
\limsup_{n\to\infty}\frac1n\log B_n\le b,
\]
then
\[
\limsup_{n\to\infty}\frac1n\log(A_n+B_n)\le b,
\]
because \(A_n+B_n\le 2e^{n(b+\eta)}\) for all large \(n\), for every \(\eta>0\).
 Therefore
\(\lambda_{con}(\omega,z)\le b<a\), proving upper semicontinuity.
\end{proof}
}

\begin{lem} \label{lem:subinvariance}  If $f_\omega$ is  continuous at $x$, then 
$\lambda_{con}(\omega,x)\leq \lambda_{con}(F(\omega,x))$. 
\end{lem}
\begin{proof} Fix \(y\) close to \(x\). For \(n\geq 1\), we have
\[
d(f_\omega^n(x),f_\omega^n(y))
=
d(
f_{T\omega}^{\,n-1}(f_\omega(x)),
f_{T\omega}^{\,n-1}(f_\omega(y))
).
\]
Since replacing \(n\) by \(n-1\) does not change the exponential rate, we obtain
\[
\limsup_{n\to\infty}
\frac1n
\log d(f_\omega^n(x),f_\omega^n(y))
=
\limsup_{n\to\infty}
\frac{1}{n-1}
\log d(
f_{T\omega}^{n-1}(f_\omega(x)),
f_{T\omega}^{n-1}(f_\omega(y))
).
\]
 Now let \(y\to x\). Since \(f_\omega\) is continuous at $x$, we have
$f_\omega(y)\to f_\omega(x)$. Therefore the points \(f_\omega(y)\) form a family of points tending to
\(f_\omega(x)\). Taking the limsup over this restricted family can only be smaller than taking the limsup over all points \(z\to f_\omega(x)\). Hence
\[
\lambda_{con}(\omega,x)
\le
\limsup_{z\to f_\omega(x)}
\limsup_{m\to\infty}
\frac1m
\log d(
f_{T\omega}^{m}(f_\omega(x)),
f_{T\omega}^{m}(z))=
\lambda_{con}(F(\omega,x)).
\]
This concludes the proof.
\end{proof}
}

{
\begin{lem} \label{lem:ergodico-contration} Let $\bar{\mu}$ be an ergodic $F$-invariant probability measure.  Assume $f_\omega$ is continuous at $x$ for $\bar \mu$-a.e.~$(\omega,x)\in \Omega \times X$. Then 
$$\text{$\lambda_{con}(\bar\mu)=\lambda_{con}(\omega,x)$ \quad for $\bar\mu$-a.e.~$(\omega,x)$.}$$
\end{lem}
\begin{proof} 
For \(M>0\), set
$\lambda_{con}^M=\max\{\lambda_{con},-M\}$.
Since the metric is normalized so that \(d\le 1\), we have
$\lambda_{con}\le 0$,
and therefore \(\lambda_{con}^M\) is bounded. Moreover, by Lemma~\ref{lem:subinvariance}, 
$
\lambda_{con}^M\le \lambda_{con}^M\circ F$ $\bar\mu$-a.e. Since \(\bar\mu\) is \(F\)-invariant,
\[
\int \lambda_{con}^M\,d\bar\mu
\le
\int \lambda_{con}^M\circ F\,d\bar\mu
=
\int \lambda_{con}^M\,d\bar\mu.
\]
Hence, equality holds, and so
$\lambda_{con}^M=\lambda_{con}^M\circ F$  $\bar\mu$-a.e.  By ergodicity, \(\lambda_{con}^M\) is constant \(\bar\mu\)-a.e. Letting
\(M\to\infty\), we conclude that \(\lambda_{con}\) is constant
\(\bar\mu\)-a.e.
\end{proof}
}

First, we show that, under Lipschitz regularity on the fibers, we can compare the maximal Lyapunov exponent and the contraction exponent as follows:

\begin{lem}  \label{eq:desigualdad}
Let $\bar{\mu}$ be an $F$-invariant probability measure satisfying 
$$
\int \log^+ \mathrm{Lip}(f_\omega)  \, d\bar{\mu}<\infty.
$$
Then, for $\bar\mu$-a.e.~$(\omega,x)$,
\begin{equation*}
  \lambda_{con}(\omega,x) \leq \lambda(\omega,x) \eqdef \lim_{n\to \infty} \frac{1}{n} \log Lf^n_\omega(x)
\end{equation*}
and hence also 
\begin{equation*} 
  \lambda_{con}(\bar{\mu})\leq \lambda(\bar{\mu}) \eqdef \int \lambda(\omega,x)\ d\bar\mu.  \end{equation*}
\end{lem}

\begin{proof}
Since $\lambda_{con}(\omega, x) \le 0$, if $\lambda(\omega, x) \ge 0$, then $\lambda_{con}(\omega,x) \le 0 \le \lambda(\omega,x)$ and the asserted inequality holds immediately. It remains to consider the case that $\lambda(\omega,x) < 0$. By Theorem~\ref{contraction}, for $\bar{\mu}$-almost every such $(\omega,x)$ and any $\chi > \lambda(\omega,x)$, there exist constants $\delta > 0$ and $C > 0$ such that 
\[
    d(f_\omega^n(x), f_\omega^n(y)) \le C e^{n\chi} d(x,y)
\]
for all $y \in B(x, \delta)$ and $n \ge 0$. For any fixed $y \in B(x, \delta) \setminus \{x\}$, taking logarithms, dividing by $n$, and passing to the limsup  as $n \to \infty$ yields
\[
    \limsup_{n \to \infty} \frac{1}{n} \log d(f_\omega^n(x), f_\omega^n(y)) \le \chi.
\]
Taking limsup as $y \to x$, we obtain $\lambda_{con}(\omega, x) \le \chi$. Since this  holds for any $\chi > \lambda(\omega, x)$, we conclude that $\lambda_{con}(\omega, x) \le \lambda(\omega, x)$, completing the proof.
\end{proof}
}

When the fiber maps are circle $C^1$ diffeomorphisms, we can conclude the following stronger conclusion: 

\begin{prop}
\label{prop:Lyapunov-contraction}
Let \(\bar\mu\) be an ergodic
\(F\)-invariant probability measure satisfying
\begin{equation} \label{eq:integrability-contraction}
   \int \|\log |f'_\omega|\,\|_\infty\,
   d\bar\mu<\infty.
\end{equation}
Then, $\lambda_{\mathrm{con}}(\bar\mu)
   =
   \min\{\lambda(\bar\mu),0\}$.
\end{prop}

This result essentially follows from~\cite[Prop.~3.3]{Mal:17}. The original
statement assumes that $(\omega,x)\mapsto \log |f_\omega'(x)|$ is bounded,
{whereas here we only require~\eqref{eq:integrability-contraction}. Note
that, for  circle $C^1$ diffeomorphisms,
\[
\|\log |f'_\omega|\,\|_\infty
=
\max\left\{
\log^+\|f'_\omega\|_\infty,\,
\log^+\|(f_\omega^{-1})'\|_\infty
\right\}.
\]
Thus~\eqref{eq:integrability-contraction} is equivalent to the integrability of both the derivative and the inverse derivative.}

\begin{proof}[Proof of Proposition~\ref{prop:Lyapunov-contraction}]
{By Lemma~\ref{eq:desigualdad}, we have $\lambda_{con}(\bar{\mu}) \leq \lambda(\bar{\mu})$. Hence, if $\lambda_{con}(\bar{\mu}) = 0$, then
\[
0=\lambda_{con}(\bar{\mu})\leq \lambda(\bar{\mu}),
\]
so $\lambda_{con}(\bar{\mu})=\min\{\lambda(\bar{\mu}),0\}$. Thus, since $\lambda_{con}(\omega,x)\leq 0$, it remains only to consider the case $\lambda_{con}(\bar{\mu}) < 0$. In this case $\min\{0,\lambda(\bar\mu)\}=\lambda(\bar\mu)$. Moreover, by Lemma~\ref{eq:desigualdad}, $\lambda_{con}(\bar \mu)\leq \lambda(\bar \mu)$. Thus, to conclude the proof, we need to show that $\lambda_{con}(\bar \mu) \geq \lambda(\bar \mu)$.  

Denote by $\mathbb{P}$ the first marginal of $\bar\mu$. Since $\bar\mu$ is $F$-invariant, $\mathbb P$ is $T$-invariant; moreover, the ergodicity of $\bar\mu$ implies the ergodicity of $\mathbb P$.}
{By the ergodicity of $\bar{\mu}$ and Lemma~\ref{lem:ergodico-contration}, 
there exists a full $\bar\mu$-measure set $E\subset \Omega\times \mathbb S^1$ such that for every $(\omega,x)\in E$,}
$$
\lambda(\bar{\mu}) = \lim_{n \to \infty} \frac{1}{n}\log |(f_\omega^n)'(x)| \quad \text{and} \quad  {\lambda_{con}(\bar{\mu}) = \lambda_{con}(\omega,x).}
$$ 
{Fix $(\omega,x)\in E$. For $r>0$, define
\[
\alpha_\omega(r)\eqdef \sup\Big\{\big|\log |f'_\omega(u)|-\log |f'_\omega(v)|\big|:\ d(u,v)\le r\Big\}.
\]}
Let $\epsilon > 0$ be small enough so that $\lambda_{con}(\bar{\mu}) + \epsilon < 0$. There exists an interval $I$ containing $x$ such that for sufficiently large $n$,
$$
\frac{1}{n} \log \left(\frac{\mathrm{diam}\, f_\omega^n(I)}{\mathrm{diam}\, I} \right) \leq \lambda_{con}(\bar{\mu}) + \epsilon.
$$ 
In particular, $\mathrm{diam}\, f_\omega^n(I) \to 0$ as $n \to \infty$. Now, let $\delta > 0$. By shrinking $I$ if necessary, we can assume that for every $n$, $\mathrm{diam}\, f_\omega^n(I) \leq \delta$. We can then control the oscillations of $\log(f_\omega^n)'$ on $I$ by the modulus of continuity $\alpha_\omega(\cdot)$ of $\log |f_\omega'|$ as follows: for every $x_1, x_2 \in I$,
\begin{align*}
\log \big|(f_\omega^n)'(x_1)\big| - \log \big|(f_\omega^n)'(x_2)\big| &= \sum_{k=0}^{n-1} \left(\log \big|f_{T^k\omega}'(f_\omega^k(x_1))\big| - \log \big|f_{T^k\omega}'(f_\omega^k(x_2))\big|\right) \\ &\leq \sum_{k=0}^{n-1} \alpha_{T^k\omega}(\delta).
\end{align*}
Taking $x_1 = x$ and choosing $x_2$ such that {$\frac{\mathrm{diam}\, f_\omega^n(I)}{\mathrm{diam}\, I} = \big|(f_\omega^n)'(x_2)\big|$} (which may depend on $n$), we obtain
$$
\log \big|(f_\omega^n)'(x)\big| \leq \log \frac{\mathrm{diam}\, f_\omega^n(I)}{\mathrm{diam}\, I} + \sum_{k=0}^{n-1} \alpha_{T^k\omega}(\delta),
$$
and thus 
$$
\frac{1}{n} \log \big|(f_\omega^n)'(x)\big| \leq \lambda_{con} (\bar{\mu}) + \epsilon + \frac{1}{n} \sum_{k=0}^{n-1} \alpha_{T^k\omega}(\delta).
$$ 
Letting $n \to \infty$, we get
$$
\lambda(\bar{\mu}) \leq \lambda_{con}(\bar{\mu}) + \epsilon + \limsup_{n \to \infty} \frac{1}{n} \sum_{k=0}^{n-1} \alpha_{T^k\omega}(\delta).
$$ 
{For each fixed $\delta > 0$, this inequality holds for $\mathbb{P}$-a.e.~$\omega$. Moreover,
\[
\alpha_\omega(\delta)\le 2\|\log |f'_\omega|\,\|_\infty.
\]
Hence, by~\eqref{eq:integrability-contraction} and Birkhoff's ergodic theorem, for each fixed $\delta>0$ and for $\mathbb{P}$-a.e.~$\omega$,}
$$
\lim_{n \to \infty} \frac{1}{n} \sum_{k=0}^{n-1} \alpha_{T^k\omega}(\delta) = \int \alpha_{\omega'}(\delta) \, d\mathbb{P}(\omega').
$$ 
{Since $\log |f'_\omega|$ is continuous on $\mathbb S^1$, we have $\alpha_\omega(\delta)\to 0$ as $\delta\to 0$. Therefore, by the dominated convergence theorem,
\[
\int \alpha_{\omega'}(\delta)\, d\mathbb{P}(\omega') \longrightarrow 0
\qquad\text{as }\delta\to 0.
\]
Hence, $\lambda(\bar{\mu}) \leq \lambda_{con}(\bar{\mu}) + \epsilon$. Letting $\epsilon\to 0$, we obtain $\lambda(\bar{\mu}) \leq \lambda_{con}(\bar{\mu})$ as we wanted to show, concluding the proof of the proposition. }
\end{proof}

\section{Random circle diffeomorphisms: invariance principle} 
The invariance principle refers to a rigidity result for the invariant measure, provided that the extremal Lyapunov exponents are equal. In the one-dimensional settings, this condition is equivalent to the maximal Lyapunov exponent being zero. The first instance of this principle can be found in the pioneering work of Ledrappier~\cite{Le86}, which deals with the setting of linear cocycles from random products of matrices. The smooth (or differentiable) non-linear version of this principle was established by Crauel~\cite{crauel:1990}. 

The following {invariance} principle for the contraction exponent was proved in~\cite[Thm~F]{Mal:17}. 

\index{Lyapunov exponents!invariance principle}

\begin{thm} \label{thm:invariant-Malicet}
Let $(\Omega, \mathscr{F}, \mathbb{P})$ be a standard Borel probability space. Consider a  measurable transformation
$$
F\colon\Omega\times \mathbb{S}^1 \to \Omega\times \mathbb{S}^1, \quad F(\omega,x) = (T(\omega), f(\omega,x))
$$  
where $f_\omega \eqdef f(\omega,\cdot) : \mathbb{S}^1 \to \mathbb{S}^1$ are homeomorphisms. Then, for every $F$-invariant probability measure $\bar{\mu}$ of the form $d\bar{\mu} = \mu_\omega \, d\mathbb{P}(\omega)$, one of the following holds:
\begin{itemize}
    \item {Contraction:} $\lambda_{con}(\bar{\mu}) < 0$, or
    \item {Invariance:} $\mu_{T(\omega)} = (f_\omega)_*\mu_\omega$ for $\mathbb{P}$-a.e.~$\omega \in \Omega$.
\end{itemize}
\end{thm}

As an application of the previous results, we get the following invariance principle for the maximal Lyapunov exponent. For a similar result in the context of continuous-time random dynamical systems, see~\cite[Cor.~4.4]{Crauel:02}.

\begin{thm} \label{thm:Markovian-inv-principle}
Let $(\Omega, \mathscr{F}, \mathbb{P})$ be a Bernoulli product probability space and consider a random map $f: \Omega \times \mathbb{S}^1 \to \mathbb{S}^1$ such that  {$\| \log |f'_\omega|\,\|_\infty$} is $\mathbb{P}$-integrable. Then, for every ergodic $f$-stationary probability measure $\mu$, one of the following holds:
   \begin{itemize}
       \item \text{Contraction:} $\lambda(\mu) < 0$, or
       \item \text{Invariance:} $(f_\omega)_*\mu = \mu$ for $\mathbb{P}$-a.e.~$\omega \in \Omega$.
   \end{itemize}
\end{thm}

\begin{proof}
Assume that $\lambda(\mu) \geq 0$. Note that $\lambda(\mu)$ refers to $\lambda(\bar{\mu})$, where $\bar{\mu} = \mathbb{P} \times \mu$. Moreover, the ergodicity of $\mu$ implies that $\bar{\mu}$ is an ergodic invariant probability measure for the associated skew-shift $F(\omega,x)=(\sigma(\omega),f_\omega(x))$ on $\Omega \times \mathbb{S}^1$. 
{Since $\bar\mu=\mathbb P\times \mu$,  the assumption that $\|\log |f'_\omega|\,\|_\infty$ is $\mathbb P$-integrable is exactly~\eqref{eq:integrability-contraction}. Hence Proposition~\ref{prop:Lyapunov-contraction} applies and yields
$\lambda_{con}(\bar\mu)=\min\{\lambda(\bar\mu),0\}=0$.
Therefore, by Theorem~\ref{thm:invariant-Malicet}, the disintegration of $\bar\mu$,  satisfies
\[
\mu_{\sigma(\omega)}=(f_\omega)_*\mu_\omega
\qquad\text{for }\mathbb P\text{-a.e. }\omega\in\Omega.
\]
Since $\mu_\omega=\mu$, we conclude that
\[
(f_\omega)_*\mu=\mu
\qquad\text{for }\mathbb P\text{-a.e. }\omega\in\Omega.
\]
Thus, the maps $f_\omega$ share a common invariant measure, which gives the invariance alternative.}
\end{proof}

\section{Proof of Proposition~\ref{cor:mostly-contraction-circle-cantor}}

By assumption, $f$  is a random map of circle $C^1$ diffeomorphisms with no common invariant measure, and {$\| \log |f'_\omega|\,\|_\infty$} is $\mathbb{P}$-integrable. Therefore, for every ergodic $f$-stationary probability measure $\mu$, Theorem~\ref{thm:Markovian-inv-principle} implies $\lambda(\mu) < 0$. Consequently, by Theorem~\ref{prop:equivalence}, we have $$\lambda(f) = \max\{\lambda(\mu) : \mu \text{ is ergodic and } f\text{-{stationary}} \} < 0.$$ This establishes that $f$ is mostly contracting.
Furthermore, since $\mathrm{Lip}(f_\omega) = \|f'_\omega\|_\infty$, Theorem~\ref{thmA} tells us that the Koopman operator is quasicompact, and that it has a spectral gap if $f$ is mingled. Moreover, if $f$ is either proximal or minimal, then it is mingled by Remark~\ref{rem:topological-notions}(ii), since $\mathbb{S}^1$ is connected. 
Finally, the following lemma concludes the proof of Proposition~\ref{cor:mostly-contraction-circle-cantor}.

\begin{lem} If \(f:\Omega \times \mathbb{S}^1 \to \mathbb{S}^1\) is quasi-symmetric and has no common invariant probability
measure, then $f$ is mingled. 
\end{lem}

\begin{proof} Let \(G\) be the smallest closed semigroup of
\((\operatorname{Homeo}(\mathbb S^1),\circ)\) containing
\(\mathbb P\)-a.e. realization \(f_\omega\). By quasi-symmetry, \(G\) is a
group. Fix \(n\geq 1\). Let \(G_n\) be the smallest closed semigroup of
\((\operatorname{Homeo}(\mathbb S^1),\circ)\) containing
\(\mathbb P\)-a.e. realization \(f_\omega^n\). We first claim that \(G_n\) is
also a group.  Since $f_\omega^n$ is also quasi-symmetric (by using for instead that $f$ is quasi-symmetric if and only $\mathbb{P}(\{\omega : \exists k\ge 1 \ \text{such that } f_\omega^k\in V\})>0$ for any neighborhood $V$ of $\mathrm{Id}$ in $\mathrm{Homeo}(X)$), we have that \(G_n\) is
also a group. 
Moreover, since \(G\) contains \(\mathbb P\)-a.e. realization \(f_\omega\), it
also contains \(\mathbb P\)-a.e. realization \(f_\omega^n\). Thus
\(G_n\subset G\).

We next prove that \(G_n\) has finite index in \(G\). For
\(0\leq r<n\), let \(C_r\) denote the smallest closed subset of
\(\operatorname{Homeo}(\mathbb S^1)\) containing
\(\mathbb P\)-a.e. realization \(f_\omega^r\), with \(C_0=\{\operatorname{Id}\}\). Fix \(0<r<n\), and take \(a,b\in C_r\). Choose
\(c\in C_{n-r}\). Since \(a\circ c\) and \(b\circ c\) are limits of
\(n\)-step realizations, they belong to \(G_n\). Hence $(a\circ c)\circ (b\circ c)^{-1}
   a\circ b^{-1}
   \in G_n$.
Thus \(a\) and \(b\) belong to the same left coset of \(G_n\). Consequently,
\(C_r\) is contained in one left coset of \(G_n\).
Every finite composition of \(\mathbb P\)-a.e. realizations has length
\(qn+r\), with \(q\geq0\) and \(0\leq r<n\). Splitting it into a product of
\(qn\) factors followed by \(r\) factors, we see that it belongs to
\(G_n C_r\). Since \(C_r\) is contained in one left coset of \(G_n\), it follows
that \(G\) is covered by finitely many left cosets of \(G_n\). Hence \(G_n\) has
finite index in \(G\).

Now let \(F\subset\mathbb S^1\) be closed. Then \(F\) is $f^n$-invariant if and only if it is invariant by the group \(G_n\).
Indeed, invariance under \(\mathbb P\)-a.e. realization \(f_\omega^n\) extends,
by closedness of \(F\) and continuity of composition, to the closed semigroup
\(G_n\). Since \(G_n\) is a group, forward invariance is the same as invariance.

Suppose, by contradiction, that there are two disjoint non-empty closed subsets
of \(\mathbb S^1\) invariant by \(f^n\). Equivalently, they are invariant by
\(G_n\). Each of them contains a minimal closed \(G_n\)-invariant set. Therefore
\(G_n\) has two disjoint minimal closed invariant sets. By the standard
classification of minimal closed invariant sets for groups of homeomorphisms of
the circle, this can happen only if \(G_n\) has a finite orbit.

Since \(G_n\) has finite index in \(G\), a finite \(G_n\)-orbit gives a finite
\(G\)-invariant set: indeed, the union of its images under finitely many coset
representatives of \(G_n\) in \(G\) is finite and \(G\)-invariant. Therefore
\(G\) has a finite orbit. The uniform probability measure on this finite orbit
is invariant under every element of \(G\), and hence is invariant under
\(\mathbb P\)-a.e. realization \(f_\omega\). This gives a common invariant
probability measure, contradicting the assumption.
Thus, for every \(n\geq1\), the \(n\)-step random map admits no two disjoint
non-empty closed invariant subsets. Hence \(f\) is mingled.
\end{proof}


\section{Random interval diffeomorphisms onto their images}

Let \(\mathbb I\) be a compact interval. We consider a random map \(f:\Omega\times\mathbb I\to\mathbb I\) of \(C^1\) diffeomorphisms onto their images such that {\(\|\log |f'_\omega|\,\|_\infty\)} is \(\mathbb P\)-integrable. 
{We shall use the following extension lemma.

\begin{lem}
\label{lem:extension-exterior-measure}
Let \(I\subset\mathbb S^1\) be a compact interval and let \(J=\mathbb S^1\setminus I\). Let \(g:I\to I\) be a \(C^1\) diffeomorphism onto its image.

\begin{enumerate}
\item If \(g(I)\ne I\), then, for every \(c\in J\), there exists a \(C^1\) circle diffeomorphism \(\tilde g:\mathbb S^1\to\mathbb S^1\) such that \(\tilde g|_I=g\), \(\tilde g(I)\subset I\), and every \(\tilde g\)-invariant probability measure \(\nu\) satisfies \(\nu(J\setminus\{c\})=0\).
\item For every \(c\in J\), there exists a \(C^1\) circle diffeomorphism \(\tilde g:\mathbb S^1\to\mathbb S^1\) such that \(\tilde g|_I=g\), \(\tilde g(I)\subset I\), and \(\tilde g(c)\ne c\).
\end{enumerate}
\end{lem}

\begin{proof}
Write \(I=[a,b]\) in the cyclic order so that \(J=(b,a)\). We identify \(\overline J\) with \([0,1]\). In these coordinates, the target interval \(J_g = {\mathbb S^1 \setminus g(\operatorname{int} I)}\) is given by \([\alpha, \beta]\). Since \(g(I)\subset I\) and \(g(I)\ne I\), we have \(\alpha \le 0 \le 1 \le \beta\), with at least one inequality being strict. To obtain \(\tilde g\), we construct a \(C^1\) diffeomorphism \(h: [0,1] \to [\alpha, \beta]\) whose endpoint values and derivatives (\(r_0\) at \(0\), and \(r_1\) at \(1\)) match those prescribed by \(g\).

We first prove item (1). We divide the construction into two cases based on the orientation of \(g\).

\emph{Case 1: \(g\) preserves orientation.} The extension \(h\) must be strictly increasing with \(h(0) = \alpha\) and \(h(1) = \beta\).
If \(\alpha < 0\), or if \(\alpha = 0\) and \(r_0 \le 1\), we can choose \(h(x) < x\) for \(x > 0\) sufficiently small. 
Similarly, if \(\beta > 1\), or if \(\beta = 1\) and \(r_1 \le 1\), we can choose \(h(x) > x\) for \(x < 1\) sufficiently close to \(1\).  Thus, in these cases, we can construct \(h\) with \(c_0\) as its unique fixed point in \((0,1)\), satisfying
\[
h(x) < x \quad \text{for } 0 < x < c_0, \qquad h(x) > x \quad \text{for } c_0 < x < 1.
\]
Thus, \(c_0\) acts as a repelling fixed point, and all other points in \((0,1)\) move monotonically toward the boundaries \(0\) or \(1\). If \(\alpha=0\) and \(r_0 > 1\), the graph is forced above the diagonal near \(0\). However, since \(g(I)\ne I\), we must have \(\beta > 1\), allowing the graph to remain above the diagonal near \(1\). We can then simply choose \(h(x) > x\) for all \(x \in (0,1)\). All points in \((0,1)\) move monotonically to the right. By symmetry, if \(\beta=1\) and \(r_1 > 1\), we choose \(h(x) < x\) for all \(x \in (0,1)\), moving all points to the left.

\emph{Case 2: \(g\) reverses orientation.} The extension \(h\) must be strictly decreasing with \(h(0) = \beta \ge 1\) and \(h(1) = \alpha \le 0\). The graph of \(h\) necessarily crosses the diagonal exactly once; we place this fixed point at \(c_0\). To prevent the formation of period-2 orbits, we make \(h\) sufficiently steep near \(c_0\) such that \(c_0\) is a repelling fixed point for \(h^2\). That is, we interpolate \(h\) such that:
\[
h^2(x) < x \quad \text{for } 0 < x < c_0, \qquad h^2(x) > x \quad \text{for } c_0 < x < 1.
\]
Under this condition, every point in \((0,1) \setminus \{c_0\}\) eventually has its iterates (or even iterates) move away from \(c_0\) until it exits \((0,1)\).

In all cases, defining \(\tilde g = g\) on \(I\) and \(\tilde g = h\) on \(J\) yields the desired \(C^1\) circle diffeomorphism. 
Let \(\nu\) be a \(\tilde g\)-invariant probability measure. 
We now conclude (1) by proving that \(\nu(J\setminus\{c_0\})=0\).  Indeed, in the orientation-preserving cases, the interval \(J\setminus\{c_0\}\) is the union of at most two intervals on which either \(h(x)<x\) for all \(x\), or \(h(x)>x\) for all \(x\).  
Consider one such interval \(L\subset J\setminus\{c_0\}\).  If \(h(x)<x\) for all \(x\in L\), then for any \(x\in L\), the sequence \(h^n(x)\) is strictly decreasing. Hence the intervals
$(h^{n+1}(x),h^n(x)]$, $n\ge0$, are pairwise disjoint. Since \(\nu\) is finite, invariant and $h$ is injective, all these intervals must have zero measure, which implies \(\nu(L)=0\).  
The case \(h(x)>x\) is analogous.   In the orientation-reversing case, we apply the same argument to the increasing map \(h^2\). By construction, on each component of \(J\setminus\{c_0\}\), either \(h^2(x)<x\) for all \(x\), or \(h^2(x)>x\) for all \(x\). Since \(\nu\) is \(h\)-invariant, it is also \(h^2\)-invariant, and the same argument gives zero measure.  

For item (2), we choose any \(C^1\) extension \(h\) matching the prescribed endpoint jets. Since \(c\) is in the interior of \(J\), we may modify \(h\) exclusively within a small neighborhood of \(c\) (preserving monotonicity and the boundary jets) to ensure \(h(c) \ne c\). 
\end{proof}}

We also use the following elementary criterion.

\begin{prop}
\label{prop:no-invariant}
Let \(f:\Omega\times\mathbb I\to\mathbb I\) be a random map such that \(f_\omega:\mathbb I\to\mathbb I\) is a homeomorphism onto its image for every \(\omega\in\Omega\). If there are no \(p,q\in\mathbb I\) such that \(f_\omega(\{p,q\})=\{p,q\}\) for \(\mathbb P\)-a.e. \(\omega\in\Omega\), then there is no common \(f\)-invariant probability measure.
\end{prop}

{
\begin{proof}
Suppose, by contradiction, that \(\mu\) is a common \(f\)-invariant probability measure: \((f_\omega)_*\mu=\mu\) for \(\mathbb P\)-a.e. \(\omega \in \Omega\). Let \(K=\operatorname{supp}\mu\). Set \(p=\min K\) and \(q=\max K\). Since \(f_\omega\) is a homeomorphism onto its image and preserves \(\mu\), it maps \(K\) onto \(K\) for \(\mathbb P\)-a.e. \(\omega \in \Omega\). Therefore, it must map the endpoints of the convex hull of \(K\) to themselves, possibly interchanging them. Hence \(f_\omega(\{p,q\})=\{p,q\}\) for \(\mathbb P\)-a.e. \(\omega \in \Omega\), contradicting the hypothesis.
\end{proof}
}

{
\section{Proof of Proposition~\ref{cor:interval}}
Let \(I=\mathbb I\) and \(J=\mathbb S^1\setminus I\). By hypothesis, there is no pair of points \(p,q\in I\) such that \(f_\omega(\{p,q\})=\{p,q\}\) for \(\mathbb P\)-a.e. \(\omega\in\Omega\). This hypothesis has two immediate consequences. First, the set 
\[
A=\{\omega\in\Omega:f_\omega(I)\ne I\}
\]
must have positive \(\mathbb P\)-measure. Indeed, if \(f_\omega(I)=I\) almost surely, then \(f_\omega\) would preserve the boundary \(\partial I\), which consists of two endpoints. This yields an invariant two-point set, contradicting the hypothesis. Second, if the system were deterministic (\(f_\omega \equiv g\) almost surely), the continuous interval map \(g\) would necessarily possess a fixed point \(p\), yielding an invariant set \(\{p,p\}\), which again violates the hypothesis. 

Since \(\mathbb P(A)>0\) and the system is not deterministic, 
we can choose a measurable sets \(A_0\subset A\) and $A_1=\Omega \setminus A_0$ with \(0<\mathbb P(A_i)<1\), $i=0,1$. Fix a point \(c\in J\). We construct a \(C^1\) circle extension \(\tilde f_\omega : \mathbb S^1 \to \mathbb S^1\) for each \(\omega \in \Omega\) as follows:
\begin{itemize}
    \item For \(\omega \in A_0\), we have \(f_\omega(I) \ne I\). By Lemma~\ref{lem:extension-exterior-measure}(1), \(f_\omega\) extends to $\tilde{f}_\omega$ such that every \(\tilde f_\omega\)-invariant probability measure satisfies \(\nu(J\setminus\{c\})=0\).
    \item For \(\omega \in A_1\), we apply Lemma~\ref{lem:extension-exterior-measure}(2) to extend \(f_\omega\) such that \(\tilde f_\omega(c)\ne c\).
\end{itemize}
We claim that every common invariant probability measure \(\nu\) for the extended circle system \(\tilde f\) is supported entirely on \(I\). Suppose \((\tilde f_\omega)_*\nu=\nu\) for \(\mathbb P\)-a.e. \(\omega\in\Omega\). Since \(\mathbb P(A_0)>0\), there is an \(\omega_0 \in A_0\) for which this invariance holds. The construction on \(A_0\) then forces 
\[
\nu(J\setminus\{c\})=0, \quad \text{which implies} \quad \nu(J)=\nu(\{c\}).
\]
Similarly, since \(\mathbb P(A_1)>0\), there is an \(\omega_1 \in A_1\) for which \((\tilde f_{\omega_1})_*\nu=\nu\). Because \(\tilde f_{\omega_1}(I) \subset I\), the complementary open interval \(J\) is backward invariant, meaning \(\tilde f_{\omega_1}^{-1}(J) \subset J\). In particular, \(\tilde f_{\omega_1}^{-1}(c)\in J\). Furthermore, since \(\tilde f_{\omega_1}(c)\ne c\), it follows that \(\tilde f_{\omega_1}^{-1}(c)\in J\setminus\{c\}\). Invariance then yields
\[
\nu(\{c\}) = \nu\big(\tilde f_{\omega_1}^{-1}(\{c\})\big) \le \nu(J\setminus\{c\}) = 0.
\]
Consequently, \(\nu(J)=0\), proving the claim.}

If \(\tilde f\) admitted a common invariant probability measure on \(\mathbb S^1\), its restriction to \(I\) would be a common invariant probability measure for \(f\), which is forbidden by Proposition~\ref{prop:no-invariant}. Therefore, the extended system \(\tilde f\) admits no common invariant measure. Applying Proposition~\ref{cor:mostly-contraction-circle-cantor} to \(\tilde f\), we conclude that \(\tilde f\) is mostly contracting. Since \(I\) is forward invariant and \(\tilde f_\omega|_I=f_\omega\), every \(f\)-stationary probability measure is also a \(\tilde f\)-stationary probability measure supported on \(I\), and its Lyapunov exponent for \(f\) coincides with its Lyapunov exponent for \(\tilde f\). Since \(\tilde f\) is mostly contracting, all these exponents are negative. Thus \(f\) is mostly contracting.
Under the exponential moment condition, the conclusions follow by restricting Theorem~\ref{thmA} from \(\tilde f\) to \(I\).

\section{Random Cantor diffeomorphisms: Proof of Proposition~\ref{prop:mostly-cantor}}
Let $f: \Omega \times \mathbb{K} \to \mathbb{K}$ be a random map of $C^1$ diffeomorphisms on a Cantor set $\mathbb{K}$, where $\Omega = T^\mathbb{N}$ with $T$ finite. Since we assume that $f$ has no common invariant measure, according to~\cite[Prop.~6.1]{MM23}, there are  positive constants $\lambda$ and $C$, and  $m\in \mathbb{N}$ such that the following holds: For $\mathbb{P}$-a.e.~$\omega \in \Omega$, there is a finite set $E$ with cardinality less than $m$, and for any neighborhood $V$ of $E$, there is $n_0$ such that for all $n \geq n_0$, we have
$$
\sup_{x \in \mathbb{K} \setminus V} |(f^n_\omega)'(x)| \leq e^{-n\lambda}.
$$

Since any stationary measure $\mu$ (which is not a common invariant) has no atoms (see~\cite[Lemma~3.3]{MM23}), we get
$$
\frac{1}{n} \int \log |Df^n_\omega(x)| \, d\mu \leq - \lambda \quad \text{for $\mathbb{P}$-a.e.~$\omega \in \Omega$.}
$$
This implies that $f$ is mostly contracting. Moreover, since $T$ is finite, by Remark~\ref{rem:thmA-finite-moment}, $f$ also satisfies the conclusions of Theorem~\ref{thmA}. This completes the proof of Proposition~\ref{prop:mostly-cantor}.

\chapter{Locally constant nonlinear cocycles: circle random maps}
\label{s:circle}

\abstract{ {
This chapter applies the mostly contracting theory to locally constant random
products of circle diffeomorphisms with no common
invariant probability measure. We study the H\"older continuity and the asymptotic distribution of the Lyapunov exponents associated
with the stationary measures and their behavior under perturbations. The results provide nonlinear analogues of the regularity and limit theorems
proved earlier for linear cocycles. 
}
}

\section{Integral representation of the Lyapunov exponent}
\label{sec:circle-derivative-cocycle}

We first deduce an integral formula for the derivative Lyapunov exponent.  In
the smooth setting, this formula is standard, but the formulation below requires weaker regularity and integrability assumptions. Thus, it applies in certain non-smooth settings with discontinuities or singularities, allowing the integral of the logarithmic derivative to take the value \(-\infty\).

\begin{lem} \label{lemma:5exponents} 
Let \(X\subset\mathbb R\) be a compact interval or \(X=\mathbb S^1\) and consider a Bernoulli random map $f:\Omega \times X \to X$. Let $\mu$ be a $f$-stationary measure and assume that $f'_\omega(x)$ exists for $(\mathbb{P}\times \mu)$-a.e.~$(\omega,x)$ and $\log^+ |f'_\omega(x)|\in L^1(\mathbb{P}\times \mu)$. Then, 
$$\lambda(\mu)=\int \log |f'_\omega(x)|\, d\mathbb{P} d\mu.$$
\end{lem}
\begin{proof}
    As explained in~\S\ref{sec:lyapunov-def}, the (maximal) Lyapunov exponent of $\mu$ is given~by 
$$
\lambda(\mu) = \lim_{n\to \infty} \frac{1}{n} \int \log |(f^n_\omega)'(x)|\, d\mathbb{P}d\mu.
$$
On the other hand, note that 
$$
\phi_n(x)\eqdef \int \log |(f^n_\omega)'(x)|\, d\mathbb{P} = \sum_{i=0}^{n-1} P^i\phi_1(x) 
$$
where $P$ is the annealed Koopman operator. Then, since $(\phi_n)_{n\geq 1}$ is an $P$-additive sequence and $\phi_1^+\in L^1(\mathbb{P}\times\mu)$, according to the first part of Corollary~\ref{cor:Kingman-uniform}, we get the following integral formula for the Lyapunov exponent
$$
\lambda(\mu)=\int \phi_1 \, d\mu. 
$$
This concludes the proof.
\end{proof}

\section{Limit theorems: proof of Proposition~\ref{mainprop:CLT-circle}}

Let $f\in \DE{1+\epsilon}(\mathbb{S}^1)$ be a uniquely ergodic random map of $C^{1+\epsilon}$ diffeomorphisms ($0<\epsilon\leq 1$) of the circle $(\mathbb{S}^1,d)$ with no common invariant measure, and satisfying a finite exponential moment 
\begin{equation} \label{eq:finite-exponential-moment}
\int \ell_\epsilon(f_\omega)^\beta \, d\mathbb{P} < \infty \quad \text{for some $\beta > 0$},
\end{equation}
where $\ell_\epsilon(g)=\max\{\|g'\|_\infty, \|(g^{-1})'\|_\infty, |g'|_\epsilon\}$ for a diffeomorphism $g$ on $\mathbb{S}^1$. Without loss of generality, we can choose a representative of $f$ satisfying that $f_\omega=f_{\omega_0}$ for $\mathbb{P}$-a.e.~$\omega=(\omega_i)_{i\ge 0}\in \Omega$.   

According to Proposition~\ref{cor:mostly-contraction-circle-cantor}, $f$ is mostly contracting and Theorem~\ref{thmA} holds. In particular, the annealed Koopman operator associated with $f$ is quasi-compact on $C^\alpha(\mathbb{S}^1)$ for any $\alpha > 0$ small enough. We can write the following identity as a Birkhoff sum
\begin{align*}
\log |(f_\omega^n)'(x)| - n\lambda(\mu) &= \sum_{i=1}^{n} \left( \log|f_{\omega_{i-1}}'(f_\omega^{i-1}(x))| - \lambda(\mu)\right) \\ &= \sum_{i=1}^n \xi(\omega_{i-1},f^{i-1}_\omega(x)) = S_n\xi(\omega,(e,x)),
\end{align*}
where $\mu$ is the unique $f$-stationary measure, $f_e=\mathrm{id}_{\mathbb{S}^1}$, and 
$$
\xi(t,x) = \log |f'_t(x)| - \lambda(\mu) \quad \text{for $(t,x)\in T\times \mathbb{S}^1$}.
$$

Although the corollaries in~\S\ref{ss:limit laws} and \S\ref{ss:large-deviation} apply to $f$, we cannot immediately use them because the potential $\xi$ does not belong to $C^\alpha(\mathbb{S}^1)$. However, as explained in~\S\ref{ss:limit-theorems-HH}, we can fix this. We only need to check that $\int \xi \, d\hat{\mu} = 0$ and $K_\xi(\theta) < \infty$ for some $\theta > 0$, where $\hat{\mu} = p \times \mu$ and $K_\xi$ is defined in~\eqref{eq:moment-HH}. 

The orthogonality of $\xi$ and $\hat{\mu}$ follows by Lemma~\ref{lemma:5exponents}. On the other hand, we choose $\theta > 0$ and $0 < \alpha \leq 1$ such that $\theta + \max\{\epsilon \alpha, 2\alpha\} < \beta$. We endow $\mathbb{S}^1$ with the metric $d(\cdot,\cdot)^{\epsilon \alpha}$. In this metric, $\mathrm{Lip}(f_t) \leq \ell(f_t)^{\epsilon \alpha}$, and
$$
|\xi_t|_1 = \sup_{x \neq y} \frac{\left|\log|f_t'(x)| - \log|f'_t(y)|\right|}{d(x,y)^{\epsilon \alpha}}.
$$
Using that $\log s \leq \frac{1}{\alpha}(s^\alpha - 1)$ for $s > 0$, we obtain 
$$
\left|\log|f_t'(x)| - \log|f'_t(y)|\right| \leq \frac{1}{\alpha} \frac{|f'_t(x) - f'_t(y)|^\alpha}{|f'_t(x)|^\alpha} \leq \frac{1}{\alpha} \big(\|(f_t)^{-1}\|_\infty\big)^\alpha \big(|f'_t|_\epsilon\big)^\alpha d(x,y)^{\epsilon \alpha}. 
$$
Consequently, $|\xi_t|_1 \leq \frac{1}{\alpha} \ell(f_t)^{2\alpha}$. Finally, since
$e^{\theta\|\xi_t\|_\infty} \leq \ell(f_t)^\theta$, we get that
\begin{align*}
K_\xi(\theta) &= \int e^{\theta \|\xi_t\|_\infty}(1 + |\xi_t|_1 + \mathrm{Lip}(f_t)) \, dp \\& \leq \int \ell(f_t)^\theta \, dp + \frac{1}{\alpha}\int \ell(f_t)^{\theta + 2\alpha} \, dp + \int \ell(f_t)^{\theta + \epsilon \alpha}\,dp.
\end{align*}
By the choice of $\theta$, $\alpha$, and~\eqref{eq:finite-exponential-moment}, all of the above integrals are finite, and therefore $K_\xi(\theta)$ is finite as well. Thus, the conclusions of the limit theorems in~\S\ref{ss:limit laws} and~\S\ref{ss:large-deviation} hold (with $r=1$) for $S_n\xi(\omega,(e,x)) = \log |(f_\omega^n)'(x)| - n\lambda(\mu)$ instead of $S_n\phi(\omega,x)$. We also have that the variance $\sigma^2 = 0$ if and only if there is $\psi \in C^\alpha(\mathbb{S}^1)$ such that
\begin{equation}
 \label{eq:coboundary}   
\xi(t, x) = \psi(x) - \psi(f_{t}(x)) \quad  \text{for \  $\hat{\mu}$-a.e.~$(t,x)\in T\times \mathbb{S}^1$}.
\end{equation}
However, notice that Corollary~\ref{maincor4} only provides estimates of the decay of the tail event $\{Z_n^x>\epsilon\}$ where $Z^x_n(\omega)\eqdef\frac{1}{n}\log |(f_\omega^n)'(x)| - \lambda(\mu)$. As we did in the proof of Theorem~\ref{thm:CLT-linear-general}, to get the large deviation principle for the modulus random variable $|Z^x_n|$, we apply~\S\ref{ss:limit-theorems-HH} to~$-\xi$. This also provides estimates for the rate of decay of the tail event $\{Z_n^x<-\epsilon\}$. Hence, we conclude in an analogous way the corresponding decay for $|Z^x_n|$ using~\cite[Lemma~4.4]{huang2012moments}.

The final requirement to conclude Proposition~\ref{mainprop:CLT-circle} is to show the positivity of $\sigma^2$ under the assumption that $f$ is minimal. To do this, assume $\sigma^2 = 0$. Then~\eqref{eq:coboundary} holds. Since the action is minimal, the topological support of $\mu$ is $\mathbb{S}^1$, and so by continuity, we get that for $p$-a.e.~$t \in T$, the equality above holds for every $x$ in $\mathbb{S}^1$. Thus, setting 
$\Psi(x) = e^{\psi(x)}/\int e^{\psi}\, d\mathrm{Leb}$,
we can rewrite~\eqref{eq:coboundary} for $p$-a.e.~$t\in T$ as 
$$ 
e^{-\lambda(\mu)}\left|f'_t(x)\right| = e^{\xi(t,x)} = e^{\psi(x) - \psi(f_t(x)) } = \frac{\Psi(x)}{\Psi(f_t(x))} \quad \text{for any $x\in \mathbb{S}^1$.} 
$$ 
Since the map $\Psi$ is positive and $\int \Psi \, d\mathrm{Leb} = 1$, we can write $\Psi = g'$ for some diffeomorphism $g$ of the circle. Then we have that for $p$-a.e.~$t \in T$ and every $x \in \mathbb{S}^1$,
\begin{align*}
\left|(g\circ f_t \circ g^{-1})'(x)\right| &=\left|\frac{(g\circ f_t)'(g^{-1}(x))}{g'(g^{-1}(x))}\right| =\left|\frac{\Psi\left(f_t\circ g^{-1}(x)\right)}{\Psi\left(g^{-1}(x)\right)} \cdot f'_t\left(g^{-1}(x)\right)\right| \\ 
&= \left|\frac{\Psi\left(f_t\circ g^{-1}(x)\right)}{\Psi\left(g^{-1}(x)\right)} \cdot \frac{\Psi\left(g^{-1}(x)\right)}{\Psi\left(f_t\circ g^{-1}(x)\right)} \cdot e^{\lambda(\mu)} \right| = e^{\lambda(\mu)} < 1,
\end{align*} 
since $f$ is mostly contracting and thus $\lambda(\mu) < 0$. But this is a contradiction, as $g\circ f_t \circ g^{-1}$ is a circle diffeomorphism. Thus, $\sigma^2 > 0$, and we conclude the proof.

\begin{rem}
    To get the statement on CLT and LD provided in Proposition~\ref{mainprop:CLT-circle} we choose $\nu=\delta_x$ in Corollaries~\ref{maincor2},~\ref{maincor3} and~\ref{maincor4}. However, Proposition~\ref{mainprop:CLT-circle}  holds with respect to any probability measure $\nu$ on $\mathbb{S}^1$.    
\end{rem}

\section{Strong law of large numbers: proof of Proposition~\ref{mainprop:SLLN}}




We now prove the strong law. Unlike the previous limit theorems, this part only uses the Furstenberg law for one-step observables and does not require the quasi-compact perturbation argument. Given $f \in \DE{1}(\mathbb{S}^1)$, observe that 
$$
\log |(f^n_\omega)'(x)| = \sum_{i=0}^{n-1} \phi(F^i(\omega,x)),
$$
where $F$ is the skew-product associated with $f$, and 
$\phi(\omega,x) = \log |f'_\omega(x)|$. By the assumptions of Proposition~\ref{mainprop:SLLN}, $\phi(\omega,\cdot)\in C(\mathbb{S}^1)$ for every $\omega\in \Omega$ and 
$\int \|\phi(\omega, \cdot)\|_\infty \,d\mathbb{P} < \infty$.
 Furthermore, since $f$ is a uniquely ergodic random map, and by Lemma~\ref{lemma:5exponents} we have $\lambda(\mu) = \int \phi(\omega,x)\, d\mathbb{P} d\mu$, it follows that this integral is independent of $\mu$. Hence, Corollary~\ref{thm:Bierman-Furstenberg} directly implies Proposition~\ref{mainprop:SLLN}.

\section{H\"older continuity: proof of Proposition~\ref{mainprop:nolinear-exponent}}

We now turn to the perturbative result. Recall the metric spaces
\((\DE{\,r}(\mathbb S^1), d_{C^r}^{\pm})\) introduced in
\S\ref{ss:holder-circle}. We start with the following elementary estimate,
which will be used to keep the exponential moment condition under perturbation.



\begin{lem} \label{lemma:1exponents} For any $\beta>0$, $0<\alpha \leq \min\{1,\beta\}$ and $0<\epsilon\leq 1$, 
    \begin{align*}
\int \big(\|(h^{-1}_\omega)'\|_\infty\big)^\alpha \,d\mathbb{P}
&\leq 1 + \int \big(\|(f_\omega^{-1})'\|_\infty\big)^\beta \, d\mathbb{P} +  \D_{\smash{C^1}}^\pm(h,f)^{\alpha} \ \ \text{for $h,f\in \DE{\,1}(\mathbb{S}^1)$} 
\end{align*}
and
    \begin{align*}
    \int \big(|(h^{-1}_\omega)'|_\epsilon\big)^\alpha \,d\mathbb{P}
&\leq 1 + \int \big(|(f_\omega^{-1})'|_\epsilon \big)^\beta \, d\mathbb{P} +  \D_{\smash{C^{1+\epsilon}}}^\pm(h,f)^{\alpha} \ \ \text{for $h,f\in \DE{1+\epsilon}(\mathbb{S}^1)$}.
\end{align*}
\end{lem}

\begin{proof} 
Let us prove the first inequality. The second inequality follows similarly. To do this,  observe that, since $t^\alpha - s^\alpha \leq |t-s|^\alpha$ for any $t,s\in\mathbb{R}$, 
\begin{align*}
\int \big(\|(h^{-1}_\omega)'\|_\infty\big)^\alpha  \, d\mathbb{P} 
\leq \int\big(\|(f_\omega^{-1})'\|_\infty \big)^\alpha \, d\mathbb{P} + 
\int \big| \|(h_\omega^{-1})'\|_\infty-\|(f_\omega^{-1})'\|_\infty\big|^\alpha \, d\mathbb{P}.
\end{align*}
From here, taking into account that $\alpha\leq \beta$ and $\alpha \leq 1$, using the Jensen inequality and that $\big|\|(h^{-1}_\omega)'\|_\infty - \|(f^{-1}_\omega)'\|_\infty\big| \leq \|(h^{-1}_\omega)' - (f^{-1}_\omega)' \|_\infty \leq d_{C^1}(f^{-1}_\omega,h_{\omega}^{-1})$, we conclude the required inequality. 
\end{proof}


\begin{thm} \label{mainprop:nolinear-exponent-improved}
    Let $f\in \DE{r}(\mathbb{S}^1)$, $r>1$, be a random map with      no common invariant probability measure and satisfying the exponential moment condition~\eqref{eq:exponential-moment-circulo}. Let $\mathcal{C}$ be the stratum of random maps with the same number $s\in\mathbb{N}$ of ergodic stationary measures as $f$. Then, there exists a neighborhood $\mathcal{B}$ of $f$ in $(\DE{r}(\mathbb{S}^1),\D_{\smash{C^{r}}}^\pm)$ such that for every $g\in \mathcal{C}\cap \mathcal{B}$, the ergodic $g$-stationary  measures can be labeled as $\mu_{g,1},\dots,\mu_{g,s}$ so  that $\mu_{g,i}$ varies continuously in the weak$^*$ topology for $g\in\mathcal{C}\cap\mathcal{B}$ and
     $$g\in\mathcal{C}\cap\mathcal{B} \mapsto \lambda(\mu_{g,i})\in (-\infty,0)$$ is H\"older continuous for $i=1,\dots,s$. 
    Moreover, there are $C>0$, $\gamma \in (0,1]$ such that
$$
   |\lambda(\mu_{g,i})-\lambda(\mu_{h,i})|\leq C  \D_{C^1}(g,h)^\gamma_{} \quad \text{for all $g,h\in \mathcal{B}\cap \mathcal{C}$.}
$$ 
In particular, $$|\lambda(g)-\lambda(h)|\leq C \D_{C^1}(g,h)^\gamma_{}.$$
\end{thm}

\begin{proof}
Let $f\in \DE{\,r}(\mathbb{S}^1)$ be 
{a random map with} $0<\epsilon\leq 1$ such that the exponential
moment~\eqref{eq:exponential-moment-circulo} holds for some $\beta>0$ and with
no common invariant measure. According to
Proposition~\ref{cor:mostly-contraction-circle-cantor}, $f$ is mostly
contracting and, since~\eqref{eq:exponential-moment-circulo}
implies~\eqref{eq:integral_condition} in Theorem~\ref{thmA}, we can apply Theorem~\ref{prop:statistical-stability-length}. Moreover, since
the inclusion
\[
   (\DE{\,r}(\mathbb{S}^1),\D_{\smash{C^{1}}}^{\pm})
   \hookrightarrow
   (\CE(\mathbb{S}^1),\D_{C^{1}})
\]
is continuous, we can take a neighborhood $\mathcal{B}$ of $f$ in
$(\DE{\,r}(\mathbb{S}^1),\D_{\smash{C^{1+\epsilon}}}^{\pm})$
such that 
{Theorem~\ref{prop:statistical-stability-length} works on
\(\mathcal C\cap\mathcal B\). Thus, after shrinking \(\mathcal B\) if
necessary, the ergodic \(g\)-stationary measures, \(g\in\mathcal C\cap
\mathcal B\), can be labelled
$ \mu_{g,1},\dots,\mu_{g,s}$, 
so that each \(g\mapsto\mu_{g,i}\) is weak\(^*\)-continuous on
\(\mathcal C\cap\mathcal B\), and the H\"older estimates of
Theorem~\ref{prop:statistical-stability-length} hold on this stratum}.

For each $h\in \mathcal{B}$, define the map
\[
\phi_h(x)= \int \log |h'_\omega(x)|\, d\mathbb{P},
\qquad x\in\mathbb{S}^1 .
\]

\begin{claim} \label{claim:holder}
\(\phi_h:\mathbb{S}^1\to \mathbb{R}\) is
\(\alpha\)-H\"older continuous
{for every sufficiently small \(\alpha>0\). Moreover, for each such
\(\alpha\), after shrinking \(\mathcal B\) if necessary, the quantities
\(\|\phi_h\|_\alpha\) are uniformly bounded for \(h\in\mathcal B\).}
\end{claim}

\begin{proof}
Since $h\in\mathcal{B}$, by Proposition~\ref{prop:lasota-york-uniforme} we have
that $(\|h'_\omega \|_\infty)^\beta$ is $\mathbb{P}$-integrable. Moreover,
Lemma~\ref{lemma:1exponents} and the exponential
moment~\eqref{eq:exponential-moment-circulo} also imply the
$\mathbb{P}$-integrability of $(\|(h^{-1}_\omega)'\|_\infty)^\beta$. Using that
$\log t < t^\beta /\beta$ for $t>0$, this implies that
\[
-\infty
<
-\frac{1}{\beta}
\int  \big(\|(h_\omega^{-1})'\|_\infty\big)^\beta \, d\mathbb{P}
\leq
\phi_h(x)
\leq
\frac{1}{\beta}
\int  \big(\|h_\omega'\|_\infty\big)^\beta\,d\mathbb{P}
<\infty.
\]

Let $0< a \leq \min\{1,\beta/2\}$. Using that
$ \log t \leq \frac{1}{a} (t^a-1)$ for $t>0$,
and $|t|^a-|u|^a \leq |t-u|^a$ for $t,u\in \mathbb R$, we obtain
\begin{align*}
    \big|\phi_h(x)-\phi_h(y)\big| &
    =
    \left| \int \log \frac{|h'_\omega(x)|}{|h'_\omega(y)|}  \, d\mathbb{P}\right| \leq
    \frac{1}{a}
    \int
    \frac{\big| h'_\omega(x)-h'_\omega(y) \big|^a}
    {|h'_\omega(y)|^a} \, d\mathbb{P} \\
     &\leq
     \frac{1}{a}
     \int
     \big(\|(h^{-1}_\omega)'\|_\infty\big)^a
     \big(|h'_\omega|_\epsilon\big)^a
     \, d\mathbb{P}\cdot d(x,y)^{\epsilon a} \\
     &\leq
     \frac{1}{a}
     \left(\int
     \big(\|(h^{-1}_\omega)'\|_\infty\big)^{2a}
     \, d\mathbb{P}\right)^{1/2}
     \left(\int
     \big(|h'_\omega|_\epsilon \big)^{2a}
     \, d\mathbb{P}\right)^{1/2}
     d(x,y)^{\epsilon a}.
\end{align*}
The last inequality follows from H\"older's inequality. Since $h\in\mathcal B$
and $2a\leq\beta$, the exponential
moment~\eqref{eq:exponential-moment-circulo} and
Lemma~\ref{lemma:1exponents} imply that the factors multiplying
$d(x,y)^{\epsilon a}$ are finite. {Moreover, after shrinking
\(\mathcal B\), these factors are uniformly bounded for \(h\in\mathcal B\).}
This proves that $\phi_h$ is H\"older continuous with arbitrarily small
H\"older exponent, {and gives the asserted uniform bound on
\(\|\phi_h\|_\alpha\) for each sufficiently small \(\alpha>0\).}
\end{proof}

{By the choice of \(\mathcal B\), for every \(h\in\mathcal C\cap
\mathcal B\) the ergodic \(h\)-stationary measures are labelled as
\(\mu_{h,1},\dots,\mu_{h,s}\), and each \(\mu_{h,i}\) varies continuously in
the weak\(^*\) topology on  \(\mathcal C\cap\mathcal B\).}

By Lemma~\ref{lemma:5exponents}, for every \(h\in\mathcal C\cap\mathcal B\)
and every \(i=1,\dots,s\), we get 
\[
   \lambda(\mu_{h,i})
   =
   \int \phi_h\,d\mu_{h,i}.
\]
Moreover, since 
{every \(h\in\mathcal B\) is mostly contracting,} we have
\[
   \lambda(\mu_{h,i})\in (-\infty,0),
   \qquad h\in\mathcal C\cap\mathcal B,\quad i=1,\dots,s.
\]
Now, for $g,h \in \mathcal{C}\cap\mathcal{B}$ and \(i=1,\dots,s\), it holds
\begin{align*}
\big|\lambda(\mu_{g,i})-\lambda(\mu_{h,i})\big|
&=
\left|
   \int \phi_g \, d\mu_{g,i}
   -
   \int \phi_h \, d\mu_{h,i}
\right| \\
&\leq
\int |\phi_g - \phi_h |\, d\mu_{g,i}
+
\left|
   \int \phi_h \, d\mu_{g,i}
   -
   \int \phi_h \, d\mu_{h,i}
\right|.
\end{align*}

From Claim~\ref{claim:holder}, \(\phi_h\) is H\"older continuous. Therefore,
using the H\"older estimate in
Theorem~\ref{prop:statistical-stability-length} on the stratum
\(\mathcal C\cap\mathcal B\), we get
\begin{equation} \label{eq:lambda-triangular}
\big|\lambda(\mu_{g,i})-\lambda(\mu_{h,i})\big|
\leq
\|\phi_g-\phi_h\|_\infty
+
C\,\D_{C^0}(g,h)^\gamma
\end{equation}
for some constant $C>0$ and exponent $\gamma\in(0,1]$, uniform for
\(g,h\in\mathcal C\cap\mathcal B\) and \(i=1,\dots,s\).
{Shrinking \(\gamma=O(\alpha)\) if necessary, we may assume that
\(\gamma<\min\{\beta/2,1\}\).} Then, as before, using $\log t \leq \frac{1}{\gamma}(t^\gamma-1)$, for $t>0$, and $|t|^\gamma-|u|^\gamma\leq |t-u|^\gamma$, for $t,u\in\mathbb R$, together with H\"older's inequality, we obtain
\begin{align*}
\big\|\phi_g-\phi_h\big\|_\infty
&=
\sup_{x\in\mathbb S^1}
\left|
   \int \log \frac{|g'_\omega(x)|}{|h'_\omega(x)|}
   \, d\mathbb P
\right| 
\leq
\frac{1}{\gamma}
\sup_{x\in\mathbb S^1}
\int
\frac{\|g'_\omega-h'_\omega\|_\infty^\gamma}
{|h'_\omega(x)|^\gamma}
\, d\mathbb P \\
&\leq
\frac{1}{\gamma}
\sup_{x\in\mathbb S^1}
\left(
   \int |h'_\omega(x)|^{-2\gamma}\,d\mathbb P
\right)^{1/2}
\left(
   \int \|g'_\omega-h'_\omega\|_\infty^{2\gamma}\,d\mathbb P
\right)^{1/2} \\
&\leq
\frac{1}{\gamma}
\left(
   \int
   \big\|(h^{-1}_\omega)'\big\|_\infty^{2\gamma}\,d\mathbb P
\right)^{1/2}
\left(
   \int
   \|g'_\omega-h'_\omega\|_\infty\,d\mathbb P
\right)^\gamma.
\end{align*}
The above inequality, the exponential
moment~\eqref{eq:exponential-moment-circulo}, Lemma~\ref{lemma:1exponents},
and~\eqref{eq:lambda-triangular} imply
\[
   |\lambda(\mu_{g,i})-\lambda(\mu_{h,i})|
   \leq
   K\,\D_{C^1}(g,h)^\gamma
\]
for every \(g,h\in\mathcal C\cap\mathcal B\) and every \(i=1,\dots,s\), where
\(K>0\) is a constant depending only on the neighborhood \(\mathcal B\). This
proves the H\"older continuity of
\[
   g\in\mathcal C\cap\mathcal B
   \longmapsto
   \lambda(\mu_{g,i})
\]
for every \(i=1,\dots,s\). {Finally  $|\lambda(g)-\lambda(h)|\leq K \,\D_{C^1}(g,h)^\gamma$  follows as in Corollary~\ref{cor:holder-maximal-average}, completing the proof of the theorem.} 
\end{proof}

We then conclude Proposition~\ref{mainprop:nolinear-exponent} from the previous theorem:

\begin{proof}[Proof of Proposition~\ref{mainprop:nolinear-exponent}]
We apply Theorem~\ref{mainprop:nolinear-exponent-improved} to $f$. Since in this case we are also assuming that $f$ is uniquely ergodic, by the upper
semicontinuity of the number of ergodic stationary measures, the constant-rank stratum coincides, after shrinking the
neighborhood if necessary, with the whole neighborhood. Now, Proposition~\ref{mainprop:nolinear-exponent} follows from Theorem~\ref{mainprop:nolinear-exponent-improved}. 
\end{proof}

{As in Theorem~\ref{prop:generic-statistical-stability},  the stability result on the stratum $\mathcal{C}$ 
and the upper semicontinuity of the number of ergodic stationary measures for mostly contracting random maps yields the genericity of the persistence and continuity of Lyapunov exponents in the region of no common invariant measures. The following proposition studies the genericity of this region in the space $(\DE{1}(\mathbb{S}^1),\D_{\smash{C^{r}}}^{\pm})$.

\begin{prop}
\label{prop:generic-no-common-circle-DE}
Let \( r \geq 1 \), and assume that the probability measure \(\mathbb{P}\) is not a Dirac measure. Then, the set of random maps with no common invariant probability measure contains an open and dense subset of the space \( (\DE{r}(\mathbb{S}^1), \D_{C^r}^{\pm}) \).  Moreover, this conclusion holds when restricted to the subspace of random maps satisfying a finite moment condition of the form
\[
   \int \ell_\epsilon(f_\omega)^\beta\,d\mathbb{P}<\infty ,
\]
where \(0<\epsilon\leq\min\{r-1,1\}\), \(\beta>0\), and
\[
   \ell_\epsilon(g)
   =
   \max\left\{
      \|g'\|_\infty,\,
      \|(g^{-1})'\|_\infty,\,
      |g'|_\epsilon
   \right\}.
\]
\end{prop}

\begin{proof}
Take Morse-Smale circle $C^r$ diffeomorphisms
\(a,b\) such that
$\mathrm{Per}(a)\cap \mathrm{Per}(b)=\emptyset$.
After taking sufficiently small \(C^r\)-neighborhoods \(U\) of \(a\) and \(V\) of \(b\), every \(u\in U\) and every \(v\in V\) are Morse-Smale and
satisfy $\mathrm{Per}(u)\cap\mathrm{Per}(v)=\emptyset$. Define
\[
   \mathcal G(U,V)
   =
   \left\{
      f\in\DE{r}(\mathbb S^1):
      p(\{t:f_t\in U\})>0
      \ \text{and}\
      p(\{t:f_t\in V\})>0
   \right\}
\]
where $p$ is the initial law of $\mathbb{P}$, that is, $\mathbb{P}=p^\mathbb{N}$. 

We first show that every \(f\in\mathcal G(U,V)\) has no common invariant
probability measure. Suppose, by contradiction, that a probability measure 
\(\mu\) on $\mathbb{S}^1$ satisfies
$(f_t)_*\mu=\mu$ for \(\mathbb P\)-a.e.~$t$. 
Since $p(\{t:f_t\in U\})>0$ and $p(\{t:f_t\in V\})>0$, we can choose
\(u\in U\) and \(v\in V\) such that
$ u_*\mu=\mu$ and $v_*\mu=\mu$. Every invariant probability measure of a Morse-Smale circle diffeomorphism is
supported on its periodic orbits. Hence
$ \mu(\mathrm{Per}(u))=1$ and $\mu(\mathrm{Per}(v))=1$,
contradicting \(\mathrm{Per}(u)\cap\mathrm{Per}(v)=\emptyset\). Thus \(f\) has no common
invariant probability measure.

We now prove that \(\mathcal{G}(U,V)\) is open. Let \(f \in \mathcal{G}(U,V)\), and let \(d_{C^r}^{\pm}\) denote the pointwise metric on \(\mathrm{Diff}^r(\mathbb{S}^1)\) given by the integrand of \(\D_{C^r}^{\pm}\), namely
\[
   d_{C^r}^{\pm}(u,v) = d_{C^{\lfloor{r}\rfloor}}(u,v) + d_{C^{\lfloor{r}\rfloor}}(u^{-1},v^{-1}) + |u'-v'|_{r-\lfloor{r}\rfloor}.
\]
Since \(U\) is open and $p(\{t:f_t\in U\})>0$, there is
\(\rho>0\) and a measurable set \(A\) with \(p(A)>0\) such that
\[
   f_t \in U
   \quad\text{and}\quad
   d_{C^r}^{\pm}(f_t,\mathrm{Diff}^r_+(\mathbb S^1)\setminus U)>\rho
   \qquad\text{for every }t\in A .
\]
Thus, for any random map \(g\) and any \(t \in A\), the condition \(g_t \notin U\) forces the pointwise distance to satisfy \(d_{C^r}^{\pm}(f_t, g_t) > \rho\). Consequently, Markov's inequality yields
\[
   p(\{t \in A : g_t \notin U\}) 
   \leq \frac{1}{\rho} \int_A d_{C^r}^{\pm}(f_t, g_t) \, dp 
   \leq \frac{1}{\rho} \D_{C^r}^{\pm}(f,g) .
\]
If \(\D_{C^r}^{\pm}(f,g) < \delta_U\eqdef \rho \cdot p(A)\), we get that $p(\{t \in A : g_t \notin U\})  < p(A)\) and thus, 
\[
   p(\{t \in A : g_t \in U\}) = p(A) - p(\{t \in A : g_t \notin U\}) > 0.
\]
This directly implies that the total probability \(p(\{t : g_t \in U\})\) is strictly positive. Applying an identical geometric argument to \(V\) provides a constant \(\delta_V > 0\) such that if \(\D_{C^r}^{\pm}(f,g) < \delta=\min\{\delta_U,\delta_V\}\), then  \(p(\{t : g_t \in U\})\cdot p(\{t : g_t \in V\}) > 0\). This concludes that \(\mathcal{G}(U,V)\) is open.

Finally, let \(f \in \DE{r}(\mathbb{S}^1)\) and \(\varepsilon > 0\). Since \(p\) is not a Dirac measure, there exist disjoint measurable sets \(A\) and \(B\) such that \(p(A) > 0\) and \(p(B) > 0\). 
Because Morse-Smale diffeomorphisms are dense in \(\mathrm{Diff}^r(\mathbb{S}^1)\), we can cover the space of diffeomorphisms with open balls of radius \(\varepsilon/2\) centered at a Morse-Smale. Since \(A\) has positive measure, there exists an open ball \(B_{\varepsilon/2}(a)\) such that the set 
\[
   \tilde{A} = \{t \in A : f_t \in B_{\varepsilon/2}(a)\}
\]
has strictly positive \(p\)-measure. Similarly, there exists an open ball \(B_{\varepsilon/2}(b)\) such that the set
\[
   \tilde{B} = \{t \in B : f_t \in B_{\varepsilon/2}(b)\}
\]
has strictly positive \(p\)-measure. Here $a$ and $b$ are Morse-Smale diffeomorphisms. Because the set of periodic orbits of a Morse-Smale diffeomorphism is finite, we can perturb \(a\) and \(b\) infinitesimally (keeping them within their respective \(\varepsilon/2\)-balls) to ensure that \(\mathrm{Per}(a) \cap \mathrm{Per}(b) = \emptyset\). Taking sufficiently small neighborhoods \(U\) of \(a\) and \(V\) of \(b\), every \(u \in U\) and \(v \in V\) are Morse-Smale and satisfy \(\mathrm{Per}(u) \cap \mathrm{Per}(v) = \emptyset\). 

Now, define the perturbed random map \(h\) by
\[
   h_t
   =
   \begin{cases}
      a, & t \in \tilde{A},\\
      b, & t \in \tilde{B},\\
      f_t, & t \notin \tilde{A} \cup \tilde{B} .
   \end{cases}
\]
By construction, for \(t \in \tilde{A}\), the distance \(d_{C^r}^{\pm}(f_t, h_t) = d_{C^r}^{\pm}(f_t, a) < \varepsilon/2\), and similarly for \(t \in \tilde{B}\), \(d_{C^r}^{\pm}(f_t, b) < \varepsilon/2\). Integrating this pointwise bound yields
\[
   \D_{C^r}^{\pm}(f,h) = \int_{\tilde{A}} d_{C^r}^{\pm}(f_t, a) \, dp + \int_{\tilde{B}} d_{C^r}^{\pm}(f_t, b) \, dp < \varepsilon.
\]
Moreover, \(h \in \mathcal{G}(U,V)\) because \(h_t = a \in U\) on \(\tilde{A}\) and \(h_t = b \in V\) on \(\tilde{B}\), with \(p(\tilde{A}) > 0\) and \(p(\tilde{B}) > 0\). Thus since $\mathcal{G}(U,V)$ is open and its elements has no common invariant probability measure, we get that $f$ is $\varepsilon$-approximated by an open set of random maps with no common invariant measures.  This proves that the set of random maps without a common invariant measure contains an open and dense subset of \( (\DE{r}(\mathbb{S}^1), \D_{C^r}^{\pm}) \). 

If \(f\) satisfies the moment condition
\[
   \int \ell_\epsilon(f_t)^\beta \, dp(t) < \infty,
\]
then the perturbed map \(h\) also satisfies it, as \(h\) differs from \(f\) only by the insertion of the two fixed smooth diffeomorphisms \(a\) and \(b\). Hence, the density statement remains valid when restricted to this subspace. This completes the proof.
\end{proof}

Now, we get the mentioned corollary: 

\begin{cor} \label{cor:generic-lyapunov-circle}
For \(r>1\), the set of random maps for which the Lyapunov
exponents associated with all ergodic stationary measures persist and vary H\"older continuously is open and dense in the subspace of $(\DE{r}(\mathbb{S}^1),\D_{\smash{C^{r}}}^{\pm})$ 
satisfying~\eqref{eq:exponential-moment-circulo}.
\end{cor}

\begin{proof}
By Proposition~\ref{prop:generic-no-common-circle-DE}, the set of random maps
having no common invariant probability measure contains an open and dense in $(\DE{r}(\mathbb{S}^1),\D_{\smash{C^{r}}}^{\pm})$ satisfying the finite moment condition.  By Proposition~\ref{cor:mostly-contraction-circle-cantor}, every random map of
circle diffeomorphisms with no common invariant probability measure is mostly
contracting. On the mostly contracting region, the number of ergodic stationary probability
measures is upper semicontinuous by
Theorem~\ref{prop:statistical-stability-length}.  Moreover, since \(r>1\), we have
\(f\in \DE{1+\epsilon}(\mathbb S^1)\) for some
\(0<\epsilon\leq1\). Thus Theorem~\ref{mainprop:nolinear-exponent-improved} applies to the constant-rank stratum. Arguing now as in Theorem~\ref{prop:generic-statistical-stability}, we conclude the proof. 
\end{proof}
}

\chapter{Local contraction}
\label{s:local-contraction}

\abstract{{
This chapter proves a local contraction theorem for measurable Lipschitz
cocycles over an arbitrary measurable base. Namely, we show that a negative pointwise Lyapunov exponent forces exponential contraction on a neighborhood of the base point along the corresponding fiberwise orbit. This non-smooth Pesin-type theorem is then applied to mostly contracting random maps, yielding
exponential local contraction in compact spaces and, more generally, under a
tightness assumption in separable complete metric spaces. The chapter concludes
with a statistical application inspired by the global Palis' conjecture: random systems
with the local contraction property admit only finitely many physical measures,
and their basins cover almost every point for the natural product reference
measure.}
}

\section{Negative  Lyapunov exponent} \index{local contraction!non-smooth Pesin theorem}
In this subsection, we consider a general skew product 
\begin{equation} \label{eq:skew-product-F}
F(\omega,x)=(T(\omega),f_\omega(x)), \quad \omega\in \Omega, \ \ x \in X
\end{equation}
where $\Omega\times X$  is a product space of a 
 measurable space $(\Omega,\mathscr{F})$ and a metric space~$(X,d)$. Moreover, we assume that $T:\Omega \to \Omega$ is a measurable transformation  and  $f_\omega\in \mbox{Lip}(X)$ for all $\omega\in \Omega$. As usual,  the second coordinate of $F^n(\omega,x)$ is denoted by $f_\omega^n(x)$ and  the maximal Lyapunov exponent of $F$ is pointwise defined at  $(\omega,x)$ by
$$\lambda(\omega,x) \eqdef \limsup_{n\to +\infty} \frac{1}{n}\log Lf_\omega^n(x).$$

The following theorem is a well-known fact of Pesin's theory in the smooth case.  Despite our Lipschitz setting not allowing for the application of Pesin theory on stable manifolds, we will demonstrate that the exponential contraction remains valid under the assumption of a negative maximal Lyapunov exponent.

\begin{mainthm}\label{contraction}
{Let $F$ be a skew-product as~\eqref{eq:skew-product-F} of the product $\Omega \times X$ of a 
 measurable space $(\Omega,\mathscr{F})$ and a metric space~$(X,d)$.} Let $\bar{\mu}$ be an $F$-invariant measure satisfying the integrability condition
\begin{equation}\label{integrability}
\int \log^+ \mathrm{Lip}(f_{\omega}) \, d\bar{\mu}<\infty .
\end{equation}
For $\bar{\mu}$-a.e.~$(\omega,x)\in \Omega\times X$, if $\lambda(\omega,x)<0$, then
{for every $\chi>\lambda(\omega,x)$,
there exist $\delta=\delta(\omega,x,\chi)>0$ and
\(C=C(\omega,x,\chi)>0\) such that}
\[
d(f_\omega^n(y),f_\omega^n(z))
\leq C e^{n\chi} d(y,z)
\quad
\text{for all } y,z\in B(x,\delta) \text{ and all } n\ge 0.
\]
\end{mainthm}

The proof of the above theorem will be provided in the subsequent sections.  

{
\begin{rem}
The skew-product notation in Theorem~\ref{contraction} is only a convenient
way of writing a fiberwise statement. The proof applies, without essential
changes, to standard Borel bundles with metric fibers. That is, one may work with a measurable bundle \(\pi:\hat{X}\to\Omega\) whose fibers are metric spaces modeled by $(X,d)$ in the following sense: there is a Borel set $\mathcal{E}$ of $\Omega \times X$ with nonempty sections and a fiber-preserving Borel isomorphism $\iota:\hat{X} \to \mathcal E$  
whose restriction to each fiber is a bi-Lipschitz homeomorphism. The skew-product map corresponds in this setting to a bundle map covering $T$, i.e.,
\[
   \mathcal F:\mathcal E\to\mathcal E,
   \qquad
   \pi\circ\mathcal F=T\circ\pi,
\]
whose fiber maps are ($\bar\mu$-a.e.) Lipschitz and satisfy the corresponding integrability condition~\eqref{integrability}. Thus, the theorem
applies, for instance, to derivative cocycles on tangent bundles, or more
generally to cocycles arising from differentiable dynamics after passing to the appropriate fiber bundle.
 \end{rem}
}

\subsection{{The flat Lyapunov exponent}} \label{ss:equivalent} 
To prove Theorem~\ref{contraction}, we define the {\emph{flat Lyapunov exponent}}\index{Lyapunov exponents!0@\(\lambda^\flat(\omega,x)\), flat Lyapunov exponent} as follows:
$$
\lambda^{\flat}
(\omega,x)\eqdef\lim_{r\to 0^+} \limsup_{n\to \infty} \frac{1}{n}\log L_rf_\omega^n(x).$$
Note that the limit in $r$ exists due to monotonicity.  In fact, since $Lf^n_\omega(x)=\lim_{r\to 0^+} L_rf^n_\omega(x)$, if we reverse the order of the two limits in the above expression, we get $\lambda(\omega,x)$. {Thus, the \emph{flat} exponent detects asymptotic growth at a neighborhood level,
whereas the usual \emph{pointwise} exponent detects it at an infinitesimal level.} On the other hand, it is clear that 
\begin{equation} \label{eq:inequailty0}
\lambda^{\flat}(\omega,x)\ge \lambda(\omega,x) 
\quad \text{for all $(\omega,x)\in \Omega\times X$}.
\end{equation}
It can be verified that the functions $\lambda$ and $\lambda^{\flat}$ are both sub-invariant by $F$, i.e., $\lambda\leq \lambda\circ F$ and $\lambda^{\flat}\leq \lambda^{\flat}\circ F$. Therefore, if $\bar{\mu}$ is an ergodic $F$-invariant measure, these two functions are $\bar{\mu}$-almost everywhere constants, denoted respectively by  $\lambda(\bar{\mu})$ and $\lambda^{\flat}(\bar{\mu})$.

We can reformulate Theorem~\ref{contraction} as follows:

\begin{prop}\label{contraction'}
If $\bar{\mu}$ is an $F$-invariant measure satisfying~\eqref{integrability}, then the equality $\lambda^{\flat}=\lambda$ holds $\bar{\mu}$-almost everywhere in the set $\{\lambda<0\}$. 
\end{prop}

{
Let us explain why Proposition~\ref{contraction'} is equivalent to
Theorem~\ref{contraction}. Assume first that Theorem~\ref{contraction} holds.
Fix a point \((\omega,x)\) with \(\lambda(\omega,x)<0\), and let
\(\chi\) satisfy \(\chi>\lambda(\omega,x)\). Then the theorem gives
\(\delta>0\) and \(C>0\) such that
$L_\delta f_\omega^n(x)\le C e^{n\chi}$ for every $n\ge0$.
Hence $\lambda^\flat(\omega,x)\le \chi$. Letting \(\chi\to \lambda(\omega,x)\), and using
\(\lambda(\omega,x)\le \lambda^\flat(\omega,x)\), we obtain
$\lambda^\flat(\omega,x)=\lambda(\omega,x)$. Conversely, if
$\lambda^\flat(\omega,x)=\lambda(\omega,x)<0$, then for any \(\chi>\lambda(\omega,x)\), by the definition of \(\lambda^\flat\), there exist \(r>0\) and \(n_0\ge1\)
such that $L_r f_\omega^n(x)\le e^{n\chi}
\le e^{n\chi}$  for all $n\ge n_0$.  Decreasing \(r\) if
necessary, setting $\delta=r$, and absorbing the finitely many values \(0\le n<n_0\) into a constant
\(C\), one obtains the estimate of Theorem~\ref{contraction}.
}

\subsection{Preliminary lemmas}  We begin with a simpler statement.

\begin{lem}\label{contractionlem1}
Let $\bar{\mu}$ be an ergodic $F$-invariant measure satisfying~\eqref{integrability}. If 
\[
\int \log Lf_{\omega}(x) \, d\bar{\mu} < 0,
\]
then 
\[
\lambda^{\flat}(\bar{\mu}) \leq \int \log Lf_{\omega}(x) \, d\bar{\mu}.
\]
\end{lem}

\begin{proof}
Take $\chi<0$ such that $\int \log Lf_{\omega} \, d\bar{\mu}<\chi$.  
Since  $\log^+ \mathrm{Lip} (f_{\omega})$ is   
$\bar{\mu}$-integrable by~\eqref{integrability} and 
$\log L_rf_\omega(x) \leq \log \mathrm{Lip}(f_\omega)$, the reverse Fatou lemma implies that
$$\limsup_{r\to 0^+} \int \log L_rf_{\omega}(x) \, d\bar{\mu}\leq \int \log Lf_{\omega}(x) \, d\bar{\mu}.$$
Hence, we can find an $r>0$ such that 
$ 
\int \log L_rf_{\omega}(x)\, d\bar{\mu}<\chi$. 
Again, by the $\bar{\mu}$-integrability condition~\eqref{integrability} and 
the Birkhoff ergodic theorem\footnote{The Birkhoff ergodic theorem is usually stated for observables $\phi$ in $L^1$. However, since this theorem is a consequence of Kingman's subadditive ergodic theorem (see Theorem~\ref{Kingman}), it also holds if $\phi^+\in L^1$.}, 
for $\bar{\mu}$-a.e.~$(\omega,x)$, we have 
\begin{equation} \label{eq:birkhoff}
\lim_{n\to \infty}\frac{1}{n}\sum_{k=0}^{n-1} 
\log L_rf_{\omega_k}(f_\omega^k(x))=\int \log L_rf_{\omega}(x)\,  d\bar{\mu}<\chi
\end{equation}
where $\omega_k=T^k(\omega)$. Let us fix such a point $(\omega,x)$. In view of~\eqref{eq:birkhoff}, we find $n_0\in\N$ such that
\begin{equation*} \label{eq:Birkoff2}
\prod_{k=0}^{n-1}L_rf_{\omega_k}(f_\omega^k(x)) \leq e^{\chi n} \quad \text{for all  $n \geq n_0$.}
\end{equation*}
Let $B=B(x,s)$ be a small ball centered at $x$ such that $\diam f^{i}_\omega(B) \leq r$ for $i=0,\dots,n_0$. 
\begin{claim}
$f_\omega^n$ is $e^{\chi n}$-contracting on $B$ for any $n\ge n_0$.
\end{claim}
\begin{proof} 
We proceed by induction. Take $n\geq  n_0$ and let us assume that $f_\omega^k$ is $e^{\chi k}$-contracting on~$B$ for any $k=n_0,\dots, n-1$ if $n>n_0$. 
This implies that $f_\omega^k(B)$ has diameter less than $r$ for every $k=0,\dots,n-1$ since $\chi<0$ and $\diam f^{i}_\omega(B) \leq r$ for $i=0,\dots,n_0$. Thus, 
the Lipschitz constant of $f_{\omega_k}$ on  $f_\omega^k(B)$ is less than $L_rf_{\omega_k}(f_\omega^k(x))$. So we get
$$\mbox{Lip}(f_\omega^n|_B)\leq  \prod_{k=0}^{n-1}L_rf_{\omega_k}(f_\omega^k(x)) \leq e^{\chi n},$$ 
which completes the induction.
\end{proof}

As a consequence of the above claim, for $\mu$-a.e.~$(\omega,x)$, it holds that $L_s f_\omega^n(x)\leq e^{\chi n}$ for every $n\ge n_0$. Then $\lambda^{\flat}(\omega,x)\leq \chi$  for $\bar{\mu}$-a.e.~$(\omega,x)$. By the ergodicity, this implies that $\lambda^+(\bar{\mu})\leq \chi$.   Since $\chi>\int \log Lf_{\omega}(x)\,  d\bar{\mu}$ is arbitrary, the conclusion follows.
\end{proof}

For any integer $m$, let us denote by $\lambda_m^{\flat}$ the function $\lambda^{\flat}$ associated with the iterated skew product $F^m$, i.e.,
$$\lambda_m^{\flat}(\omega,x)=\lim_{r\to 0^+} \limsup_{k\to \infty} \frac{1}{k}\log L_rf_\omega^{km}(x).$$
\begin{lem}\label{contractionlem2}
If $\bar{\mu}$ is an $F$-invariant measure, then  $\lambda^{\flat}_m=m\lambda^{\flat}$ holds $\bar{\mu}$-almost everywhere. 
\end{lem}
\begin{proof}
For any $(\omega,x)\in \Omega\times X$, we have
\begin{align*}
    \lambda^{\flat}_m(\omega,x) &=m \cdot  \lim_{r\to 0^+} \limsup_{k\to \infty} \frac{1}{km}\log L_rf_\omega^{km}(x) \\ &\leq   m\cdot \lim_{r\to 0^+} \limsup_{n\to \infty} \frac{1}{n}\log L_rf_\omega^n(x)=m \cdot \lambda^{\flat}(\omega,x).
\end{align*}
Thus we get $\lambda^{\flat}_m\leq m\lambda^{\flat}$. 
On the other hand, let us fix an $\epsilon>0$ arbitrarily. For $\bar{\mu}$-almost every~$(\omega,x)$, denoting $\omega_i=T^i(\omega)$, we have by a corollary of the Birkhoff ergodic theorem that $\frac{1}{i}\log^+ \mathrm{Lip}(f_{\omega_{i}})$ tends to $0$ as $i\to \infty$.  Thus, for $i$ large enough, we have $$\mathrm{Lip}(f_{\omega_{i}})\leq e^{\frac{i\epsilon}{m}}.$$ Then, for $n$ large, writing  $n=km+l$ with $0\leq l<m$, we have
$$L_rf_\omega^{n}(x)\leq L_rf_\omega^{km}(x)\left(\prod_{i=km}^{n-1}\mbox{Lip}(f_{\omega_{i}})\right)\leq L_rf_\omega^{km}(x)\prod_{i=km}^{n-1}e^{\frac{n\epsilon}{m}} \leq L_rf_\omega^{km}(x)e^{n\epsilon}.$$
From here it follows that
$$\frac{1}{n}\log L_rf_\omega^{n}(x)\leq \frac{1}{km}\log L_rf_\omega^{km}(x)+\epsilon$$
Letting $n\to \infty$ (and so $k$) and then $r\to 0^+$, we obtain
$\lambda^{\flat}(\omega,x)\leq \frac{1}{m}\lambda^{\flat}_m(\omega,x)+\epsilon$.
Since $\epsilon$ is arbitrary, we get the converse inequality $\lambda_m^{\flat}\geq m\lambda^{\flat}$ and thus conclude the proof.
\end{proof}

\subsection{Proof of Theorem~\ref{contraction}} 
As we have seen in~\S\ref{ss:equivalent}, it suffices to show Proposition~\ref{contraction'}.  Moreover, in light of~\eqref{eq:inequailty0}, to demonstrate Proposition~\ref{contraction'},
it is sufficient to prove that if $\bar{\mu}$ is an ergodic $F$-invariant measure satisfying~\eqref{integrability}, then the following implication holds: 
\begin{equation}\label{objectif}\lambda(\bar{\mu})<0 \ \ \Longrightarrow \ \ \lambda^{\flat}(\bar{\mu})\leq \lambda(\bar{\mu}).\end{equation}
Indeed, from the ergodic decomposition we can write $d\bar{\mu}=\bar{\mu}_\alpha \, d\alpha$ where $\bar{\mu}_\alpha$ are  ergodic $F$-invariant  measures and hence,~\eqref{objectif} implies 
$$
 \int_{\{\lambda<0\}} \lambda^{\flat}(\omega,x)\,d\bar{\mu} =\int_{\{\alpha: \, \lambda<0\}} \lambda^{\flat}(\bar{\mu}_\alpha)\,d\alpha
\leq \int_{\{\alpha: \,\lambda<0\}} \lambda(\bar{\mu}_\alpha) \, d\alpha = \int_{\{\lambda<0\}} \lambda(\omega,x)\, d\bar{\mu}. 
$$
Consequently, by ~\eqref{eq:inequailty0}, $\lambda^{\flat}(\omega,x)=\lambda(\omega,x)$ for $\bar{\mu}$-almost everywhere on $\{\lambda<0\}$.   

We will now proceed to show~\eqref{objectif}.
Let $\bar{\mu}$ be an ergodic $F$-invariant measure satisfying~\eqref{integrability}. Lemma~\ref{contractionlem1} can be expressed as the following inequality:
$$
\min\{\lambda^{\flat}(\bar{\mu}),0\} \leq \int \log Lf_{\omega}(x) \, d\bar{\mu}.
$$
 For any integer $m$, by applying this inequality on  $F^m$ instead of $F$, we get
$$
\min\{\lambda_m^{\flat}(\bar{\mu}),0\} 
\leq \int \log Lf_{\omega}^m(x) \, d\bar{\mu}.$$
Using Lemma~\ref{contractionlem2}, we obtain
$$
\min\{\lambda^{\flat}(\bar{\mu}),0\} 
\leq \frac{1}{m}\int \log Lf_{\omega}^m(x) \, d\bar{\mu}.$$
Letting $m$ tend to $\infty$ and using the reverse Fatou lemma (or Kingman's subadditive ergodic theorem), we conclude that  
$$\min\{\lambda^{\flat}(\bar{\mu}),0\} \leq \int \lambda(\omega,x) \, d\bar{\mu} = \lambda(\bar{\mu}).$$
This implies~\eqref{objectif} and concludes the proof.

\section{Proof of Theorem~\ref{cor:local-contraction}} 
In this subsection, we apply Theorem~\ref{contraction} to the setting of a skew-product $F$ arising from a random map $f$. 
{Namely, we assume that $f:\Omega\times X \to X$ is a random map where $(\Omega,\mathscr{F},\mathbb{P})$ is a Bernoulli product probability space and $(X,d)$ is a separable complete metric space.}  

{
We say that $f$ is \emph{tight} if for every $x\in X$ the sequence $\{e_n(\omega,x)\}_{n\geq 1}$ of empirical measure 
\[
e_n(\omega,x)\eqdef \frac1n\sum_{k=0}^{n-1}\delta_{f_\omega^k(x)}
\]
is relatively compact on the space of probability measures endowed with the weak$^*$ topology for $\mathbb{P}$-a.e.~$\omega\in \Omega$. If $X$ is compact, then $f$ is tight. Thus, the following result proves Theorem~\ref{cor:local-contraction}.

\begin{thm} \label{thm:generalization-thmA}
    Let $f$ be a tight continuous random map of a separable complete metric space $X$. If $f$ is mostly contracting, then $f$ has exponential local contraction. Namely, for every $\chi \in (\lambda(\mu),0)$ and $x\in X$,  for $\mathbb{P}$-a.e.~$\omega\in \Omega$ there exists a neighborhood $B$ of $x$ and a constant $C$ such that 
    $$ \mathrm{diam} f^n_\omega(B) \leq C e^{n\chi} \quad \text{for all $n\geq 1$}
    $$
    where $\lambda(f) =
   \sup\{\lambda(\mu): \mu \text{ is $f$-stationary} \}<0$. 
\end{thm}}
{
\begin{proof} Since $f$ is mostly contracting $\lambda(f) <0$. 
Choose $\lambda(f)<\chi<0$  and set \(q=e^\chi<1\). Since $X$ is a separable metric space, we can take a countable basis $\mathcal{U}$ of open sets. Define
\(E_q\subset\Omega\times X\) by declaring that \((\omega,x)\in E_q\) if there
exist \(U\in\mathcal U\) and \(C>0\) such that
\[
   x\in U
   \quad\text{and}\quad
   \operatorname{diam} f_\omega^n(U)\le Cq^n
   \quad\text{for every } n\ge0 .
\]
Equivalently, $E_q$ can be expressed as a countable union of  rectangles
\[
   E_q
   =
   \bigcup_{U \in \mathcal{U}} \bigcup_{C \in \mathbb{N}} 
   \left( \Omega_{U,C} \times U \right),
\]
where 
\[
   \Omega_{U,C} = \bigcap_{n \ge 0} \left\{ \omega \in \Omega : \operatorname{diam} f_\omega^n(U) \le Cq^n \right\}.
\]

\begin{claim}
    $E_q$ is a measurable set in $\Omega \times X$. 
\end{claim}
\begin{proof}
    Since $X$ is separable, it contains a countable dense subset $D$. For any open set $U$, the intersection $D \cap U$ is countable and dense in $U$. By the continuity of the fiber maps $f_\omega^n$, we can compute the diameter using only this countable subset
\[
    \operatorname{diam} f_\omega^n(U) = \sup_{y, z \in D \cap U} d(f_\omega^n(y), f_\omega^n(z)).
\]
Because the supremum is taken over a countable set and $\omega \mapsto d(f_\omega^n(y), f_\omega^n(z))$ is measurable, the map $\omega \mapsto \operatorname{diam} f_\omega^n(U)$ is measurable. Consequently, the sets $\Omega_{U,C}$ are measurable, ensuring that $E_q \in \mathscr{F} \otimes \mathscr{B}$ where $\mathscr{B}$ is the Borel $\sigma$-algebra of $X$.
\end{proof}

\begin{claim}  For each $\omega \in \Omega$, $E_q(\omega) = \{x \in X : (\omega,x) \in E_q\}$ is an open set in $X$. 
\end{claim}
\begin{proof} Observe from the product representation of $E_q$ that the fiber for a fixed $\omega$ can be written exactly as
\[
   E_q(\omega) = \bigcup \big\{ U \in \mathcal{U} : \text{there exists } C \in \mathbb{N} \text{ such that } \omega \in \Omega_{U,C} \big\}.
\]
Since $E_q(\omega)$ is a union of open basis elements, it is necessarily open.
\end{proof}
\begin{claim} The set \(E_q\) is backward invariant, i.e., $  F^{-1}(E_q)\subset E_q$.
\end{claim}
\begin{proof} Suppose that \(F(\omega,x)=(\sigma\omega,f_\omega(x))\in E_q\).
Then there are \(U\in\mathcal U\) with \(f_\omega(x)\in U\) and \(C>0\) such that
\[
   \operatorname{diam} f_{\sigma\omega}^n(U)\le Cq^n
   \quad\text{for all }n\ge0 .
\]
By continuity of \(f_\omega\), there is \(V\in\mathcal U\) with \(x\in V\) and
\(f_\omega(V)\subset U\). Hence, for every \(n\ge1\),
\[
   \operatorname{diam} f_\omega^n(V)
   =
   \operatorname{diam}
   f_{\sigma\omega}^{\,n-1}(f_\omega(V))
   \le
   \operatorname{diam}
   f_{\sigma\omega}^{\,n-1}(U)
   \le
   Cq^{n-1}.
\]
Increasing the constant if necessary to take care of \(n=0\), we get
\[
   \operatorname{diam} f_\omega^n(V)\le C' q^n
   \quad\text{for every }n\ge0.
\]
Thus \((\omega,x)\in E_q\), proving \(F^{-1}(E_q)\subset E_q\).
\end{proof}

Now we want to apply the zero-one law for random open sets, Corollary~\ref{cor:malicet-zero-one-open-sets}, to $E_q$. Since $f$ is a continuous random map, the assumption of Feller continuity of such a result holds. Moreover,  since $f$ is tight, we have that for every $x\in X$ the set of accumulation points $\Pi(\omega,x)$ of the empirical measure $e_n(\omega,x)$ is non-empty for $\mathbb{P}$-a.e.~$\omega$. Thus, in view of the three previous claims, to apply Corollary~\ref{cor:malicet-zero-one-open-sets}, it remains to prove that $(\mathbb{P}\times \mu)(E_q)>0$ for every ergodic $f$-stationary measure $\mu$.

Let \(\mu\) be an ergodic $f$-stationary measure.  Then \(\mathbb P\times\mu\) is an ergodic
\(F\)-invariant probability measure. Since \(f\) is mostly contracting, 
$$\lambda(\mu)\le \lambda(f)<\chi<0.$$ 
Thus, for \((\mathbb P\times\mu)\)-a.e.~\((\omega,x)\), the pointwise maximal
Lyapunov exponent satisfies
$\lambda(\omega,x)=\lambda(\mu)<\chi<0$.
By Theorem~\ref{contraction}, applied with this fixed number \(\chi\), for
\((\mathbb P\times\mu)\)-a.e.~\((\omega,x)\) there exist a neighborhood \(B\)
of \(x\) and a constant \(C>0\) such that
\[
   d(f_\omega^n(y),f_\omega^n(z))
   \le
   C e^{n\chi} d(y,z)
   =
   Cq^n d(y,z)
\]
for every \(y,z\in B\) and every \(n\ge0\). Taking a basis element
\(U\in\mathcal U\) with \(x\in U\subset B\), we obtain
\[
   \operatorname{diam} f_\omega^n(U)
   \le
   C\,\operatorname{diam}(U)\,q^n
   \quad\text{for every }n\ge0.
\]
Hence \((\omega,x)\in E_q\). Thus  $(\mathbb P\times\mu)(E_q)=1$.

Since \(E_q\) is measurable, has open fibers, and satisfies
\(F^{-1}(E_q)\subset E_q\) and $(\mathbb P\times\mu)(E_q)>0$ for every ergodic $f$-stationary measure $\mu$, the zero-one law for random open sets gives
\[
   (\mathbb P\times\nu)(E_q)=1
   \qquad\text{for every probability measure }\nu\text{ on }X.
\]
Hence, for every \(x\in X\), taking \(\nu=\delta_x\), we conclude that $\mathbb P(\Omega_x)=1$ where $\Omega_x=\{\omega:(\omega,x)\in E_q\}$. Therefore, for every $\omega\in \Omega_x$ we have $(\omega,x)\in E_q$ and by the definition of \(E_q\), there exist a neighborhood \(B\) of \(x\) and
a constant \(C>0\) such that $\operatorname{diam} f_\omega^n(B)\le Cq^n$ for every $n\ge0$.
This proves the exponential local contraction property.
\end{proof}
}

\section{Palis' global conjecture} \label{ss:palis-conjeture} \index{Palis' global conjecture}
Palis proposed~\cite{palis2000global} that for most dynamical systems, there exists a finite number of physical measures that capture the statistics of almost every orbit.
Physicality means that the basin of attraction has positive measure (with respect to the reference measure). In this section, we prove that the skew-product associated with a mostly contracting random map satisfies this statistical description of the dynamics. Moreover, we prove this conjecture for the class of dynamics that have the following local contraction property. 

\index{local contraction!1@local contraction property}
\begin{defi} A random map $f:\Omega\times X \to X$ has the \emph{local contraction property} if
for every $x\in X$, for $\mathbb{P}$-a.e.~$\omega\in \Omega$ there is a neighborhood $B$ of $x$ such that 
$$
\diam f^n_\omega(B) \to 0 \quad \text{as $n\to\infty$.}
$$
\end{defi}
As mentioned, by Theorem~\ref{cor:local-contraction}, every mostly contracting random map has the (exponential) local contraction property. 

\begin{thm} \label{thm:palis-conjecture}
Let $(X ,d,m)$ and $(\Omega,\mathscr{F},\mathbb{P})$ be a compact metric Borel probability space and a Bernoulli product probability space of a Polish probability space, respectively. Let $f:\Omega\times X\to X$ be a random map with the local contraction property and consider the associated skew-product $F$. 
Then, there exist finitely many ergodic $F$-invariant probability measures $\bar{\mu}_1,\dots,\bar{\mu}_s$ on $\Omega\times X$, such that
\begin{enumerate}[leftmargin=1.5cm,label=(\roman*)]
  \item $\bar{\mu}_i=\mathbb{P}\times \mu_i$ where $\mu_i$ is an $f$-stationary measure, 
  \item $\bar{m}(B(\bar{\mu}_i))>0$, i.e., $\bar{\mu}_i$ is a physical measure with respect to $\bar{m}=\mathbb{P}\times m$,  
  \item $\bar{m}\left(B(\bar{\mu}_1)\cup \dots \cup B(\bar{\mu}_s)\right)=1$,
\end{enumerate} 
wherein
$$B(\bar{\mu}_i)=\left\{(\omega,x)\in \Omega\times X:  \lim_{n\to\infty}\frac{1}{n}\sum_{k=0}^{n-1}\delta_{F^k(\omega,x)}=\bar{\mu}_i \ \text{in the weak$^*$ topology }  \right\}.$$
\end{thm}
\begin{proof}
    The proof follows the arguments of~\cite[Prop.~4.9]{Mal:17}. 
    
    According to~\cite[Prop.~4.8]{Mal:17} (see Corollary~\ref{rem:measure-mean-quasicompactness}(b)), a random map $f$ satisfying the local contraction property has a finite number of ergodic $f$-stationary probability measures. Denote these measures by $\mu_1,\dots, \mu_r$.
   Note that the ergodicity and $f$-invariance of $\mu_i$ implies that $\bar{\mu}_i\eqdef \mathbb{P}\times \mu_i$ is an ergodic $F$-invariant probability measure for all $i=1,\dots,r$. 
    Consider $\mathcal{E}_0$ to be the set of points $(\omega, x)$ such that there exists a neighborhood of $x$ contracted by $f^n_\omega$ for $n\geq 1$. Note that since $f$ satisfies the local contraction property, for
    every $x\in X$, we have a set $\Omega_x \subset \Omega$ with $\mathbb{P}(\Omega_x)=1$ such that 
    $\Omega_x \times \{x\}\subset \mathcal{E}_0$. 
    Hence, $(\mathbb{P}\times\mu)(\mathcal{E}_0)=1$ 
    for any   probability measure $\mu$ on $X$. 
    
    Let
$$
\mathcal{E}_i = \left\{(\omega, x) \in \mathcal{E}_0: 
\, \lim_{n\to\infty}\frac{1}{n} \sum_{k=0}^{n-1}
\delta_{F^k(\omega,x)} = \bar{\mu}_i \right\}  \quad \text{for $i=1,\dots,r$}.
$$
Here, the limits of measures are taken in the weak$^*$ topology. That is, $\nu_n\to \nu$ if and only if $\int \varphi \, d\nu_n \to \int \varphi \, d\nu$ for any {bounded} real-valued continuous bounded function $\varphi$.  Clearly, $F^{-1}(\mathcal{E}_i) \subset \mathcal{E}_i$ and 
$\mathcal{E}_i=\mathcal{E}_0\cap B(\bar{\mu}_i)$. Since, $\bar\mu_i$ is ergodic, by Birkhoff theorem\footnote{Here we are using that $\Omega\times X$ is a Polish space, and thus the convergence in the weak$^*$ topology is countably determined. See~\cite[Prop.~5.1 and Lemma~5.2]{BNNT22}.}, $\bar{\mu}_i(B(\bar{\mu}_i))=1$. Moreover, since $\bar{\mu}_i(\mathcal{E}_0)=1$, we also get that $\bar{\mu}_i(\mathcal{E}_i)=1$ for all $i=1,\dots,r$.

We write 
$$
\mathcal{E}_i = \bigcup_{\omega\in\Omega} \{\omega\} \times U_i(\omega).
$$ 
Moreover, if $\omega$ belongs to $\Omega$ and $x\in U_i(\omega)$, by the local contraction property, we have a neighborhood $B$ of $x$ such that $\diam(f^n_\omega(B))\to 0$ as $n\to\infty$.  Then, for any $y\in B$, we have $d(F^n(\omega,x),F^n(\omega,y))\to 0$ as $n\to\infty$, which implies that $(\omega,y)\in \mathcal{E}_i$. Therefore, $B \subset U_i(\omega)$, and consequently, $U_i(\omega)$ is open for all $i=1,\dots,r$.

Hence, the set $\mathcal{E} = \mathcal{E}_1\cup \dots \cup \mathcal{E}_r$ is forward $F$-invariant ($F^{-1}(\mathcal{E})=\mathcal{E}$) measurable set, has open fibers $U(\omega)=U_1(\omega)\cup \dots \cup U_r(\omega)$ for all $\omega\in \Omega$, and $(\mathbb{P}\times \mu)(\mathcal{E})>0$ for every ergodic $f$-stationary measure. Thus, we can apply the zero-one law for open random maps, {Corollary~\ref{cor:malicet-zero-one-open-sets},} to obtain that
$$
\bar{m}\left(B(\bar{\mu}_1)\cup \dots \cup B(\bar{\mu}_r)\right) \geq (\mathbb{P} \times m)(\mathcal{E}) = 1.
$$
Since the basins of attraction are pairwise disjoint, this implies (iii) and also (ii) for a number $0<s\leq r$ of measures (omitting the measures whose basin has null reference measure).
\end{proof}

\appendix

\chapter{Limit and integration lemmas}
\label{s:preliminar}

In what follows, $(X,d)$ is a metric space equipped with the Borel $\sigma$-algebra $\mathscr{B}$ and with a
probability measure $\mu$ on $\mathscr{B}$.

\section{Semicontinuity and weak convergence}
\label{app:semicontinuity-weak-convergence}

An extended real-valued function $g$ on \(X\) is said to be \emph{upper semicontinuous} if
$$
\limsup_{x \rightarrow x_0} g(x)\leq g(x_0) \quad \text{for every $x_0\in X$},
$$
or equivalently, if the set \(\{x: g(x) < a\}\) is open for every real number \(a\).

\begin{lem}[Baire's Theorem~{\cite[Prop.~1.4.17]{HL:12}}] \label{lem:Baire-teorema} An extended real-valued function \(g\) on \(X\) is upper semicontinuous if and only if there exists a decreasing sequence of continuous functions \(g_k\) of $X$ such that \(g_k(x) \to g(x)\) for all \(x \in X\).
\end{lem}


\begin{lem}[{\cite[Theorem 1.4.18]{HL:12}}]\label{weakineq}
Let $(\mu_n)$ be a sequence of probability measures on $X$ weakly$^*$ converging to  $\mu$. Then for any $\phi:X\to\R$ which is upper semicontinuous,
$$\limsup_{n\to \infty}  \int \phi \, d\mu_n \leq \int \phi \, d\mu.$$
\end{lem}

\section{Integration and subadditive limit lemmas}
\label{app:integration-subadditive}

\begin{lem}[Reverse Fatou's lemma~{\cite[Prop.~1.5.3(a)]{HL:12}}] \label{Fatou} 
Let $g$, $(f_n)_{n\geq 1}$ be extended real-valued measurable functions such that $\int g \, d\mu < \infty$ and  $f_n\leq g$  for all $n$ large enough. Then
\[\limsup_{n\to\infty} \int f_n \, d\mu \leq \int \limsup_{n\to\infty} f_n \, d\mu.\]
\end{lem}


\begin{lem}[Dominated convergence~{\cite[Prop.~1.5.3(c)]{HL:12}}] Let $g$, $(f_n)_{n\geq1}$  be extended real-valued measurable functions such
that if $|f_n| \leq g$ for all $n$, and $g$ is $\mu$-integrable. If $f_n \to f$ pointwise, then
\[
\lim_{n \to \infty} \int f_n \, d\mu = \int f \, d\mu.
\]
\end{lem}


\begin{lem}[Fekete's lemma~{\cite[Lem.~1.2.1]{steele1997probability}}] \label{lem:Fekete}
 If a sequence $(a_n)_{n\geq 1} \subset [-\infty,\infty)$  satisfies the subadditive condition $a_{n+m} \leq a_n+a_m$, then
$$\lim_{n\to\infty} \frac{1}{n} a_n = \inf_{n\geq 1} \frac{1}{n} a_n \in [-\infty,\infty). $$
\end{lem}

\section{Separability and Stability lemmas}
\label{app:holder-separation}

Although the following lemma is a consequence of McShane's theorem~\cite{mcshane1934extension}, we provide here an easy direct proof: 

\begin{lem}[Lipschitz Urysohn's lemma] \label{lem:Urysohn}
Let $A,B$ be subsets of $X$ such that $$d(A,B)\eqdef \inf\{d(a,b): a\in A, b\in B\} > 0.$$ Then for each $0 < \alpha \leq 1$, there is a function $h \in C^\alpha(X)$ such that 
$0 \leq h \leq 1$,  $h|_A = 0$ and $h|_B = 1$. 
Moreover, if $U$ is a proper open set of $X$, there exists a sequence $(h_n)_{n\geq 1}$ of functions in $C^\alpha(X)$ that converge pointwise everywhere and monotonically increase to $1_U$.   
\end{lem}
\begin{proof} Consider the case $\alpha=1$. The case $0<\alpha<1$ follows by applying the same argument to the metric space $(X,d_\alpha)$ where $d_\alpha(x,y)=d(x,y)^\alpha$. 

Since $d(x,A)+d(x,B)\geq d(A,B) >0$, we can define the function 
$$h(x)=\frac{d(x,A)}{d(x,A)+d(x,B)}, \qquad x\in X.$$
Clearly $0\leq h\leq 1$, $h(x)=0$ for $x\in A$ and $h(x)=1$ for $x\in B$.
Moreover, since the functions $$x\mapsto d(x,Y)
\quad \text{and} \quad (x,y)  \mapsto \frac{x}{x+y}$$ 
where $Y$ is a subset of $X$ and $(x,y)$ runs over 
$\{(x,y): x,y\geq 0 \ \text{and} \ x+y\geq d(A,B)\}$ 
are Lipschitz,  $h$ is also Lipschitz, which concludes the proof of the first part. The second part of lemma follows by considering $A=X\setminus U$ and $B_n=\{x\in X : 
 d(x,A)\geq 1/n\}$, noting that $U=\cup B_n$ and taking $h_n$ the functions above associated with $A$ and $B_n$. 
\end{proof}

The following lemma establishes the Lipschitz stability of the vertices of a simplex with respect to the Hausdorff metric. While this result belongs to the folklore of convex analysis, its formulation in the generality of arbitrary normed spaces is not easily found in standard references (cf.~\cite[Lemma~1(b)]{Nguyen15}). For the sake of completeness, we provide a detailed proof below.

\begin{lem}[Stability of the vertices of a simplex]
\label{lem:vertex-stability-simplex}
Let \(E\) be a normed vector space and let
$S_0=\operatorname{conv}\{v_1,\dots,v_r\}\subset E$  be a simplex. Then there exist \(\varepsilon_0>0\) and \(C_0>0\) such that
any simplex \(S\subset E\) with exactly \(r\) vertices satisfying \(d_H(S,S_0)<\varepsilon_0\) admits a unique labelling of its vertices \(w_1(S),\dots,w_r(S)\) where \(w_i(S) \to v_i\) as \(d_H(S, S_0) \to 0\). Moreover, for any two such simplices \(S,S'\), 
\[
   \max_{1\le i\le r}\|w_i(S)-w_i(S')\| \le C_0\,d_H(S,S').
\]
\end{lem}

\begin{proof}
The case \(r=1\) is immediate. Let \(A=\operatorname{aff}(S_0)\) be the affine hull of \(S_0\), i.e., the smallest affine subspace containing the simplex.  Let
\(\lambda_i:A\to\mathbb R\), \(i=1,\dots,r\), be the barycentric coordinates of
\(S_0\), i.e., the unique set of coefficients  such that $x = \sum_i \lambda_i(x) v_i$ with $\sum_i \lambda_i = 1$. Thus \(\lambda_i(v_j)=\delta_{ij}\) and
\(\sum_i\lambda_i\equiv1\) on \(A\). Extend \(\lambda_1,\dots,\lambda_{r-1}\)
to continuous affine functionals on \(E\) by Hahn-Banach, and set
\[
   \lambda_r=1-\sum_{i=1}^{r-1}\lambda_i .
\]
Then \(\sum_i\lambda_i\equiv1\) on \(E\). Let
$L=\max_i \operatorname{Lip}(\lambda_i)$. 

Let \(S\) be a simplex with \(r\) vertices and put
\(\delta=d_H(S,S_0)\). For each \(i\), choose \(y_i\in S\) with
\(\|y_i-v_i\|\le\delta\). Since \(\lambda_i(y_i)\ge1-L\delta\), some vertex
\(u_i\) of \(S\) satisfies
$\lambda_i(u_i)\ge1-L\delta$. On the other hand, every vertex \(u\) of \(S\) is \(\delta\)-close to \(S_0\). Hence $\lambda_j(u)\ge -L\delta$ for any $j$. If the same vertex \(u\) satisfied
\(\lambda_i(u)\ge1-L\delta\) and \(\lambda_k(u)\ge1-L\delta\) for \(i\neq k\),
then
\[
   1=\sum_{j=1}^r\lambda_j(u)
   \ge 2(1-L\delta)-(r-2)L\delta
   =2-rL\delta,
\]
which is impossible for \(\delta<1/(rL)\). Therefore the vertices
\(u_1,\dots,u_r\) are distinct, and hence they are all the vertices of \(S\).

We now show that \(u_i\) is close to \(v_i\). Choose \(x_i\in S_0\) with
\(\|u_i-x_i\|\le\delta\). Then
$\lambda_i(x_i)\ge1-2L\delta$.
Writing \(x_i\) in barycentric coordinates in \(S_0\), and denoting
\(D=\operatorname{diam}(S_0)\),~we~get
\[
   \|x_i-v_i\|
   \le
   D\sum_{j\neq i}\lambda_j(x_i)
   =
   D(1-\lambda_i(x_i))
   \le 2DL\delta.
\]
Thus $\|u_i-v_i\|\le(1+2DL)\delta$. After decreasing $\varepsilon_0 \geq \delta$, the balls around the vertices \(v_i\) are
pairwise disjoint, and this gives the unique labelling \(w_i(S)=u_i\).

It remains to prove the Lipschitz estimate. Decreasing \(\varepsilon_0\) again,
we may assume that there is \(\gamma>0\) such that, for every such simplex
\(T\),
\[
   \lambda_i(w_i(T))-\lambda_i(w_j(T))\ge\gamma
   \qquad\text{whenever }j\neq i.
\]
Let \(S,S'\) be two simplexes in this neighbourhood, set
\[
   \eta=d_H(S,S'),
   \qquad
   u_i=w_i(S),
   \qquad
   u_i'=w_i(S'),
\]
and choose \(y\in S'\) with \(\|u_i-y\|\le\eta\). Write
$y=a_i u_i' +(1-a_i)z$, 
where \(z\) belongs to the face of \(S'\) opposite \(u_i'\). By the preceding
separation,
$\lambda_i(u_i')-\lambda_i(z)\ge\gamma$.
Hence
$\gamma(1-a_i)
   \le
   \lambda_i(u_i')-\lambda_i(y)$.
Since \(u_i\) and \(u_i'\) are the maximizers of \(\lambda_i\) on \(S\) and
\(S'\), respectively, the map \(T\mapsto\max_T\lambda_i\) is
\(L\)-Lipschitz for the Hausdorff metric. Therefore
\[
   |\lambda_i(u_i')-\lambda_i(u_i)|\le L\eta,
   \qquad
   |\lambda_i(u_i)-\lambda_i(y)|\le L\eta,
\]
and so \enlargethispage{0.25cm}
$1-a_i\le \eta (2L/\gamma)$. Finally, all simplexes in the \(\varepsilon_0\)-neighbourhood of \(S_0\) have
diameter bounded by \(D_*:=D+2\varepsilon_0\). Thus
\[
   \|u_i-u_i'\|
   \le
   \|u_i-y\|+\|y-u_i'\|
   \le
   \eta+(1-a_i)D_*
   \le
   \left(1+\frac{2LD_*}{\gamma}\right)\eta. 
\]
Taking the maximum over \(i\) proves the result. \end{proof}

\chapter{Martingale laws of random variables}
\label{s:appendix-martingale-zero-one}

Throughout this
appendix, \((\Omega,\mathscr F,\mathbb P)\) is a probability space.  If
\(\mathscr G\subset \mathscr F\) is a sub-\(\sigma\)-algebra and \(\psi\in
L^1(\mathbb P)\), we denote by
\(\mathbb E[\psi\mid\mathscr G]\) the conditional expectation of \(\psi\) with
respect to \(\mathscr G\).  If \(A\in\mathscr F\), we write
\(\mathbb P(A\mid\mathscr G)=\mathbb E[1_A\mid\mathscr G]\).

\section{L\'evy's zero-one law} \index{zero-one laws!L\'evy's zero-one law}

\begin{thm}[{\cite[Cor.~C.9]{Ber03}}]
\label{cor:levy-zero-one}
Let \((\mathscr F_n)_{n\ge 1}\) be an increasing filtration and set
$
   \mathscr F_\infty \eqdef \sigma(\bigcup_{n\ge 1}\mathscr F_n)$. 
If \(\psi\in L^1(\mathbb P)\), then
\[
   \mathbb E[\psi\mid \mathscr F_n]
   \longrightarrow
   \mathbb E[\psi\mid \mathscr F_\infty] \quad\text{almost surely and in \(L^1\).}
\]
  In particular, if \(A\in\mathscr F_\infty\), then $
   \mathbb P(A\mid \mathscr F_n)\to 1_A$ 
 almost surely,
and, if \(\psi\) is \(\mathscr F_\infty\)-measurable, then
\[
   \mathbb E(\psi\mid\mathscr F_n)\longrightarrow \psi
   \qquad\text{almost surely and in }L^1.
\]
\end{thm}

\section{Kolmogorov's law of large numbers}
\index{Kolmogorov's law of large numbers}

\begin{thm}[{\cite[{Thm.~A.6}]{benoist2016random}}]
\label{thm:BQ-A6}
Let \((\xi_n)_{n\ge 1}\) be a sequence of random variables and let
\((\mathscr F_n)_{n\ge 0}\) be an increasing filtration such that \(\xi_n\) is
\(\mathscr F_n\)-measurable.  If there is a non-negative integrable
random variable \(\xi\) such that, for every \(t\ge 0\) and  \(n\ge 1\),
\[
   \mathbb P\big(|\xi_n|>t\mid \mathscr F_{n-1}\big)
   \le
   \mathbb P(\xi>t)
   \qquad\text{almost surely},
\]
then
\[
   \frac1n\sum_{k=1}^n
   \Big(\xi_k-\mathbb E[\xi_k\mid\mathscr F_{k-1}]\Big)
   \longrightarrow 0 \quad \text{almost surely and in \(L^1\).}
\]
\end{thm} 

Corollary~\ref{cor:levy-zero-one} is
 a direct consequence of \emph{Doob's martingale convergence theorem}~\cite[Thm.~A.3]{benoist2016random}. Also, Benoist and Quint derive Theorem~\ref{thm:BQ-A6} from Doob's theorem by a
truncation and Kronecker-lemma argument.  

\chapter{Law of large numbers and zero-one law}
\label{s:appendix-feller-bernoulli}


Let
\((T,\mathscr A,p)\) be a probability space and let $(\Omega,\mathscr F,\mathbb P)
   =(T^{\mathbb N},\mathscr A^{\mathbb N},p^{\mathbb N})$
be a one-sided Bernoulli product space.  We write
\(\omega=(\omega_0,\omega_1,\ldots)\) and denote by \(\sigma:\Omega\to\Omega\) the left shift.  Unless explicitly stated otherwise, \(X\) is a Polish space with 
Borel \(\sigma\)-algebra $\mathscr{B}$. 

A Bernoulli random map, or simply a \emph{random map}, \(f:\Omega\times X\to X\) is a measurable map such that
\(f(\omega,\cdot)=f_{\omega_0}\) for \(\mathbb P\)-a.e.~\(\omega\).  We use the notation
\[
   f^0_\omega=\operatorname{id},
   \qquad
   f^n_\omega=f_{\omega_{n-1}}\circ\cdots\circ f_{\omega_0},
   \qquad n\ge 1,
\]
and thus the associated skew-product map $F:\Omega\times X \to \Omega \times X$ satisfies 
$F^n(\omega,x)=(\sigma^n(\omega),f^n_\omega(x))$, for $n\geq 0$.   
The random map \(f\) also induces the transition probability 
\[
   P(x,A)
   \eqdef
   \mathbb P(\{\omega\in\Omega: f_\omega(x)\in A\}),
   \qquad x\in X, \ \ A\in \mathscr{B}  
\]
from which the associated annealed Koopman operator is defined as
\[
   P\varphi(x)
   \eqdef
   \int \varphi(y)\,P(x,dy)
      =
   \int \varphi(f_\omega(x))\,d\mathbb P, \qquad x\in X
\]
for Borel measurable bounded functions $\varphi:X\to \mathbb{R}$. 

\section{Random Birkhoff average comparison} We compare the asymptotic behavior of the Birkhoff average along a random orbit for different types of observables. First, we consider a bounded observable $\varphi$ and $P\varphi$. 

\index{random Birkhoff averages}
\index{Breiman's lemma}

\begin{prop}[Breiman's lemma, {\cite{Bre:60}}]
\label{lem:breiman-feller} Let $f$ be a random map and consider a Borel measurable bounded function $\varphi: X \to \mathbb{R}$.  Then, for every \(x\in X\),
\[
   \frac1n\sum_{k=0}^{n-1}
   \Big(\varphi(f^k_\omega(x))-P\varphi(f^k_\omega(x))\Big)
   \longrightarrow 0 \quad \text{for \(\mathbb P\)-a.e.~\(\omega\).}
\]
\end{prop}

\begin{proof}
Fix \(x\in X\) and set $X_k(\omega)\eqdef f^k_\omega(x)$, $k\ge 0$. Thus \(X_0=x\) and $X_k=f_{\omega_{k-1}}(X_{k-1})$ for $k\ge 1$. Let $\mathscr F_k\eqdef \sigma(\omega_0,\ldots,\omega_{k-1})$, $k\ge 1$, and let \(\mathscr F_0\) be the trivial \(\sigma\)-algebra. Then \(X_k\) is
\(\mathscr F_k\)-measurable. Since \(\omega_{k-1}\) is independent of \(\mathscr F_{k-1}\) and has law \(p\), we have
\[
   \mathbb E\big[\varphi(X_k)\mid \mathscr F_{k-1}\big]
   =
   \int \varphi(f_t(X_{k-1}))\,dp(t)
   =
   P\varphi(X_{k-1}).
\]
For \(k\ge 1\), define
$\xi_k
   \eqdef
   \varphi(X_k)-P\varphi(X_{k-1})$.
Then \(\xi_k\) is \(\mathscr F_k\)-measurable,
$\mathbb E[\xi_k\mid \mathscr F_{k-1}]=0$, 
and $|\xi_k|\le 2\|\varphi\|_\infty$. 
Hence $
   \mathbb P(|\xi_k|>t\mid \mathscr F_{k-1})
   \le
   \mathbb P(\xi>t)$ for every \(t\ge0\) 
where \(\xi=2\|\varphi\|_\infty\) is viewed as a constant integrable random variable. Thus the sequence \((\xi_k)_{k\ge1}\) satisfies the hypotheses of
Theorem~\ref{thm:BQ-A6}.  Therefore
\[
   \frac1n\sum_{k=1}^n \xi_k
   =
   \frac1n\sum_{k=1}^n
   \Big(\varphi(X_k)-P\varphi(X_{k-1})\Big)
   \longrightarrow 0
   \quad \text{almost surely.}
\]
It remains only to shift the index. Since
\[
\begin{aligned}
&\frac1n\sum_{k=1}^n
   \Big(\varphi(X_k)-P\varphi(X_{k-1})\Big)
 -
 \frac1n\sum_{k=0}^{n-1}
   \Big(\varphi(X_k)-P\varphi(X_k)\Big)  =
   \frac{\varphi(X_n)-\varphi(X_0)}{n},
\end{aligned}
\]
the absolute value of this difference is at most
\(2\|\varphi\|_\infty/n\), which tends to zero. Therefore
\[
   \frac1n\sum_{k=0}^{n-1}
   \Big(\varphi(X_k)-P\varphi(X_k)\Big)
   \longrightarrow 0
   \quad \text{almost surely.}
\]
Since \(X_k(\omega)=f^k_\omega(x)\), this is the desired conclusion.
\end{proof}

We now compare the quenched and annealed asymptotic behaviors of the Birkhoff average along random orbits. First, denote by $C_b(X)$ \index{space of functions!$Cb@\(C_b(X)\), bounded continuous functions} the space of bounded real-valued continuous functions of $X$.

\begin{prop}[{\cite[Prop.~3.2]{benoist2016random}}]
\label{prop:one-step-observables} Let $f$ be a random map and consider a measurable one-step map  \(\phi:\Omega\times X\to\mathbb R\), i.e., such that $\phi(\omega,x)=\phi_{\omega_0}(x)$ only depends on the zero-coordinate of~$\omega$.  Assume that
\[
\text{\(\phi(\omega,\cdot)\in C_b(X)\) \  \ for  \(\mathbb P\)-a.e.~\(\omega\)} \quad \text{and}
\quad 
    \int \| \phi(\omega,\cdot)\|_{\infty} \, d\mathbb{P}<\infty. 
\]
Define $\bar \phi(x)\eqdef \int \phi(\omega,x)\,d\mathbb{P}$. Then \(\bar \phi\in C_b(X)\) and, 
for every \(x\in X\) and \(\mathbb P\)-a.e.~\(\omega\)
\[
   \frac1n\sum_{k=0}^{n-1}
   \big(
      \phi_{\omega_k}(f^k_\omega(x))-\bar\phi(f^k_\omega(x))
   \big)= \frac1n\sum_{k=0}^{n-1}
   \big(\phi(F^k(\omega,x))-\bar\phi(f^k_\omega(x))\big)
   \longrightarrow 0.
\]
\end{prop}

\begin{proof}
The function \(\bar\phi\) is well defined and bounded because \enlargethispage{0.5cm}
\[
   |\bar\phi(x)|
   \le
   \int_T |\phi_t(x)|\,dp(t)
   \le
   \int_T \|\phi_t\|_\infty\,dp(t)<\infty
   \qquad \text{for every }x\in X .
\]
Moreover, \(\bar\phi\in C_b(X)\) by dominated convergence.  Fix \(x\in X\)
and let \(X_k(\omega)=f^k_\omega(x)\) for $k\geq 0$.  With \(\mathscr F_{k+1}=\sigma(\omega_0,\ldots,
\omega_{k})\), the random variable \(X_{k+1}\) is \(\mathscr F_{k+1}\)-measurable.

Define  $\xi_{k+1}(\omega)\eqdef \phi_{\omega_k}(X_k(\omega))$ for $k\ge0$. Then $\xi_{k+1}$ is $\mathscr{F}_{k+1}$-measurable, and since \(\omega_k\) is
independent of \(\mathscr F_k\) and has law $p$, we have  
\[
   \mathbb E[\xi_{k+1} \mid\mathscr F_k]=\int  \phi_t(X_k)\, dp(t) = \bar \phi(X_k).
\]
We now check the domination hypothesis in Theorem~\ref{thm:BQ-A6}.  Put $\xi(\omega)\eqdef \|\phi_{\omega_0}\|_\infty$.  Then \(\xi\in L^1(\mathbb{P})\), and
$ |\xi_{k+1}(\omega)|   =
   |\phi_{\omega_k}(X_k(\omega))|
   \le
   \|\phi_{\omega_k}\|_\infty=\xi(\sigma^k\omega)$.
Since \(\omega_k\) is independent of \(\mathscr F_k\), and using that the shift map $\sigma$ is $\mathbb{P}$-invariant, we get that
\[
   \mathbb P(|\xi_{k+1}|>u\mid\mathscr F_k)
   \le
   \mathbb P (\xi >u) \quad  \text{for every \(u\ge0\)}.
\]
Therefore, the hypotheses of Theorem~\ref{thm:BQ-A6} are satisfied, and we obtain
\[
   \frac1n\sum_{k=0}^{n-1}
   \Big(
      \xi_{k+1}
      -
      \mathbb E(\xi_{k+1}\mid\mathscr F_k)
   \Big)
   \longrightarrow 0
   \quad\text{for \(\mathbb P\)-a.e.~\(\omega\).}
\]
Using the expressions above for \(\xi_{k+1}\) and its conditional expectation, this is
\[
   \frac1n\sum_{k=0}^{n-1}
   \big(
      \phi_{\omega_k}(X_k)-\bar\phi(X_k)
   \big)
   \longrightarrow 0.
\]
Finally, since \(X_k(\omega)=f^k_\omega(x)\) and
$\phi(F^k(\omega,x))  =
   \phi(\sigma^k\omega,f^k_\omega(x))
   =
   \phi_{\omega_k}(f^k_\omega(x))$,
the last convergence is exactly the desired statement.
\end{proof}

\section{Law of Large Numbers}  We consider the convergence of probability measures in the weak$^*$ topology\index{stationary measures!weak$^*$ topology}. That is, $\mu_n \to \mu$ if and only if $\int \varphi \,d\mu_n \to \int \varphi \, d\mu$ for every $\varphi \in C_b(X)$. For \(x\in X\) and \(\omega\in\Omega\), denote by $\Pi(\omega,x)$ the set of accumulation points in the weak$^*$-topology  on the space of probability measures on $X$ of the empirical measures along the random orbit $f^k_\omega(x)$, $k\geq 0$. That is,  
\[
   \Pi(\omega,x)
   \eqdef
   \operatorname{Acc}_{w^*}
   \big\{
      e_n(\omega,x): n\ge 1
   \big\} \quad \text{where} \ \ \  e_n(\omega,x)\eqdef \frac1n\sum_{k=0}^{n-1}\delta_{f^k_\omega(x)}.
   \index{empirical measures!\(\Pi(\omega,x)\), accumulation set}
\]
We say that the sequence of empirical measures $\{e_n(\omega,x)\}_{n\geq 1}$ is \emph{tight} \index{empirical measures!tightness} if for every \(\varepsilon>0\), there exists a compact set
\(K_\varepsilon\subset X\) such that
\[
    e_n(X\setminus K_\varepsilon)=\frac{1}{n}\sharp \{0\leq k \leq n-1: \ f^k_\omega(x) \not\in K_\varepsilon\}<\varepsilon \quad \text{for every $n\geq 1$}.
\] 
By Prokhorov's theorem, this is equivalent to relative compactness of the sequence \(\{e_n(\omega,x)\}_{n\ge1}\) in the weak$^*$ topology (on the space of probability measures), or equivalently to the property that every subsequence admits a weak$^*$ convergent further subsequence to a probability measure. In particular, under tightness the set of weak$^*$ accumulation points $\Pi(\omega,x)$ is non-empty
and compact. If \(X\) is compact metric, tightness is automatic.

\begin{lem}
\label{lem:limsup-max}
Let \(f\) be a random map on a Polish space $X$ and consider \(\phi\in C_b(X)\). Then, for every $(\omega,x)\in \Omega\times X$ such that $\{e_n(\omega,x)\}_{n\geq 1}$ is tight, it holds 
\[
   \limsup_{n\to\infty}
   \frac1n\sum_{k=0}^{n-1}\phi(f^k_\omega(x))
   =
   \max_{\mu\in\Pi(\omega,x)}\int\phi\,d\mu
\]
and
\[
   \liminf_{n\to\infty}
   \frac1n\sum_{k=0}^{n-1}\phi(f^k_\omega(x))
   =
   \min_{\mu\in\Pi(\omega,x)}\int\phi\,d\mu.
\]
\end{lem}

\begin{proof}   
Set $e_n=e_n(\omega,x)$. Notice that
$$
   \frac1n\sum_{k=0}^{n-1}\phi(f^k_\omega(x))=\int\phi\,de_n. 
$$
Since \(\mu\mapsto\int\phi\,d\mu\) is continuous and the sequence $\{e_n\}_{n\geq 1}$ is tight, the set of accumulation values of the real sequence \(\int\phi\,de_n\) is exactly $\{\int\phi\,d\mu:\mu\in\Pi(\omega,x)\}$. Moreover, since the set  $\Pi(\omega,x)$ is a non-empty compact,  the asserted identities of limsup and liminf follow.  
\end{proof}

A probability measure \(\mu\) on \(X\) is called \emph{(ergodic) \(f\)-stationary} if the product measure \(\mathbb P\times\mu\) is (ergodic) $F$-invariant.  Equivalently, $\mu$ is $f$-stationary if and only $P^*\mu=\mu$. We denote by \(\mathcal I\) the set of \(f\)-stationary probability
measures and by \(\mathcal I_{\rm erg}\) the set of its ergodic elements. 

\begin{defi} \index{random maps!Feller continuous random map}
\label{def:feller-random-map}
A random map \(f:\Omega \times X \to X\) is said to be \emph{Feller continuous} if the associated annealed Koopman operator $P$ satisfies  $P(C_b(X))\subset C_b(X)$. 
\end{defi}
Feller continuity is equivalent to the transition probabilities \(P(x,\cdot)\) varying continuously
with \(x\) in the weak$^*$ topology.  In particular, \(f\) is a Feller
continuous random map whenever \(f_\omega:X\to X\) is continuous for \(\mathbb P\)-a.e.~\(\omega\).

When $f$ is Feller continuous, $\mathcal{I}$ is a closed convex subset of the space of probability measures and its extreme points are precisely the ergodic \(f\)-stationary measures, i.e., $\mathcal{I}_{\rm erg}$. We emphasize that  $\mathcal I$  may be empty unless one assumes tightness of an empirical measure and need not be compact, as it automatically follows when $X$ is compact. 
The following result is the analogue of the Krylov-Bogolyubov theorem for random maps. In the case that $X$ is compact, this result can be found in~{\cite[Cor.~3.4]{benoist2016random}}. Here we adapt the proof for random maps on Polish spaces. 

\index{Krylov-Bogolyubov random theorem}
\begin{prop}
\label{cor:empirical-limits-stationary}
Let $f$ be a Feller continuous random map on a Polish space $X$.   Then, for every \(x\in X\), it holds that 
$\Pi(\omega,x)\subset \mathcal I$ for \(\mathbb P\)-a.e.
\(\omega\in \Omega\).
\end{prop}

\begin{proof} Since $X$ is a Polish space, we have a metric $d$ compatible with the topology such that $(X, d)$ is a complete separable metric space. In such a case, the convergence in the weak$^*$ topology on the space of probability measures on $X$ is
countably determined by a set $S$ of bounded Lipschitz (with respect to the metric $d$) functions, c.f.~\cite[Prop.~5.1]{BNNT22}.  By
Proposition~\ref{lem:breiman-feller}, after intersecting countably many full-measure
sets, we may assume that,  for every $x\in X$, for $\mathbb{P}$-a.e.~$\omega\in \Omega$ it holds
\[
   \frac1n\sum_{k=0}^{n-1}
   \big(\varphi(f^k_\omega(x))-P\varphi(f^k_\omega(x))\big)	\to 0 \quad \text{for every $\varphi \in S$}.
\]
If $\Pi(\omega,x)$ is empty, we have nothing to prove. Otherwise, let \(\mu\in\Pi(\omega,x)\).  Taking a subsequence
\(n_\ell\to\infty\) such that
\[
   \frac1{n_\ell}\sum_{k=0}^{n_\ell-1}\delta_{f^k_\omega(x)}
   \longrightarrow \mu
   \qquad\text{in the weak* topology},
\]
and using that \(P\varphi\in C_b(X)\) for every $\varphi \in S$, we obtain
\[
   \int \varphi\,d\mu=\int P\varphi\,d\mu = \int \varphi \, dP^*\mu,
   \qquad \text{$\varphi \in S$}.
\]
Therefore \(P^*\mu=\mu\), i.e., \(\mu\in\mathcal I\).
\end{proof}



The following result is due to Furstenberg-Kifer~\cite[Theorem~1.1]{furstenberg1983random} when $X$ is compact. This result implies the classical Breiman strong law of large numbers when $f$ is uniquely ergodic. We follow the approach of the proof provided by Benoist-Quint~\cite[Corollary~3.6]{benoist2016random}. 

\index{strong law of large numbers}
\begin{cor}[Furstenberg-Kifer]
\label{thm:Furstenber-kifer}
Let \(f\) be a Feller continuous random map on a Polish space $X$ and consider \(\phi\in C_b(X)\).  Then, for every \(x\in X\), for \(\mathbb P\)-a.e.~\(\omega\in \Omega\) such that $\{e_n(\omega,x)\}_{n\geq 1}$ is tight, it holds
\begin{align*}
\min_{\mu\in\mathcal I_{\rm erg}}\int\phi\,d\mu
&\le
\liminf_{n\to\infty}
\frac1n\sum_{k=0}^{n-1}\phi(f^k_\omega(x))  \\
&\le
\limsup_{n\to\infty}
\frac1n\sum_{k=0}^{n-1}\phi(f^k_\omega(x))
\le
\max_{\mu\in\mathcal I_{\rm erg}}\int\phi\,d\mu .
\end{align*}
\end{cor}

\begin{proof}
Since Proposition~\ref{cor:empirical-limits-stationary} implies the inclusion \(\Pi(\omega,x)\subset\mathcal I\),  the asserted bounds for limsup and liminf follow from Lemma~\ref{lem:limsup-max} and from the fact that a continuous affine
functional on the compact convex set \(\mathcal I\) attains its maximum and minimum at ergodic stationary measures. 
\end{proof}

The next consequence is an extension of a result of Furstenberg~\cite[Lemma~7.3]{Fur:63}, which also appears in~\cite[Theorem~3.9]{benoist2016random} for one-step observables with a unique average. 

\begin{cor}[Furstenberg] \label{thm:Bierman-Furstenberg} 
Let $f$  be a Feller continuous random map on a Polish space $X$ and consider a measurable one-step map $\phi: \Omega \times X \to \mathbb{R}$ such that 
\[
\text{\(\phi(\omega,\cdot)\in C_b(X)\) \  \ for  \(\mathbb P\)-a.e.~\(\omega\)} \quad \text{and}
\quad 
    \int \| \phi(\omega,\cdot)\|_{\infty} \, d\mathbb{P}<\infty. 
\]
Define $\bar\phi(x) = \int \phi(\omega,x)\,d\mathbb{P}$. Then, for every~$x \in X$, for $\mathbb{P}$-a.e.~$\omega\in \Omega$,
\begin{align*} 
    \min_{\mu \in \mathcal{I}_{\rm erg}}  \int \bar\phi\, d\mu &\leq \liminf_{n\to\infty} \frac{1}{n}\sum_{i=0}^{n-1}\phi(F^i(\omega,x))   \\ &\leq \limsup_{n\to\infty} \frac{1}{n}\sum_{i=0}^{n-1}\phi(F^i(\omega,x)) \leq \max_{\mu \in \mathcal{I}_{\rm erg}}\int \bar\phi\, d\mu.  
\end{align*}
\end{cor}
\begin{proof} According to Proposition~\ref{prop:one-step-observables}, we have that for any $x\in X$, 
$$\lim_{n\to\infty} \frac{1}{n}\sum_{i=0}^{n-1} \big(\phi(F^i(\omega,x))-\bar\phi(f_\omega^i(x))\big) =0  \quad \text{for $\mathbb{P}$-a.e.~$\omega \in \Omega$.}
$$ 
This implies that for every $x\in X$ and $\mathbb{P}$-a.e.~$\omega \in \Omega$, 
\begin{align*}
\liminf_{n\to\infty} \frac{1}{n}\sum_{i=0}^{n-1} \phi(F^i(\omega,x)) &= \liminf_{n\to\infty} \frac{1} {n}\sum_{i=0}^{n-1} \phi_1(f^i_\omega(x))  \quad \text{and} \\ \limsup_{n\to\infty} \frac{1}{n}\sum_{i=0}^{n-1} \phi(F^i(\omega,x)) &= \limsup_{n\to\infty} \frac{1}{n}\sum_{i=0}^{n-1} \phi_1(f^i_\omega(x)).
\end{align*}
Then, since $\bar\phi\in C_b(X)$,  applying Corollary~\ref{thm:Furstenber-kifer}, we conclude the required inequalities and complete the proof.  
\end{proof}

\section{Occupation bound and zero-one law} The following result is a random analogue of the classical mean-value inequality for superharmonic functions. That is,  superharmonic observables dominate their occupation averages.

\index{superharmonic occupation bound}
\begin{prop}[{\cite[Lemma~3.20]{Mal:17}}]
\label{lem:malicet-comparison}
Let $f$ be a random map on a Polish space $X$ and consider a measurable function \(u:\Omega\times X\to[0,\infty]\).  Assume
also that
\begin{enumerate}[label=(\roman*)]
   \item the map \(x\mapsto u(\omega,x)\) is lower
   semicontinuous for $\mathbb{P}$-a.e.~\(\omega\in\Omega\),
   \item \(u\circ F\le u\) on \(\Omega\times X\).
\end{enumerate}
Then, for every \(x\in X\),
\[
   u(\omega,x)
   \ge
   \sup_{\mu\in\Pi(\omega,x)}
   \int u \,d(\mathbb P\times \mu) \quad \text{for \(\mathbb P\)-a.e.~\(\omega\in \Omega\)}.
\]
\end{prop}

\begin{proof} We first assume that \(u\) is bounded. 
Fix \(x\in X\), and let \(\mathscr F_n\) be the \(\sigma\)-algebra generated by
\(\omega_0,\ldots,\omega_{n-1}\).  Put
\[
   \bar u(x) \eqdef \int u(\omega,x)\,d\mathbb P,
   \qquad
   M_n(\omega)\eqdef \mathbb E\big[u(\cdot,x)\mid\mathscr F_n\big](\omega).
\]
By Corollary~\ref{cor:levy-zero-one}, \(M_n(\omega)\to u(\omega,x)\) for
\(\mathbb P\)-a.e.~\(\omega\in \Omega\).  Since \(u\circ F^n\le u\), taking conditional
expectation with respect to \(\mathscr F_n\) gives
\[
   \bar u(f^n_\omega(x)) = \mathbb E\big[u( F^n(\cdot,x))\mid\mathscr F_n\big](\omega) \leq  M_n(\omega).
\]
Therefore, by Ces\`aro averaging,
\begin{equation}
\label{eq:malicet-cesaro-appendix}
   \limsup_{N\to\infty}
   \frac1N\sum_{n=0}^{N-1}\bar u(f^n_\omega(x)) \leq u(\omega,x) \quad \text{for \(\mathbb P\)-a.e.~\(\omega\in \Omega\).}
\end{equation}
By Fatou's lemma and the lower semicontinuity of \(x\mapsto u(\omega,x)\), the
function \(\bar u\) is lower semicontinuous.  Hence, by Lemma~\ref{lem:Baire-teorema}, \(\bar u\) is the pointwise
supremum of an increasing sequence of continuous functions \(\psi_m\le \bar u\). Let  $\omega$ be in the set of probability one for which~\eqref{eq:malicet-cesaro-appendix} holds. If $\Pi(\omega,x)$ is empty, we have nothing to prove. Otherwise, take \(\mu\in\Pi(\omega,x)\), and choose \(N_j\to\infty\) such that
\[ 
   \frac1{N_j}\sum_{n=0}^{N_j-1}\delta_{f^n_\omega(x)}
   \longrightarrow \mu.
\]
For every \(m\), using \(\psi_m\le\bar u\) and~\eqref{eq:malicet-cesaro-appendix},
\[
   \int \psi_m\,d\mu = \lim_{j\to\infty}
   \frac1{N_j}\sum_{n=0}^{N_j-1}\psi_m(f^n_\omega(x))  \leq u(\omega,x).
\]
Letting \(m\to\infty\) and using monotone convergence gives
\[
   \int u\,d(\mathbb P\times \mu)=\int \bar u\,d\mu = \lim_{m\to \infty} \int \psi_m \, d\mu \leq u(\omega,x).
\]
Since \(\mu\in\Pi(\omega,x)\) was arbitrary, the result follows when $u$ is bounded. 

For the general case, define $u_M\eqdef \min\{u,M\}$. 
Then \(u_M\) is bounded, non-negative, measurable, lower semicontinuous in
the \(X\)-variable, and satisfies
$u_M\circ F\le u_M$. 
Applying the bounded case to \(u_M\), we obtain, for every
\(\mu\in\Pi(\omega,x)\),
\[
   u_M(\omega,x)
   \ge
   \int u_M\,d(\mathbb P\times\mu).
\]
Letting \(M\to\infty\) and using monotone convergence gives
$   u(\omega,x)
   \ge
   \int  u\,d(\mathbb P\times\mu)$. 
Taking the supremum over \(\mu\in\Pi(\omega,x)\) gives the desired inequality.
\end{proof}

The preceding proposition becomes especially useful when applied to indicators of backward invariant random open sets $U(\omega)$, i.e., satisfying $1_{U(\omega)}\circ f_\omega(x)\leq 1_{U(\omega)}$. Hence, the superharmonic occupation
bound forces
$$
1_{U(\omega)}(x) \geq  \int \mu(U(\omega))\,d\mathbb P \quad \text{for every $\mu \in \Pi(\omega,x)$.}
$$
In particular, one has the following result:

\index{zero-one laws!zero-one law for random open sets}
\begin{cor}[Zero--one law for random open sets, {\cite[Prop.~4.2]{Mal:17}}]
\label{cor:malicet-zero-one-open-sets}
Let $f$ be a Feller continuous random map on a Polish space $X$ such that for every $x\in X$ the set  $\Pi(\omega,x)$ is non-empty for $\mathbb{P}$-a.e.~$\omega\in \Omega$.  Let
\[
   E=\bigcup_{\omega\in\Omega}\{\omega\}\times U(\omega)
   \subset \Omega\times X
\]
be a measurable set such that the fiber \(U(\omega)\) is open in \(X\) for $\mathbb{P}$-a.e.~$\omega\in \Omega$.  Assume
that \(E\) is backward invariant, that is, $F^{-1}(E)\subset E$.
If
\[
   (\mathbb P\times\mu)(E)>0
   \qquad\text{for every }\mu\in\mathcal I_{\rm erg},
\]
then
\[
   (\mathbb P\times\nu)(E)=1
   \qquad\text{for every probability measure }\nu\text{ on }X.
\]
\end{cor}

\begin{proof}
By ergodic decomposition, the assumption also implies
\((\mathbb P\times\mu)(E)>0\) for every \(\mu\in\mathcal I\).  Since almost surely the fibers
\(U(\omega)\) are open sets, the function \(1_E(\omega,\cdot)\) is lower semicontinuous for $\mathbb{P}$-a.e.~$\omega\in \Omega$.
The backward invariance \(F^{-1}(E)\subset E\) is equivalent to
\(1_E\circ F\le 1_E\). Hence, fixing \(x\in X\), applying Proposition~\ref{lem:malicet-comparison} to
\(u=1_E\), and using that $\Pi(\omega,x)\not=\emptyset$ for $\mathbb{P}$-a.e.~$\omega\in\Omega$ and  Proposition~\ref{cor:empirical-limits-stationary},  we get that
\[
   1_E(\omega,x)
   \ge
   \sup_{\mu\in\Pi(\omega,x)}(\mathbb P\times\mu)(E)>0 \quad  \text{for \(\mathbb P\)-a.e.~\(\omega\in \Omega\)}.
\]
Hence \(1_E(\omega,x)=1\) for \(\mathbb P\)-a.e.~\(\omega\).  Integrating this
identity with respect to an arbitrary probability measure \(\nu\) on \(X\), we obtain
\((\mathbb P\times\nu)(E)=1\).
\end{proof}

\bibliographystyle{alpha3}
\bibliography{bibliography_cleaned}


\end{document}